%% file: apex_main.tex
\documentclass{article}
\usepackage{graphicx} 
\usepackage{fullpage}
\usepackage[utf8]{inputenc}
\usepackage{amsmath,amsfonts,amsthm,amssymb}
\usepackage{graphics}
\usepackage{url}
\usepackage{hyperref}
\usepackage{cleveref}
\usepackage{color}
\usepackage{enumitem}
\usepackage{algorithm}
\usepackage{algpseudocode}
\usepackage{lineno}
\usepackage{subfiles}
\usepackage{mathrsfs}
\usepackage{caption}
\usepackage{mathabx}
\usepackage{subcaption}
\usepackage{xspace}
\usepackage{booktabs}
\usepackage{multicol}
\usepackage{multirow}
\usepackage[export]{adjustbox}

\usepackage{tikz}
\usepackage{tkz-graph}
\usetikzlibrary{shapes.geometric}
\usetikzlibrary{positioning}
\usetikzlibrary {decorations.markings}
\usetikzlibrary{arrows.meta}
\usetikzlibrary{calc}

\newtheorem{theorem}{Theorem}[section]
\crefname{theorem}{Theorem}{Theorems}
\newtheorem{lem}[theorem]{Lemma}
\crefname{lem}{Lemma}{Lemmas}

\crefname{proposition}{Proposition}{Propositions}
\newtheorem{claim}[theorem]{Claim}

\newtheorem{dfn}[theorem]{Definition}

\newtheorem{conjecture}{Conjecture}

\crefname{figure}{Figure}{Figures}

\newcommand\drop[1]{}

\def\showlabel#1{}

\algnewcommand{\Break}{\State \textbf{break} }
\algnewcommand{\Continue}{\State \textbf{continue} }
\algnewcommand{\True}{\textbf{true}\xspace}
\algnewcommand{\False}{\textbf{false}\xspace}

\counterwithin{algorithm}{subsection}

\input{commands}

\title{Three-edge-coloring apex cubic graphs}

\author{
Yuta Inoue\thanks{The University of Tokyo, Tokyo, Japan, \texttt{yutainoue@is.s.u-tokyo.ac.jp}.  Supported by JSPS Kakenhi 26K21777 and JP25K24465, by JST ASPIRE JPMJAP2302 and by Hirose Foundation Scholarship.}
\and
Ken-ichi Kawarabayashi\thanks{National Institute of Informatics \& The University of Tokyo, Tokyo, Japan, \texttt{k\_keniti@nii.ac.jp}.  Supported by JSPS Kakenhi 26K21777 and JP25K24465 and by JST ASPIRE JPMJAP2302.}
\and
Ritarou Matsuo\thanks{The University of Tokyo, Tokyo, Japan, \texttt{rin2004@g.ecc.u-tokyo.ac.jp} Supported by JSPS Kakenhi 26K21777 and JP25K24465 and by JST ASPIRE JPMJAP2302.}
\and
Atsuyuki Miyashita\thanks{The University of Tokyo, Tokyo, Japan, \texttt{miyashita-atsuyuki869@is.s.u-tokyo.ac.jp}. Supported by JSPS Kakenhi 26K21777 and JP25K24465 and by JST ASPIRE JPMJAP2302.}
\and
Bojan Mohar\thanks{Simon Fraser University, Burnaby, BC, Canada \& FMF, University of Ljubljana, Slovenia, \texttt{mohar@sfu.ca}. Supported in part by NSERC Discovery Grant R832714 (Canada), by the ERC Synergy grant (European Union, ERC, KARST, project number 101071836), and by the Research Core Grant P1-0297 of ARIS (Slovenia).}
\and
Tomohiro Sonobe\thanks{National Institute of Informatics \& The University of Tokyo, Tokyo, Japan, \texttt{tomohiro\_sonobe@nii.ac.jp}.  Supported by JSPS Kakenhi 26K21777 and JP25K24465 and by JST ASPIRE JPMJAP2302.}
 }
\date{\today}

\begin{document}
\maketitle
\begin{abstract}
A graph $G$ is \emph{apex} if $G$ has a vertex $v$ such that $G-v$ is planar. 
We prove that every $2$-connected apex cubic graph is three-edge-colorable. This result gives the final piece of the proof for the well-known Tutte's three-edge-coloring conjecture from 1966 \cite{tutte}.
The proof, as well as the result, generalizes that of the Four Color Theorem, which requires computer checks. 

As in \cite{inoue2026four, RSST}, the proof is constructive. More precisely, given a $2$-connected apex cubic graph $G$ on $n$ vertices, our reducibility and discharging procedure yields a three-edge-coloring of $G$ in $O(n^2)$ time.

As an additional reproducibility check for our computer checks, independent implementations reconstructed from the detailed pseudocode (given in the appendix, as in \cite{inoue2026four}) using generative-AI systems reproduced the required computational results. These reconstructions are not part of the mathematical justification of the theorem, but provide additional evidence for the reproducibility of the computations.
\end{abstract}

\subfile{apex_intro}

\subfile{apex_proofs}

\bibliographystyle{alpha}
\bibliography{4CT}

\appendix

\subfile{apex_connectivity}

\section{Pseudocode}

\subfile{pseudo/reducibility_check}

\subfile{pseudo/free_homomorphism}

\subfile{pseudo/cartwheel_algorithm}

\subfile{pseudo/conf_hom}

\end{document}

%% file: commands.tex
\newcommand{\nil}{\textit{nil}\xspace}
\newcommand{\head}{\textit{head}\xspace}
\newcommand\tail{\textit{tail}\xspace}
\newcommand\reverse{\textit{rev}\xspace}
\newcommand\successor{\textit{succ}\xspace}
\newcommand\predecessor{\textit{pred}\xspace}

\newcommand{\deltaout}{\delta^{\textsf{out}}\xspace}

\newcommand{\efirst}{e_\textsf{first}\xspace}
\newcommand{\elast}{e_\textsf{last}\xspace}

\newcommand{\neverApply}{\textsf{neverApply}\xspace}

\newcommand{\CARTWHEELDEGREES}{\texttt{CARTWHEEL\_DEGREES}\xspace}

\newcommand{\allHomImages}{\textsf{allHomImages}\xspace}
\newcommand{\makeOuterExtension}{\textsf{makeOuterExtension}\xspace}
\newcommand{\isPlanar}{\textsf{isPlanar}\xspace}

\newcommand{\hasSeparatingCycle}{\textsf{hasSeparatingCycle}\xspace}
\newcommand{\getWalks}{\textsf{getWalks}\xspace}

\newcommand{\addBoundaryDartsDirectly}{\textsf{addBoundaryDartsDirectly}\xspace}
\newcommand{\linkIncidenceListEnds}{\textsf{linkIncidenceListEnds}\xspace}

\newcommand{\id}[1]{\operatorname{id}_{#1}\xspace}

\newcommand{\eufirst}{e_{u,\textsf{first}}\xspace}
\newcommand{\ewlast}{e_{w,\textsf{last}}\xspace}

\newcommand{\cRstarnoK}{\mathcal{R}^{*-\mathcal{K}}\xspace}
\newcommand{\cRstar}{\mathcal{R}^{*}\xspace}

\newcommand{\Ksmaller}{\mathcal{K}_\textsf{smaller}\xspace}

\newcommand{\Rauxiliary}{\mathcal{R}_\textsf{auxiliary}\xspace}

\newcommand{\Ca}{\mathcal{C}_\textsf{all}\xspace}

%% file: apex_intro.tex
\section{Introduction}
\label{sect:intro}
\showlabel{sect:intro}

The Four Color Theorem (4CT) states that every loopless planar graph is 4-colorable. 
It was first shown by Appel and Haken \cite{4ct1,4ct2} in 1977, and a simplified version of the proof was given in 1997 by Robertson et al.~\cite{RSST}. 
The same authors used the latter proof \cite{RSST-STOC} to obtain a quadratic-time algorithm for 4-coloring planar graphs. 

From the viewpoint of vertex-coloring, it seems extremely difficult to extend the 4CT to larger families of graphs from planar graphs.
Indeed, it has been known that for every surface with positive Euler genus there are infinitely many \emph{$5$-color-critical graphs}\footnote{A graph $G$ is said to be \emph{$k$-color-critical} if $G$ is not $(k-1)$-colorable, but every proper subgraph of $G$ is.} that can be embedded in the surface \cite{MT}. 
This means that while there are no 5-color-critical planar graphs, there are infinitely many 5-color-critical graphs on the projective plane.
Regarding vertex-coloring, there is a major gap between planar graphs and graphs of higher genus.

On the other hand, the 4CT has an equivalent formulation using planar duality:

\begin{theorem}\label{thm:3ECplanar}
   Every $2$-connected cubic planar graph is three-edge-colorable. 
\end{theorem}

While the vertex-coloring formulation of the Four Color Theorem has no extension beyond planar graphs, the edge-coloring formulation in Theorem \ref{thm:3ECplanar} has some chance to extend.
This equivalence is a consequence of Tutte's coloring-flow duality theorem. This duality uses the notion of nowhere-zero flows.
A \emph{nowhere-zero $4$-flow} on a graph $G$ is a function $\phi: E(G)\to \{(0,1),(1,0),(1,1)\}$, where the three values in the range of $\phi$ are viewed as the nonzero elements of the group $\mathbb{Z}_2\times\mathbb{Z}_2$, satisfying the \emph{flow condition}: $\sum_{e\in C}\phi(e)=(0,0)$ for every edge-cut $C\subseteq E(G)$. 
Nowhere-zero $k$-flows can be defined for any integer $k\ge2$ by using nonzero elements of any abelian group of order $k$. For a precise definition, we refer the reader to the textbook by Diestel \cite{Diestel_book}. 

\begin{theorem}[Tutte \cite{tutte}]\label{thm:coloring-flow duality}
   Let $k\ge2$ be an integer and $G$ be a plane graph. Then $G$ is $k$-colorable if and only if its dual graph $G^*$ admits a nowhere-zero $k$-flow. 
\end{theorem}

Note that a cubic graph is three-edge-colorable if and only if it admits a nowhere-zero 4-flow.
Thus, Theorem \ref{thm:3ECplanar} is equivalent to the 4CT.
Motivated by this formulation, Tutte proposed his famous \emph{$4$-Flow Conjecture}.

\begin{conjecture}[Tutte \cite{tutte}]\label{conj:Tutte4Flow}
   Every 2-connected graph without Petersen graph as a minor admits a nowhere-zero $4$-flow. 
\end{conjecture}

This conjecture remains open.
The special case of this conjecture restricted to cubic graphs is known as \emph{Tutte's three-edge-coloring conjecture} and it predicts the following.

\begin{conjecture}[Tutte \cite{tutte}]\label{conj:Tutte3edge}
   Every $2$-connected cubic graph without Petersen graph as a minor is three-edge-colorable.
\end{conjecture}

In 1997, Robertson, Seymour, and Thomas \cite{3edgecoloring} showed the following, with the results in \cite{excludedcubic,cyclically5}. 

\begin{theorem}\label{thm;equi}\showlabel{equi}
    Conjecture \ref{conj:Tutte3edge} follows if both of the following statements hold:
    \begin{itemize}
        \item Every 2-connected doublecross cubic graph is three-edge-colorable.
        \item Every 2-connected apex cubic graph is three-edge-colorable.
    \end{itemize}
\end{theorem}

A graph $G$ is \emph{doublecross} if it can be drawn in the plane with only two crossings, both on the same region.
A graph $G$ is \emph{apex} if $G$ has a vertex $v$ such that $G-v$ is planar.
While the first result was already published in \cite{doublecross} 10 years ago, the second part has never been published, although a proof was announced by Sanders and Thomas around 30 years ago.

The main theorem of this paper is to prove the apex case:
\begin{theorem}\label{mainth}\showlabel{mainth}
Every 2-connected apex cubic graph is three-edge-colorable.  
\end{theorem}

Together with Theorem \ref{thm;equi} and \cite{doublecross}, our result finally proves Tutte's three-edge-coloring conjecture. Our proof is based on several computer-assisted tasks. 
Source code and data used in this paper are available on GitHub\footnote{\url{https://github.com/three-edge-coloring-apex-cubic-graphs}}.

A graph is \emph{cyclically $k$-edge-connected} if, after deleting fewer than $k$ edges, there is at most one component that contains a cycle.
A minimal counterexample of Theorem \ref{mainth} can be reduced to cyclically 4-edge-connected graphs with girth of at least 5.
Cyclically 4-edge-connected cubic graphs with girth of at least 5 that are not three-edge-colorable are called \emph{snarks}. 
The most famous snark is the Petersen graph, which is the first snark discovered in \cite{petersen-1898}.
Theorem \ref{mainth} says, in particular, that no snark is apex.

\paragraph{Algorithmic consequence:}  As in the proof of \cite{inoue2026four, RSST}, the proof is constructive. More precisely, given a two-edge-connected apex cubic graph $G$ on $n$ vertices, our reducibility and discharging procedure yield a three-edge-coloring of $G$ in $O(n^2)$ time. The finite reducibility and unavoidability checks are precomputed once and used as constant-size lookup tables. More details will be given in Section \ref{sec:alg}. 

\paragraph{Detailed pseudocode and reproducibility:} 
In the appendix, following the approach in \cite{inoue2026four}, we provide detailed pseudocode specifying the computer-assisted parts of the proof.
Each pseudocode is linked to the corresponding function in the source code in GitHub, allowing the reader to verify that it has been implemented correctly.
The specification is sufficiently precise that independent implementations were reconstructed from it and reproduced the authors’ computational results, as described below.

\paragraph{The Use of Generative AI:}
GitHub Copilot was used to assist in implementing parts of the authors' verification software.
All generated code was reviewed by the authors, who take responsibility for the correctness of the implementation and its correspondence with the pseudocode in the appendices.
ChatGPT Pro and Gemini Deep Think were also used to assist with exposition and to review the consistency between the mathematical descriptions and the pseudocode.

As an additional reproducibility check, independent implementations were reconstructed by ChatGPT Pro and Claude Code, directly from the pseudocode (together with the referenced pseudocode in \cite{inoue2026four}) and the public input data, without access to the authors' source code.
Both independently reproduced the computational results required in this paper, including the verifications of Lemmas \ref{comp:lem:combined_rules}, \ref{comp:lem:deg7-11}, and \ref{comp:lem:reducible}.
The instructions for checking reproducibility and the link to the generated source code are available on GitHub\footnote{\url{https://github.com/three-edge-coloring-apex-cubic-graphs/instructions-for-checking-reproducibility}}.

\subsection{Sketch of the proof}
\label{subsect:sketch}
\showlabel{subsect:sketch}

The basic outline of our proof of Theorem \ref{mainth} is similar to the proof of the 4CT \cite{RSST} (or a recent generalization \cite{inoue2026four}) and its extension to the doublecross case in \cite{doublecross}.

First, suppose that a minimal counterexample of \cref{mainth}, say $G_{mc}$, exists.
$G_{mc}$ is apex, so it cannot be embedded in the plane.
To focus on the planar structure of $G_{mc}$, we remove the apex vertex of $G_{mc}$ to obtain a subcubic planar graph $G$ (the precise operation to get $G$ from $G_{mc}$ is described in Section \ref{sect:minimal}).
Since $G_{mc}$ is not three-edge-colorable, $G$ is not three-edge-colorable either by the standard parity argument.
We also use $G^*$, which is the planar dual of $G$.
We aim to apply proof techniques similar to those used in \cite{RSST}, \cite{inoue2026four}, and \cite{doublecross}. 
The proof has two main parts:

\medskip
\noindent
{\bf (Reducibility).}
    The crucial ingredient of our proof is the notion of a \emph{reducible} subgraph\footnote{This subgraph is called \emph{a multi-boundary island} in Section \ref{subsect:island}.}.
    A reducible subgraph has the following property: if this subgraph is contained in $G$, then we can replace it in $G$ with a smaller subgraph without changing the non-three-edge-colorability of $G$ (and with keeping the property that the graph was obtained from an apex graph by removing the apex vertex).
    By this property, if a reducible subgraph is contained in the graph $G$, then a smaller non-three-edge-colorable apex graph exists.
    This contradicts that $G$ is obtained from a minimal counterexample $G_{mc}$.
    Thus, a reducible subgraph cannot be contained in $G$.

    When checking the reducibility of a subgraph $I$ of $G$, we show that, for any possible edge-coloring of $\delta(V(I))$\footnote{For a set of vertices $Y$, $\delta(Y)$ denotes the set of edges with one endpoint in $Y$ and the other endpoint not in $Y$.}, by changing this coloring using \emph{Kempe chains} outside $I$, we necessarily find a coloring of $\delta(V(I))$ that can be extended to $I$.
    The precise definition is in Definitions \ref{dfn:D-reducible} and \ref{dfn:C-reducible}.
    
\medskip
\noindent
{\bf (Unavoidability).}
    We aim to construct a set of reducible subgraphs such that $G$ must contain at least one of the reducible subgraphs in this set.
    This property is called \emph{unavoidability}.
    If we succeed in constructing an unavoidable set of reducible subgraphs, this would contradict the minimality of $G$ (and thus $G_{mc}$), thereby proving Theorem \ref{mainth}.

    The standard technique to show unavoidability is the \emph{discharging method}.
    This method has traditionally been used on triangulations, which is $G^*$ in our case, and we adhere to this established approach for consistency and clarity.
    In the discharging method, every vertex in $G^*$ has \emph{initial charge} depending on its degree, and the initial charges are assigned so that the total charge is positive.
    The charge is redistributed by the predetermined \emph{discharging rules}.
    If the final charge of some vertex is positive after redistribution, a desired structure exists in the vicinity of $v$ in $G^*$.
    This desired structure in $G^*$ leads to finding a reducible subgraph in $G$.
\medskip

The main difference from the previous work is that since $G$ has two\footnote{In this paper, we prove both cases where the number of vertices of degree two is two or three; indeed, when there are exactly two vertices of degree two, the result follows from the face-width one case in the torus \cite{torus2024}.} or three vertices of degree two, it is not a cubic graph.
In the dual view, $G^*$ has two or three faces of size two, which are called \emph{digons}, so it is not a triangulation.
We call such a plane graph with all faces triangles or digons \emph{a triangulation with digons}.
This differs significantly from \cite{RSST} and \cite{doublecross}, and it causes a lot of trouble, as we list below:
 
\begin{enumerate}
    \item Digons may mess up many structures in a triangulation. 
    \item $G^*$ may have short contractible cycles. This may mess up reducible configurations, as well as the local structure to determine the final charge.
    \item  We want to find a reducible configuration in $G^*$ that is not close to digons.    
\end{enumerate}

The first issue is that the embedding of ``configurations" in  $G^*$ is more flexible than the embedding into a minimal counterexample for 4CT.
A \emph{configuration} is, informally, a subgraph in which each vertex has a specified degree; see Figure \ref{fig:75555}. It is exactly the desired structure found by the discharging method.
The left figure represents a configuration and its embedding.
A vertex represented by a solid black circle has degree 5, and a vertex represented by a black-outlined open circle has degree 7.
If this configuration is contained in the plane graph $H$, where every face of $H$ is a triangle, then the embedding of the configuration in $H$ is uniquely determined (by Whitney's theorem), as shown in the part of the figure.
This is the case for 4CT or the doublecross case.
This makes it easy to take the dual graph of a configuration and find a reducible subgraph when a configuration is contained in a minimal counterexample of 4CT.
However, in our case, $G^*$ also has digons.
This makes the embedding flexible, as shown in Figure \ref{fig:75555-flexible}.
A double edge represents a pair of parallel edges bounding a digon.
An edge incident to the exterior of a configuration may actually be a pair of parallel edges bounding a digon, and the same can occur for an edge with both endpoints in the configuration.
Thus, even if a configuration is contained in $G^*$, the graph structure here is flexible.
This makes it difficult to find a reducible subgraph in $G$ after finding a desired configuration by applying the discharging method to $G^*$.
We handle these difficulties in Section \ref{sect:hom-conf}.

\begin{figure}[htbp]
    \centering
    \begin{minipage}[b]{0.47\columnwidth}
        \centering
        \includegraphics[width=0.8\columnwidth]{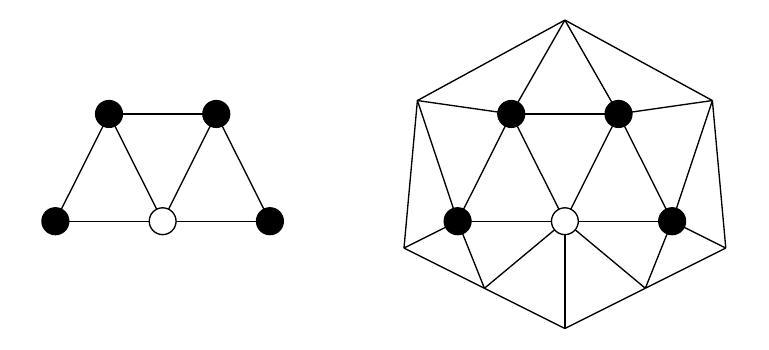}
        \caption{An example of a configuration and its completion where every face is a triangle.}
        \label{fig:75555}
    \end{minipage}
    \quad
    \begin{minipage}[b]{0.49\columnwidth}
        \centering
        \includegraphics[width=0.9\columnwidth]{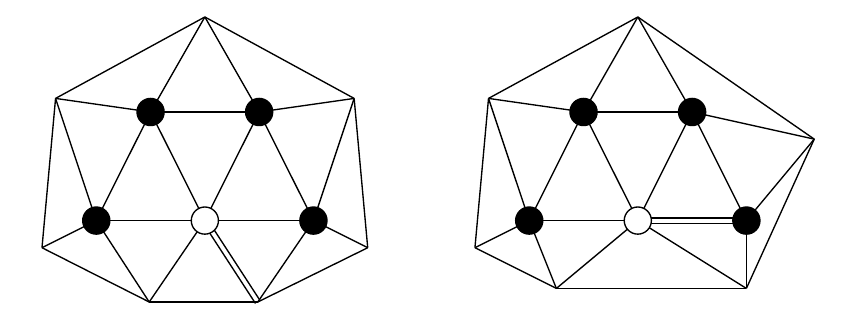}
        \caption{Examples of completions of the configuration in Figure \ref{fig:75555} containing digons.}
        \label{fig:75555-flexible}
    \end{minipage}
\end{figure}

The second issue is that $G$ is not highly connected compared to a minimal counterexample for 4CT and to the doublecross case.
A minimal counterexample for 4CT is an \emph{internally 6-connected triangulation}: for a cycle $C$ of length at most 5 that bounds two open disks, if $|C| \leq 4$, one of them contains no vertex, and if $|C|=5$, one of them contains at most one vertex.
On the other hand, for the apex case, we can only show the following for the triangulation with digons $G^*$: for every cycle $C$ of length five or less in the dual graph $G^*$ that  bounds two open disks, the following holds: 
\begin{itemize}
    \item If $|C| \leq 3$, then one of the two disks contains no vertices.
    \item If $|C| = 4$, then one of them contains at most two vertices.
    \item If $|C| = 5$, then either one of them contains at most one vertex, or each of them contains at least one digon.
\end{itemize}
See the precise statement in Lemma \ref{lem:triangulation}.
This weaker connectivity condition causes two troublesome differences.

First, the embedding of a configuration in $G^*$ may not be induced or may be self-intersecting.
As in \cite{RSST} or \cite{doublecross}, we only use a configuration whose diameter is at most 4.
Internal 6-connectivity easily ensures that the embedding of a configuration is induced.
However, in our case, $G^*$ may have a cycle $C$ of length 5, where both disks separated by $C$ have arbitrarily many vertices, so the embedding of a configuration may not be induced.
Also, $G^*$ may have a cycle $C$ of length 4, where each of the two disks separated by $C$ has at least two vertices, so the embedding of a configuration may be self-intersecting.
This issue is handled by a configuration with ``more than one" boundary, which we call a \emph{multi-boundary configuration} (the precise description is in Definition \ref{dfn:conf}).
If the embedding of a configuration is self-intersecting or not induced, it is considered the embedding of a multi-boundary configuration.
In Figure \ref{fig:multi-boundary-configuration}, the graph on the left represents a configuration, where plain vertices have degree 6.
The center figure describes a multi-boundary configuration obtained if the left configuration is not induced.
We prove that we can find a reducible subgraph in $G$ also when such a multi-boundary configuration is contained in $G^*$.

\begin{figure}[htbp]
    \centering
    \includegraphics[width=0.9\linewidth]{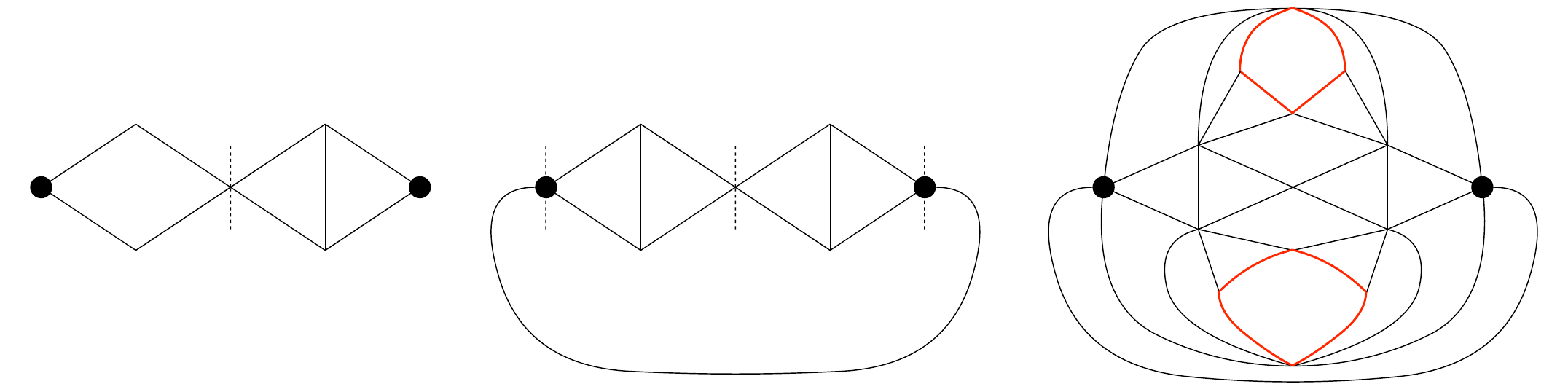}
    \caption{The left figure: a configuration (dotted lines represent edges between the configuration and its exterior); the center figure: a multi-boundary configuration obtained if the left configuration is not induced; the right figure: the completion of a multi-boundary configuration if every face is a triangle.}
    \label{fig:multi-boundary-configuration}
\end{figure}

Second, the local structure of a vertex $v$ in $G^*$ that determines the final charge of $v$
need not be ``well-behaved".
Let $B_2(v)$ be the subgraph induced by vertices within distance at most 2 from $v$.
The final charge of $v$ is determined by the degrees and the structure of faces in $B_2(v)$ (this is explicitly mentioned in \cite{inoue2026four}).
Internal 6-connectivity ensures that $B_2(v)$ is ``well-behaved", i.e., $B_2(v)$ consists of $v$ and two cycles $C_1, C_2$ such that $C_1$ is obtained from the first neighbors of $v$, and $C_2$ is obtained from the second neighbors of $v$. 
However, in the apex case, $B_2(v)$ in $G^*$ does not behave this way.
In fact, the faces in $B_2(v)$ may contain digons or triangles, so there are many more face structures compared to those in a minimal counterexample to 4CT.
These issues prevent us from using the same computer programs to check unavoidability with the discharging algorithm as in \cite{RSST} and \cite{doublecross}.
We handle this issue in Section \ref{sect:cartwheel}.

\medskip

The third issue is to avoid using too many configurations with digons.
Using the discharging method, we can find a desired configuration in the vicinity of a vertex with a positive final charge.
Although $G^*$ has two or three digons, these digons may be very close to a vertex with a positive final charge.
Considering all degree combinations and the structure of digons, we would need many configurations with digons in order to prove unavoidability. 
To this end, we relax the condition to find a desired configuration around digons.
Roughly, for a vertex close to a digon, we do not need to find a desired configuration even if its final charge is positive.
Instead, if the final charge of such a vertex exceeds a positive constant, we ensure that a desired configuration exists.
For this method to work, the sum of the values should be less than the total initial charge.
This is possible because the number of digons is either 2 or 3.
The more detailed discussion is in Section \ref{sect:discharging}.

\medskip

There are other differences from the proof of 4CT.
For a cubic graph with a three-edge-coloring, the three-edge-coloring induced on any edge-cut $\delta(Y)$ satisfies the parity condition, i.e., the parity of the numbers of edges colored 1, 2, or 3 is the same.
However, $G$ has vertices of degree 2, so the parity condition no longer holds.
This makes a difference in checking reducibility.
When checking reducibility of a subgraph $I$ of $G$, we need to show that, for all possible colorings for $\delta(V(I))$, by recoloring it using \emph{Kempe chains}, we find a coloring of $\delta(V(I))$ that is extendable to $I$.
This is time-consuming because of the loss of the parity condition.
Another difference regarding reducibility is that Kempe chains behave differently from 4CT because of vertices of degree 2.
This is described in Section \ref{subsect:reducibility}.

\paragraph{Organization.}
The organization of our proof of Theorem \ref{mainth} is as follows.
First, we discuss the connectivity of $G_{mc}$ (and $G, G^*$) in Section \ref{sect:minimal} and Section \ref{sec:con}.

Second, using the discharging method, we show that some desired configuration is contained in $G^*$ in Theorem \ref{thm:K-homomorphism-G} in Section \ref{sect:discharging}.
To describe the discharging method, we define multi-boundary configurations in Section \ref{sect:conf} and describe the dart representations in Section \ref{sect:dart}, which are the embedding model used in the discharging method.

Third, we show that this desired configuration in $G^*$ implies the existence of a reducible subgraph in $G$ in Theorem \ref{thm:hom-imply-reducible} in Section \ref{sect:hom-conf}.
We discuss reducibility and multi-boundary islands in Section \ref{sect:island}.

The second and third steps contradict the minimality of $G_{mc}$, which completes the proof of Theorem \ref{mainth}.

%% file: apex_proofs.tex
\section{A minimal counterexample}
\label{sect:minimal}
\showlabel{sect:minimal}
Let $G_\text{mc}$ be a 2-connected apex cubic graph of smallest order that is not three-edge-colorable.
In other words, $G_{\text{mc}}$ is a \emph{minimal counterexample} for \cref{mainth}, and deriving a contradiction to the minimality of $G_{\text{mc}}$ implies \cref{mainth}.

In order to prove Theorem \ref{mainth}, we remove the nonplanar part of $G_\text{mc}$ and focus on the planar structure.
More precisely, we define a planar graph $G$ as follows: 
Let $v$ be the apex vertex of $G_\text{mc}$.
If, for some edge $e$ incident to $v$, its removal makes the graph planar, let $G = G_\text{mc} - e$; otherwise, let $G = G_\text{mc} - v$.

Embed $G$ in the plane, and let $G^*$ be the dual of $G$.
Since every vertex in $G$ has degree three, except for two or three vertices of degree two, $G^*$ is a plane graph in which every face is a triangle, except for two or three \emph{digons}.
The definition of $G$ ensures that no face in $G$ is incident to more than one vertex of degree two; dually, this implies that no vertex in $G^*$ is incident to more than one digon.
Finally, a standard parity argument allows us to extend any three-edge-coloring of $G$ to a three-edge-coloring of $G_{\text{mc}}$.

A cubic graph $H$ is said to be \emph{theta-connected} if it has girth at least five and $|\delta_H(X)| \geq 6$ for every $X \subseteq V(H)$ such that $|X|, |V(H) \setminus X| \geq 6$.
In \cite{3edgecoloring}, it is proved that every minimal counterexample to Tutte's three-edge-coloring conjecture that is not apex is theta-connected.
This lemma is useful for the doublecross case in \cite{doublecross}, but not useful for the apex case.
Here, we prove that $G_{mc}$ is also highly connected as in the following lemma, but unfortunately the connectivity of $G_{mc}$ is weaker than theta-connectivity.
A \emph{cyclic cut} of $G_{mc}$ is an edge-set $F$ such that $G_{mc}-F$ has two connected components, each containing a cycle.

\begin{lem}
\label{lem:Gmc-connected}
\showlabel{lem:Gmc-connected}
    Let $F$ be a cyclic edge-cut in $G_{mc}$. Then $|F| \geq 5$. Moreover, if $|F|=5$, then one of the components of $G_{mc} - F$ is a 5-cycle or domino graph (see Figure \ref{fig:domino-first}) whose vertices do not contain the apex vertex of $G_{mc}$.
\end{lem}

The proof of Lemma \ref{lem:Gmc-connected} is given in the Appendix, see Section \ref{sec:con}.
The domino graph violates the theta-connectivity.

\begin{figure}[htbp]
    \centering
    \includegraphics[width=0.16\linewidth]{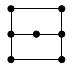}
    \caption{The domino graph}
    \label{fig:domino-first}
\end{figure}

When considering planar graphs, we consider their embeddings in the plane as embeddings in the 2-dimensional sphere. In this way, the outer face is also viewed as a topological disk. In particular, having a cycle in a plane graph, we consider the interior and exterior of $C$ as two disks bounded by the cycle.

We say that an open disk bounded by a cycle of length four in $G^*$ is \emph{tiny} if its interior consists of only one of the following configurations:
\begin{enumerate}
    \item A single vertex of degree five incident to a digon.
    \item Two vertices of degree five connected via a pair of multiple edges surrounding a digon.
\end{enumerate}
These structures are illustrated in Figure \ref{fig:tiny-4cycle}. 
Following the terminology in \cite{3edgecoloring}, we also refer to the latter configuration as a \emph{domino}.
This graph corresponds to the dual of the domino graph in Figure \ref{fig:domino-first}.
Translating Lemma \ref{lem:Gmc-connected} from $G$ to $G^*$, we have the following lemma.

\begin{figure}[htbp]
    \centering
    \includegraphics[width=0.35\linewidth]{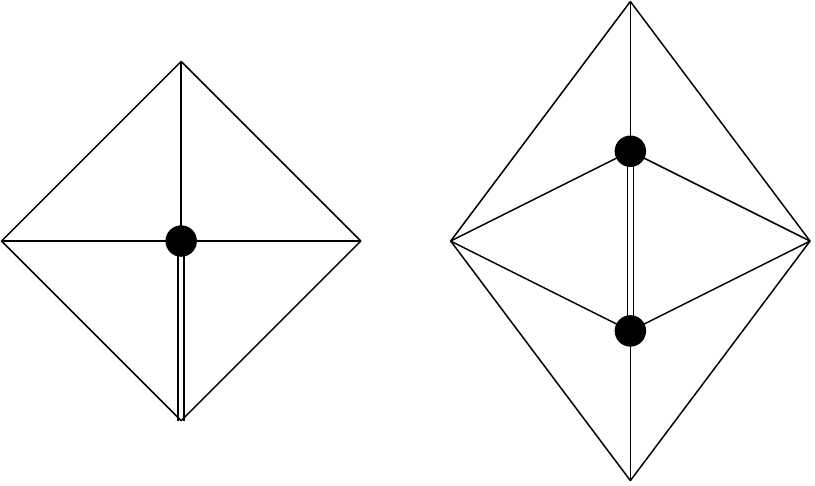}
    \caption{Tiny disks. The figure on the right represents a domino in $G^*$.}
    \label{fig:tiny-4cycle}
\end{figure}

\begin{lem}
\label{lem:triangulation}
\showlabel{lem:triangulation}
Let $G_\text{mc}, G, G^*$ be defined as above.
Then, the following holds for $G^*$ and $G$.
\begin{enumerate}
    \item $G^*$ is a loopless graph in which the only non-triangular faces are digons. There are only two or three digons, and every vertex is incident with at most one digon.
    \item The minimum degree of $G^*$ is at least five.
    \item For every cycle $C$ of length five or less in $G^*$, bounding two open disks,
    \begin{itemize}
        \item If $|C| \leq 3$, then one of the disks is a face.
        \item If $|C| = 4$, then one of them contains no vertices, or consists of a tiny disk.
        \item If $|C| = 5$, then either one of them contains at most one vertex, or both contain at least one digon.
    \end{itemize}
\end{enumerate}
\end{lem}

\section{Islands and Reducibility}
\label{sect:island}
\showlabel{sect:island}

\subsection{Islands}
\label{subsect:island}
\showlabel{subsect:island}

We define the structure we want to find in $G$ as follows:
\begin{dfn}[multi-boundary island, island]
\label{dfn:island}
\showlabel{dfn:island}
     Let $I$ be a connected plane subcubic graph. Let $E_R(I)$ be the set of pendant edges in $I$, that is the edges incident to a vertex of degree one, and let $F_R(I)$ be the set of faces whose facial walk contains at least one edge in $E_R(I)$. 
     If every edge not in $E_R(I)$ is incident to at least one face that is not in $F_R(I)$, then $I$ is said to be a \emph{multi-boundary island} and each face in $F_R(I)$ is called a \emph{boundary face} of $I$.
     An \emph{island} is a multi-boundary island with $|F_R(I)|=1$.
\end{dfn}
A similar notion of islands was introduced in \cite{doublecross}; however, while \cite{doublecross} does not allow vertices of degree two, our definition does.
Another distinction is that we allow for more than one boundary face in order to handle weaker connectivity assumptions than those in \cite{doublecross} and \cite{RSST}, see Section \ref{sect:hom-conf}.
Throughout this paper, we only consider multi-boundary islands with 0, 1, 2, or 3 boundary faces.
We show two examples of islands: in Figure \ref{fig:65555}, we have an island whose boundary face is the outer face in the figure; similarly, Figure \ref{fig:65555-2} shows an example of a multi-boundary island with two boundary faces.

For a multi-boundary island \(I\), we denote by \(I-E_R(I)\) the plane graph obtained from \(I\) by removing all leaves and all edges in \(E_R(I)\). 

\begin{figure}[htbp]
    \centering
    \begin{minipage}[b]{0.49\columnwidth}
        \centering
        \includegraphics[width=0.8\columnwidth]{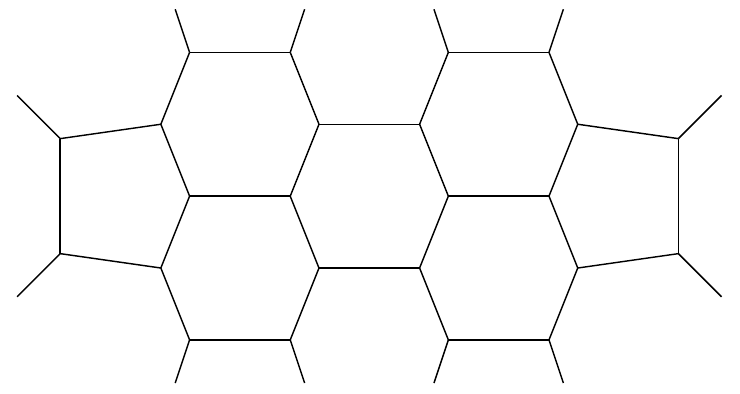}
        \caption{An example of an island}
        \label{fig:65555}
    \end{minipage}
    \begin{minipage}[b]{0.49\columnwidth}
        \centering
        \includegraphics[width=0.8\columnwidth]{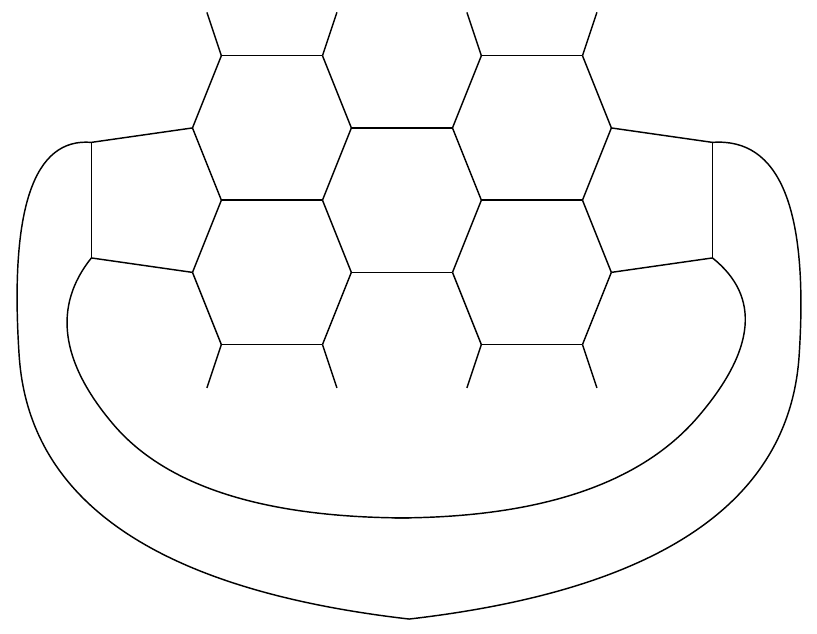}
        \caption{An example of a multi-boundary island with two faces in $F_R(I)$}
        \label{fig:65555-2}
    \end{minipage}
\end{figure}

\begin{dfn}[appear]
\label{dfn:island-appear}
\showlabel{dfn:island-appear}
    Let $I$ be a multi-boundary island and $G$ be a $2$-connected subcubic plane graph.
    We say that $I$ \emph{appears} in $G$ if $I-E_R(I)$ is a subgraph of $G$ and every vertex whose degree is two or three in $I$ has the same degree in $G$, and every face not in $F_R(I)$ is also a face of $G$.
\end{dfn}

The condition on the degrees in Definition \ref{dfn:island-appear} implies that no edge in $E(G)\setminus E(I)$ is incident with any vertex in $I-E_R(I)$. As a corollary,
note that a multi-boundary island $I$ having a boundary face incident with only one degree-one vertex cannot appear in $G$ since $G$ is bridgeless.
Also note that if a multi-boundary island $I$ with $|F_R(I)|=0$ appears in $G$, then $G$ is isomorphic to $I$.

\subsection{Reducibility}
\label{subsect:reducibility}
\showlabel{subsect:reducibility}

Informally, a multi-boundary island $I$ is said to be \emph{reducible} if we can construct a smaller non-colorable graph $G'$ from $G$ by modifying the structure of $I$ if $I$ appears in $G$.

We consider the case where $I$ appears in $G$, and all edges in $G$ except those in $I - E_R(I)$ are three-edge-colored.
Fix two colors, say $c_1$ and $c_2$, and consider how edges in $E_R(I)$ colored with either $c_1$ or $c_2$ are connected outside $I$ by alternating paths of these two colors.
Starting from an edge in $E_R(I)$ with color $c_1$ or $c_2$, we trace the path by alternately following edges of these two colors.
If we encounter a vertex of degree two in $G$ that is incident to an edge of a third color, we cannot follow an edge colored with $c_1$ or $c_2$, so we terminate the process.
Otherwise, the path must eventually reach another edge in $E_R(I)$ with color $c_1$ or $c_2$.
Thus, the edges in $E_R(I)$ colored with $c_1$ or $c_2$ are partitioned into singletons or pairs corresponding to the $\{c_1,c_2\}$-colored paths.

Such $\{c_1,c_2\}$-colored paths are called \emph{Kempe chains} (for colors $c_1$ and $c_2$). By deleting $I - E_R(I)$ from $G$, the remaining subgraphs within different faces in $F_R(I)$ become mutually disconnected; hence, two edges incident to different faces in $F_R(I)$ cannot be connected by a Kempe chain.
Furthermore, even for two edges incident to the same face, the planarity of the embedding strictly restricts the topological patterns of these connections; specifically, it forbids Kempe chains from crossing.

To formally define it, we index the faces in $F_R(I)$ as $\{F_1, \ldots, F_{|F_R(I)|}\}$ in an arbitrary order.
Since all degree-one vertices of $I$ are incident to exactly one face in $F_R(I)$, we label the edges in $E_R(I)$ as $r_{i,j}$, where $1 \leq i \leq |F_R(I)|$ and $j \geq 1$, such that $r_{i,j}$ is incident to $F_i$. The second index $j$ enumerates the edges in $E_R(I)$ incident to $F_i$ in a clockwise order around $F_i$. Finally, we define the set $R\subseteq \mathbb{N} \times \mathbb{N}$ as the set $R = \{(i,j)\mid r_{i,j}\in E_R(I)\}$ that corresponds to the enumeration of $E_R(I)$.
The definition is as follows.
\begin{dfn}[overlap, semi-matching]
\label{dfn:semi-matching}
\showlabel{dfn:semi-matching}
    \hfill
    \begin{itemize}
        \item A \emph{match} is an unordered pair of distinct elements in $\mathbb{N} \times \mathbb{N}$ with the same first coordinate. That is, a match is of the form $\{(i, a), (i, b)\}$ where $a \neq b$.
        \item Two matches $\{(i, a), (i, b)\}$ and $\{(j, c), (j, d)\}$ with $a<b$ and $c<d$ are said to \emph{overlap} if $i=j$ and, assuming without loss of generality that $a < c$, we have $a < c < b < d$.
        \item A \emph{semi-matching} is a set of pairwise disjoint matches and singletons in $\mathbb{N} \times \mathbb{N}$ such that no two matches overlap.
    \end{itemize}
\end{dfn}
We also define the operation to flip colors of $E_R(I)$ by alternating paths outside $I$.
Since this operation does not inherently depend on the structure of the island $I$ itself, we define the coloring not on $E_R(I)$, but more abstractly on a subset $R \subseteq \mathbb{N} \times \mathbb{N}$.
The set $R$ corresponds to the indices $(i,j)$ of the edges $r_{i,j}\in E_R(I)$.
\begin{dfn}[switch colors by Kempe chains]
\label{dfn:kempe}
\showlabel{dfn:kempe}
    Let $R \subseteq \mathbb{N} \times \mathbb{N}$ and $\phi \colon R \to [3]$ be a map, $x,y \in [3]$ be distinct integers, and $M'$ be a semi-matching whose \emph{support} $\operatorname{supp}(M') = \bigcup_{X \in M'}X$ is contained in $\{(i,j) \in R \mid \phi(i,j) \in \{x, y\}\}$. A map $\phi' \colon R \to [3]$ is said to be obtained by \emph{switching colors $x$ and $y$ on Kempe chains} using the semi-matching $M'$ if  
    \[
    \phi'(r) = \begin{cases}
        \phi(r) & \text{if } r \not \in \operatorname{supp}(M') \\
        x & \text{if } r \in \operatorname{supp}(M') \hbox{ and\, } \phi(r) = y \\
        y & \text{if } r \in \operatorname{supp}(M') \hbox{ and\, } \phi(r) = x. \\
    \end{cases}
    \]
\end{dfn}

For a three-edge-coloring of the edges in $G$ except those in $I - E_R(I)$, we can modify the coloring using Kempe chains.
Our goal is to extend this coloring to the inside of $I$ by repeated modifications.
Thus, it is important to know the set of colorings of $E_R(I)$ that are extendable to $I$, and to determine whether the current coloring can be transformed into one in this set.
To do so, we investigate the property of a set of colorings that is closed under switching colors on Kempe chains.
Given a particular coloring and a choice of two distinct colors $x, y \in [3]$, the connection topology of the Kempe chains is determined by the coloring of the edges of $G$ outside $I$.
All colorings obtained by switching colors by these Kempe chains must also belong to the same set.
This leads to the following concept. 
\begin{dfn}[semi-consistent]
\label{dfn:semi-consistent}
\showlabel{dfn:semi-consistent}
    Let $R \subseteq \mathbb{N} \times \mathbb{N}$, and $\mathcal{D}_R$ be the set of all maps from $R$ to $[3]$.
    A subset $\mathcal{C}$ of $\mathcal{D}_R$ is \emph{semi-consistent} if for every $\phi \in \mathcal{C}$ and for every distinct $x, y \in [3]$, there exists a semi-matching $M$ that is a partition of $\{r \in R \mid \phi(r) \in \{x, y\}\}$ such that every map obtained from $\phi$ by switching colors on Kempe chains on some $M'\subseteq M$ also belongs to $\mathcal{C}$.
\end{dfn}

\begin{dfn}[semi-D-reducible]
\label{dfn:D-reducible}
\showlabel{dfn:D-reducible}
    Let $I$ be a multi-boundary island with $E_R(I)=\{r_{i,j}\}$ labeled as above, and let $R \subseteq \mathbb{N} \times \mathbb{N}$ be the set of these indices.
    Let $\mathcal{D}_R$ be the set of all maps from $R$ to $[3]$.
    Let $\mathcal{C}_I$ be the set of maps $\phi$ in $\mathcal{D}_R$ such that there exists a three-edge-coloring $\eta$ of $I$ with $\eta(r_{i,j})=\phi((i,j))$ for all $(i,j) \in R$.
    A multi-boundary island $I$ is \emph{semi-D-reducible} if the empty set is the only semi-consistent set in $\mathcal{D}_R \setminus \mathcal{C}_I$.
\end{dfn}
Let us observe that the definition here is not restricted to apex graphs. Indeed, this definition is with respect to subcubic planar graphs (i.e., we do not have to bound the number of vertices of degree 2). 
For a semi-D-reducible multi-boundary island, any coloring of $E_R(I)$ can be transformed into a coloring that is extendable inside $I$. 
We can show the following lemma.
\begin{lem}
\label{lem:D-reducible-doesnot-appear}
\showlabel{lem:D-reducible-doesnot-appear}
    Let $I$ be a semi-D-reducible multi-boundary island such that $I - E_R(I)$ has at least one edge, and let $G$ be the subcubic graph introduced in Section \ref{sect:minimal}.
    Then, $I$ does not appear in $G$.
\end{lem}
\begin{proof}
    Suppose $I$ appears in $G$.
    In $G$, we choose an edge $e \in I - E_R(I)$, and delete it from $G$.

    Let $u_1,u_2$ be the endpoints of $e$.
    For $i=1,2$, if $d_G(u_i)=3$, then the degree becomes 2 after deleting $e$, so we suppress it, i.e., replace the two incident edges by a single edge.
    If $d_G(u_i)=2$, then the degree becomes 1 after deleting $e$, so we delete $u_i$ and its incident edge.
    After this operation, the neighbor of $u_i$ in $G-e$ has degree 2 since the degree of this vertex is 3 in $G$ by the construction of $G$.
    
    After this operation, we have a smaller subcubic graph than $G$ with the same number of vertices of degree 2.
    We denote it by $H$.
    By the third claim in Lemma \ref{lem:triangulation}, $H$ is bridgeless.
    Then, $H$ is three-edge-colorable because we obtain a smaller apex cubic graph than $G$ by adding an edge between vertices of degree 2 if $H$ has two vertices of degree 2, otherwise adding a new apex vertex adjacent to three vertices of degree 2.
    This three-edge-coloring of $H$ is transformed into three-edge-coloring of $G$ except edges in $I - E_R(I)$.
    By definition, the set of all colorings of $E_R(I)$ that can be obtained from this three-edge-coloring using Kempe chains outside $I - E_R(I)$ is semi-consistent.
    This semi-consistent set is non-empty, so by semi-D-reducibility of $I$, this set has a three-edge-coloring in $\mathcal{C}_I$.
    Every three-edge-coloring in $\mathcal{C}_I$ is extendable to $I$, so we can color all edges of $G$ with three colors.
    This contradicts that $G$ is obtained from the minimal counterexample $G_{mc}$ of Theorem \ref{mainth}.
\end{proof}

We can consider a multi-boundary island $I$ such that $E_R(I)=\emptyset$.
For this island $I$, the set of indices $R$ is $\emptyset$ and $\mathcal{D}_R$ only contains the empty function from $\emptyset$ to $[3]$.
If $I$ itself is 3-edge-colorable, $\mathcal{C}_I$ consists of the empty function.
Otherwise, $\mathcal{C}_I$ is the empty set.
Thus, $I$ is semi-D-reducible if and only if $I$ itself is 3-edge-colorable.
If such an island $I$ appears in $G$, then $G$ is isomorphic to $I$, which contradicts Lemma \ref{lem:D-reducible-doesnot-appear}.

In addition to semi-D-reducibility, we introduce a standard notion of reducibility termed \emph{C-reducibility}.
For a multi-boundary island $I$ that is not semi-D-reducible, certain colorings of $E_R(I)$ cannot be transformed to any coloring extendable to the inside of $I$.
A semi-C-reducible multi-boundary island $I$ eliminates the possibility of encountering such colorings by restricting the valid colorings on $E_R(I)$; this is achieved by specifying a particular edge-set to be deleted during the inductive step.

\begin{dfn}[deletable]
\label{dfn:deletable}
\showlabel{dfn:deletable}
    Let $I$ be a multi-boundary island.
    A subset $F \subseteq E(I) \setminus E_R(I)$ is said to be \emph{deletable} if the following conditions hold:
    \begin{enumerate}
        \item No vertex in $I$ is incident to exactly two edges of $F$.
        \item $1 \leq |F| \leq 3$, or $|F| = 4$ and either some face $f \notin F_R(I)$ is incident to at least three edges of $F$, or there exist two distinct faces $f_1, f_2 \notin F_R(I)$ such that some edge of $I$ is incident to both $f_1$ and $f_2$, and every edge of $F$ is incident to at least one of them\footnote{This criterion is the same as that used in \cite{doublecross}.}. 
    \end{enumerate}
    A \emph{three-edge-coloring modulo $F$ of $I$} is a map $\phi \colon E(I) \setminus F \to [3]$ such that for all distinct edges $e, f \in E(I) \setminus F$ sharing a common end-vertex $v$, we have $\phi(e) = \phi(f)$ if and only if $v$ is incident to exactly one edge in $F$.
\end{dfn}

In general, when $G$ is a subcubic graph, $G \dotdiv F$ is the subcubic graph obtained from $G$ by deleting $F$ and performing the following changes to endpoints of edges in $F$:
\begin{itemize}
    \item if its degree becomes 0, remove it,
    \item if its degree becomes 2, suppress it (i.e., replace the two incident edges $uv,wv$ by the single edge $uw$).
\end{itemize}
After the operation, repeatedly remove a vertex of degree one together with its incident edge until no vertex of degree one remains.
When $I$ appears in $G$, a three-edge-coloring modulo $F$ of $I$ corresponds to a coloring on $E(G) \setminus F$, which is obtained from a three-edge-coloring of $G \dotdiv F$.

\begin{dfn}[semi-C-reducible]
\label{dfn:C-reducible}
\showlabel{dfn:C-reducible}
    Let $I$ be a multi-boundary island and let $E_R(I), R, \mathcal{D}_R, \mathcal{C}_I$ be defined as in Definition \ref{dfn:D-reducible}, and $F$ be a deletable edge set of $I$.
    Let $\mathcal{C}_{I \dotdiv F}$ be the set of maps $\phi$ in $\mathcal{D}_R$ such that there exists a three-edge-coloring modulo $F$ of $I$, say $\eta$, with $\eta(r_{i,j})=\phi((i, j))$ for all $(i, j) \in R$.
    A multi-boundary island $I$ is \emph{semi-C-reducible by $F$} if every semi-consistent set in $\mathcal{D}_R \setminus \mathcal{C}_{I}$ is disjoint from $\mathcal{C}_{I \dotdiv F}$.
\end{dfn}

We remark that an island $I$ with no degree-two vertices is never
semi-D-reducible.
Let $\mathcal{C}_{\textsf{nopar}}$ be the set of colorings of $E_R(I)$
that do not satisfy the parity condition.
Then $\mathcal{C}_{\textsf{nopar}}$ is semi-consistent.
Indeed, fix $\varphi\in\mathcal{C}_{\textsf{nopar}}$ and two colors
$x,y$. List the edges colored $x$ or $y$ in their boundary order.
If their number is even, pair them consecutively.
If it is odd, pair them consecutively except for one edge, which is
taken as a singleton.
This gives a semi-matching partition of the edges colored $x$ or $y$.
Switching colors on a pair preserves all parities, while switching on
the singleton changes the parities for both $x$ and $y$.
In either case, the resulting coloring still does not satisfy the
parity condition.
Every coloring in $C_I$ satisfies the parity condition, and hence
$\mathcal{C}_{\textsf{nopar}}\cap C_I=\emptyset$.
Therefore, $I$ is not semi-D-reducible.

On the other hand, for any deletable edge set $F$, every coloring in
$C_{I\dotdiv F}$ satisfies the parity condition.
Hence,
$\mathcal{C}_{\textsf{nopar}}\cap C_{I\dotdiv F}=\emptyset$.
Thus, although $I$ cannot be semi-D-reducible, this parity obstruction
does not prevent $I$ from being semi-C-reducible.

Similarly, if $I$ has exactly one degree-two vertex, the set of
colorings satisfying the parity condition is a semi-consistent set
disjoint from $C_I$. Hence, such an island is also never
semi-D-reducible.

\begin{lem}
\label{lem:C-reducible-doesnot-appear}
\showlabel{lem:C-reducible-doesnot-appear}
    Let $I$ be a semi-C-reducible multi-boundary island by $F$, and $G$ be the subcubic graph introduced in Section \ref{sect:minimal}.
    Then, $I$ does not appear in $G$.
\end{lem}
\begin{proof}
    Suppose $I$ appears in $G$.
    First, we show the following claims:
    \begin{claim}
    \label{claim:cubic-bridge}
    \showlabel{claim:cubic-bridge}
        There is no set $Y \subsetneq V(G)$ such that $|\delta_G(Y) \setminus F|=1$ and every vertex in $Y$ has degree $3$.
    \end{claim}
    \begin{proof}
        Suppose that there exists $Y \subsetneq V(G)$ such that $|\delta_G(Y) \setminus F| = 1$ and every vertex in $Y$ has degree 3.
        By Lemma \ref{lem:triangulation} and the fact that $|F| \leq 4$, this is possible when $|\delta_G(Y)|=3,~|Y|=1$ or $|\delta_G(Y)|=4,~|Y|=2$ or 
        $|\delta_G(Y)|=5,~|Y| \leq 5$.
        By the first condition of Definition \ref{dfn:deletable}, the only possibility is the case where $|\delta_G(Y)|=5$ and $|Y|=5$.
        Then, $G[Y]$ is a 5-cycle and $|F|=4$.
        Let $w_i$ ($1 \leq i \leq 5$) be the vertices of a cycle in $G[Y]$ in this order, and $h_i$ ($1 \leq i \leq 5$) be an edge in $\delta_G(Y)$ that is incident to $w_i$, and assume $F=\{h_1,h_2,h_3,h_4\}$.
        Let $r_i$ ($1 \leq i \leq 5$) be the region that is incident to $w_iw_{i+1}$, $h_i$, and $h_{i+1}$.
        Since $G$ has no bridge and by the connectivity of $G$ in Lemma \ref{lem:triangulation}, $r_i \neq r_j$ if $i \neq j$.
        By the connectivity of $G$ in Lemma \ref{lem:triangulation}, $r_i, r_{i+2}$ are not incident to the same edge.
        Considering the second condition in Definition \ref{dfn:deletable}, $F$ cannot satisfy all requirements.
        This contradiction shows that no such $Y$ exists.
    \end{proof}
    \begin{claim}
    \label{claim:one-degree2-vertex}
    \showlabel{claim:one-degree2-vertex}
        There is no set $Y \subsetneq V(G)$ such that $\delta_G(Y) \subseteq F$ and $Y$ contains exactly one vertex of degree two in $G$.
    \end{claim}
    \begin{proof}
        Second, suppose that there exists $Y \subsetneq V(G)$ such that $\delta_G(Y) \subseteq F$ and exactly one vertex in $Y$ has degree 2.
        By Lemma \ref{lem:triangulation} and the fact that $|F| \leq 4$, this is possible when $|\delta_G(Y)|=2,|Y|=1$, or
        $|\delta_G(Y)|=3,|Y|=2$, or
        $|\delta_G(Y)|=4,|Y| \leq 7$.
        By the first condition of Definition \ref{dfn:deletable}, it is possible only when $|\delta_G(Y)|=4$, and $G[Y]$ is a 5-cycle or a domino. 
        Then, $F=\delta_G(Y)$.
        By the same proof as above, $F$ cannot satisfy the requirements of the second condition in Definition \ref{dfn:deletable}.
        This contradiction shows that no such $Y$ exists.
    \end{proof}

    \begin{claim}
    \label{claim:G-F}
    \showlabel{claim:G-F}
        $G \dotdiv F$ satisfies each of the following:
        \begin{itemize}
            \item[(i)] The number of degree-two vertices in $G \dotdiv F$ is at most the number of degree-two vertices in $G$.
            \item[(ii)] No connected component of $G \dotdiv F$ contains exactly one degree-two vertex.
            \item[(iii)] For every bridge $e$ in $G \dotdiv F$, each component of $(G \dotdiv F)-e$ contains a vertex whose degree in $G \dotdiv F$ is two.
        \end{itemize}
    \end{claim}
    \begin{proof}
        By construction of $G \dotdiv F$ from $G$, every degree-two vertex
        in $G \dotdiv F$ is either an original degree-two vertex of $G$,
        or is created during the deletion of degree-one vertices.
    
        By the first condition of Definition \ref{dfn:deletable}, every
        degree-one vertex created immediately after deleting $F$ originates
        from a degree-two vertex of $G$.
        When a degree-one vertex is deleted, the degree of its unique
        neighbor decreases by one.
        If the neighbor becomes degree-one, the deletion process continues;
        if it becomes degree-two, the process stops there.
        Thus, a deletion process starting from one degree-two vertex of $G$
        can produce at most one degree-two vertex of $G \dotdiv F$.
        Moreover, if two deletion processes meet, the number of resulting
        degree-two vertices can only decrease.
    
        Hence the number of degree-two vertices in $G \dotdiv F$ is at most
        the number of degree-two vertices in $G$.
        This proves (i).

        We next prove (ii).
        Suppose that some connected component $H$ of $G \dotdiv F$ contains exactly one degree-two vertex.
        Let $K$ be the component of $G - F$ from which $H$ is obtained, and let $Y \coloneq V(K)$.
        Then, $\delta_G(Y) \subseteq F$ holds.
        Since $H$ contains a degree-two vertex, $Y$ contains at least one degree-two vertex of $G$.
        If $Y$ does not contain all degree-two vertices of $G$, then, since $G$ has only two or three degree-two vertices, either $Y$ or $V(G) \setminus Y$ contains exactly one degree-two vertex.
        This contradicts Claim \ref{claim:one-degree2-vertex}.

        Hence, $Y$ contains all degree-two vertices of $G$.
        If the unique degree-two vertex of $H$ is an original degree-two vertex of $G$ that survives, then all the other original degree-two vertices are deleted completely.
        The vertices involved in these deletion processes form a set $Z'$ with
        $\delta_G(Z') \subseteq F$; hence either $Z'$ or its complement contains
        exactly one degree-two vertex of $G$, contradicting
        Claim \ref{claim:one-degree2-vertex}.
        Therefore, the unique degree-two vertex in $H$ is created as follows.
        Some of the corresponding deletion processes may disappear completely, while all the other processes merge and produce the unique degree-two vertex of $H$.
        Let $Z$ be the union of all vertices deleted or suppressed in these processes.
        Then $Z$ contains all degree-two vertices of $G$, and exactly
        one edge of $\delta_G(Z)$ does not belong to $F$, namely the
        edge through which the unique degree-two vertex of $H$ is
        produced.
        Thus, $|\delta_G(Z) \setminus F|=1$.
        Therefore, every vertex of $V(G) \setminus Z$ has degree three, and $|\delta_G(V(G) \setminus Z) \setminus F|=1$, contradicting Claim \ref{claim:cubic-bridge}.

        We finally prove (iii).
        Let $e$ be a bridge of $G \dotdiv F$, and suppose that one side of $e$ contains no vertex whose degree in $G \dotdiv F$ is two.
        Let $X \subsetneq V(G)$ be the corresponding side before the suppressions and deletions.
        Then $|\delta_G(X) \setminus F|=1$.
        If $X$ contains no degree-two vertex of $G$, this contradicts Claim \ref{claim:cubic-bridge}.
        Otherwise, every degree-two vertex of $G$ in $X$ is deleted by the deletion process of degree-one vertices, which follows the deletion of $F$ and suppression.
        Hence, these vertices are contained in a set $Z \subseteq X$ with $\delta_G(Z) \subseteq F$.
        Since $G$ has only two or three degree-two vertices,
        Claim~\ref{claim:one-degree2-vertex}, applied to $Z$ or to
        $V(G)\setminus Z$, implies that $Z$ contains all degree-two
        vertices of $G$.
        Then, $V(G) \setminus X$ contains no degree-two vertex of $G$ and $|\delta_G(V(G) \setminus X) \setminus F|=1$, contradicting Claim \ref{claim:cubic-bridge}.
    \end{proof}
    \begin{claim}
    \label{claim:G-F-is-colorable}
    \showlabel{claim:G-F-is-colorable}
        $G \dotdiv F$ is three-edge-colorable.
    \end{claim}
    \begin{proof}
        Let $H$ be a connected component of $G \dotdiv F$.
        By Claim \ref{claim:G-F} (i), (ii), the number of degree-two vertices in $H$ is zero, two, or three.
        
        If $H$ has no degree-two vertex, then $H$ is cubic.
        By Claim \ref{claim:G-F} (iii), $H$ has no bridge, so it is a bridgeless planar cubic graph.
        $H$ is three-edge-colorable by the Four Color Theorem.
    
        Suppose now that $H$ has two or three degree-two vertices.
        Construct a cubic graph $\widehat H$ as follows: 
        if the degree-two vertices are $u_1$, $u_2$, add the edge $u_1u_2$.
        If the degree-two vertices are $u_1$, $u_2$, and $u_3$, add a new vertex adjacent to all $u_1$, $u_2$ and $u_3$.
        By Claim \ref{claim:G-F} (iii), a bridge of $H$ is turned into a non-bridge in $\widehat H$.
        Moreover, every newly added edge is not a bridge.
        Thus, $\widehat H$ is two-edge-connected.
        $\widehat H$ is apex.
        Since $F \neq \emptyset$, $G \dotdiv F$ has strictly smaller number of vertices than $G$ has.
        If $H$ has three degree-two vertices, then we add a new vertex to get $\widehat H$.
        In this case, by Claim \ref{claim:G-F} (i), $G$ also has three degree-two vertices, so $G$ is obtained from $G_{mc}$ by deleting its apex vertex.
        If $H$ has two degree-two vertices, then we add no vertex to get $\widehat H$.
        In every case, $|V(\widehat H)| < |V(G_{mc})|$.
        By minimality of $G_{mc}$, $\widehat H$ is three-edge-colorable.
        Restricting such a coloring to $H$ gives a three-edge-coloring of $H$.
        Therefore, every connected component of $G \dotdiv F$ is three-edge-colorable, and hence $G \dotdiv F$ is three-edge-colorable.
    \end{proof}

    A three-edge-coloring of $G \dotdiv F$, obtained by Claim \ref{claim:G-F-is-colorable}, is transformed into three-edge-coloring of $G$, except edges in $I-E_R(I)$, such that the coloring in $E_R(I)$ belongs to $\mathcal{C}_{I \dotdiv F}$.
    By definition, the set of all colorings of $E_R(I)$ that can be obtained from this three-edge-coloring using Kempe chains outside $I - E_R(I)$ is semi-consistent.
    By the semi-C-reducibility of $I$ by $F$, this set has a three-edge-coloring in $\mathcal{C}_I$.
    This contradicts that $G$ is obtained from the minimal counterexample $G_{mc}$ of Theorem \ref{mainth}.
\end{proof}

Note that Lemma \ref{lem:C-reducible-doesnot-appear} holds for any triangulation with digons satisfying Lemma \ref{lem:triangulation}.
A multi-boundary island is called \emph{semi-reducible} if it is either semi-D-reducible or semi-C-reducible.

\subsection{Reducibility under deleting boundary edges}

In this subsection, we show Lemma \ref{lem:imply}.
A multi-boundary island $I'$ is called a \emph{trimmed multi-boundary island of a multi-boundary island $I$} if $I'$ is obtained from $I$ by deleting some degree-one vertices together with their incident edges, and $F_R(I')$ is obtained from $F_R(I)$ by deleting faces $F \in F_R(I)$ that are incident to no degree-one vertex after this operation.
We can show the following Lemma \ref{lem:imply}.
This lemma helps to reduce the number of multi-boundary islands whose reducibility is to be checked.
The role of the lemma is detailed in Section \ref{sect:hom-conf}.

\begin{lem}
\label{lem:imply}
\showlabel{lem:imply}
    Let $I$ be a multi-boundary island, $I'$ be a trimmed multi-boundary island of $I$, and $F$ be a deletable edge set of $I$.
    Then, the following holds:
    \begin{itemize}
        \item if $I$ is semi-D-reducible, then $I'$ is semi-D-reducible, and
        \item if $I$ is semi-C-reducible by $F$, then $I'$ is semi-C-reducible by $F$.
    \end{itemize}
\end{lem}
\begin{proof}
    It suffices to consider the case where $I'$ is obtained by deleting one degree-one vertex together with its incident edge, because the claims where more than one degree-one vertex are deleted follow by repeatedly applying the claim for the case where one degree-one vertex is deleted.

    As described above, we index the faces in $F_R(I)$ as $\{F_1,\ldots,F_{|F_R(I)|}\}$.
    Assume that $I'$ is a multi-boundary island obtained from $I$ by deleting one edge $r_{i,j}\in E_R(I)$ and its incident degree-one vertex.
    By shifting the clockwise ordering of edges in $F_i$, if necessary, we assume w.l.o.g. that $j$ is the last index of edges in $E_R(I)$ that are incident to the face $F_i$.
    Let $R, R'$ be the set of indices for $I,I'$ respectively, then $R = R' \cup \{(i, j)\}$.

    Let $\mathcal{C}'$ be semi-consistent set in $\mathcal{D}_{R'}$.
    For each $\phi' \colon R' \to [3]$ and 
    each $c \in [3]$, define $\phi_c \colon R \to [3]$ as follows:
    $$
    \phi_c((k, l)) =
    \begin{cases}
        \phi'((k, l)) & \text{if } (k, l) \in R' \\
        c & \text{if } (k,l)=(i,j)
    \end{cases}
    $$
    We define $\mathcal{C}$ by $\mathcal{C} = \bigcup_{\phi' \in \mathcal{C}'} \{\phi_1, \phi_2, \phi_3\}$.
    \begin{claim}
        \label{claim:imply-consistent}
        \showlabel{claim:imply-consistent}
        If $\mathcal{C}'$ is a semi-consistent set in $\mathcal{D}_{R'}$, then $\mathcal{C}$ is a semi-consistent set in $\mathcal{D}_R$.
    \end{claim}
    \begin{proof}[Proof of Claim \ref{claim:imply-consistent}]
        Since $\mathcal{C}'$ is semi-consistent,
        for every $\phi' \in \mathcal{C}'$, for every distinct $x, y \in [3]$, there exists a semi-matching $M_{\phi',x,y}$ that is also a partition of $\{r' \in R' \mid \phi(r') \in \{x, y\}\}$ such that every map obtained from $\phi'$ by switching colors by Kempe chains $M_{\phi',x,y}$ also belongs to $\mathcal{C}'$.

        For $\phi_c \in \mathcal{C}$, and for every distinct $x,y \in [3]$, we specify a semi-matching $M_{\phi_c,x,y}$ by 
        \[
        M_{\phi_c,x,y} = \begin{cases}
            M_{\phi',x,y} \cup \{\{(i, j)\}\} & \text{if } c \in \{x, y\} \text{ (iff }\phi_c((i,j)) \in \{x,y\}) \\
            M_{\phi',x,y} & \text{otherwise.}
        \end{cases}
        \]
        $M_{\phi_c,x,y}$ is a partition of $\{r \in R \mid \phi_c(r) \in \{x, y\}\}$.
        We need to show that every map obtained from $\phi_c$ by switching colors by Kempe chains $M_{\phi_c,x,y}$ also belongs to $\mathcal{C}$.

        Let $M \subseteq M_{\phi_c,x,y}$ and a map $\phi_\textsf{changed}$ is obtained from $\phi_c$ by switching colors of $M$.
        $\phi_\textsf{changed}$ is expressible as $\omega_d$ for some $d \in [3]$ where $\omega'$ is a map obtained from $\phi'$ by switching colors of the component of $M_{\phi',x,y}$ corresponding to $M \setminus \{(i, j)\}$.
        By the semi-consistency of $\mathcal{C}'$, $\omega'$ is in $\mathcal{C}'$.
        Hence, $\phi_\textsf{changed}=\omega_d$ is in $\mathcal{C}$, which confirms the claim.
    \end{proof}
    If $\mathcal{C}'$ is non-empty, then $\mathcal{C}$ is non-empty.
    If we cannot color $I'$ with the assignment of edges of $E_R(I')$ being $\phi'$, for all $c \in [3]$, we cannot color $I$ with the assignment of edges of $E_R(I)$ being $\phi_c$.
    This means that if $\phi' \notin \mathcal{C}_{I'}$, then for all $c \in [3]$, $\phi_c \notin \mathcal{C}_I$.
    Then, if $\mathcal{C}' \cap \mathcal{C}_{I'} = \emptyset$, then $\mathcal{C} \cap \mathcal{C}_I = \emptyset$.

    We show the contraposition of the first claim.
    If $I'$ is not semi-D-reducible, then a non-empty semi-consistent set $\mathcal{C}' \subseteq \mathcal{D}_{R'} \setminus \mathcal{C}_{I'}$.
    The set $\mathcal{C}$ constructed from $\mathcal{C}'$ is a non-empty semi-consistent set in $\mathcal{D}_{R} \setminus \mathcal{C}_{I}$, so $I$ is not semi-D-reducible.

    By definition, if $F$ is deletable for $I$, then $F$ is deletable for $I'$.
    By the same discussion as above, we can show that $\mathcal{C}' \cap \mathcal{C}_{I' \dotdiv F} \neq \emptyset$, so $\mathcal{C} \cap \mathcal{C}_{I \dotdiv F} \neq \emptyset$.

    We show the contraposition of the second claim.
    If $I'$ is not semi-C-reducible by $F$, then some semi-consistent set $\mathcal{C}'$ in $\mathcal{D}_{R'} \setminus \mathcal{C}_{I'}$ has intersection with $\mathcal{C}_{I' \dotdiv F}$.
    The set $\mathcal{C}$ constructed from $\mathcal{C}'$ is a semi-consistent set in $\mathcal{D}_R \setminus \mathcal{C}_I$ that has intersection with $\mathcal{C}_{I \dotdiv F}$, so $I$ is not semi-C-reducible by $F$.
\end{proof}

\subsection{Checking Semi-reducibility}

The semi-D/C-reducibility of a (possibly multi-boundary) island is checked using a variant of existing reducibility-checking algorithms for cubic graphs, for example, \cite{doublecross}, \cite{proj2024}, and \cite{torus2024}.

The main difference between our algorithm and those used in previous work lies in the set of Kempe chains we employ.
All previous work considered 3-edge-colorings of cubic graphs, so any Kempe chain inside (or outside) an island must have both endpoints on the island's ring.
This is not the case in our setting, where some vertices might have degree 2, and therefore Kempe chains could have their endpoints at a degree-2 vertex inside (or outside) the island.
We call such Kempe chains \emph{half-chains} in the context of reducibility checking (in contrast to the "full-chains" connecting a pair of ring edges).
Another subtle difference is that we need to consider all $3^r$ ring colorings for an island of ring size $r$. In contrast, previous work enumerates only the ring colorings that satisfy the parity criteria (the number of ring edges that are colored with a certain color must have the same parity for all three colors).

In the algorithm we implement, we compute the set of all planar Kempe chain combinations (including half-chains) and determine all feasible colorings from this set.

We also provide pseudocode to enumerate deletable edge sets.

The relevant pseudocode is given in the Appendix, see Section \ref{appreduce}.
However, we provide pseudocode only for the newly introduced parts that differ from the existing implementation.
To perform the complete semi-D- and semi-C-reducible checks, one must additionally implement the standard iterative procedure used in existing reducibility checkers to compute the maximal semi-consistent subset of the nonextendable colorings.

\section{Configurations}
\label{sect:conf}
\showlabel{sect:conf}

Our goal is to prove unavoidability of semi-reducible multi-boundary islands.
To achieve this, we use the \emph{discharging method}.
When applying the discharging method, it is standard to use the triangulation with digons $G^*$, which is the dual of $G$, in Lemma \ref{lem:triangulation}.
Thus, we define the notion of a \emph{(multi-boundary) configuration} that is the dual concept of a (multi-boundary) island.

\begin{dfn}[multi-boundary configuration, configuration]
\label{dfn:conf}\showlabel{dfn:conf}
    A \emph{multi-boundary configuration} $K$ is a tuple $(G(K), F_R(K), \delta_K, \delta^{\textsf{out}}_K)$ such that the following holds:
    \begin{itemize}
        \item $G(K) = (V(K), E(K))$ is a connected plane graph, and $F_R(K)$ is a specified subset of the faces of $G(K)$, where every face not in $F_R(K)$ is a digon or a triangle.
        \item $\delta_K \colon V(K) \to \mathbb{Z}_+$.
        \item We replace each edge by two oppositely oriented edges. For each $F \in F_R(K)$, let $W_F$ be a closed directed walk bounding $F$ in clockwise order, and let $D$ be the set of oriented edges in these walks.
        $\deltaout_K: D \to \mathbb{Z}_+$ is a map satisfying,
        for every vertex $v \in V(K)$,
        \[
            \delta_K(v) = d_{G(K)}(v) + \sum_{\substack{e \in D: \text{the head of }e\text{ is }v}} \deltaout_K(e)
        \]
    \end{itemize}
    A \emph{configuration} is a multi-boundary configuration with $|F_R(K)|=1$.
\end{dfn}

We distinguish two opposite orientations of an edge in $D$.
Since all facial walks are oriented clockwise, if an edge is incident to two faces in $F_R(K)$, then it appears in the two corresponding walks with opposite orientations.
Thus, each directed edge belongs to exactly one closed walk.

We define \emph{normality} only for configurations, not for general multi-boundary configurations, as follows.
\begin{dfn}[normal]
\label{dfn:normal}
\showlabel{dfn:normal}
    A configuration $K$ is \emph{normal} if the following holds:
    \begin{itemize}
        \item For every vertex $v \in V(K)$, $G(K) - v$ has at most two connected components. $K$ has at most one cut-vertex.
        \item For every $e \in D$, if the head of $e$ is a cutvertex of $G(K)$, then $\deltaout_K(e)=1$.
    \end{itemize}
\end{dfn}

Similar to a multi-boundary island, we only consider multi-boundary configurations with $|F_R(K)| \in \{0,1, 2,3\}$.
Informally, $\delta_K(v)$ means the degree of $v$ when $K$ ``appears" in a larger triangulation with digons, typically $G^*$.
For a directed edge $e \in D$ with head $v$, $\deltaout_K(e)$ corresponds to the number of edges incident to $v$ that lie between $e$ and the immediately succeeding edge of the boundary walk in the counterclockwise rotation around $v$.
Under this topological interpretation, the equation in Definition \ref{dfn:conf} is naturally satisfied.

Throughout the paper, every configuration we use is normal. Therefore, hereafter the term "configuration" will always mean a normal configuration unless otherwise stated.
Except for allowing digons inside configurations, normal configurations follow the standard definition of configurations used in \cite{RSST} and \cite{doublecross}.
For normal configurations, $\deltaout_K$ can be uniquely determined from $K$ and $\delta_K$.
Specifically, for a directed edge $e \in D$, if the head of $e$, say $v$, is a cut-vertex, then the value of $\deltaout_K(e)$ is $1$.
Otherwise, the boundary walk of the unique face in \(F_R(K)\) visits \(v\) only once.
Hence \(e\) is the only directed edge in \(D\) whose head is \(v\), and so $\deltaout_K(e)=\delta_K(v)-d_{G(K)}(v)$.
Therefore, we do not explicitly specify $\deltaout_K$.

To describe multi-boundary configurations, the shape of the vertices represents the value of $\delta_K$, as shown in Figure \ref{fig:degree-representation}.
In Section \ref{sect:intro}, we already showed the figure of configurations represented in this way.
\begin{figure}[htbp]
    \centering
    \includegraphics[width=0.8\linewidth]{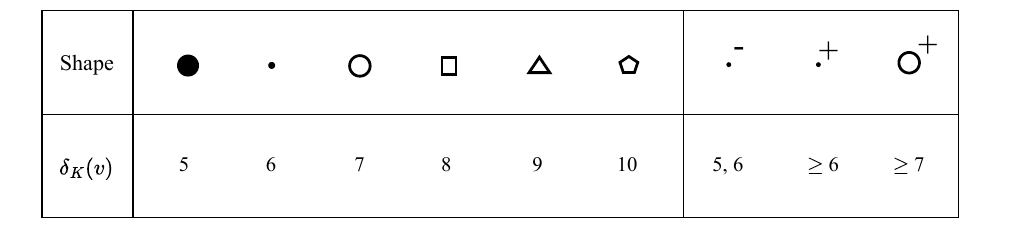}
    \caption{The shape of vertices}
    \label{fig:degree-representation}
\end{figure}

For a multi-boundary configuration $K$, we define a plane graph $\widehat{K}$ as follows.
Start with $G(K)$.
For each directed edge $e\in D$, let $v=\operatorname{head}(e)$, and let $e^+$ be the directed edge immediately following $e$ in the same clockwise boundary walk.
Thus $e^+$ has tail $v$.
For each $i\in\{1,\ldots,\deltaout_K(e)\}$, add a new vertex $x_{e,i}$ and a new edge $vx_{e,i}$.
Embed these new edges so that the cyclic order around $v$ from $e$ to $e^+$ in the counterclockwise direction is
$e, vx_{e,1}, vx_{e,2},\ldots, vx_{e,\deltaout_K(e)}, e^+$.
No other vertices or edges are added.
The resulting plane graph $\widehat{K}$ is called the \emph{outer extension} of $K$.
The vertices $x_{e,i}$ are called the \emph{outer endpoints} of $K$, and the edges $vx_{e,i}$ are called the \emph{outer edges} of $K$.

\cite{RSST} considered only triangular faces.
Accordingly, a configuration is completed by adding the outer edges as above and then identifying two outer endpoints $x_{e, \deltaout_K(e)}$ and $x_{e^+, 1}$, i.e., the last outer endpoint of $e$ and the first outer endpoint of $e^+$ for every $e$, and joining $x_{e, i}$ and $x_{e, i+1}$ for every $e$ and $i=1,\ldots,\deltaout_K(e)-1$.
This completion is called the \emph{free completion}.
The added edges are intended to form triangular faces.
In our setting, the dual graph $G^*$ contains digons; hence, the corresponding configurations must also include boundary regions that give rise to digons.
For this reason, we use the outer extension $\widehat{K}$ instead of the free completion.

For a walk $W$ bounding a face $F \in F_R(K)$, the sum $\sum_{e \in W} (\deltaout_K(e) - 1)$ becomes the size of the face in the free completion of $K$, originating from $F$.
This value is called \emph{the ring size of $F$}.
If the ring size of $F$ is $0$, all outer endpoints in $F$ are identified.
If the ring size of $F$ is $1$, all outer endpoints in $F$ are identified, and one loop bounds a face.

\subsection{Multi-boundary islands associated with a multi-boundary configuration}
\label{subsect:conf-2-island}
\showlabel{subsect:conf-2-island}
From a multi-boundary configuration $K$, we obtain a multi-boundary island as follows.
First, take the free completion of $K$, and then take its plane dual.
For each face $F \in F_R(K)$ with the positive ring size, let $v_F$ be the vertex in the dual corresponding to the face in the free completion, originating from $F$.
We replace $v_F$ by a collection of degree-one vertices, one for each edge originally incident with $v_F$, and make each such edge incident with its corresponding new vertex.
The resulting plane graph is denoted by $I(K)$.
Let $F_R(I(K))$ be the set of faces created by replacing the vertices $v_F$ for $F \in F_R(K)$ having positive ring size.
We call $I(K)$, together with $F_R(I(K))$, the \emph{multi-boundary island associated with $K$}.
Note that if the ring size of some face $F \in F_R(K)$ is 1, $I(K)$ has a face $F \in F_R(I)$ that is incident to only one degree-one vertex.
Since $G$ (obtained in Lemma \ref{lem:triangulation}) is bridgeless, $I(K)$ does not appear in $G$.

We can also obtain multi-boundary islands from the outer extension $\widehat K$ of $K$.
In this case, one multi-boundary configuration yields more than one multi-boundary island.
For each directed boundary edge $e$ of $K$, recall that the outer extension contains the outer edges
$vx_{e,1}, vx_{e,2},\ldots,vx_{e,\deltaout_K(e)}$ where $v=\operatorname{head}(e)$.
For each edge $e$ and each $i=1,\ldots,\deltaout_K(e)-1$, we figure out whether the two consecutive outer endpoints $x_{e,i}$ and $x_{e,i+1}$ are identified.
If they are identified, the corresponding boundary region yields a digon; otherwise, we add an edge between them, and the corresponding boundary region yields a triangle.
Also, we identify two outer endpoints $x_{e,\deltaout_K(e)}, x_{e^+, 1}$ for every $e$.
After making these operations, we complete $\widehat K$ accordingly.
The face corresponding to the face $F \in F_R(K)$ is constructed if at least one edge is added between $x_{e,i},x_{e,i+1}$ for a directed edge $e$ in $F$.
We take the plane dual of the graph, and replace the vertices corresponding to these faces by degree-one vertices, as above.
This yields \emph{a family of multi-boundary islands associated with the outer extension of $K$}.
One of them is $I(K)$, as described above.

Equivalently, these islands are precisely the set of trimmed multi-boundary islands of $I(K)$.
Thus, by Lemma \ref{lem:imply}, if $I(K)$ is semi-reducible, then all multi-boundary islands in the family are semi-reducible.

\section{Darts, homomorphisms}
\label{sect:dart}
\showlabel{sect:dart}
We introduce \emph{dart representations}, which will be used to encode embeddings of outer extensions of multi-boundary configurations and of $G^*$.
A dart representation is useful for defining a \emph{homomorphism} from a multi-boundary configuration to $G^*$.
Theorem \ref{thm:K-homomorphism-G}, which is shown by the discharging method, asserts the existence of a homomorphism from a configuration in the particular set $\mathcal{K}$ to $G^*$.

A dart representation is similar to a combinatorial map, developed by \cite{Klein1879} and \cite{Tutte1971}.
This representation was recently used to formalize the computer check required to develop a near-linear-time algorithm for coloring planar graphs in \cite{inoue2026four}.
We basically follow their definitions, but they consider only triangular faces, whereas we allow digons.
Thus, the definition differs from that in \cite{inoue2026four}.

Formally, our \emph{dart representation} has a set $V$ of vertices and a set $D$ of oriented edges or \emph{darts}. Each dart $e\in D$ has four pointers:
$\head(e)$ to a head vertex $v$ of $e$,
$\reverse(e)\in D$ to a reverse dart which is the \emph{opposite direction} of $e$ unless $\reverse(e)=e$, 
$\successor(e)$ to the succeeding dart with the same head $v$. 
and $\predecessor(e)$ to the preceding dart with the same head $v$. 
Both $\successor(e)$ and $\predecessor(e)$ can be $\nil$, which means that $e$ has no successor or predecessor.
These $\nil$-pointers are essential for our interpretation of a boundary.
The \emph{tail} of $e$ is defined as $\tail(e)=\head(\reverse(e))$.
If $\tail(e)=\head(e)$, then $e$ is a \emph{loop}.

Below, we define the requirements we aim to satisfy.
First, we have some \emph{basic requirements} (M1)--(M4).
\begin{itemize}
     \item[(M1)] There are no isolated vertices, i.e., for each vertex $v\in V$, there is a dart $e$ with $\head(e)=v$.  
     \item[(M2)] $\reverse: D\to D$ is an involution, that is, $e=\reverse(\reverse(e))$.
    \item[(M3)] $\successor$ and $\predecessor$ map $D$ to the set $D\cup\{\nil\,\}$. They are inverse to each other when non-$\nil$, that is, if we have two (non-$\nil$\,) darts $e$ and $f$, then
    $e=\predecessor(f)$ if and only if $f=\successor(e)$.
    \item[(M4)] Heads are consistent in the sense that if $e$ has successor $f\ne\nil$, then $\head(e)=\head(f)$.
\end{itemize}

In \cite{inoue2026four}, they also require that if $\successor(e) \ne \nil$, then there are darts $f$ and $g$ such that $f=\reverse(\successor(e))$, $g=\reverse(\successor(f))$, and $e=\reverse(\successor(g))$.
This reflects that every face is a triangle, but is not suitable for handling digons in our setting, so we do not require it.
Instead, we require the following:
\begin{itemize}
    \item[(M5)] Let $e$ be a dart with $\successor(e)\neq\nil$.
    Starting with $e_0=e$, define a sequence $e_0,e_1,e_2,e_3$ by $e_{i+1}=\reverse(\successor(e_i))$ as long as $\successor(e_i)\neq\nil$.
    Then, either (i) $e_3=e$ or (ii) $e_2=e$ holds.
\end{itemize}
Case (i) corresponds to a triangle, and case (ii) corresponds to a digon.

As in \cite{inoue2026four}, we impose one more requirement.
By (M3), the $\successor$ and $\predecessor$ pointers partition the darts into doubly-linked
lists, each of which may or may not be cyclic. If a list is acyclic, we view it as ordered following successors, starting from the dart $e$ with $\predecessor(e)=\nil$.
By the head consistency requirement (M4), the darts in each list all have the same head. 
(M6) below claims that each vertex corresponds to one doubly-linked list.
\begin{itemize}
    \item[(M6)] \emph{Single-list condition:}
    For each vertex $v$, there is exactly one list with $v$ as its head. This is called the \emph{incidence list} of $v$. The number of darts in the incidence list is the \emph{degree} $d(v)$ of $v$.
    If this list is cyclic, then $v$ is called \emph{an inner vertex}.
    If this list is acyclic, then it has $\nil$ at both ends, and $v$ is called \emph{a boundary vertex}.
\end{itemize}

A dart representation satisfying (M1)--(M4) and (M6) is called \emph{a pseudo-embedding}.
If it satisfies (M5), then we call it a \emph{pseudo-triangulation with digons}.
The \emph{pseudo-triangulations} considered in \cite{inoue2026four} correspond to the special case in which only case (i) in (M5) occurs; that is, every inner facial walk has length three.

$G^*$ can be represented as a pseudo-triangulation with digons.
The vertex set is $V(G^*)$.
Each edge of $G^*$ is replaced by two darts that are reversed relative to each other.
For each vertex $v$ in $G^*$, the darts with its head $v$ are assigned non-\nil \successor, \predecessor pointers induced from their clockwise rotations in the embedding.
It is easy to see that it satisfies (M1)--(M6).

\paragraph{Pseudo-embeddings of outer extensions of multi-boundary configurations}
A pseudo-embedding represents an outer extension of a multi-boundary configuration, but is not representable by a pseudo-triangulation with digons.
Let $\widehat{K}$ be an outer extension of a multi-boundary configuration.
The vertex set is $V(\widehat{K})$.
Each edge of $\widehat{K}$ is converted into two darts, which are reversed with respect to each other.
For each vertex $v$ in $K$, the darts whose head is $v$ are assigned non-\nil \successor, \predecessor pointers induced from their clockwise rotation in the embedding.
Note that this is possible because their clockwise order is completed by adding the outer edges of $K$.
Each dart $e$ whose head is an outer endpoint of $\widehat{K}$ has $\successor(e)=\predecessor(e)=\nil$ because outer endpoints are treated as boundary ends.
It is also easy to see that it satisfies (M1)--(M4) and (M6).
Especially, it does not satisfy (M5); let $e_0$ be a dart whose tail is an outer endpoint, then
$\successor(e_0) \neq \nil$, but neither case (i) nor case (ii) occurs.

A multi-boundary configuration does not necessarily admit a dart representation satisfying (M6), because a vertex may appear more than once on the closed walks bounding the faces in $F_R(K)$.
\paragraph{Homomorphism}
A \emph{homomorphism} from a dart representation $Z$ to another dart representation $Z'$ is a mapping $\phi$ from vertices and darts of $Z$ to vertices and darts in $Z'$ such that for any dart $e$ in $Z$:
\begin{itemize}
\item $\head(\phi(e))=\phi(\head(e))$;
\item $\reverse(\phi(e))=\phi(\reverse(e))$;
\item if $\successor(e)\neq\nil$, then $\successor(\phi(e))=\phi(\successor(e))$;
\item if $\predecessor(e)\neq\nil$, then $\predecessor(\phi(e))=\phi(\predecessor(e))$.
\end{itemize}
Note that, if we have a \emph{connected} dart representation $Z$ satisfying (M6), then its homomorphism $\phi$ into a dart representation $Z'$ is uniquely determined as soon as we have decided the image $e' \in D(Z')$ of a dart $e \in D(Z)$.
The reason is that, starting with $e$, we can trace all of $Z$ using only pointers to heads, reverse darts, successors, and predecessors, and a homomorphism must follow the corresponding pointers in the image.

For a homomorphism from outer extensions of multi-boundary configurations to another dart representation, we require one special condition: \emph{for every vertex $v \in V(K)$, the image vertex $\phi(v)$ has degree $\delta_K(v)$ in the target dart representation.}
This condition ensures that the prescribed number of incident edges at each vertex of $K$ is respected in the target.
In this understanding, we define a homomorphism from a configuration $K$ to another dart representation $Z$ as the restriction of a homomorphism from $\widehat{K}$ to $Z$ on vertices or darts in $K$.
This is a homomorphism from a configuration to $G^*$ ensured by Theorem \ref{thm:K-homomorphism-G}.

\section{Discharging}
\label{sect:discharging}
\showlabel{sect:discharging}
Recall that $G^*$ is the triangulation with digons in Lemma \ref{lem:triangulation}.
The goal of this section is to prove Theorem \ref{thm:K-homomorphism-G} below.
We start by assigning to each vertex $v \in V(G^*)$ the value 
\[
    T_0(v) := 10(6 - d(v)),
\]
which we call the \emph{initial charge at $v$}. It follows from Euler's formula (see Lemma \ref{lem:60} below) that 
\[
    \sum_{v \in V(G^*)} T_0(v) \geq 60.
\]
We redistribute the initial charges by moving charge along edges. In moving charges, we use \emph{rules} defined as follows.
Note that \emph{a near-triangulation with digons} is a plane graph whose faces are triangles or digons except for \emph{the outer face}.
\begin{dfn}[rule]
\label{dfn:rule}
\showlabel{dfn:rule}
      A \emph{rule} is a seven-tuple $R = (G(R), \delta^-_R, \delta^+_R, r(R), s(R), t(R), e(R))$, such that
      \begin{enumerate}
          \renewcommand{\labelenumi}{(\roman{enumi})}
          \item $G(R) = (V(R), E(R))$ is a near-triangulation with digons and for each $v \in V(R)$, $G(R) - v$ is connected.
          \item $\delta^-_R \colon V(R) \rightarrow \mathbb{Z}_+$, $\delta^+_R \colon V(R) \rightarrow \mathbb{Z}_+ \cup \{\infty\}$, such that $d_{G(R)}(v) < \delta^-_R(v) \leq \delta^+_R(v)$ for each $v \in V(R)$.
          \item $r(R)$ is a positive integer.
          \item $e(R) \in E(R)$ is an edge between $s(R), t(R) \in V(R)$.
      \end{enumerate}
\end{dfn}
A rule $R$ is represented as a pseudo-triangulation with digons.
Every edge of $G(R)$ is converted into two darts, which are reversed with respect to each other.
For each vertex $v \in V(R)$, the darts with head $v$ are ordered according to the clockwise order of the corresponding edges around $v$ in the embedding of $G(R)$.
The $\successor$ and $\predecessor$ pointers are assigned according to this cyclic order, except that the order is cut at the outer face.
More precisely, for a dart $e$, if the face lying to the left of $e$, when $e$ is traversed from $\tail(e)$ to $\head(e)$, is the outer face, $\successor(e)$ is $\nil$.
Similarly, if the face lying to the right of $e$ is the outer face, $\predecessor(e)$ is $\nil$.
It is easy to see that this dart representation satisfies (M1)--(M5).
Since $G(R)-v$ is connected for every $v \in V(R)$, cutting the clockwise order at the outer face gives a single incidence list for each vertex. Hence (M6) also holds.

The distinguished edge $e(R)$ has endpoints $s(R)$ and $t(R)$.
Among the two darts corresponding to $e(R)$, we distinguish the dart with tail $s(R)$ and head $t(R)$.
This dart represents the direction in which the rule moves charge, from $s(R)$ to $t(R)$.
We denote this dart by $e_R$.

We say that \emph{a rule $R$ applies to a dart $e$ in $G^*$} if there exists a homomorphism $\phi$ from $R$ to $G^*$ such that $\phi(e_R) = e$ and every vertex $v_R \in V(R)$ satisfies $\delta^-_R(v_R) \leq d_{G^*}(\phi(v_R)) \leq \delta_R^+(v_R)$.
Thus, when a rule $R$ is applied, the degree of vertices in $G^*$ is in the range defined by $\delta^-_R, \delta^+_R$.
It is equivalent to say: first choose $\delta_R(v_R) \in [\delta_R^-(v_R), \delta_R^+(v_R)]$ for each vertex $v_R$ and $v_R$ satisfies $d_{G^*}(\phi(v_R))=\delta_R(v_R)$.

\subfile{tikz/rules-RSST}

In our definition of rules, we add the edge $e(R)$ to the definition in \cite{RSST} because there may be more than one edge between $s(R)$ and $t(R)$ in our setting, due to digons.
In Figure \ref{fig:rules-RSST}, we show the rules used to prove 4CT in \cite{RSST}.
These rules are also used in \cite{doublecross}.
For these rules, there is only one edge between $s(R)$ and $t(R)$, so $e(R)$ is automatically specified.
As a set of rules used in this paper, we use the union of the rules in Figure \ref{fig:rules-RSST} and the rules in Figure \ref{fig:deg5,6-rule-multiedge}.
The symbol $\mathcal{R}$ denotes this set of rules.
We denote the 32 rules in \cite{RSST} by $R(i)$ ($1 \leq i \leq 32$) in the order listed.
Three rules in Figure \ref{fig:deg5,6-rule-multiedge} are denoted by $R(i)$ ($33 \leq i \leq 35$) in the order listed.
Note that, for a non-symmetric rule $R(i)$ around $e(R)$, its mirror rule is also included in $\mathcal{R}$.
We denote each rule drawn in Figure \ref{fig:rules-RSST}, \ref{fig:deg5,6-rule-multiedge} by $R(i, \textsf{L})$, and its mirror is denoted by $R(i, \textsf{R})$.
Specifically, $R(1)$ and $R(12)$ are symmetric, so they constitute a single rule.
$R(10)$, $R(31)$ are special rules.
They have a unique vertex $v$ with $\delta^-(v)=5, \delta^+(v)=\infty$, and if $\delta^-(v)=\delta^+(v)=5$, they become symmetric rules; otherwise, they are non-symmetric.
Thus, each of $R(10)$, $R(31)$ consists of 3 rules; one is a symmetric rule with $\delta^-(v)=\delta^+(v)=5$, and the others are non-symmetric rules that mirror each other when $\delta^-(v)=6, \delta^+(v)=\infty$.
The files representing the discharging rules in $\mathcal{R}$ are available on GitHub\footnote{\url{https://github.com/three-edge-coloring-apex-cubic-graphs/discharging-rules}}.

\begin{figure}[htbp]
    \centering
    \includegraphics[width=8cm]{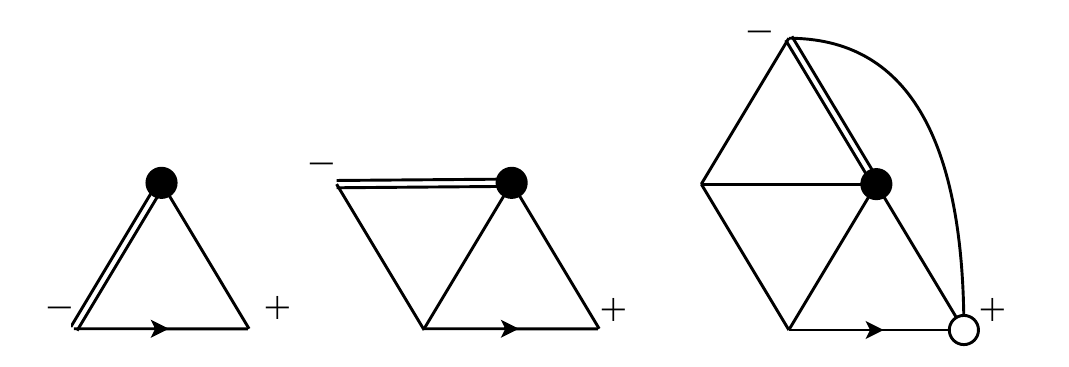}
    \caption{The new rules $R(33)$, $R(34)$, and $R(35)$ in $\mathcal{R}$. A double line represents a pair of parallel edges that bound a digon.}
    \label{fig:deg5,6-rule-multiedge}
\end{figure}

For a dart $e$ in $G^*$, let $\mathcal{R}_e \subseteq \mathcal{R}$ be the set of rules applied to $e$, and define $r(e)$ by the sum of $r(R)$ for $R \in \mathcal{R}_e$.
This value is also called the amount of charge \emph{transferred along $e$}.
For a vertex $u \in V(G^*)$, we set
\[
  T(u) := T_0(u) + \sum_{\substack{(e,f) \\ \head(e)=u, f=\reverse(e)}} r(e) - r(f)
\]
The obtained value is said to be the \emph{final charge} of $u$.
For a vertex $v \in V(G^*)$, the symbol $B_k(v)$ denotes the subgraph induced by the vertices at a distance of at most $k$ from $v$.
The final charge of $v$ is determined by the degrees of vertices in $B_2(v)$ and digons in $B_2(v)$ since for every rule $R \in \mathcal{R}$, every vertex $v \in V(R)$ is at a distance of at most 2 from $s(R)$ and $t(R)$ in $G(R)$.

\begin{lem}
\label{lem:60}
\showlabel{lem:60}
$\sum_{u \in V(G^*)} T(u) \geq 60$. Moreover, $\sum_{u \in V(G^*)} T(u) = 60$ if the number of digons in $G^*$ is three. Otherwise, its value is $80$.
\end{lem}

\begin{proof}
By applying the discharging rules, the total charge amount remains the same. Thus, $\sum_{u \in V(G^*)} T(u) = \sum_{u \in V(G^*)} 10(6-d(u))$.
Since $G^*$ is a triangulation with two or three digons, $2|E(G^*)|$ equals  $3|F(G^*)| - 3$ or $3|F(G^*)| - 2$. Combining Euler's formula ($|V(G^*)| - |E(G^*)| + |F(G^*)| = 2$) implies that $|E(G^*)|$ equals $3|V(G^*)| - 3$ or $3|V(G^*)| - 4$. We have:
 \[
 \sum_{u \in V(G^*)} T(u) = \sum_{u \in V(G^*)} 10(6-d(u)) = 60|V(G^*)| - 20|E(G^*)|.
 \]
 Thus, when $G^*$ has three digons, the total sum is $60$, otherwise $80$.
\end{proof}
Lemma \ref{lem:60} ensures that the sum of the final charges over all vertices is at least 60, which is positive.
In the standard discharging method, one aims to show that, in the vicinity of vertices with positive final charges, the desired structure must appear. 
In our context, this desired structure corresponds to the existence of a homomorphism from a configuration in a specified set $\mathcal{K}$, as described in Theorem \ref{thm:K-homomorphism-G}, and the vicinity corresponds to $B_2(v)$.
While we follow this general approach, we modify it for vertices near digons.
Specifically, for a vertex with $B_2(v)$ containing digons, we require a weaker condition. There exists an absolute constant $c > 0$ such that if the final charge exceeds $c$, then there exists a desired homomorphism.
By Lemma \ref{lem:60}, the total additional charge allowed per one digon should be less than 20.
We choose the value $c$ appropriately to satisfy this requirement.
The precise values for $c$ are detailed later.

Here, we introduce the set of configurations used in this paper.
We denote it by $\mathcal{K}$.
$\mathcal{K}$ consists of 915 normal configurations.
It is constructed based on the set of configurations used in \cite{doublecross}, but we have to add configurations with digons.
This set plays an important role, as Theorem \ref{thm:K-homomorphism-G} and Theorem \ref{thm:hom-imply-reducible} suggest.
Although we do not draw all of them in this paper, we show the first 19 configurations in Figure \ref{fig:conf-for-deg5,6}.
The files representing the configurations in $\mathcal{K}$ are available on GitHub\footnote{\url{https://github.com/three-edge-coloring-apex-cubic-graphs/configurations}}.
First, we show the following lemma:
\begin{lem}
\label{lem:deg5,6}
\showlabel{lem:deg5,6}
    Let $v$ be a vertex in $G^*$ with degree 5 or 6.
    If\/ $T(v) > 0$, then there exists a homomorphism from some $K \in \mathcal{K}$ to $B_2(v)$.
\end{lem}
The proof of Lemma \ref{lem:deg5,6} is provided in Section \ref{subsect:final-charge-deg5,6}.
The proof is computer-free.
Next, we have the following lemma.
\begin{lem}
\label{lem:deg12}
\showlabel{lem:deg12}
    Let $v$ be a vertex of degree at least 12 in $G^*$.
    If\/ $T(v) > 0$, then there exists a homomorphism from some $K \in \mathcal{K}$ to $B_2(v)$.
\end{lem}
This lemma is proved by using the following result.
\begin{lem}
\label{lem:send5} 
\showlabel{lem:send5}
    Let $e$ be a dart in $G^*$ with $u=\head(e),v=\tail(e)$.
    If there does not exist a homomorphism from any configuration $K \in \mathcal{K}$ to $B_2(v) \cap B_2(u)$, then $r(e) \leq 5$ holds, i.e., the amount of charge transferred by $\mathcal{R}$ along $e$ is at most 5.
\end{lem}
The proof of Lemma \ref{lem:send5} will be given in Section \ref{subsect:free-homomorphism}.
The proof of this lemma is computer-assisted.
Lemma \ref{lem:send5} implies Lemma \ref{lem:deg12} as follows:
\begin{proof}[Proof of Lemma \ref{lem:deg12}, assuming Lemma \ref{lem:send5}]
    We prove by contraposition.
    By Lemma \ref{lem:send5}, for each dart $e$ with head $v$, $r(e) \leq 5$ holds.
    Then, $T(v) \leq 10 \cdot (6 - d(v)) + 5 \cdot d(v)= 60 - 5 \cdot d(v) \leq 0$.
\end{proof}
Finally, we show a similar lemma for vertices of degree $7,8,9,10,11$.
Here, we determine the value of the final charge.
We assign a specific constant to $v$ from each digon in $B_2(v)$, depending on the degrees and the distance from the endpoints of this digon.
The value for $v$ is then defined as the sum of these constants for all digons contained in $B_2(v)$.
Let $u_1,u_2$ be the endpoints of a digon in $B_2(v)$, and let $\textsf{dist}(u_i,v)$ denote the distance from $u_i$ to $v$ in $G^*$ for $i=1,2$.
We may assume that $\textsf{dist}(u_1, v) \leq \textsf{dist}(u_2, v)$, and if equal, the degree of $u_1$ is less than or equal to the degree of $u_2$ (i.e., $d_{G^*}(u_1) \leq d_{G^*}(u_2)$).
The constant assigned to $v$ from this digon is defined as follows:
\begin{enumerate}[label=(\roman*), ref=(\roman*)]
    \item \label{constaint-(i)} The case where $u_1=v$, $\textsf{dist}(u_2,v)=1$:
    the constant assigned to $v$ is $4$ if $7 \le d_{G^*}(u_1) \le 10$, and $0$ otherwise.

    \item \label{constaint-(ii)} The case where $\textsf{dist}(u_1,v)=1$, $\textsf{dist}(u_2,v)=1$:
    the constant assigned to $v$ is $5$ if $d_{G^*}(u_1)=d_{G^*}(u_2)=5$, and
    $3$ if $d_{G^*}(u_1)=5$, $d_{G^*}(u_2)=6$, and
    $2$ otherwise.

    \item \label{constaint-(iii)} The case where $\textsf{dist}(u_1,v)=1$, $\textsf{dist}(u_2,v)=2$:
    the constant assigned to $v$ is $4$ if $d_{G^*}(u_1)=d_{G^*}(u_2)=5$, and
    $2$ if $d_{G^*}(u_1)=5$, $6 \leq d_{G^*}(u_2) \leq 7$ or $d_{G^*}(u_2)=5$, $6 \leq d_{G^*}(u_1) \leq 7$, and
    $0$ otherwise.

    \item \label{constaint-(iv)} The case where $\textsf{dist}(u_1,v)=2$, $\textsf{dist}(u_2,v)=2$: the constant assigned to $v$ is $0$.
\end{enumerate}
For a vertex $v \in G^*$, let $c(v)$ be the sum of these constants for all digons contained in $B_2(v)$.
Note that if no digon is contained in $B_2(v)$, then $c(v)=0$.
\begin{lem}
\label{lem:deg7-11}
\showlabel{lem:deg7-11}
    Let $v$ be a vertex in $G^*$ with degree between $7$ and $11$.
    If $T(v) > c(v)$, then there exists a homomorphism from $K \in \mathcal{K}$ to $B_2(v)$.
\end{lem}
The proof of this lemma is shown in Section \ref{sect:cartwheel}.
The proof of this lemma requires an extensive case analysis, which is checked by computer.

We leave the proof of Lemmas \ref{lem:deg5,6}, \ref{lem:send5}, and \ref{lem:deg7-11} for later, but we now show that the combination of these lemmas leads to the proof of Theorem \ref{thm:K-homomorphism-G} below. In Section \ref{sect:hom-conf}, we show that this desired configuration in $G^*$ implies the existence of a reducible multi-boundary island in $G$ in Theorem \ref{thm:hom-imply-reducible}. This completes the proof of Theorem \ref{mainth}.

\begin{theorem}
\label{thm:K-homomorphism-G}
\showlabel{thm:K-homomorphism-G}
    For some configuration $K \in \mathcal{K}$, there exists a homomorphism from $K$ to $G^*$.
\end{theorem}
\begin{proof}[Proof of Theorem \ref{thm:K-homomorphism-G}, assuming Lemmas \ref{lem:deg5,6}, \ref{lem:send5}, and \ref{lem:deg7-11}]
    \hfill
    
    We show the following claim:
    \begin{itemize}
        \item[(1)] $\sum_{v \in V(G^*)} T(v) > \sum_{v \in V(G^*)} c(v)$.
    \end{itemize}
    Before we show (1), we see that (1) implies this Theorem.
    By (1), at least one vertex $v \in V(G^*)$ has $T(v) > c(v) \geq 0$.
    Depending on the degree of $v$, we apply either of Lemmas \ref{lem:deg5,6}, \ref{lem:deg12}, and \ref{lem:deg7-11}.
    This implies the claim of Theorem  \ref{thm:K-homomorphism-G}.
    
    It remains to show (1).
    $c(v)$ is defined by the sum of constants for each digon contained in $B_2(v)$, so a digon in $G^*$ adds some constant to $c(v)$ for vertices near it.
    We calculate the sum of added constants by one digon and ensure that it is less than 20.
    It implies (1) since, by Lemma \ref{lem:60}, $\sum_{v \in V(G^*)} T(v) \geq 60$ and the number of digons in $G^*$ is at most 3.

    Let $u_1,u_2$ be the endpoints of a digon.    
    A vertex $v$ that forms a facial triangle with $u_1,u_2$ is of distance 1 from both $u_1$ and $u_2$.
    No other vertex is of distance 1 from both $u_1$ and $u_2$ since it contradicts the connectivity of $G^*$ in Lemma \ref{lem:triangulation}.
    Thus, there are only 2 vertices such that $\textsf{dist}(u_1, v)=\textsf{dist}(u_2, v)=1$.
    Similarly there are $d_{G^*}(u_1)+d_{G^*}(u_2)-8$ vertices such that $\textsf{dist}(u_1, v)=1, \textsf{dist}(u_2, v)=2$ or $\textsf{dist}(u_1, v)=2, \textsf{dist}(u_2, v)=1$.

    We divide into the following cases, depending on the degrees of $u_1,u_2$.
    We assume w.l.o.g., that $d_{G^*}(u_1) \leq d_{G^*}(u_2)$.
    \\
    \textbf{Case 1:} when $d_{G^*}(u_1)=d_{G^*}(u_2)=5$, by considering \ref{constaint-(ii)}, \ref{constaint-(iii)}, the distributed sum is at most $5 \cdot 2 + 4 \cdot 2 = 18$.
    \\
    \textbf{Case 2:} when $d_{G^*}(u_1)=5, d_{G^*}(u_2)=6$, by considering \ref{constaint-(ii)}, \ref{constaint-(iii)}, the distributed sum is at most $3 \cdot 2 + 2 \cdot 3 = 12$.
    \\
    \textbf{Case 3:} when $d_{G^*}(u_1)=5, d_{G^*}(u_2)=7$, by considering \ref{constaint-(i)}, \ref{constaint-(ii)}, \ref{constaint-(iii)}, the distributed sum is at most $4 \cdot 1 + 2 \cdot 2 + 2 \cdot 4=16$.
    \\
    \textbf{Case 4:} when $d_{G^*}(u_1)=5, d_{G^*}(u_2)\geq8$, by considering \ref{constaint-(i)}, \ref{constaint-(ii)}, the distributed sum is at most $4 \cdot 1 + 2 \cdot 2=8$.
    \\
    \textbf{Case 5:} when $d_{G^*}(u_1)\geq6, d_{G^*}(u_2)\geq6$, by considering \ref{constaint-(i)}, \ref{constaint-(ii)}, the distributed sum is at most $4 \cdot 2 + 2 \cdot 2 = 12$.

    This shows the sum per one digon is less than 20, implying (1).
\end{proof}
It remains to show Lemmas \ref{lem:deg5,6}, \ref{lem:send5}, and \ref{lem:deg7-11}.
In the subsequent subsections, we prove each of them.

\subsection{The final charge of vertices of degree 5 or 6}
\label{subsect:final-charge-deg5,6}
\showlabel{subsect:final-charge-deg5,6}
In this subsection, we prove Lemma \ref{lem:deg5,6}.
The set of configurations drawn in Figure \ref{fig:conf-for-deg5,6}, including their mirror image, is included in $\mathcal{K}$.
We use them to prove Lemma \ref{lem:I=O for degree 5,6} below.
The configurations in Figure \ref{fig:conf-for-deg5,6} are denoted by $K(i)$ for $1 \leq i \leq 19$ in the natural order listed.
Lemma \ref{lem:deg5,6} is proved by using the following lemma.
Refer to the similarity of Lemma (4.5) in \cite{RSST}.

\begin{figure}[htbp]
    \centering
    \includegraphics[width=0.9\linewidth]{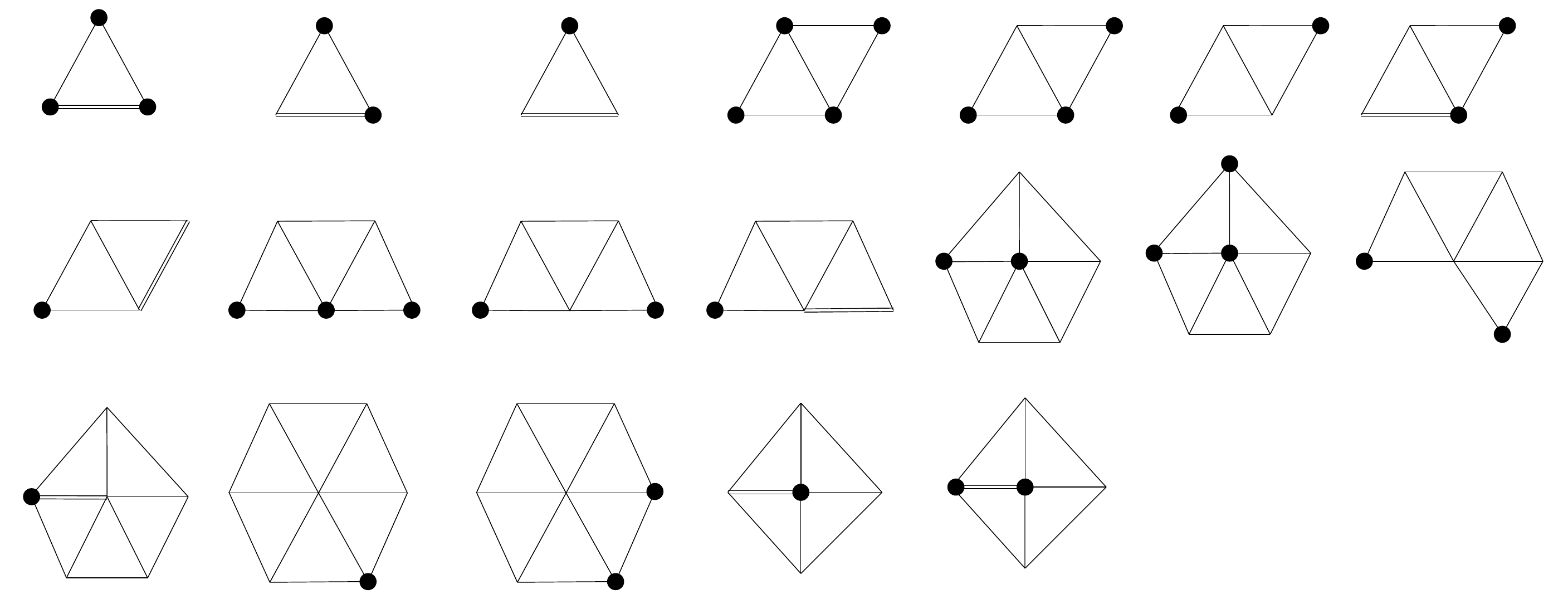}
    \caption{The subset of configurations $K(1)$--$K(19)$ in $\mathcal{K}$.}
    \label{fig:conf-for-deg5,6}
\end{figure}

\begin{lem}
\label{lem:I=O for degree 5,6}
\showlabel{lem:I=O for degree 5,6}
    Let $v$ be a vertex of degree 5 or 6 in $G^*$.
    Let $I(i), O(i)$ be the amount of charge $v$ receives (resp., sends) by rule R($i$) for $1 \leq i \leq 35$.
    If there does not exist any homomorphism from $K \in \mathcal{K}$ to $B_2(v)$, then
    \begin{itemize}
        \item[(1)] $I(1)=O(2)+O(3)+O(33)$,
        \item[(2)] $I(3)=O(4)$,
        \item[(3)] $I(4)=O(5)+O(6)$,
        \item[(4)] $I(5)=O(7)$,
        \item[(5)] $I(33)=O(34)$, and
        \item[(6)] $I(34)=O(35)$ holds.
    \end{itemize}
\end{lem}
\begin{proof}
    Let $e_1,\ldots,e_d$ ($d=5$ or $6$) be darts whose head is $v$ in clockwise order in the embedding.
    Recall that for each non-symmetric rule in Figure \ref{fig:rules-RSST} or \ref{fig:deg5,6-rule-multiedge}, the drawing defines the rule $R(i, \textsf{L})$, and its mirror defines $R(i, \textsf{R})$.
    For such a rule $R(i)$, $I(i)$ and $O(i)$ are computed by the amount transferred by both $R(i, \textsf{L})$ and $R(i, \textsf{R})$. 
    We also note that, from Lemma \ref{lem:triangulation}, every vertex in $G^*$ is incident to at most one digon, and the minimum degree is $5$.
    
    \textbf{Proof of (1):}
    Let $e_i$ be a dart with $d_{G^*}(\tail(e_i))=5$.
    Then, $R(1)$ applies to $e_i$.
    For the clockwise direction, we show that either of $R(2, \textsf{L}), R(3, \textsf{L}), R(33, \textsf{L})$ applies to the reverse of some dart whose head is $v$.
    For the counter-clockwise direction, we show that either of $R(2, \textsf{R}), R(3, \textsf{R}), R(33, \textsf{R})$ applies to the reverse of some dart whose head is $v$.
    This shows (1) since the discharge of $R(1)$ is $2$, while the discharge of $R(2),R(3),R(33)$ is $1$.
    Since the proof is the same, we only consider the clockwise direction.

    First, we consider the case where $\tail(e_i) \neq \tail(e_{i+1})$.
    If $d_{G^*}(\tail(e_{i+1}))\geq 7$, then $R(2, \textsf{L})$ applies to $\reverse(e_{i+1})$.
    Thus, assume that this is not the case, then $d_{G^*}(\tail(e_{i+1})) \leq 6$.
    If $\tail(e_{i+1})=\tail(e_{i+2})$, then a homomorphism from $K(1)$, $K(2)$ or $K(3)$ to $B_2(v)$ exists.
    If $\tail(e_{i+1}) \neq \tail(e_{i+2})$ and $d_{G^*}(\tail(e_{i+2})) \geq 6$, then $R(3, \textsf{L})$ applies to $\reverse(e_{i+2})$.
    If $\tail(e_{i+1}) \neq \tail(e_{i+2})$ and $d_{G^*}(\tail(e_{i+2})) = 5$, then a homomorphism from $K(4)$, $K(5)$, or $K(6)$ to $B_2(v)$ exists.

    Second, we consider the case where $\tail(e_i) = \tail(e_{i+1})$, then $\tail(e_{i+1}) \neq \tail(e_{i+2})$.
    If $d_{G^*}(\tail(e_{i+2})) = 5$, a homomorphism from $K(1)$ or $K(2)$ to $B_2(v)$ exists.
    If $d_{G^*}(\tail(e_{i+2})) \geq 6$, $R(33, \textsf{L})$ applies to $\reverse(e_{i+2})$.
    Thus (1) holds.
    
    \textbf{Proof of (2):}
    Consider the case where $R(3, \textsf{L})$ applies to $e_i$.
    If $\tail(e_{i+1})=\tail(e_{i+2})$, then a homomorphism from $K(2)$, $K(3)$, $K(7)$, or $K(8)$ to $B_2(v)$ exists.
    Then, we consider that $\tail(e_{i+1}) \neq \tail(e_{i+2})$.
    If $d_{G^*}(\tail(e_{i+2})) = 5$, then a homomorphism from $K(5)$, $K(6)$, $K(9)$, or $K(10)$ to $B_2(v)$ exists.
    If $d_{G^*}(\tail(e_{i+2})) \geq 6$, then $R(4, \textsf{L})$ applies to $\reverse(e_{i+2})$.
    
    The case where $R(3, \textsf{R})$ applies to $e_i$ is shown in the same way. In the following proof, we only consider $R(i, \textsf{L})$ for every $i$ since the proof for $R(i, \textsf{R})$ is the same as that for $R(i, \textsf{L})$.
    
    \textbf{Proof of (3):}
    Consider the case where $R(4, \textsf{L})$ applies to $e_i$.
    If $\tail(e_{i+1})=\tail(e_{i+2})$, then a homomorphism from $K(2)$, $K(8)$, or $K(11)$ to $B_2(v)$ exists.
    Then, we consider that $\tail(e_{i+1}) \neq \tail(e_{i+2})$.

    First, we consider the case where $d_{G^*}(\tail(e_{i+1}))=5$.
    If $d_{G^*}(\tail(e_{i+2})) \leq 6$, then a homomorphism from $K(4)$, $K(12)$, or $K(13)$ to $B_2(v)$ exists.
    If $d_{G^*}(\tail(e_{i+2})) \geq 7$, then $R(6, \textsf{L})$ applies to $\reverse(e_{i+2})$.
    
    Second, we consider the case where $d_{G^*}(\tail(e_{i+1}))=6$.
    If $d_{G^*}(\tail(e_{i+2})) = 5$, then a homomorphism from $K(10)$ or $K(14)$ to $B_2(v)$ exists.
    If $d_{G^*}(\tail(e_{i+2})) \geq 6$, then $R(5, \textsf{L})$ applies to $\reverse(e_{i+2})$.
    
    \textbf{Proof of (4):}
    Consider the case where $R(5, \textsf{L})$ applies to $e_i$.
    If $\tail(e_{i+1}) = \tail(e_{i+2})$, a homomorphism from $K(3)$ to $B_2(v)$ exists.
    Then, we assume that $\tail(e_{i+1}) \neq \tail(e_{i+2})$.

    If $\predecessor(\reverse(e_{i+1}))$ forms a digon with its predecessor, then a homomorphism from $K(2)$ or $K(15)$ to $B_2(v)$ exists.
    Then, we assume that this is not the case.
    If $d_{G^*}(\tail(e_{i+2})) \leq 6$, then a homomorphism from $K(5)$, $K(16)$ or $K(17)$ to $B_2(v)$ exists.
    If $d_{G^*}(\tail(e_{i+2})) \geq 7$, then $R(7, \textsf{L})$ applies to $\reverse(e_{i+2})$.
    
    \textbf{Proof of (5):}
    Consider the case where $R(33, \textsf{L})$ applies to $e_i$.
    Since $\tail(e_{i+1})$ is incident to at most one digon, then $\tail(e_{i+1}) \neq \tail(e_{i+2})$.
    If $d_{G^*}(\tail(e_{i+2}))=5$, then a homomorphism from $K(5)$ or $K(7)$ to $B_2(v)$ exists.
    If $d_{G^*}(\tail(e_{i+2})) \geq 6$, then $R(34, \textsf{L})$ applies to $\reverse(e_{i+2})$.
    
    \textbf{Proof of (6):}
    Consider the case where $R(34, \textsf{L})$ applies to $e_i$.
    Since $\tail(e_{i+1})$ is incident to at most one digon, then $\tail(e_{i+1}) \neq \tail(e_{i+2})$.
    If $d_{G^*}(\tail(e_{i+2})) \leq 6$, then a homomorphism from $K(1)$, $K(2)$, $K(18)$, or $K(19)$ to $B_2(v)$ exists.
    If $d_{G^*}(\tail(e_{i+2})) \geq 7$, then $R(35, \textsf{L})$ applies to $\reverse(e_{i+2})$.
\end{proof}

\begin{proof}[Proof of Lemma \ref{lem:deg5,6}]
    Rules in $\mathcal{R}$ except for R($1$)--R($7$), R($33$)-R($35$) do not affect the final charge of $v$ since such a rule $R$ satisfies $\delta_R^-(s(R)) > 6$ and $\delta_R^-(t(R)) > 6$.
    If there is no homomorphism from configurations in $\mathcal{K}$ to $B_2(v)$, Lemma \ref{lem:I=O for degree 5,6} implies the following:
    if the degree of $v$ is $5$, then $T(v)=10+I(1)-O(1)-O(2)-O(3)-O(33)=0$ and,
    if the degree of $v$ is $6$, then $T(v)=I(1)+I(3)+I(4)+I(5)+I(33)+I(34)-O(2)-O(3)-O(4)-O(5)-O(6)-O(7)-O(33)-O(34)-O(35)=0$.
\end{proof}

\subsection{Free combination of discharging rules}
\label{subsect:free-homomorphism}
\showlabel{subsect:free-homomorphism}
In this subsection, we provide a computer-assisted proof of Lemma \ref{lem:send5}.
This proof is analogous to the proof of Lemma 7.2 in \cite{inoue2026four}, and the verification program is nearly the same, except that we have to handle digons. 
We now briefly explain the approach in \cite{inoue2026four}.

Recall that $\mathcal{R}_e \subseteq \mathcal{R}$ is the set of rules applied to $e$.
For every rule $R$ in $\mathcal{R}_e$, a homomorphism from $R$ to $G^*$ exists such that $\phi(e_R)=e$ and the vertices in the image have degrees in the specified degree ranges.
We need to check whether these homomorphisms are realizable. 
We explain how to check it. Below, we present the notation in \cite{inoue2026four}.

\paragraph{Dart representation with degree functions}
We shall consider a dart representation $Z=(V, D)$ equipped with a degree function, $\delta \colon V \to \mathbb{Z}_+$.
For example, $G^*$ is naturally equipped with the degree function determined by its degree.
We also consider a dart representation with degree-range functions $\delta^-, \delta^+ \colon V \to \mathbb{Z}_+ \cup \{\infty\}$.
This represents the set of all dart representations with degree function, where each vertex has a fixed degree in the specified degree range.
Here the notation $(Z,\delta)\in (Z,\delta^-,\delta^+)$ means that the underlying dart representation is fixed and only the degree function varies within the prescribed ranges.
A dart representation of a rule is naturally equipped with its degree-range functions.
Also, we can consider that the dart representation of $G^*$ has degree ranges with both lower and upper bounds equal to $d_{G^*}(v)$.

Besides the basic condition of homomorphisms, a homomorphism $\phi$ from $(Z, \delta)$ to $(Z', \delta')$, i.e., with a degree function, satisfies $\delta'(\phi(v))=\delta(v)$ for every vertex $v \in Z$.
A homomorphism $\phi$ from $(Z, \delta^-, \delta^+)$ to $(Z', \delta^{'-}, \delta^{'+})$, i.e., with degree-range functions, has several variants.
One typical condition is that $[\delta^{'-}(\phi(v)), \delta^{'+}(\phi(v))] \subseteq [\delta^{-}(v), \delta^{+}(v)]$ for every vertex $v \in Z$.
This means that for every $(Z', \delta') \in (Z', \delta^{'-}, \delta^{'+})$, $\phi$ is a homomorphism from some $(Z, \delta) \in (Z, \delta^-, \delta^+)$ to $(Z', \delta')$.
In the following paragraph, we use this typical condition.

\paragraph{Free homomorphisms of pseudo-embeddings with a degree function}
Let $(Z, \delta)$ be a pseudo-embedding with a degree function.
An \emph{identification request} is a pair of darts $(e,f)$ of $Z$.
Let $S=\{(e_1,f_1),\ldots,(e_k,f_k)\}$ be a set of identification requests.
A homomorphism $\phi':(Z, \delta) \to (Z', \delta')$ \emph{respects} $S$ if $\phi(e_i)=\phi(f_i)$ for all $i=1,\ldots,k$.

A homomorphism $\phi^*:(Z, \delta) \to (Z^*, \delta^*)$ respecting $S$ is called a \emph{free homomorphism respecting $S$} if it satisfies the following property:
for every homomorphism $\phi':(Z, \delta) \to (Z', \delta')$ respecting $S$, there exists a homomorphism $\phi^{*'} \colon (Z^*, \delta^*) \to (Z', \delta')$ such that
$\phi'=\phi^{*'} \circ \phi^*$.
In this case, $(Z^*, \delta^*)$ is called the \emph{free homomorphic image of $(Z, \delta)$ respecting $S$}.
It is the most general homomorphic image with respect to $S$.

We can also define a free homomorphism and free homomorphic image of a pseudo-embedding without a degree function in the same way, but it is unnecessary here.

\paragraph{Free homomorphism of pseudo-embeddings with degree-range functions}
A pseudo-embedding with degree-range function $(Z, \delta^-, \delta^+)$ is considered as the set of pseudo-embeddings with degree function, so free homomorphisms and free homomorphic images become sets.
Namely, for a pseudo-embedding with degree-range function $(Z, \delta^-, \delta^+)$ with a set of identification requests $S$, a set $\mathcal{Z}^*$ of pairs of a pseudo-embedding and a homomorphism from $(Z, \delta^-, \delta^+)$ to it is called a \emph{set of free homomorphic images with free homomorphisms} if $(Z^*, \delta^*) \in ((Z^*, \delta^{*-}, \delta{*+}), \phi^*) \in \mathcal{Z}^*$ if and only if $(Z^*, \delta^*)$ is a free homomorphic image of some $(Z, \delta) \in (Z, \delta^-, \delta^+)$ under free homomorphism $\phi^*$.  

\paragraph{Free homomorphisms of pseudo-triangulations with digons with a degree function.}
In this paper, we compute a free homomorphic image of
\begin{itemize}
    \item an outer extension of a multi-boundary configuration (in Section \ref{sect:hom-conf}), or
    \item a pseudo-triangulation with digons, e.g., a rule, or a cartwheel, which is described later.
\end{itemize}
When computing free homomorphic images of pseudo-triangulations with digons, we require the resulting free homomorphic images also to be pseudo-triangulations with digons.
By this requirement, a free homomorphic image is not uniquely determined.
This is because there are multiple choices of the size of a face, i.e., a digon or a triangle.
Thus, when determining the size of a face during computing free homomorphic images, there are multiple choices.

Then, we define a \emph{set} of free homomorphic images of a pseudo-triangulation with digons $(Z, \delta)$ respecting a set of identification requests $S$.
A set $\mathcal{Z}^*$ of pairs of a pseudo-triangulation with digons and a homomorphism from $(Z,\delta)$ to it is \emph{a set of free homomorphic images} if, for every homomorphism $\phi' \colon (Z,\delta) \to (Z',\delta')$ respecting $S$, there exists $((Z^*, \delta^*), \phi^*) \in \mathcal{Z}^*$ such that a homomorphism $\phi^{*'} \colon (Z^*, \delta^*) \to (Z', \delta')$ exists such that $\phi' = \phi^{*'} \circ \phi^*$, and $\mathcal{Z}^*$ is minimal with this property.

We have the following two remarks.
First, $\mathcal{Z}^*$ is a non-empty set if and only if a homomorphism from $(Z, \delta)$ to some pseudo-triangulation with digons respecting $S$ exists.
Second, a target $Z'$ is also restricted to be a pseudo-triangulation with digons.
In this paper, a target $Z'$ is typically a graph embedded in some surface, or a subgraph with its induced embedding.
For instance, it is the plane graph $G^*$, or its subgraph $B_2(v)$.
Then, they are pseudo-triangulations with digons.

\paragraph{Free homomorphisms of pseudo-triangulations with digons with degree-range functions}
For a pseudo-triangulation with digons with degree-range functions $(Z,\delta^-,\delta^+)$, consider a set $\mathcal{Z}^*$ of pairs of a pseudo-triangulation with digons with degree-range functions and a homomorphism from $(Z,\delta^-,\delta^+)$ to it. 
This set $\mathcal{Z}^*$ is \emph{a set of free homomorphic images} if, $(Z^*, \delta^*) \in (Z^*, \delta^{*-}, \delta^{*+})$ for some $((Z^*, \delta^{*-}, \delta^{*+}), \phi^*) \in \mathcal{Z}^*$ if and only if there is a $(Z, \delta) \in (Z, \delta^-, \delta^+)$ such that $((Z^*, \delta^*), \phi^*)$ belongs to a set of free homomorphic images of $(Z, \delta)$ respecting $S$.

We have the same remark as above.
First, $\mathcal{Z}^*$ is a non-empty set if and only if for some $(Z, \delta) \in (Z, \delta^-, \delta^+)$, there exists a homomorphism from $(Z, \delta)$ to some pseudo-triangulation with digons respecting $S$.
Second, a target $Z'$ is also restricted to be a pseudo-triangulation with digons.

\paragraph{Free combination of discharging rules}
To check whether several rules $R_1,\ldots,R_k$ apply to the same dart $e$ in $G^*$, we want to check whether a pseudo-triangulation with digons $(R^*, \delta^*)$ with a distinguished dart $e^*$ exists such that for every $1 \leq i \leq k$, there exists a homomorphism $\phi_i$ from a rule $R_i$ to $R^*$ such that $\phi_i(\vec e_{R_i})=e^*$ and
every vertex $v$ in $R_i$ satisfies that
$\delta^-_{R_i}(v)\leq \delta^*(\phi_i(v))\leq \delta^+_{R_i}(v)$.
If a homomorphism from $(R^*, \delta^*)$ to $G^*$ exists, all rules $R_1,\ldots,R_k$ apply.
Conversely, if all rules $R_1,\ldots, R_k$ apply to the same dart in $G^*$, the union of their images satisfies the same requirement as $R^*$.
We will check whether such an $R^*$ exists as follows.

We take the disjoint union of the pseudo-triangulation with digons of $R_1,\ldots, R_k$, together with their degree-range functions, denoted by $R^{\sqcup}$.
For $R^{\sqcup}$, we impose the identification requests 
$$S=\{(\vec e_{R_1}, \vec e_{R_2}), (\vec e_{R_2},\vec e_{R_3}), \ldots, (\vec e_{R_{k-1}}, \vec e_{R_k})\}.$$
We then compute a set $\mathcal{R}^*$ of free homomorphic images of $R^{\sqcup}$ respecting $S$.
By the remark in the previous paragraph, $\mathcal{R}^*$ is a non-empty set if and only if $R^*$, described as above, exists.
A pseudo-triangulation with digons $R^*$ in $\mathcal R^*$, together with the common image of the distinguished darts and $r(R^*)=\sum_{i=1}^k r(R_i)$, is called a \emph{combined rule} or \emph{a free combination of discharging rules} $R_1,\ldots,R_k$.

\paragraph{The algorithm to compute free homomorphisms.}
In the following, we aim to compute free homomorphisms with degree-range functions.
The algorithm to compute free homomorphic images with free homomorphisms of a pseudo-triangulation is given in \cite{inoue2026four}.
However, we need to compute free homomorphic images and free homomorphisms of pseudo-embeddings, or pseudo-triangulations with digons.
Therefore, we use a modified implementation of the algorithm.
The only place where the implementation differs is the boundary completion step: during the construction of the free homomorphic image $(Z^*, \delta^{*-}, \delta^{*+})$, some boundary vertex $v$ of $Z^*$ may satisfy $\delta^{*-}(v) = \delta^{*+}(v) = d(v)$.
The algorithm completes its incidence list of $v$. 
For the completeness, we present all algorithmic descriptions below and all necessary pseudocode in the appendix, see Section \ref{sec:pseudo1}. 

\begin{itemize}
    \item \emph{The algorithm for pseudo-triangulations in \cite{inoue2026four}.}
    Let $e_\textsf{first},e_\textsf{last}$ be the first, the last dart in the incidence list of $v$.
    If $\tail(\efirst)=\tail(\elast)$, then the following operation creates a loop, so it fails to complete the incidence list, and the algorithm returns \texttt{null}.
    Otherwise, the algorithm completes the incidence list of $v$ by assigning $\predecessor(e_\textsf{first})=e_\textsf{last}, \successor(e_\textsf{last})=e_\textsf{first}$.
    It then adds two darts $f,g$ with $\reverse(f)=g$ and $\reverse(g)=f$, and set the pointers as follows:
    $\head(f)=\tail(\efirst)$,
    $\successor(f)=\nil$,
    $\predecessor(f)=\reverse(\efirst)$,
    $\head(g)=\tail(\elast)$,
    $\successor(g)=\reverse(\elast)$, and
    $\predecessor(g)=\nil$.
    Finally, it updates $\successor(\reverse(\efirst))=f$ and $\predecessor(\reverse(\elast))=g$.

    This is the \textsf{addBoundaryDarts} routine in Algorithm A.4.8 of \cite{inoue2026four}.
    We also write the same pseudocode in Algorithm \ref{alg:add_boundary_darts}.

    \item \emph{The algorithm for pseudo-embeddings.}
    The algorithm completes the incidence list of $v$ by only assigning $\predecessor(e_\textsf{first})=e_\textsf{last}, \successor(e_\textsf{last})=e_\textsf{first}$.
    The corresponding pseudocode is given in the \textsf{boundaryCompletions} routine in Algorithm \ref{alg:boundary_completions_normal}.
    
    \item \emph{The algorithm for pseudo-triangulations with digons.}
    The difference from a pseudo-triangulation is that $e_\textsf{first}, e_\textsf{last}$ may form a digon in a target $Z'$.
    We complete the incidence list of $v$ in two ways.
    The first one is the same as the above algorithm for pseudo-triangulations.
    For the second way, we try to identify $\tail(\elast)$ with $\tail(\efirst)$.
    First, if the degree intersection $[\delta^{*-}(\tail(\efirst)), \delta^{*+}(\tail(\efirst))] \cap [\delta^{*-}(\tail(\elast)), \delta^{*+}(\tail(\elast))]$ is empty, we cannot identify them, so returns \texttt{null}.
    Otherwise, the algorithm completes the incidence list of $v$ by assigning $\predecessor(\efirst)=\elast, \successor(\elast)=\efirst$, and link the incidence list of $\tail(\elast)$ and $\tail(\efirst)$ by calling \textsf{linkIncidenceListEnds} routine with input $\eufirst=\reverse(\elast), \ewlast=\reverse(\efirst)$.
    This routine proceeds as follows.
    Let $u$ be $\head(\eufirst)$, and $w$ be $\head(\ewlast)$.
    First, assign $\predecessor(\eufirst)=\ewlast$, and $\successor(\ewlast)=\eufirst$.
    If $u=w$, the algorithm terminates here.
    Otherwise, for every dart whose head is $u$, we change its head to $w$ and remove $u$ from the vertex set.
    The degree range of $w$ is replaced by the intersection $[\delta^-(u), \delta^+(u)] \cap [\delta^-(w), \delta^+(w)]$.
    This routine is described in Algorithm \ref{alg:link_indidence_list_ends}.
    We discard it if the operation creates a loop.

    The corresponding pseudocode is given in the \textsf{identifyNeighbors} routine in Algorithm \ref{alg:identify_neighbors}.
    The overall algorithm for the boundary completion step is presented in the \textsf{boundaryCompletions} routine in Algorithm \ref{alg:boundary_completions_digons}.
\end{itemize}

In the other parts of the algorithm, as in \cite{inoue2026four}, it does not use the fact that all faces are triangles, so we use the same algorithm.
Namely, within the algorithm for computing free homomorphic images of pseudo-triangulation with degree-range functions in \cite{inoue2026four}, we replace the \textsf{addBoundaryDarts} routine with the corresponding \textsf{boundaryCompletions} routine, depending on whether the target is a pseudo-embedding or a pseudo-triangulation with digons.
Specifically, the \textsf{addBoundaryDarts} routine is called once inside the \textsf{fixSingleDegreeIssue} routine in Algorithm A.4.7, so the replacement occurs only once.

For pseudo-triangulations with digons, we have two choices for completing the incidence list, depending on the face size.
This makes a free homomorphic image with degree functions non-unique.

The proof in \cite{inoue2026four} with the above observation leads to the following lemma.

\begin{lem}
\label{lem:free-hom}
\showlabel{lem:free-hom}
    Given a pseudo-embedding $(Z, \delta^-, \delta^+)$ and a set of identification requests $S$,
    the modified algorithm as above computes a set $\mathcal{Z}^*$ of free homomorphic images of $(Z, \delta^-, \delta^+)$ respecting $S$.
\end{lem}

\begin{lem}
\label{lem:free-hom-tri-digon}
\showlabel{lem:free-hom-tri-digon}
    Given a pseudo-triangulation with digons $(Z, \delta^-, \delta^+)$ and a set of identification requests $S$,
    the modified algorithm as above computes a set $\mathcal{Z}^*$ of free homomorphic images of $(Z, \delta^-, \delta^+)$ respecting $S$, where the targets are restricted to pseudo-triangulations with digons.
\end{lem}

\paragraph{Enforce single digon incidence}
By Lemma \ref{lem:triangulation}, every vertex in $G^*$ is incident to at most one digon.
Consider a pseudo-embedding $(Z, \delta^-, \delta^+)$ such that some vertex is incident to at least two digons, and a homomorphism from $(Z, \delta^-, \delta^+)$ to $G^*$ exists.
Since the images of two digons are also incident to the same vertex in $G^*$, these two digons in $Z$ must be mapped to the same digon in $G^*$.

We explain it in terms of darts.
Consider a dart $e_0$ with
$\successor(e_0) \neq \nil$ and $e_{1}=\reverse(\successor(e_0))$ with $\successor(e_1) \neq \nil$ and 
$e_{2}=\reverse(\successor(e_1))$.
If $e_0=e_2$, then $e_0,e_1$ bound a digon.
For a vertex $v$, assume that there exist two distinct such darts whose head is $v$, say $e_0, f_0$.
These darts must be mapped to the same dart in $G^*$ since $\head(e_0)=\head(f_0)$, so their images have the same head in $G^*$, and hence they are equal.
From this observation, a homomorphism from $(Z, \delta^-, \delta^+)$ to $G^*$ respects $\{(e_0, f_0)\}$.
In the following procedure, we convert $(Z, \delta^-, \delta^+)$ to a set $\mathcal{Z}^*$ of pseudo-embeddings with a homomorphism from $(Z, \delta^-, \delta^+)$ to it.
We ensure that for a pseudo-embedding in $\mathcal{Z}^*$, no vertex is incident to more than one digon.
We also ensure that for every $(Z, \delta) \in (Z, \delta^-, \delta^+)$, if there exists a homomorphism $\phi'$ from $(Z, \delta)$ to $G^*$, then there exists $((Z^*, \delta^{*-}, \delta^{*+}), \phi^*) \in \mathcal{Z}^*$ and $(Z^*, \delta^*) \in (Z^*, \delta^{*-}, \delta^{*+})$ such that $\phi' = \phi^{*'} \circ \phi^*$, where $\phi^{*'}$ is a homomorphism from $(Z^*, \delta^*)$ to $G^*$ and
$\mathcal{Z}^*$ is minimal with this property.
We call this set \emph{a set of free homomorphic images enforcing single digon incidence}.

We explain the process to compute this set.
The algorithm makes an empty queue, and pushes $((Z, \delta^-, \delta^+), \allowbreak \id{V(Z) \cup D(Z)})$, where $\id{V(Z) \cup D(Z)}$ is an identity map of vertices and darts in $Z$.
Each element in the queue consists of a pseudo-embedding with degree-range functions together with a homomorphism from the original $(Z, \delta^-,\delta^+)$ to it.
While the queue is not empty, pop one element $((\tilde{Z}, \tilde{\delta}^-, \tilde{\delta}^+), \tilde{\phi})$ from the queue.
If every vertex $v$ in $\tilde{Z}$ is incident to at most one digon, we push it into the final resulting set $\mathcal{Z}^*$.
Otherwise, some vertex $v$ in $\tilde{Z}$ is incident to at least two digons.
We choose two distinct darts $e_0, f_0$ whose head is $v$ as above.
The pseudocode corresponding to the algorithm to take two such darts is presented in Algorithm \ref{alg:two_digons_incident_with_same_vertex}.
We compute a set $\tilde{\mathcal{Z}}^*$ of free homomorphic images of $(\tilde{Z}, \tilde{\delta}^-, \tilde{\delta}^+)$ respecting $\{(e_0, f_0)\}$. 
For each $((\tilde{Z}^*, \tilde{\delta}^{*-}, \tilde{\delta}^{*+}), \tilde{\phi}^*) \in \tilde{\mathcal{Z}}^*$, we push $(\tilde{Z}^*, \tilde{\delta}^{*-}, \tilde{\delta}^{*+})$ paired with the homomorphism $\tilde{\phi}^* \circ \tilde{\phi}$ into the queue.
This algorithm finally terminates since every time we take free homomorphic images, the number of darts decreases.
The pseudocode of this overall algorithm is presented in Algorithm \ref{alg:enforce_single_digon_incidence}.

The following lemma follows from Lemma \ref{lem:free-hom} and the fact that two darts $e_0,f_0$ described above are mapped to the same dart.
\begin{lem}
\label{lem:pseud-embed-single-digon}
\showlabel{lem:pseud-embed-single-digon}
    Given a pseudo-embedding with degree-range functions $(Z, \delta^-, \delta^+)$, the above algorithm computes a set of free homomorphic images enforcing single digon incidence.
\end{lem}

A set of free homomorphic images enforcing single digon incidence can also be considered for pseudo-triangulations with digons.
The algorithm to compute this set is the same as that for pseudo-embeddings, except that whenever the procedure calls the subroutine for free homomorphic images, we use the version for pseudo-triangulations with digons instead of the version for pseudo-embeddings.
The following lemma similarly follows from Lemma \ref{lem:free-hom-tri-digon}.
\begin{lem}
\label{lem:pseud-tri-single-digon}
\showlabel{lem:pseud-tri-single-digon}
    Given a pseudo-triangulation with digons with degree-range functions $(Z, \delta^-, \delta^+)$, the above algorithm computes a set of free homomorphic images enforcing single digon incidence.
\end{lem}

We often compute a set of free homomorphic images respecting a set of identification requests $S$, and after that, for each element in this set, compute free homomorphic images enforcing single digon incidence, and take the union of them.
We call the resulting set \emph{a set of free homomorphic images respecting $S$ and enforcing single digon incidence}.
The pseudocode corresponding to the algorithm to compute this set is presented in Algorithm \ref{alg:free_hom_single_digon}.
We can also compute a set of free homomorphic images enforcing single digon incidence of a free combination of discharging rules.
An element in this set together with the image of the distinguished dart and the same amount of charge is regarded as a discharging rule.
We call it \emph{a combined rule enforcing single digon incidence} or \emph{a free combination of discharging rules enforcing single digon incidence}.

\paragraph{Blocked by configurations}
To show Lemma \ref{lem:send5}, we want to check whether, for some configuration $K \in \mathcal{K}$, there exists a homomorphism from $K$ to a combined rule $R^*$.

Then, we say a pseudo-embedding with degree-range functions $(Z, \delta^-, \delta^+)$ is \emph{blocked by configuration in $\mathcal{K}$} if for every $(Z, \delta) \in (Z, \delta^{-}, \delta^{+})$, there exists a homomorphism from some configuration $K \in \mathcal{K}$ to $(Z, \delta)$.

We use the same algorithm in \cite{inoue2026four} to check it.
See Section 10 in \cite{inoue2026four}.
However, handling a cut-vertex in a configuration is different from that, so we only explain it here.

Recall that all configurations in $\mathcal{K}$ are normal, see Definition \ref{dfn:normal}.
For a normal configuration, if a cut-vertex of $K$ exists, it is incident to two outer edges in its outer extension $\widehat K$.
For each outer edge, we consider the plane graph $K$ with this outer edge and its outer endpoints.
This graph is represented by pseudo-embeddings, as follows.
From a pseudo-embedding of the outer extension $\widehat K$, delete outer edges and outer endpoints except the chosen outer edge and outer endpoint.
We make a vertex in $K$ a boundary vertex if it is incident to a deleted outer edge.
Specifically, if the $\successor, \predecessor$ pointer for some dart points to a dart corresponding to a deleted outer edge, we reassign \nil to this pointer.
Since we leave one chosen outer edge that is incident to a cut-vertex, this dart representation satisfies (M6), so it is a pseudo-embedding.
Then, we have two pseudo-embeddings from a normal configuration $K$ with a cut-vertex.
To check whether a homomorphism from $K$ exists, we check whether a homomorphism from either of these two pseudo-embeddings exists.

We also check the mirror images of configurations in $\mathcal{K}$.

\paragraph{Free combinations of discharging rules enforcing single digon incidence, which are not blocked by $\mathcal{K}$}
We implemented the above algorithms to compute a set of free homomorphic images respecting the identification requests and enforcing single digon incidence.
We consider all subsets of $\mathcal{R}$, and compute a set of all free combinations of discharging rules of this set, enforcing single digon incidence.
Let $\mathcal{R}^*$ denote this set.
Note that $\mathcal{R}^*$ contains a trivial rule with charge zero.
Moreover, we remove a combined rule in $\mathcal{R}^*$ when it is blocked by configurations in $\mathcal{K}$.
Let $\cRstarnoK$ denote this set.
The algorithm to compute it is presented in Algorithm A.8.2 in \cite{inoue2026four}, but inside this routine, we use Algorithm \ref{alg:free_hom_single_digon} for pseudo-triangulation with digons instead of the routine to compute a set of free homomorphic images of pseudo-triangulation.
We check that every combined rule $R^* \in \cRstarnoK$ has the amount of charge $r(R^*) \leq 5$, see Lemma \ref{comp:lem:combined_rules}.
This shows Lemma \ref{lem:send5}.

Note that the algorithm in this section does not use the fact that $G^*$ has at most 3 digons, so the algorithm is generally applicable to triangulations with digons, but with the constraint that every vertex is incident to at most one digon.

\subsection{The final charge of vertices with degrees between 7 and 11}\label{sect:cartwheel}\showlabel{sect:cartwheel}

In this subsection, we show Lemma \ref{lem:deg7-11}.
We describe the algorithm for checking the lemma.
The algorithm searches for a possible structure of $B_2(v)$ such that $T(v) > c(v)$, i.e., the final charge exceeds the value depending on digons, but a homomorphism from any configuration $K \in \mathcal{K}$ does not exist.
If the algorithm concludes that no such structure exists, then Lemma \ref{lem:deg7-11} follows.
The strategy of the algorithm is the same as the one used in Section 11.2 in \cite{inoue2026four}.
However, there are two obstructions to apply that algorithm directly, and therefore, we need a modified algorithm.
We first explain these two obstructions.

The first is that $G^*$ lacks sufficient connectivity.
As we described in Section \ref{subsect:sketch}, a minimal counterexample to 4CT is an internally 6-connected triangulation, but $G^*$ is not.
The ball $B_2(v)$ for a vertex $v$ is ``well-behaved" for a minimal counterexample to 4CT, i.e., this ball consists of $v$ and two cycles $C_1, C_2$ and edges between them.
However, in our case, $B_2(v)$ in $G^*$ is not like that.

The second is that the structure of $B_2(v)$ varies.
For a minimal counterexample to 4CT, all faces are triangles.
However, $G^*$ has digons, so the faces of $B_2(v)$ in $G^*$ may contain digons and triangles.
Thus, when constructing $B_2(v)$, there are exponentially many structures of faces compared to those for a minimal counterexample to 4CT.

These two obstructions prevent us from preconstructing all possible local structures that determine the final charge.
To overcome this difficulty, our algorithm constructs the structure of $B_2(v)$ while simultaneously computing the final charge of $v$.

\subsubsection{Wheels and cartwheels}

\paragraph{Wheel}
\emph{A wheel of center degree $d$} is a plane graph obtained from a cycle $v_1v_2\ldots v_dv_1$ by adding a \emph{center vertex} $v$ inside the cycle and joining $v$ to every $v_i$.
The vertices $v_1, \ldots, v_d$ appear in this clockwise rotation around $v$.
\emph{A wheel of center degree $d$ with a digon} is obtained from a wheel of center degree $d-1$ by replacing one edge $vv_1$ with a pair of multiple edges bounding a digon.
Thus, the center vertex has degree $d$.
A wheel of center degree $d$ possibly with a digon means either a wheel of center degree $d$ or a wheel of center degree $d$ with a digon.
A wheel possibly with a digon is a near-triangulation possibly with a digon, so it is
represented as a pseudo-triangulation with digons, denoted by $Z_W$.
Specifically, an edge is converted to two darts that are reverse to each other, and \successor/\predecessor are assigned from the embedding, but \nil is assigned if the face between the successor/predecessor edge is the outer face of $W$.
$Z_W$ is equipped with degree-range functions defined by $\delta_W^-(v)=\delta_W^+(v)=d$ and $\delta_W^-(v_i)=5, \delta_W^+(v_i)=\infty$.

\begin{claim}
\label{claim:hom-wheel}
\showlabel{claim:hom-wheel}
    Let $v$ be a vertex of degree $d$ in $G^*$.
    Let $(Z_W, \delta_W^-, \delta_W^+)$ be a wheel of center degree $d$ if $v$ is not incident to a digon; otherwise, a wheel of center degree $d$ with a digon incident to $ v$.
    Then, a homomorphism from some $(Z_W, \delta_W) \in (Z_W, \delta_W^-, \delta_W^+)$ to $G^*$ such that the center vertex is mapped to $v$ exists.
\end{claim}
\begin{proof}
    By Lemma \ref{lem:triangulation}, all faces that are incident to $v$ are triangles or digons, but $v$ is not incident to more than one digon.
    Since the minimum degree of $G^*$ is five, each neighbor of $v$ has the degree of at least five.
    Thus, the claim holds.
    Note that there may be a pair of parallel edges between two neighbors of $v$ bounding a digon, but this does not obstruct the homomorphism because homomorphisms are not required to preserve \nil-pointers.
\end{proof}

We refine the degree ranges of non-center vertices $v_i$ of a wheel possibly with a digon.
Especially, we refine the range $[\delta_W^-(v_i), \delta_W^+(v_i)] = [5, \infty]$ to one of the ranges $[5,5]$, $[6,6]$, $[7,7]$, $[8,8]$, or $[9,\infty]$.
We consider all combinations of these degree ranges for $v_1, \ldots, v_m$, with $m=d$ or $m=d-1$ depending on whether $v$ is incident to a digon.
Let $\mathcal{W}_d$ be the set of pseudo-triangulations with digons with degree-range functions representing all wheels with center of degree $d$ possibly with a digon obtained in this way.
By Claim \ref{claim:hom-wheel}, for a vertex of degree $d$ in $G^*$, we know that there exists $(Z_W, \delta_W) \in (Z_W, \delta^-_W, \delta^+_W)$ such that a homomorphism from $(Z_W, \delta_W)$ to $G^*$ where the center vertex is mapped to $v$ exists.

The pseudocode for the algorithm that constructs such a set of all the wheels with the center of degree $d$ is presented in Algorithm \ref{alg:enum_wheels}.
The pseudocode for the algorithm that constructs such a set of all the wheels with the center of degree $d$ with a digon is presented in Algorithm \ref{alg:enum_digon_incident_wheels}.
By calling both of them, we can get $\mathcal{W}_d$.

\paragraph{Cartwheel}
A wheel only contains the neighbors of the center vertex.
We will extend it by adding vertices and darts that represent a part of the second-neighborhood structure of the center vertex, but we do not necessarily add all second neighbors.
We do not add third neighbors of the center vertex.
We call such a pseudo-triangulation with digons with degree-range functions \emph{a cartwheel}.
During the algorithm, we ensure that for every vertex $u$ of a cartwheel, the degree range is either fixed (i.e., $\delta^-(u)=\delta^+(u)$) or $\delta^+(u)=\infty$.
The latter type of range is called \emph{tail range}.
We also ensure that every vertex $u$ is incident to at most one digon.
Note that we also regard elements in $\mathcal{W}_d$ as cartwheels.

We distinguish darts whose head is the center vertex $v$ in a cartwheel.
Since the degree of the center vertex in a cartwheel in $\mathcal{W}_d$ is $d$, there are $d$ darts whose head is $v$.
They are denoted by $e_1, \ldots, e_d$ in the clockwise order.
All cartwheels we process in this algorithm are obtained from the disjoint union of some cartwheel in $\mathcal{W}_d$ and several discharging rules by taking a free homomorphic image.
We maintain the image of $e_1, \ldots, e_d$ together with a cartwheel.
By the requirement of homomorphisms with a degree function, different darts in $\{e_i\}_{1 \leq i \leq d}$ are mapped to different darts.
When we use $e_1,\ldots,e_d$ to designate darts in some cartwheel, they are the images of the original $e_1,\ldots,e_d$ in a cartwheel in $\mathcal{W}_d$.

A key perspective for understanding the algorithm is that, when we handle a cartwheel $(Z_C, \delta_C^-, \delta_C^+)$, we consider the final charge of a vertex $v$ in $G^*$ when for some $(Z_C, \delta_C) \in (Z_C, \delta_C^-, \delta_C^+)$, a homomorphism from $(Z_C, \delta_C)$ to $G^*$ exists and $v$ is the image of the center vertex.

\subsubsection{Charges along an edge and a bound on the final charge}
Let $R=(Z_R, \delta_R^-, \delta_R^+)$ be a discharging rule with a distinguished dart $e_R$, and $(Z_C, \delta_C^-, \delta_C^+)$ be a cartwheel.

A rule $R$ \emph{always applies to a dart $e_C$} in $(Z_C, \delta_C^-, \delta_C^+)$ if a homomorphism $\phi$ from $(Z_R, \delta_R^-, \delta_R^+)$ to $(Z_C, \delta_C^-, \delta_C^+)$ exists such that $\phi(e_R)=e_C$ and for every vertex $v_R$ in $Z_R$, $[\delta_C^-(\phi(v_R)), \delta_C^+(\phi(v_R))] \subseteq [\delta_R^-(v_R), \delta_R^+(v_R)]$.
This implies that for every $(Z_C, \delta_C) \in (Z_C, \delta_C^-, \delta_C^+)$, a rule $R$ applies to a dart $e_C$ in $(Z_C, \delta_C)$.
This means that if for some $(Z_C, \delta_C) \in (Z_C, \delta_C^-, \delta_C^+)$, a homomorphism $\phi_C$ from $(Z_C, \delta_C)$ to $G^*$ exists, then a rule $R$ applies to the dart $\phi_C(e_C)$ in $G^*$.
The pseudocode for the algorithm that checks whether a rule always applies to a dart is the same as Algorithm A.9.1 in \cite{inoue2026four}.

A rule $R$ \emph{never applies to a dart $e_C$} in $(Z_C, \delta_C^-, \delta_C^+)$ if a set of free homomorphic images of the disjoint union $(Z_C, \delta_C^-, \delta_C^+) \sqcup (Z_R, \delta_R^-, \delta_R^+)$ respecting $\{(e_R, e_C)\}$ and enforcing single digon incidence is an empty set.
This implies that, for every $(Z_C, \delta_C) \in (Z_C, \delta_C^-, \delta_C^+)$ and every $(Z_R, \delta_R) \in (Z_R, \delta_R^-, \delta_R^+)$, a homomorphism $\phi_R$ from $(Z_R, \delta_R)$ to $G^*$ and a homomorphism $\phi_C$ from $(Z_C, \delta_C)$ to $G^*$ such that $\phi_R(e_R)=\phi_C(e_c)$ do not exist simultaneously.
This means that if for some $(Z_C, \delta_C) \in (Z_C, \delta_C^-, \delta_C^+)$, a homomorphism $\phi_C$ from $(Z_C, \delta_C)$ to $G^*$ exists, then a rule $R$ does not apply to the dart $\phi_C(e_C)$ in $G^*$.
Algorithm \ref{alg:never_apply} presents the pseudocode for checking whether a rule never applies to a dart.

A rule $R$ \emph{sometimes applies to a dart $e_C$} in $(Z_C, \delta_C^-, \delta_C^+)$ if the set of free homomorphic images used in the preceding definition of ``never applies'' is nonempty; equivalently, $R$ does not never apply to a dart $e_C$.

\paragraph{An upper bound of the final charge}
Recall that the set of combined rules $\cRstar$ and $\cRstarnoK$ is constructed in Section \ref{subsect:free-homomorphism}.
Every combined rule $R^*$ in $\cRstar \setminus \cRstarnoK$ is blocked by $\mathcal{K}$.
If such a combined rule $R^*$ always applies to a dart in a cartwheel, then a homomorphism from some configuration $K \in \mathcal{K}$ to this cartwheel exists.
As we only consider cartwheels such that there does not exist a homomorphism from any configuration in $\mathcal{K}$ to it, we have only to consider combined rules in $\cRstarnoK$, not $\cRstar$.

For a cartwheel $(Z_C, \delta_C^-, \delta_C^+)$ of center degree $d$, consider that, for some $(Z_C, \delta_C) \in (Z_C, \delta_C^-, \delta_C^+)$, a homomorphism $\phi$ from $(Z_C, \delta_C)$ to $G^*$ exists.
Let $v \in V(G^*)$ be the image of the center vertex.
We aim to calculate an upper bound of the final charge of $v$ in this situation.

We explain how to calculate an upper bound of the final charge.
We consider darts $e_1, \ldots, e_d$ whose heads are the center vertices of a cartwheel.
For each $1 \leq j \leq d$, we calculate an upper bound of the charge transferred along $\phi(e_j)$.
We compute the maximum value of $r(R^*)$ for all combined rules $R^* \in \cRstarnoK$ such that $R^*$ sometimes applies to $e_j$ in $(Z_C, \delta_C^-, \delta_C^+)$.
We denote this value by $r_j$.
We also calculate a lower bound of the charge transferred along $\phi(\reverse(e_j))$.
We compute the sum of $r(R)$ for all rules $R \in \mathcal{R}$ such that $R$ always applies to $\reverse(e_j)$ in $(Z_C, \delta_C^-, \delta_C^+)$.
We denote this value by $s_j$.
Then, combining the initial charge $10 \cdot (6 - d)$, we calculate $10 \cdot (6-d) + \sum_{1 \leq j \leq d} r_j - \sum_{1 \leq j \leq d} s_j$ as an upper bound of the final charge.

\begin{lem}
\label{lem:upperbound-charge}
\showlabel{lem:upperbound-charge}
    For a cartwheel $(Z_C, \delta_C^-, \delta_C^+)$, assume that for some $(Z_C, \delta_C) \in (Z_C, \delta_C^-, \delta_C^+)$, a homomorphism from $(Z_C, \delta_C)$ to $G^*$ exists.
    Let $v \in V(G^*)$ be the image of the center vertex.
    If there does not exist a homomorphism from any configuration in $\mathcal{K}$ to $B_2(v)$, then the final charge of $v$ is at most $10 \cdot (6-d) + \sum_{1 \leq j \leq d} r_j - \sum_{1 \leq j \leq d} s_j$, calculated as above.
\end{lem}
\begin{proof}
    Let $\phi$ be a homomorphism from $(Z_C, \delta_C)$ to $G^*$.
    Let $\mathcal{R}_j$ be a set of rules applied to $\phi(e_j)$.
    Then, some combined rule $R^*_j$ obtained from $\mathcal{R}_j$ applies to $\phi(e_j)$ in $G^*$.
    $R^*_j$ belongs to $\cRstar$.
    Moreover, $R^*_j$ belongs to $\cRstarnoK$ since there does not exist a homomorphism from any configuration in $\mathcal{K}$ to $B_2(v)$.
    $R^*_j$ sometimes applies to $e_j$ in $(Z_C, \delta_C^-, \delta_C^+)$ since otherwise, $R^*_j$ never applies to $e_j$, which contradicts that $R^*_j$ applies to $\phi(e_j)$ in $G^*$.
    Then, $r_j$ is at least the discharge of $R^*_j$ i.e., $r(R^*_j)$.

    Let $R$ be a rule that always applies to $\reverse(e_j)$ in $(Z_C, \delta_C^-, \delta_C^+)$.
    Then, $R$ also applies to $\phi(\reverse(e_j))$ in $G^*$.
    Thus, the actual charge sent along $\phi(\reverse(e_j))$ is at least $s_j$.
    Therefore, the claim holds.
\end{proof}

The pseudocode for the algorithm that computes $s_j, r_j$, and the upper bound of the final charge is the same as Algorithm A.9.3, A.9.4, A.9.13 in \cite{inoue2026four} respectively, except that the \neverApply routine called inside these algorithms is replaced with the \neverApply routine in this paper, which is Algorithm \ref{alg:never_apply}.

\paragraph{A lower bound of charges given to digons}
For a cartwheel $(Z_C, \delta_C^-, \delta_C^+)$ of center degree $d$, consider that, for some $(Z_C, \delta_C) \in (Z_C, \delta_C^-, \delta_C^+)$, a homomorphism $\phi$ from $(Z_C, \delta_C)$ to $G^*$ exists.
Let $v \in V(G^*)$ be the image of the center vertex.
We aim to compute a lower bound on $c(v)$ in Lemma \ref{lem:deg7-11}.

For each digon in $(Z_C, \delta_C^-, \delta_C^+)$, we consider endpoints $u_1, u_2$ of this digon.
We compute the distances from the center vertex $v$ to $u_1,u_2$.
We assume w.l.o.g., that $\textsf{dist}(u_1, v) \leq \textsf{dist}(u_2, v)$ in $Z_C$.
Depending on these distances, we consider one case from \ref{constaint-(i)},  \ref{constaint-(ii)}, \ref{constaint-(iii)}, or \ref{constaint-(iv)}.
We check whether the degree ranges specified in this case include the degree ranges $[\delta_C^-(u_1), \delta_C^+(u_1)]$, $[\delta_C^-(u_2), \delta_C^+(u_2)]$.
If so, we add the specified value for this case.
For example, if $\textsf{dist}(u_1, v)=1$, $\textsf{dist}(u_2, v)=2$, $5 = \delta_C^-(u_1) = \delta_C^+(u_1) = 5$ and $6 \leq \delta_C^-(u_2) \leq \delta_C^+(u_2) \leq 7$, then it corresponds to one case in \ref{constaint-(iii)}, so add the value 2.
We compute the sum of these values for each digon, which is denoted by $c(Z_C)$.
The pseudocode for the algorithm to compute $c(Z_C)$ is presented in Algorithm \ref{alg:lower_bound_of_digon_charge}.
\begin{lem}
\label{lem:charge-given-to-digon}
\showlabel{lem:charge-given-to-digon}
    For a cartwheel $(Z_C, \delta_C^-, \delta_C^+)$, assume that for some $(Z_C, \delta_C) \in (Z_C, \delta_C^-, \delta_C^+)$, a homomorphism from $(Z_C, \delta_C)$ to $G^*$ exists.
    Let $v \in V(G^*)$ be the image of the center vertex.
    The value $c(v)$ is at least the calculated value $c(Z_C)$.
\end{lem}
\begin{proof}
    Let $\phi$ be a homomorphism from $(Z_C, \delta_C)$ to $G^*$.
    By definition, for every vertex $u$ in $Z_C$, the image of $u$ by $\phi$ has a degree $\delta_C(u)$.

    Recall that in $(Z_C, \delta_C^-, \delta_C^+)$, every inner vertex $v$ satisfies $\delta_C^-(v)=\delta_C^+(v)=d_C(v)$ and every boundary vertex $v$ satisfies $\delta_C^-(v) > d_C(v)$ ($d_C(v)$ is the number of darts whose head is $v$) from the algorithm to compute free homomorphism, see Section \ref{subsect:free-homomorphism} and \cite{inoue2026four}.
    Suppose that two distinct darts, say $e,f$, whose heads are the same vertex are mapped to the same dart in $G^*$.
    By exchanging $e$ and $f$ if necessary, there exists an integer $k < d_C(\head(e))$ such that $f=\successor^k(e)$.
    Since a homomorphism preserves \successor, $\phi(f)=\successor^k(\phi(e))$ holds.
    Then, the degree of $\phi(\head(e))$ is at most $k (< d_C(\head(e)))$, which contradicts $\phi(\head(e))$ has degree $\delta_C(\head(e))$.
    
    Two distinct neighbors of the same vertex are mapped to two distinct vertices in $G^*$ since otherwise it makes a separating 2-cycle or a loop in $G^*$, which contradicts Lemma \ref{lem:triangulation}.
    
    Recall that in computing $c(Z_C)$ (also $c(v)$), we only use digons whose both endpoints are at a distance of $(0,1),(1,1)$, or $(1,2)$.
    Let $v_C$ be the center vertex of $Z_C$ and $u_1,u_2$ be endpoints of a digon such that $(\textsf{dist}(v_C, u_1), \textsf{dist}(v_C, u_2))$ is one of $(0,1),(1,1),(1,2)$.

    We show the distance between $v_C$ and $u_i$ ($i=1,2$) are preserved under $\phi$, i.e., $\textsf{dist}(v_C, u_i)=\textsf{dist}(\phi(v_C), \phi(u_i))$ for $i=1,2$.
    It is easy to see that $\textsf{dist}(v_C, u_i) \geq \textsf{dist}(\phi(v_C), \phi(u_i))$.
    If $\textsf{dist}(v_C, u_i)=1$, but $\textsf{dist}(\phi(v_C), \phi(u_i))=0$, then it makes a loop in $G^*$, which contradicts Lemma \ref{lem:triangulation}.
    If $\textsf{dist}(v_C, u_i)=2$, but $\textsf{dist}(\phi(v_C), \phi(u_i))=0$, then two neighbors of $u_{3-i}$ ($u_i$ and $v_C$) are mapped to the same vertex, which contradicts the above discussion.
    If $\textsf{dist}(v_C, u_i)=2$, but $\textsf{dist}(\phi(v_C), \phi(u_i))=1$, then it makes a cycle of length 3 in $G^*$.
    If it is separating, then it contradicts Lemma \ref{lem:triangulation}.
    Otherwise, this cycle is facial, so two neighbors of $u_{3-i}$ ($u_i$ and a neighbor of $u_{3-i}$ at a distance 1 from $v_C$) are mapped to the same vertex, which contradicts the above discussion.
    Thus, we can show that the distances are preserved.

    It remains to show that two such distinct digons in $Z_C$
    are mapped to two distinct digons in $G^*$ under $\phi$.
    Since the distances are preserved, for two digons that are mapped to the same digon in $G^*$, their endpoints have the same distance pair from the center vertex.
    At least one endpoint of a digon is at a distance of 1 from $v_C$.
    Taking into account that at most one digon is incident to a vertex in $Z_C$, it does not occur, since otherwise two neighbors of the center vertex would be mapped to the same vertex in $G^*$.
\end{proof}

By the above computations, we cut a cartwheel $(Z_C, \delta_C^-, \delta_C^+)$ if $10 \cdot (6-d) + \sum_{1 \leq j \leq d} r_j - \sum_{1 \leq j \leq d} s_j \leq c(Z_C)$.
This is because $T(v) \leq c(v)$ for the vertex $v$, which is the image of the center vertex.

\subsubsection{Fixing rules applied from neighbors to the center}
\label{subsect:fix-in-rules}
\showlabel{subsect:fix-in-rules}
For a vertex $v \in V(G^*)$, we aim to fix the set of discharging rules in $\mathcal{R}$ that apply to darts whose head is $v$.
Recall that each combined rule in $\cRstar$ or $\cRstarnoK$ is obtained from a subset of rules in $\mathcal{R}$.
As such a set of rules, we consider a subset of discharging rules corresponding to some combined rule in $\cRstarnoK$.

\paragraph{Algorithm}
We process each cartwheel $(Z_W, \delta_W^-, \delta_W^+) \in \mathcal{W}_d$ in the following way.
We first create $\mathcal{C}_0=\{(Z_W, \delta_W^-, \delta_W^+)\}$.
We construct the sets $\mathcal{C}_1,\ldots,\mathcal{C}_d$ in order from $\mathcal{C}_0$.
Each element in $\mathcal{C}_i$ is a cartwheel.
We construct $\mathcal{C}_{i}$ from $\mathcal{C}_{i-1}$ as follows: for each $(Z_C, \delta_C^-, \delta_C^+) \in \mathcal{C}_{i-1}$, and for each $(Z_{R^*_i}, \delta_{R^*_i}^-, \delta_{R^*_i}^+)$ with the distinguished dart $e_{R^*_i}$ in $\cRstarnoK$, we compute a set of free homomorphic images of the disjoint union $(Z_C, \delta_C^-, \delta_C^+) \sqcup (Z_{R^*_i}, \delta_{R^*_i}^-, \delta_{R^*_i}^+)$ respecting $\{(e_i, e_{R^*_i})\}$ and enforcing single digon incidence.
For each element $((Z_{C'}, \delta_{C'}^-, \delta_{C'}^+), \phi^*)$ in this set, if some vertices have a degree range except for the tail range, we refine them to fixed degrees and enumerate all combinations.
We add all resulting cartwheels to $\mathcal{C}_i$.
The pseudocode for the algorithm that constructs $\mathcal{C}_d$ is presented in Algorithm \ref{alg:fix_in_rules}.
The pseudocode for updating a cartwheel by a combined rule $R^*_i \in \cRstarnoK$ is presented in Algorithm \ref{alg:update_degree_by_rule}.

For every rule $R$ in $\mathcal{R}$, every vertex is of distance at most 2 from both $s(R)$ and $t(R)$, so it holds for every combined rule in $\cRstarnoK$ and $\cRstar$. 
Thus, this process adds only second neighbors of the center vertex and no third neighbors.

\paragraph{Pruning}
We apply three pruning methods when adding $(Z_{C'}, \delta_{C'}^-, \delta_{C'}^+)$ to $\mathcal{C}_i$ for $i=1, \ldots, d$.
We consider darts $e_1, \ldots, e_d$ whose head is the center vertex in $Z_{C'}$.

Recall that for such $(Z_{C'}, \delta_{C'}^-, \delta_{C'}^+)$ and for every $j \leq i$, we have decided the exact set of rules applying to $e_j$, which corresponds to a combined rule $R_j^*$ identified with $e_j$.
We denote this set by $\mathcal{R}_j$.
This has two implications.
One is that the amount of charge $r(R_j^*)$ is the exact amount of charge transferred along the image of $e_j$ under a homomorphism from $(Z_{C'}, \delta_{C'}^-, \delta_{C'}^+)$ to $G^*$.
The other is for pruning: we can cut $(Z_{C'}, \delta_{C'}^-, \delta_{C'}^+)$ if there exists a rule $R \in \mathcal{R} \setminus \mathcal{R}_j$ that always applies to $e_j$.

We also calculate an upper bound of the final charge.
For $1 \leq j \leq i$, we know that $r(R_j^*)$ is the exact discharge.
For $i < j \leq d$, we compute $r_j$, which is the maximum value of $r(R^*)$ for all combined rules $R^* \in \cRstarnoK$ such that $R^*$ sometimes applies to $e_j$ in $(Z_{C'}, \delta^-_{C'}, \delta^+_{C'})$.
We also calculate $s_j$, which is the sum of $r(R)$ for all rules $R \in \mathcal{R}$ such that $R$ always applies to $\reverse(e_j)$ in $(Z_{C'}, \delta^-_{C'}, \delta^+_{C'})$.
Then, we calculate $10 \cdot (6-d) + \sum_{1 \leq j \leq i} r(R_j^*) + \sum_{i < j \leq d} r_j - \sum_{1 \leq j \leq d} s_j$ as an upper bound of the final charge.
Corresponding to Lemma \ref{lem:upperbound-charge}, we have the following lemma stating the validity of this pruning.

\begin{lem}
\label{lem:upperbound-charge2}
\showlabel{lem:upperbound-charge2}
    Let $(Z_{C'}, \delta_{C'}^-, \delta_{C'}^+)$, $R_j^*$, $\mathcal{R}_j (1 \leq j \leq i)$, and $\{e_j\}_{1 \leq j \leq d}$ be defined as above.
    Assume that for some $(Z_{C'}, \delta_{C'}) \in (Z_{C'}, \delta_{C'}^-, \delta_{C'}^+)$, a homomorphism from $(Z_{C'}, \delta_{C'})$ to $G^*$ exists, and the set of rules applying to the image of $e_j$ equals $\mathcal{R}_j$ for $1 \leq j \leq i$.
    Let $v \in V(G^*)$ be the image of the center vertex.
    If there does not exist a homomorphism from any configuration in $\mathcal{K}$ to $B_2(v)$, the final charge of $v$ is at most $10 \cdot (6-d) + \sum_{1 \leq j \leq i} r(R_j^*) + \sum_{i < j \leq d} r_j - \sum_{1 \leq j \leq d} s_j$, calculated as above.
\end{lem}
Thus, we can cut $(Z_{C'}, \delta_{C'}^-, \delta_{C'}^+)$ if the computed upper bound of the final charge is at most $c(Z_{C'})$, similarly as before.
Finally, we cut $(Z_{C'}, \delta_{C'}^-, \delta_{C'}^+)$ if $(Z_{C'}, \delta_{C'}^-, \delta_{C'}^+)$ is blocked by $\mathcal{K}$.
Note that we apply the same pruning for the initial case where $i=0$.
The pseudocode for checking whether applying three pruning methods to the current cartwheel is presented in Algorithm \ref{alg:prune}.

\subsubsection{Refinement}
For all cartwheels in $\mathcal{C}_d$, we have determined the set of rules applied to $e_i$ for every $1 \leq i \leq d$.
However, some cartwheels in $\mathcal{C}_d$ may not have all second neighbors of the center.
We construct the final output $\mathcal{C}$ of the algorithm from $\mathcal{C}_d$ by selectively adding them.

The algorithm we will explain is not generalizable to any set of rules, as a general implementation is very complicated due to the flexibility of face sizes.
Thus, the algorithm is tailored to our set of rules, which makes our implementation and the proof of correctness simpler.
An unnecessary general implementation may cause mistakes, so we choose a simpler algorithm tailored to our set of rules.

\paragraph{A homomorphic cover}
For a rule $R$ represented by $(Z_R, \delta_R^-, \delta_R^+)$ with a distinguished dart $e_R$, consider the rules $R_1,\ldots,R_k$ with the following properties: if there exists a homomorphism $\phi$ from some $(Z_R, \delta_R) \in (Z_R, \delta_R^-, \delta_R^+)$ to $G^*$, then there exists an index $i$ ($1 \leq i \leq k$) such that a homomorphism $\phi_i$ from some $(Z_{R_i}, \delta_{R_i}) \in (Z_{R_i}, \delta^-_{R_i}, \delta^+_{R_i})$ to $G^*$ satisfying $\phi(e_R)=\phi_i(e_{R_i})$ exists.
We call such rules $R_1,\ldots,R_k$ \emph{a homomorphic cover of $R$}.
We show two examples.
Actually, all homomorphic covers used in this algorithm are only these two and their mirror images.
One is in Figure \ref{fig:rule27-aux}.
The other is in Figure \ref{fig:rule31-aux}.
For both of them, the three rules from the right form a homomorphic cover of the left rule $R$.

Actually, these two originate from the rule $R(27)$ and $R(31)$, see Figure \ref{fig:rules-RSST}.
Thus, the rules in Figure \ref{fig:rule27-aux} are denoted by $R_{\textsf{aux}, 27}, R_{\textsf{aux}, 27}^1, R_{\textsf{aux}, 27}^2,
R_{\textsf{aux}, 27}^3$ from the left.
Similarly the rules in Figure \ref{fig:rule31-aux} are denoted by $R_{\textsf{aux}, 31}, R_{\textsf{aux}, 31}^1, R_{\textsf{aux}, 31}^2,
R_{\textsf{aux}, 31}^3$ from the left.
We can check that $R_{\textsf{aux}, 27}^1, R_{\textsf{aux}, 27}^2,
R_{\textsf{aux}, 27}^3$ form a homomorphic cover of $R_{\textsf{aux}, 27}$ by observing the boundary edge, designated by $e$ in $R_{\textsf{aux}, 27}$.
In $R_{\textsf{aux}, 27}^1$, the edge corresponding to $e$ bounds a digon with another edge.
In $R_{\textsf{aux}, 27}^2$, the edge corresponding to $e$ bounds a triangle with other edges, and the vertex that forms a triangle with $e$ has degree $5$.
In $R_{\textsf{aux}, 27}^3$, the edge corresponding to $e$ bounds a triangle with other edges, and the vertex that forms a triangle with $e$ has degree-range $[6, \infty]$.
Since the minimum degree of $G^*$ is 5 and all faces of $G^*$ are digons or triangles, they form a homomorphic cover.

\begin{figure}[htbp]
    \centering
    \includegraphics[width=0.7\linewidth]{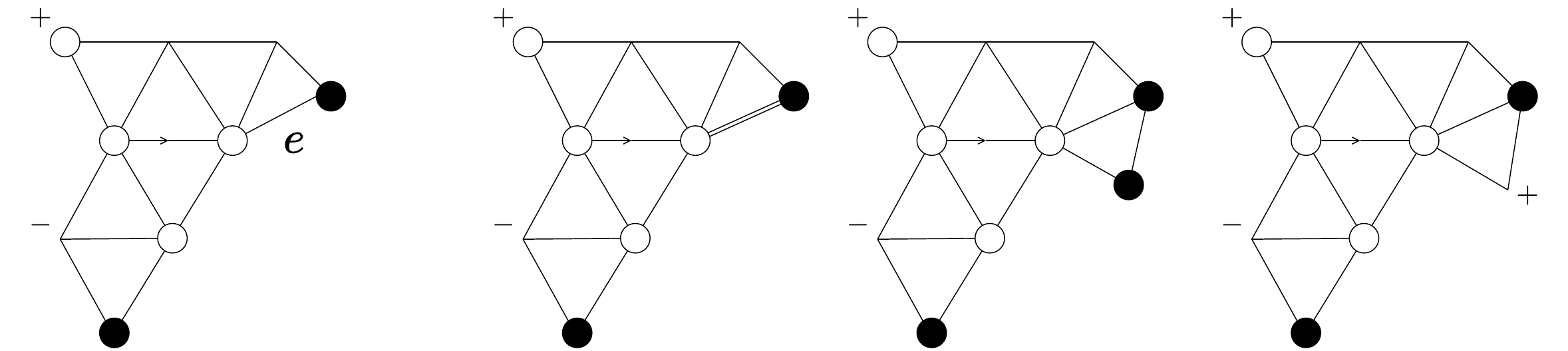}
    \caption{The right three rules, denoted by $R_{\textsf{aux}, 27}^1, R_{\textsf{aux}, 27}^2,
R_{\textsf{aux}, 27}^3$ form a homomorphic cover of the left rule $R_{\textsf{aux}, 27}$.}
    \label{fig:rule27-aux}
\end{figure}

\begin{figure}[htbp]
    \centering
    \includegraphics[width=0.7\linewidth]{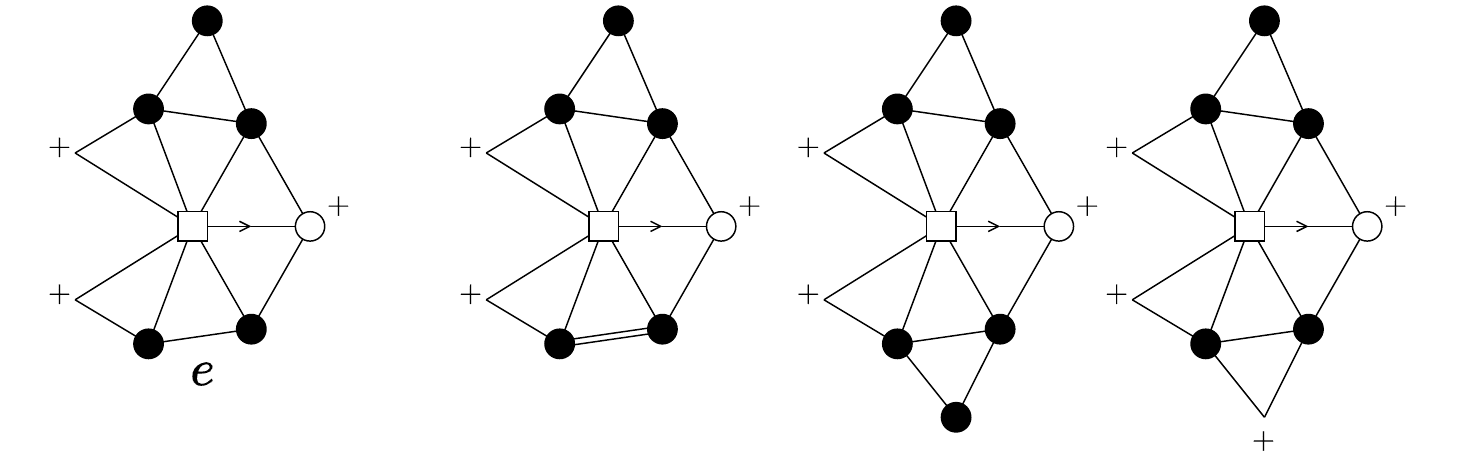}
    \caption{The right three rules, denoted by $R_{\textsf{aux}, 31}^1, R_{\textsf{aux}, 31}^2,
R_{\textsf{aux}, 31}^3$ form a homomorphic cover of the left rule $R_{\textsf{aux}, 31}$. We denote them by $R$.}
    \label{fig:rule31-aux}
\end{figure}

\paragraph{Algorithm}
We explain the algorithm.
Let $\Rauxiliary$ be the set that consists of $(R_{\textsf{aux}, 27}, R_{\textsf{aux}, 27}^1, R_{\textsf{aux}, 27}^2,
R_{\textsf{aux}, 27}^3)$, $(R_{\textsf{aux}, 31}, R_{\textsf{aux}, 31}^1, R_{\textsf{aux}, 31}^2,
R_{\textsf{aux}, 31}^3)$, and their mirrors.
Initially, we add every cartwheel in $\mathcal{C}_d$ to a queue.
While the queue is not empty, we pop one cartwheel $(Z_C, \delta_C^-, \delta_C^+)$ from the queue.
Then, we iterate over all pairs of homomorphic covers $(R,R_1,R_2,R_3) \in \Rauxiliary$ and a dart $e_i$ whose head is the center vertex in $(Z_C, \delta_C^-, \delta_C^+)$.
For each pair $((R, R_1, R_2, R_3), e_i)$, we check whether $R$ always applies to $\reverse(e_i)$ in $(Z_C, \delta_C^-, \delta_C^+)$, but no rule in $\{R_1, R_2, R_3\}$ always applies to $\reverse(e_i)$ in $(Z_C, \delta_C^-, \delta_C^+)$.
If so, we perform the refinement.
Otherwise, we proceed to the next pair.
The pseudocode for this algorithm that determines whether we should refine is presented in Algorithm \ref{alg:should_refine}.
For the refinement procedure, for each $i=1,2,3$, we take a set of free homomorphic images of the disjoint union $(Z_C, \delta_C^-, \delta_C^+) \sqcup (Z_{R_i}, \delta^-_{R_i}, \delta^+_{R_i})$ respecting $\{(\reverse(e_i), e_{R_i})\}$.
The pseudocode for this refinement algorithm is presented in Algorithm \ref{alg:refinement}.
We apply the three previously described pruning methods to the resulting set.
For the surviving cartwheels, if some vertices have a non-tail degree range, we refine them to fixed degrees and enumerate all combinations.
We push the resulting cartwheels into the queue.
If $(Z_C, \delta_C^-, \delta_C^+)$ is not refined for any pair, we add it to the final set $\mathcal{C}$.
We repeat until the queue becomes empty.
The pseudocode for the overall algorithm for the refinement is presented in Algorithm \ref{alg:fix_out_rules}.

Note that this algorithm terminates in finite time.
After we perform refinement by a pair $((R,R_1,R_2,R_3), e_i)$, we do not apply refinement for the same pair again since either of $R_1,R_2,R_3$ always applies to $\reverse(e_i)$ after refinement.
Thus, for one cartwheel in $\mathcal{C}_d$, the refinement occurs at most $2 \cdot d$ times.

The algorithm to compute $\mathcal{C}$ from a singleton $\mathcal{C}_0$ is presented in Algorithm \ref{alg:enum_bad_cartwheels}.
Let $\Ca$ be the union of all $\mathcal{C}$ obtained from each singleton of cartwheels in $\mathcal{W}_d$.

\begin{lem}
\label{lem:cartwheel-algorithm-is-correct}    
\showlabel{lem:cartwheel-algorithm-is-correct}
    For $d \in \{7,8,9,10,11\}$, let $\Ca$ be obtained from $\mathcal{W}_d$ as above.
    If $\Ca = \emptyset$, then for every vertex $v$ of degree $d$ in $V(G^*)$, at least one of the following holds: (i) $T(v) \leq c(v)$ or (ii) a homomorphism from configurations in $\mathcal{K}$ to $B_2(v)$ exists.
\end{lem}
\begin{proof}
    Let $v$ be a vertex of degree $d$ in $G^*$.
    By Claim \ref{claim:hom-wheel}, for some cartwheel $(Z_W, \delta_W^-, \delta_W^+)$ in $\mathcal{W}_d$, a homomorphism from some $(Z_W, \delta_W) \in (Z_W, \delta_W^-, \delta_W^+)$ to $G^*$ mapping its center to $v$ exists.
    We consider the algorithm in Section \ref{subsect:fix-in-rules} to generate $\mathcal{C}_i$ ($1 \leq i \leq d$) starting from $\mathcal{C}_0$ that consists of this cartwheel $(Z_W, \delta_W^-, \delta_W^+)$.
    We show the following claim:
    \begin{itemize}
        \item[(1)] If $v$ satisfies neither (i) nor (ii), then there exists a cartwheel $(Z_C, \delta_C^-, \delta_C^+) \in \mathcal{C}_d$ such that a homomorphism from some $(Z_C, \delta_C) \in (Z_C, \delta_C^-, \delta_C^+)$ to $G^*$, where the center vertex is mapped to $v$, exists.
    \end{itemize}
    First, we show claim (1) if we did not apply pruning.
    
    Let $e_j$ ($1 \leq j \leq d$) be darts in $(Z_W, \delta_W^-, \delta_W^+)$ whose head is the center vertex.
    For $j=1,\ldots,d$, let $\mathcal{R}_j$ be the set of rules that apply to the image of $e_j$ in $G^*$.
    Then, a combined rule $R_j^*$ is obtained from $\mathcal{R}_j$ by applying it to the image of $e_j$ in $G^*$.
    If $R_j^*$ is blocked by $\mathcal{K}$, case (ii) holds, so we can assume that $R_j^* \in \cRstarnoK$.
    In the algorithm in Section \ref{subsect:fix-in-rules}, every time we generate $\mathcal{C}_j$ from $\mathcal{C}_{j-1}$, we consider the case where we choose the specified combined rule $R_j^*$ and compute free homomorphic images respecting $\{(e_j, e_{R_j^*})\}$ and enforcing single digon incidence.
    By the property of free homomorphic images respecting identification requests and enforcing single digon incidence with Lemmas \ref{lem:free-hom-tri-digon} and \ref{lem:pseud-tri-single-digon}, some cartwheel $(Z_C, \delta_C^-, \delta_C^+)$ in $\mathcal{C}_d$ obtained in this case satisfies claim (1).

    We show that this cartwheel $(Z_C, \delta_C^-, \delta_C^+)$ survives the pruning if $v$ satisfies neither (i) nor (ii).
    By Lemmas \ref{lem:upperbound-charge}, \ref{lem:charge-given-to-digon}, and \ref{lem:upperbound-charge2}, $(Z_C, \delta_C^-, \delta_C^+)$ is not pruned by the charge condition since otherwise $v$ satisfies (i).
    $(Z_C, \delta_C^-, \delta_C^+)$ is not blocked by $\mathcal{K}$, since otherwise $v$ satisfies (ii).
    If some rule $R \in \mathcal{R} \setminus \mathcal{R}_j$ always applies to $e_j$ in $(Z_C, \delta_C^-, \delta_C^+)$, then $R$ also applies to the image of $e_j$ in $G^*$ under a homomorphism from $(Z_C, \delta_C^-, \delta_C^+)$.
    This contradicts the definition of $\mathcal{R}_j$, so the pruning does not cut $(Z_C, \delta_C^-, \delta_C^+)$.
    Therefore, claim (1) holds.

    We obtain the claim from (1) by replacing $\mathcal{C}_d$ with $\mathcal{C}$.
    We consider the cartwheel $(Z_C, \delta_C^-, \delta_C^+)$ in claim (1).
    When a rule $R$ always applies to $\reverse(e_i)$, a homomorphism from some $(Z_R, \delta_R) \in (Z_R, \delta_R^-, \delta_R^+)$ to $G^*$ such that $e_R$ is mapped to the image of $\reverse(e_i)$ exists.
    By the definition of a homomorphic cover, for some $R_i$ ($i=1,2,3$), a homomorphism from some $(Z_{R_i}, \delta_{R_i}) \in (Z_{R_i}, \delta_{R_i}^-, \delta_{R_i}^+)$ to $G^*$ such that $e_{R_i}$ is mapped to the image of $\reverse(e_i)$ exists.
    Thus, by the property of free homomorphic images, for one of the cartwheels obtained by computing free homomorphic images with $(Z_{R_i}, \delta_{R_i}^-, \delta_{R_i}^+)$ respecting $\{(\reverse(e_i), e_{R_i}\}$, a desired homomorphism to $G^*$ exists.
    In the same discussion as claim (1), this cartwheel survives the pruning, so a desired claim holds.

    This implies that if some vertex $v$ satisfies neither (i) nor (ii), then $\mathcal{C} \neq \emptyset$.
    Therefore, this lemma holds.
\end{proof}
Finally, we use Algorithm \ref{alg:enum_bad_cartwheels} to compute $\mathcal{C}$ from $\mathcal{C}_0$ on a computer and verified that $\mathcal{C}=\emptyset$ for each possible $\mathcal{C}_0$, see Lemma \ref{comp:lem:deg7-11}.
This proves Lemma \ref{lem:deg7-11}.

\section{Homomorphic images of configurations}
\label{sect:hom-conf}
\showlabel{sect:hom-conf}
The goal of this section is to prove Theorem \ref{thm:hom-imply-reducible} below. Together with Theorem \ref{thm:K-homomorphism-G}, this completes the proof of Theorem \ref{mainth}. 

Recall that $G^*$ is the triangulation with digons, and $G$ is its dual subcubic graph in Lemma \ref{lem:triangulation}.
We aim to construct the set of multi-boundary islands such that, if there exists a homomorphism from a configuration in $\mathcal{K}$ to $G^*$, then at least one of them appears in $G$.
Then, we show that all multi-boundary islands in this set are semi-reducible.
We will construct a set of multi-boundary islands, not multi-boundary configurations.

\begin{theorem}
\label{thm:hom-imply-reducible}
\showlabel{thm:hom-imply-reducible}
    If there exists a homomorphism from a configuration $K$ in $\mathcal{K}$ to $G^*$, then some semi-reducible multi-boundary island appears in $G$.
\end{theorem}

Actually, our constructed set of multi-boundary islands has the form $\mathcal{I} \cup \mathcal{T}$, where $\mathcal{T}$ is the set of trimmed multi-boundary islands of some multi-boundary island $I \in \mathcal{I}$.
If all multi-boundary islands in $\mathcal{I}$ are semi-reducible, then all multi-boundary islands in $\mathcal{T}$ are semi-reducible by Lemma \ref{lem:imply}.
Thus, we only maintain the set $\mathcal{I}$.
Note that we construct $\mathcal{I}$ from $\mathcal{K}$ by executing the algorithms described in this section on a computer.

\subsection{Algorithms for computer checks}
\label{subsect:alg-comp-check}
\showlabel{subsect:alg-comp-check}
The algorithm in this section branches, for each $K\in\mathcal K$, according to whether a homomorphism from $\widehat K$ is injective on darts.
If it is injective, we show that $I(K)$ or its trimmed multi-boundary island appears in $G$, see Claim \ref{claim:island-appear}.
Otherwise, some two distinct darts are mapped to the same dart, and hence there exists a homomorphism from a free homomorphic image respecting this pair of darts.
However, this simple branching creates too many multi-boundary islands.
We therefore apply pruning based on planarity and on the structure of separating cycles in $G^*$.
Although we first describe the pruning method, the main algorithm is \allHomImages, which is presented at the end of this section.

Note that in the algorithm in this Section, we only use the fact that the number of vertices of degree 2 is at most 3 in only one part.
This is the part where, if the constructed multi-boundary island has more than 3 vertices of degree 2, we do not add it to $\mathcal{I}$.
We add a footnote in this part of the algorithm.

\paragraph{Free homomorphic images of an outer extension of a multi-boundary configuration}
In the following algorithm, we compute free homomorphic images of an outer extension of a multi-boundary configuration $\widehat K$.
Note that an outer extension of a multi-boundary configuration is a pseudo-embedding, so we use the algorithm for a pseudo-embedding (not a pseudo-triangulation with digons).
The homomorphism from $\widehat K$ preserves the degrees of vertices in $K$, see Section \ref{sect:dart}.
When computing free homomorphic images of $\widehat K$, it is considered as a pseudo-embedding with degree-range functions $\delta^-,\delta^+$.
Specifically, for a vertex $v$ in $K$, $\delta^-(v)=\delta^+(v)=\delta_K(v)$, and for an outer endpoint $v$, $\delta^-(v)=5, \delta^+(v)=\infty$.
Since the minimum degree of $G^*$ is 5, $\delta^-(v)=5$.

In $\widehat K$, a boundary vertex has degree one.
The following claim shows that this property is preserved under taking free homomorphic images.
This claim helps us to see that a free homomorphic image of an outer extension can also be regarded as an outer extension of some multi-boundary configuration.
\begin{claim}
\label{claim:suc=nil iff head=degrre-one}
\showlabel{claim:suc=nil iff head=degrre-one}
    Let $(Z, \delta^-, \delta^+)$ be a pseudo-embedding such that
    every boundary vertex $v$ has degree one, but $\delta^-(v) > 1$.
    Let $(Z^*, \delta^{*-}, \delta^{*+})$ be a free homomorphic image of $(Z, \delta^-, \delta^+)$ respecting some set of identification requests $S$.
    
    In $Z^*$, every boundary vertex $v^*$ has degree one, but $\delta^{*-}(v^*) > 1$.
\end{claim}
\begin{proof}
    Let $\phi^*$ be a free homomorphism from $(Z, \delta^-, \delta^+)$ to $(Z^*, \delta^{*-}, \delta^{*+})$.
    Suppose some dart $e^*$ in $Z^*$ has $\successor(e^*)=\nil$, but the degree of $\head(e^*)$ is more than one.
    Since $\successor(e^*)=\nil$, the vertex $\head(e^*)$ cannot be the image of an inner vertex of $Z$.

    Thus, while computing the set of free homomorphic images, two degree-one boundary vertices are identified, but two darts whose head is the identified vertex are not identified.
    This does not occur because, by observing the algorithm to compute free homomorphic images of pseudo-embeddings, including the algorithm in \cite{inoue2026four}, two vertices are identified only if two darts whose heads are these two vertices are identified.
\end{proof}

Note that when executing the algorithm for computer checks, as described below, we confirm the property in Claim \ref{claim:suc=nil iff head=degrre-one} every time we take free homomorphic images.
In this section, as a pseudo-embedding, we only consider an outer extension of a multi-boundary configuration and its free homomorphic images.
Thus, a boundary vertex has degree one by the above claim.
Connectivity is also preserved under taking a free homomorphism.

All pseudo-embeddings with degree, or degree-range functions, are loopless.
Since $G^*$ has no loops, and the image of a loop under a homomorphism is also a loop, restricting the input $Z$ to be loopless is fine.
Specifically, the algorithm to compute free homomorphic images does not return a pseudo-embedding with loops.

In this section, as a homomorphism $\phi$ from dart representations with degree-range functions $(Z,\delta^-,\delta^+)$ to $(Z', \delta^{'-}, \delta^{'+})$, we consider the standard one: every vertex $v$ in $Z$ satisfies $[\delta^{'-}(\phi(v)), \delta^{'+}(\phi(v))] \subseteq [\delta^-(v), \delta^+(v)]$.
We consider that $G^*$, which is represented as a pseudo-triangulation with digons, has a degree range with both lower bound and upper bound exactly $d_{G^*}(v)$.
Thus, a homomorphism to $G^*$ satisfies $\delta^-(v) \leq d_{G^*}(\phi(v)) \leq \delta^+(v)$.

\paragraph{Check planarity}
By identifying several darts of an outer extension of a multi-boundary configuration, we may have ``non-planar" pseudo-embeddings.
We consider the homomorphism to the plane graph $G^*$.
However, a homomorphism from ``non-planar" pseudo-embeddings to $G^*$ may exist.

The subroutine $\isPlanar$ checks planarity in some sense.
The input of this subroutine is a connected loopless pseudo-embedding $Z=(V,D)$.
The procedure $\isPlanar(Z)$ proceeds as follows.
From an arbitrary dart $e_0$, we compute a closed walk starting from $e_0$ as follows.
For a dart $e_i$, if $\successor(e_i) \neq \nil$, we determine the next dart $e_{i+1}=\reverse(\successor(e_i))$.
Otherwise, $\successor(e_i) = \nil$, then we determine the next dart $e_{i+1}=\reverse(f)$, where $f$ is the dart with $\head(f)=\head(e_i)$ and $\predecessor(f)=\nil$.
Let $F$ be the set of closed walks obtained by starting from every dart in this way, where cyclic shifts are considered as the same walk.
The pseudocode corresponding to the algorithm to compute $F$ is presented in the \getWalks subroutine in Algorithm \ref{alg:get_walks}.
The routine \isPlanar returns \textsf{true} if $|V|-|D|/2+|F|=2$, otherwise \textsf{false}.
Since $Z$ has no loops, every dart $e$ in $Z$ satisfies $\reverse(e) \neq e$, so the number of darts is even.
The pseudocode of the \isPlanar subroutine is presented in Algorithm \ref{alg:is_planar}.

For every boundary vertex, we can complete its incidence list by assigning $\predecessor(e)=f,\successor(f)=e$ for the first dart $e$ and the last dart $f$ of its incidence list.
The obtained pseudo-embedding in this way from $Z$ is considered as a rotation system, which is equivalent to a 2-cell embedding on some orientable surface (see \cite{MT}), but not necessarily the plane.
By considering the Euler formula for orientable surfaces, it follows that $|V|-|D|/2+|F| \leq 2$.
The following claim shows that $\isPlanar(Z)=\textsf{false}$ implies non-planarity of $Z$ in some sense.

\begin{claim}
\label{claim:is-planar}
\showlabel{claim:is-planar}
    Let $Z=(V,D)$ be a connected loopless pseudo-embedding.
    We assume that a homomorphism from $Z$ to the plane graph $G^*$ exists.
    If $\isPlanar(Z)=\textsf{false}$, then this homomorphism is not injective on darts.
\end{claim}
\begin{proof}
    Let $Z^*$ be the image of $Z$ under $\phi$.
    $Z^*$ can be regarded as a subgraph of $G^*$.
    Since $G^*$ is embedded in the plane, $Z^*$ is also embedded in the plane.
    Let $V(Z^*), E(Z^*), F(Z^*)$ denote the set of vertices, edges, faces of $Z^*$ in this embedding.
    Suppose $\phi$ is injective on darts.
    We prove that $|V(Z^*)|-|E(Z^*)|+|F(Z^*)| \leq |V|-|D|/2+|F|$ by induction on the number of vertices identified by $\phi$.
    By Euler Formula, $|V(Z^*)|-|E(Z^*)|+|F(Z^*)|=2$, but since $\isPlanar(Z)=\textsf{false}$, $ |V|-|D|/2+|F| \leq 1$ holds.
    It leads to a contradiction.
    
    As a base case, suppose $\phi$ is injective on vertices.
    Then, $|V(Z^*)|=|V|$ and $|E(Z^*)|=|D|/2$.
    For each closed walk in $F$ of $Z$, its image bounds a face in $F(Z^*)$, and hence $|F|=|F(Z^*)|$.
    Thus, $|V(Z^*)|-|E(Z^*)|+|F(Z^*)|=|V|-|D|/2+|F|$.

    As an inductive step, we consider the case where two distinct vertices $u,v$ in $Z$ are mapped to the same vertex, i.e., $\phi(u)=\phi(v)$.
    Since $\phi$ is injective on darts, both $u$ and $v$ are boundary vertices.
    Then, there exists another homomorphism $\phi'$ from $Z$
    such that $\phi$ is represented by $\phi'' \circ \phi'$, where $\phi''$ is a homomorphism satisfying, for all vertices or darts $x,y$ in $\phi'(Z)$,
    $\phi''(x)=\phi''(y)$ if and only if $\{x,y\}=\{\phi'(u), \phi'(v)\}$.
    In $Z^*$, which is the image of $\phi$, the number of vertices decreases by $1$ from $\phi'(Z)$.
    The number of faces increases by at most $1$.
    Applying the induction hypothesis to $\phi'(Z)$, $|V(Z^*)|-|E(Z^*)|+|F(Z^*)| \leq |V|-|D|/2+|F|$ holds.
\end{proof}

\paragraph{Check separating cycles}
Recall in Lemma \ref{lem:triangulation}, $G^*$ has no separating cycle of length at most 3.
$G^*$ neither has a cycle of length 4 such that more than 2 vertices are in both disks separated by this cycle.
The \hasSeparatingCycle subroutine below checks these kinds of cycles in a given pseudo-embedding.
We will show Claim \ref{claim:cycle} below.

A \emph{walk} of length $l$ is a sequence of darts $(d_i)_{i=0}^{l-1}$ where $\head(d_i)=\tail(d_{i+1})$ for every $0 \leq i < l-1$.
A \emph{closed walk} of length $l$ is a walk of length $l$ such that $\head(d_{l-1})=\tail(d_0)$.
A \emph{cycle} of length $l$ is a closed walk of length $l$ where if $i \neq j$ then $\head(d_i) \neq \head(d_j)$ for all $0 \leq i,j \leq l-1$.

The \hasSeparatingCycle subroutine takes a connected loopless pseudo-embedding $(Z, \delta^-, \delta^+)$ with degree-range functions such that $\isPlanar(Z)=\textsf{true}$.
The \hasSeparatingCycle subroutine proceeds as follows.
First, we enumerate cycles of length at most 4 in $Z$.
The pseudocode to enumerate such cycles is presented in Algorithm \ref{alg:enum_cycles}.
This algorithm searches cycles starting from each dart in $Z$, and simply runs DFS.

For each cycle $C=(d_i)_{i=0}^{l-1}$, the algorithm proceeds as follows.
The indices of $d_i$ are cyclic, i.e., if $i < 0$ or $i \geq l$, then $d_i=d_{i \bmod l}$.
We add the label \textsf{L} or \textsf{R} to all darts in $Z$.
In the following, the symbol $L_D(d)$ denotes the label of a dart $d$.
First, for each $0 \leq i \leq l-1$, assign $L_D(d_i)=\textsf{L}$, $L_D(\reverse(d_i))=\textsf{R}$. 
After that, we propagate these labels as follows.
For each $0 \leq i \leq l-1$, if $\successor(d_i) \neq \nil$ and $\successor(d_i) \neq \reverse(d_{i+1})$, and assign $L_D(\successor(d_i))=\textsf{L}$.
Moreover, if an unlabeled dart $d$ is reachable from $\successor(d_i)$ by following $\successor, \predecessor, \reverse$ pointers without intersecting any other darts in $C$ or their reverse, assign $L_D(d)=\textsf{L}$.
We also propagate labels from $\reverse(d_i)$ similarly.
Specifically, if $\successor(\reverse(d_{i+1})) \neq \nil$ and $\successor(\reverse(d_{i+1})) \neq d_i$, assign $L_D(\successor(\reverse(d_{i+1})))=\textsf{R}$, and propagate this label in the same way as above.
Under the assumption of Claim \ref{claim:cycle} below, $\isPlanar(Z)=\textsf{true}$, so no two labels are assigned to the same dart.
The pseudocode corresponding to the algorithm for adding labels to darts is presented in Algorithm \ref{alg:label_darts}.

Then, we consider the vertex set
$V_C \coloneq \{\head(d) \mid d \in C\}$,
$V_\textsf{L} \coloneq \{\head(d) \mid L_D(d)=\textsf{L}\} \setminus V_C$, and
$V_\textsf{R} \coloneq \{\head(d) \mid L_D(d)=\textsf{R}\} \setminus V_C$.
By the label propagation step, two darts whose heads are the same vertex not in $C$ have the same label, so $V_\textsf{L} \cap V_\textsf{R} = \emptyset$.
If all vertices in $V_\textsf{L}$ are inner,
we compute the value $n_\textsf{L}$ by the number of inner vertices in $V_\textsf{L}$.
Otherwise, $n_\textsf{L}$ is this number plus 1.
We compute the value $n_\textsf{R}$ in the same way from $V_\textsf{R}$.
The pseudocode to compute $n_\textsf{L}, n_\textsf{R}$ is presented in Algorithm \ref{alg:num_seperated_vertices}.

If the length of $C$ is at most 3, $n_\textsf{L}>0$, and $n_\textsf{R}>0$, return \textsf{true}.
Also, if the length of $C$ is 4, $n_\textsf{L}>2$, and $n_\textsf{R}>2$, return \textsf{true}.
Otherwise, we process the next cycle.
If we finish processing all cycles, return \textsf{false}.
The pseudocode of the \hasSeparatingCycle subroutine is presented in Algorithm \ref{alg:has_separating_cycle}.

We can show the following claim.
\begin{claim}
\label{claim:cycle}
\showlabel{claim:cycle}
    Let $(Z, \delta^-, \delta^+)$ be a connected loopless pseudo-embedding with degree-range functions with $\isPlanar(Z)=\textsf{true}$.
    We assume a homomorphism from $(Z, \delta^-, \delta^+)$ to $G^*$ exists.
    If $\hasSeparatingCycle((Z, \delta^-, \delta^+))=\textsf{true}$ then this homomorphism is not injective on darts.
\end{claim}
\begin{proof}
    Suppose that $\hasSeparatingCycle((Z, \delta^-, \delta^+))=\textsf{true}$, but the homomorphism is injective on darts.
    $Z$ has a cycle $C$ with length at most 3, $n_\textsf{L} > 0$, and $n_\textsf{R} > 0$, or length at most 4, $n_\textsf{L} > 2$, and $n_\textsf{R} > 2$.
    The image of $C$ under a homomorphism also forms a cycle in $G^*$.
    Since this homomorphism is injective on darts, two distinct vertices of $Z$ can be mapped to the same vertex only if they are both boundary vertices.
    Then, the number of vertices in the two disks separated by the image of $C$ in $G^*$ is at least $n_\textsf{L}$ and $n_\textsf{R}$ respectively.
    The image of $C$ forming a cycle in $G^*$ contradicts Lemma \ref{lem:triangulation}.
\end{proof}

Note that the pruning by Claim \ref{claim:cycle} is important for reducing the size of the set $\mathcal{I}$.

\paragraph{All homomorphic images}
We also allow a multi-boundary configuration to be embedded in non-planar surfaces, and call it a \emph{possibly non-planar multi-boundary configuration}.
We consider the routine $\allHomImages(K, \mathcal{K}_\textsf{smaller})$ where the input $K$ is a possibly non-planar multi-boundary configuration, and $\mathcal{K}_\textsf{smaller}$ is a set of configurations.
This routine outputs the set $\mathcal{L}$ of multi-boundary islands with the following property: if there exists a homomorphism from $K$ to $G^*$ but there does not exist any homomorphism from a configuration in $\mathcal{K}_\textsf{smaller}$ to $G^*$, then there exists one multi-boundary island $I \in \mathcal{L}$ such that $I$ or a trimmed multi-boundary island of $I$ appears in $G$.
The routine $\allHomImages(K, \Ksmaller)$ proceeds as follows.

As an output set, we initialize $\mathcal{L}$ as an empty set.
Let $\widehat K$ be the outer extension of $K$.
If $\widehat K$ is blocked by the configuration set $\mathcal{K}_\textsf{smaller}$, then return $\mathcal{L}=\emptyset$, and finish this routine.
Otherwise, the routine proceeds as follows.
If $\isPlanar(\widehat K)=\textsf{true}$, then $K$ can be regarded as a multi-boundary configuration (in the plane).
If $\hasSeparatingCycle(\widehat K)=\textsf{false}$, we add one island to $\mathcal{L}$ as follows.
Recall that $I(K)$ is the multi-boundary island associated with $K$ described in Section \ref{subsect:conf-2-island}.
If no face $F \in F_R(K)$ has ring size of $1$, we choose $I(K)$.
Otherwise, the face in $F_R(I(K))$ originating from $F \in F_R(K)$ of ring size 1 is incident to only one degree-one vertex.
We obtain a trimmed island $I(K)$ by deleting every degree-one vertex incident to such a face.
We choose this island in this case.
If the chosen island has at most 3 degree-2 vertices\footnote{This is the only part where we use the fact that the number of vertices of degree 2 is at most 3 in the algorithm in Section~\ref{sect:hom-conf}.}, we add it to $\mathcal{L}$.
In this part, we have to implement the algorithm to construct such an island from an outer extension $\widehat K$.
Since the implementation details of the construction are rather technical, we defer them to Section \ref{subsect:island-from-outer-extension}.
For other cases, i.e., $\isPlanar(\widehat K)=\textsf{false}$ or $\hasSeparatingCycle(\widehat K)=\textsf{true}$, we add nothing to $\mathcal{L}$ here.

Next, we consider all pairs of distinct darts in $\widehat K$, and for each pair $(e,f)$, compute the set of free homomorphic images of $\widehat K$ respecting $S=\{(e,f)\}$ and forcing single digon incidence, denoted by $\mathcal{Z}_{e,f}^*$.
However, as we know $\mathcal{Z}^*_{\reverse(e),f}=\mathcal{Z}^*_{e,\reverse(f)}$, if we consider all pairs containing a dart $e$, we do not need to consider a pair containing $\reverse(e)$.
Specifically, in this enumeration process, we loop over a dart $e$ in the outer loop, and we loop over a dart $f$ in the inner loop.
In the implementation, we maintain a set of marked darts.
When an unmarked dart $e$ is chosen in the outer loop, the inner loop ranges over all unmarked darts $f$ distinct from $e$.
After this iteration, we mark both $e$ and $\reverse(e)$.
Marked darts are not used later because every pair containing a marked dart has already been represented in an earlier iteration.
This pruning reduces the size of the set $\mathcal{I}$.

For the enumerated pair $(e, f)$, consider each element $((Z_{e,f}^*, \delta_{e,f}^{*-}, \delta_{e,f}^{*+}), \phi^*_{e,f}) \in \mathcal{Z}_{e,f}^*$.
By Claim \ref{claim:suc=nil iff head=degrre-one}, a boundary vertex has degree one in $Z_{e,f}^*$.
The \makeOuterExtension subroutine makes $(Z_{e,f}^*, \delta_{e,f}^{*-}, \delta_{e,f}^{*+})$ an outer extension of some possibly non-planar multi-boundary configuration.
The subroutine $\makeOuterExtension((Z_{e,f}^*, \allowbreak \delta_{e,f}^{*-}, \delta_{e,f}^{*+}))$ proceeds as follows.

We initialize an output set and an empty queue.
We push $(Z_{e,f}^*, \delta_{e,f}^{*-}, \delta_{e,f}^{*+})$ into the queue.
While this queue is non-empty, we pop one element $(Z,\delta^-,\delta^+)$ and process it as follows.
We consider an arbitrary dart $e_0$ where $\successor(e_0) \neq \nil$.
Let $e_{i+1}$ be $\reverse(\successor(e_i))$ for $i=0,1,2,3$.
Note that if $\successor(e_i) = \nil$ for some $i$, the later darts are not defined.
Since $\successor(e_0) \neq \nil$, $e_1$ is always defined, and $e_1 \neq e_0$ since otherwise $e_0$ is a loop.
We check whether every such dart $e_0$ satisfies at least one of the following: $\head(e_1)$ or $\head(e_2)$ is a boundary vertex, $e_2=e_0$, or $e_3=e_0$.
Note that $e_2$ is defined when $\head(e_1)$ is an inner vertex and $e_3$ is defined when $\head(e_2)$ is an inner vertex.
If every such dart satisfies this condition, we add $(Z, \delta^-,\delta^+)$ to the output set.
Otherwise, a dart $e_0$ exists where neither $\head(e_1)$ nor $\head(e_2)$ is a boundary vertex, $e_2 \neq e_0$ and $e_3 \neq e_0$.
The pseudocode to find such darts $e_i$ ($0 \leq i \leq 3$) is presented in Algorithm \ref{alg:find_four_darts}.
We compute a set $\mathcal{Z}_2$ of free homomorphic images of $(Z,\delta^-,\delta^+)$ respecting $\{(e_0,e_2)\}$ and enforcing single digon incidence.
Also, we compute a set $\mathcal{Z}_3$ of free homomorphic images of $(Z,\delta^-,\delta^+)$ respecting $\{(e_0, e_3)\}$ and enforcing single digon incidence.
The pseudocode to compute $\mathcal{Z}_2$ and $\mathcal{Z}_3$ is presented in Algorithm \ref{alg:enusre_outer_extension}.
We push all elements in $\mathcal{Z}_2$ and $\mathcal{Z}_3$ into the queue, then finish the process for $(Z,\delta^-,\delta^+)$.
This iteration terminates in finite time because the number of darts decreases at each iteration.
The pseudocode of the \makeOuterExtension subroutine is presented in Algorithm \ref{alg:make_outer_extension}.

Let $\mathcal{S}=\makeOuterExtension((Z_{e,f}^*, \delta_{e,f}^{*-}, \delta_{e,f}^{*+}))$.
Every pseudo-embedding in $\mathcal{S}$ can be regarded as an outer extension of some possibly non-planar multi-boundary configuration, see Claim \ref{claim:make-outer-extension} below.
For each of them, we recursively call $\allHomImages$ with the same $\mathcal{K}_\textsf{smaller}$, and add its output to $\mathcal{L}$.
We process other pseudo-embeddings in $\mathcal{Z}_{e,f}^*$ in the same way, and update $\mathcal{L}$.
For other pairs of darts except $(e, f)$, we process them in the same way.
We return $\mathcal{L}$ as the output of $\allHomImages(K, \mathcal{K}_\textsf{smaller})$.
The recursive call of \allHomImages terminates in finite time since at every recursive call, the number of darts decreases.
The pseudocode of \allHomImages is presented in Algorithm \ref{alg:all_hom_images}.

Using the $\allHomImages$ routine, we construct the set of islands $\mathcal{I}$ as follows.
We order configurations in $\mathcal{K}$ by $K_1,K_2,\ldots,K_{|\mathcal{K}|}$, and initialize $\mathcal{K}_\textsf{smaller}=\emptyset$.
For each $K_i$ ($i=1,2,\ldots,|\mathcal{K}|$) in this order, we call $\allHomImages(K_i, \mathcal{K}_\textsf{smaller})$.
After calling it, we update $\mathcal{K}_\textsf{smaller}$ by adding $K_i$ and its mirror image, say $K_i^{\textsf{mirror}}$.
The union of all islands obtained by calling the \allHomImages routine is $\mathcal{I}$.
In other words, $\mathcal{I} = \bigcup_{1 \leq i \leq |\mathcal{K}|} \allHomImages(K_i, \{K_j, K_j^{\textsf{mirror}}\}_{j=1}^{i-1})$.
Lemma \ref{lem:I-is-reducible} is shown by constructing the set $\mathcal{I}$ and checking the semi-reducibility of all multi-boundary islands in $\mathcal{I}$, see Lemma \ref{comp:lem:reducible}.

\begin{lem}
\label{lem:I-is-reducible}
\showlabel{lem:I-is-reducible}
    All multi-boundary islands in $\mathcal{I}$ are semi-D-reducible or semi-C-reducible.
\end{lem}

\subsection{Analysis}
Our goal is to show Theorem \ref{thm:hom-imply-reducible} by analyzing the above \allHomImages routine.

\begin{claim}
\label{claim:island-appear}
\showlabel{claim:island-appear}
    Let $K$ be a multi-boundary configuration, and $\widehat K$ be its outer extension.
    Assume a homomorphism $\phi$ from $\widehat K$ to $G^*$ exists.
    If $\phi$ is injective on darts, then $I(K)$ or a trimmed multi-boundary island of $I(K)$ appears in $G$.
\end{claim}
\begin{proof}
    Since $\phi$ is injective on darts, $\phi$ is also injective for vertices in $K$.
    A pair of vertices that may be mapped to the same vertex is the outer endpoints of $K$.
    We consider the image of $\widehat K$ and darts in $G^*$ that form a facial triangle with two consecutive outer edges of $K$.
    If all faces bounded by the image of the outer edges of $K$ are triangles, $I(K)$ appears in $G$.
    Otherwise, one of the trimmed multi-boundary islands of $I(K)$ appears in $G$.
\end{proof}

\begin{claim}
\label{claim:make-outer-extension}
\showlabel{claim:make-outer-extension}
    Let $(Z, \delta^-, \delta^+)$ be a connected pseudo-embedding such that every boundary vertex has degree one, and $\mathcal{S}$ be the output of $\makeOuterExtension((Z, \delta^-, \delta^+))$. Then, the following holds.
    \begin{itemize}
        \item A pseudo-embedding in $\mathcal{S}$ can be considered as an outer extension of some possibly non-planar multi-boundary configuration.
        \item If there exists a homomorphism from $(Z, \delta^-, \delta^+)$ to $G^*$, then a homomorphism from some possibly non-planar multi-boundary configuration in $\mathcal{S}$ to $G^*$ exists.
    \end{itemize}
\end{claim}
\begin{proof}
    Let $(Z', \delta^{'-}, \delta^{'+})$ be a pseudo-embedding in $\mathcal{S}$.
    By Claim \ref{claim:suc=nil iff head=degrre-one}, every boundary vertex in $Z'$ has degree one.
    Every inner vertex in $Z'$ is the image of an inner vertex of $Z$, so it has a fixed degree-range, i.e., $\delta^{'-}(v)=\delta^{'+}(v)$.
    Considering the process in \makeOuterExtension, a dart $e_0$ in $Z'$ where $\successor(e_0) \neq \nil$ satisfies at least one of the following: the head of $e_1 \coloneq \reverse(\successor(e_0))$ is a boundary vertex, the head of $e_2 \coloneq \reverse(\successor(e_1))$ is a boundary vertex, $e_0=e_2$, or $e_0=e_3 \coloneq \reverse(\successor(e_2))$.
    Consider a walk computed in the $\getWalks$ subroutine.
    If this walk contains a dart whose head is a boundary vertex, no three consecutive darts have inner vertices as heads.
    Otherwise, the number of darts in this walk is either 2 or 3.
    Thus, the first claim holds.

    In \makeOuterExtension, we convert a pseudo-embedding by taking free homomorphic images respecting $\{(e_0,e_2)\}$ or $\{(e_0, e_3)\}$, where $e_{i+1}=\reverse(\successor(e_i))$ for $i=0,1,2$.
    To prove the second claim, it suffices to show that a homomorphism from $(Z, \delta^-, \delta^+)$ to $G^*$ respects $\{(e_0, e_2)\}$ or $\{(e_0, e_3)\}$.
    Since $G^*$ is a pseudo-triangulation with digons, the images of $e_0,e_1,e_2,e_3$ satisfy either $\phi(e_0)=\phi(e_2)$ or $\phi(e_0)=\phi(e_3)$ by (M5).
    Hence, the homomorphism respects either $\{(e_0, e_2)\}$ or $\{(e_0, e_3)\}$.
\end{proof}

\begin{lem}
\label{lem:K-imply-I}
\showlabel{lem:K-imply-I}
    Let $K$ be a possibly non-planar multi-boundary configuration, $\mathcal{K}_\textsf{smaller}$ be a set of configurations, and $\mathcal{L}$ be the output of $\allHomImages(K, \mathcal{K}_\textsf{smaller})$.
    If there exists a homomorphism from $K$ to $G^*$ but there does not exist a homomorphism from any configuration in $\mathcal{K}_\textsf{smaller}$ to $G^*$, then there exists a multi-boundary island $I \in \mathcal{L}$ where $I$ or a trimmed multi-boundary island of $I$ appears in $G$.
\end{lem}
\begin{proof}
    We prove by induction on the number of darts in the outer extension $\widehat K$.
    By definition, a homomorphism from $K$ to $G^*$ can be extended to a homomorphism from $\widehat K$ to $G^*$.
    Let $\phi$ be such a homomorphism.

    If $\widehat K$ is blocked by $\mathcal{K}_\textsf{smaller}$, a homomorphism $\phi$ from some configuration $L \in \mathcal{K}_\textsf{smaller}$ to $\widehat K$ exists.
    By composing a homomorphism from $\widehat K$ to $G^*$, a homomorphism from $L$ to $G^*$ exists.
    This contradicts the assumption, so it is fine that \allHomImages returns $\mathcal{L}=\emptyset$.
    
    If $\isPlanar(\widehat K)=\textsf{false}$, then by Claim \ref{claim:is-planar}, $\phi$ is not injective on darts.
    If $\isPlanar(\widehat K)=\textsf{true}$ and $\hasSeparatingCycle(\widehat K)=\textsf{true}$, then by Claim \ref{claim:cycle}, $\phi$ is not injective on darts.
    Consider the remaining case: $\isPlanar(\widehat K)=\textsf{true}$ and $\hasSeparatingCycle(\widehat K)=\textsf{false}$.
    $K$ is a multi-boundary configuration.
    If $\phi$ is injective on darts, by Claim \ref{claim:island-appear}, $I(K)$ or a trimmed multi-boundary island of $I(K)$ appears in $G$.
    If $I(K)$ has a face incident to only one degree-one vertex, then $I(K)$ does not appear in $G$, since $G$ is bridgeless.
    We add a trimmed multi-boundary island of $I(K)$ by deleting such degree-one vertices to $\mathcal{L}$ (if no such vertex exists, it is the same as $I(K)$).
     
    We consider the case where $\phi$ is not injective on darts.
    Let $e,f$ be two distinct darts in $\widehat K$ that are mapped to the same dart in $G^*$.
    By the definition of free homomorphic images, there exists a homomorphism from some pseudo-embedding, say $(Z_{e,f}^*, \delta_{e,f}^{*-}, \delta_{e,f}^{*+})$, in $\mathcal{Z}_{e,f}^*$ to $G^*$.
    $Z_{e,f}^*$ has smaller number of darts than that of $\widehat K$.
    By Claim \ref{claim:make-outer-extension}, there exists a homomorphism from some possibly non-planar multi-boundary configuration in $\mathcal{S}=\makeOuterExtension((Z_{e,f}^*, \delta_{e,f}^{*-}, \delta_{e,f}^{*+}))$ to $G^*$.
    During the algorithm, we add the resulting set of multi-boundary islands for the \allHomImages routine with input some multi-boundary configuration represented by elements in $\mathcal{S}$ to $\mathcal{L}$.
    Applying the induction hypothesis to it shows that the claim holds.

    Finally, the base case of this induction is when, for all pairs of darts $e,f$, the set of free homomorphic images is empty.
\end{proof}

\begin{proof}[Proof of Theorem \ref{thm:hom-imply-reducible}]
    We order configurations in $\mathcal{K}$ by $K_1,K_2,\ldots,K_{|\mathcal{K}|}$.
    We consider the smallest index $i$ such that a homomorphism from $K_i$ or $K_i^{\textsf{mirror}}$ to $G^*$ exists.
    By reversing the orientation of the plane embedding if necessary, assume that this homomorphism is from $K_i$.
    Let $\mathcal{L}$ be the output of $\allHomImages(K_i, \{K_j, K_j^{\textsf{mirror}}\}_{j=1}^{i-1})$.
    By Lemma \ref{lem:K-imply-I}, there exists $I \in \mathcal{L}$ such that either $I$ or a trimmed multi-boundary island of $I$ appears in $G$.
    Since $\mathcal{L} \subseteq \mathcal{I}$, Lemma \ref{lem:I-is-reducible} implies that $I$ is semi-reducible.
    By Lemma \ref{lem:imply}, every trimmed multi-boundary island of $I$ is also semi-reducible.
    Hence, the claim holds. 
\end{proof}

\subsection{How to get a multi-boundary island from an outer extension}
\label{subsect:island-from-outer-extension}
\showlabel{subsect:island-from-outer-extension}
In this section, we explain how to get a multi-boundary island, which is added to $\mathcal{I}$ in the \allHomImages routine.

First, we explain how to convert an outer extension $\widehat K$ of a multi-boundary configuration $K$ to a free completion of $K$.
Let $F$ be the set of walks obtained by calling $\getWalks(\widehat K)$.
For each walk $W \in F$, and for each dart $e_0$ in $W$, if $\predecessor(\reverse(e_0))=\nil$, then $\successor(e_0) \neq \nil$ since $\widehat K$ is an outer extension of a multi-boundary configuration.
Let $e_1$ be $\reverse(\successor(e_0))$.
If $\successor(e_1) = \nil$, then we add two darts between $\tail(e_0)$ and $\head(e_1)$.
This operation is done by calling \addBoundaryDartsDirectly in Algorithm \ref{alg:add_boundary_darts_directly} with input $\efirst=\reverse(e_1)$, $\elast=e_0$.
If $\successor(e_1) \neq \nil$, then let $e_2=\reverse(\successor(e_1))$.
Since $\widehat K$ is an outer-extension of a multi-boundary configuration, $\successor(e_2)=\nil$ holds.
In this case, we identify $\tail(e_0)$ and $\head(e_2)$.
This operation is done by calling \linkIncidenceListEnds in Algorithm \ref{alg:link_indidence_list_ends} with input $\eufirst=\reverse(e_0)$ and $\ewlast=e_2$.
The two subroutines above are the same as those used to compute free homomorphic images of pseudo-triangulations with digons.
After applying these operations for each walk $W$ and each dart $e_0$ in $W$, we obtain a free completion of $K$, which is denoted by $S$.
The pseudocode corresponding to this algorithm is presented in Algorithm \ref{alg:free_completion_from_outer_extension}.

Second, we explain how to get a multi-boundary island $I(K)$ from a free completion $S$ of $K$.
For each pair of darts that are reverse to each other, we create one edge.
Let $F_S$ be the set of walks computed by $\getWalks(S)$.
Since $S$ is a free completion, for each walk in $F_S$, either all darts in the walk have non-$\nil$ successors, or all darts in the walk have successor $\nil$.
Also, every walk that satisfies the former condition has a length of 2 or 3.
If $W$ in $F_S$ satisfies the former condition, then we create one vertex.
This vertex is incident to edges corresponding to pairs of darts containing darts in $W$.
If $W$ in $F_S$ satisfies the latter condition, we create $|W|$ vertices.
Each vertex is incident to an edge corresponding to a pair of darts containing darts in $W$.
The obtained graph is an island $I(K)$.

As described in Section \ref{subsect:alg-comp-check}, we delete every edge of $I(K)$ if this edge is the only edge that is incident to some face $F_R(I(K))$.
The conceptual algorithm to obtain such an island from a free completion $S$ is straightforward.
Its implementation, however, must take into account the file format of multi-boundary islands used as input to the semi-reducibility checker.
In this file format, the edges in $I$ are assigned indices, with the edges in $E_R(I)$ required to appear first.
Therefore, the construction must carefully maintain edge indices.

The implemented algorithm also adds an auxiliary edge incident with each vertex of degree two to make its degree 3.
The overall pseudocode to implement the algorithm for constructing such an island from a free completion is presented in Algorithm \ref{alg:island-from-free-completion}.
Inside this function, the algorithms in Algorithm \ref{alg:index-boundary-edges}, \ref{alg:index-pendant-edges}, \ref{alg:index-other-edges} maintain the indices of edges.

\section{Algorithm}\label{sec:alg}\showlabel{sec:alg}

In this section, we consider the algorithmic corollary of our main theorem, Theorem \ref{mainth}. Let us mention the problem again. 

\begin{quote}{\bf Algorithmic Problem}\\
{\bf Input}: A 2-connected apex cubic graph $G$.\\
{\bf Output}: A three-edge-coloring of $G$.\\
{\bf Time complexity}: $O(n^2)$, where $n=|V(G)|$.
\end{quote}

\begin{itemize}
\item
Find an apex vertex $v$. 
Since the graph is cubic, testing $G-v$ for planarity for all $v$ costs $O(n^2)$, using linear-time planarity testing, say, by Hopcroft and Tarjan \cite{HopcroftT74}. Construct the planar subcubic graph $G$ as in Section 2.
\item 
Find one homomorphic image of the unavoidable configurations from $\mathcal{K}$ in $G^*$. The set $\mathcal{K}$ has 915 normal configurations, so by scanning local neighborhoods and testing homomorphisms, we can find one configuration in $\mathcal{K}$ in $O(n)$ time; see \cite{proj2024} for more details. 
\item
Convert the found homomorphic image of the configuration in $G^*$ into a semi-reducible multi-boundary island in $G$. This is exactly Theorem \ref{thm:hom-imply-reducible} and Lemma \ref{lem:K-imply-I}.
\item 
Reduce the graph, recursively three-edge-color the smaller graph, and extend the coloring back using the semi-D or semi-C reducibility certificate.
\end{itemize}
Each reduction decreases the graph size, and each extension requires $O(n)$ time to find a configuration and linear-time Kempe chain operations. We have $O(n)$ reductions, so the total time complexity is $O(n^2)$.

\section*{Acknowledgment}
We are grateful for the availability of modern generative-AI systems, which made possible an unusual additional reproducibility test: independent reconstruction of the verification software directly from the detailed pseudocode. We acknowledge in particular ChatGPT (OpenAI) and Claude (Anthropic), whose independently generated implementations reproduced the computational results of the paper.

%% file: tikz/rules-RSST.tex
\tikzset{deg1/.style={thick, circle, draw, fill=black, inner sep=0pt,}}
\tikzset{deg5/.style={thick, circle, draw, fill=black, inner sep=1.5pt,}}
\tikzset{deg6/.style={thick, circle, draw, fill=black, inner sep=0pt,}}
\tikzset{deg7/.style={thick, circle, draw, fill=white, inner sep=2pt,}}
\tikzset{deg8/.style={thick, rectangle, draw, fill=white, inner sep=2pt,}}
\tikzset{deg9/.style={thick, regular polygon, regular polygon sides=3, rotate=180, draw, fill=white, inner sep=1pt,}}
\tikzset{deg10/.style={thick, regular polygon, draw, fill=white, inner sep=2pt,}}
\tikzset{->-/.style={decoration={
    markings,
    mark=at position .6 with {\arrow{>}}}, postaction={decorate}}}
\tikzset{->>-/.style={decoration={
    markings, 
    mark=at position .5 with {\arrow{>}};, 
    mark=at position .6 with {\arrow{>}};}, postaction={decorate}}}

\begin{figure}[htbp]
\scalebox{0.9}{%
\begin{tabular}{ccccccc}
\begin{minipage}[t]{0.13\hsize}
\centering
\begin{tikzpicture} []
    \node [deg5] at (0.0, 0.0) (v0) {};
    \node [deg5] at (0.9, 0.0) (v1) {};
    \node [above = 0.15 cm of v1, anchor=center] (v1+) { $+$ };
    \draw [->>-] (v0) -- (v1);
    \foreach \u / \v in {v0/v1}
        \draw (\u) -- (\v);
    \foreach \u / \v in {}
        \draw [double] (\u) -- (\v);
\end{tikzpicture}
\end{minipage}
&
\begin{minipage}[t]{0.13\hsize}
\centering
\begin{tikzpicture} []
    \node [deg6] at (0.0, 0.0) (v0) {};
    \node [above = 0.15 cm of v0, anchor=center] (v0-) { $-$ };
    \node [deg7] at (0.9, 0.0) (v1) {};
    \node [above = 0.15 cm of v1, anchor=center] (v1+) { $+$ };
    \node [deg5] at (0.4500000000000001, 0.7794228634059948) (v2) {};
    \draw [->-] (v0) -- (v1);
    \foreach \u / \v in {v0/v1, v0/v2, v1/v2}
        \draw (\u) -- (\v);
    \foreach \u / \v in {}
        \draw [double] (\u) -- (\v);
\end{tikzpicture}
\end{minipage}
&
\begin{minipage}[t]{0.13\hsize}
\centering
\begin{tikzpicture} []
    \node [deg6] at (0.0, 0.0) (v0) {};
    \node [above = 0.15 cm of v0, anchor=center] (v0-) { $-$ };
    \node [deg6] at (0.9, 0.0) (v1) {};
    \node [above = 0.15 cm of v1, anchor=center] (v1+) { $+$ };
    \node [deg6] at (0.4500000000000001, 0.7794228634059948) (v2) {};
    \node [above = 0.15 cm of v2, anchor=center] (v2-) { $-$ };
    \node [deg5] at (-0.44999999999999984, 0.7794228634059949) (v3) {};
    \draw [->-] (v0) -- (v1);
    \foreach \u / \v in {v0/v1, v0/v2, v0/v3, v1/v2, v2/v3}
        \draw (\u) -- (\v);
    \foreach \u / \v in {}
        \draw [double] (\u) -- (\v);
\end{tikzpicture}
\end{minipage}
&
\begin{minipage}[t]{0.13\hsize}
\centering
\begin{tikzpicture} []
    \node [deg6] at (0.0, 0.0) (v0) {};
    \node [deg6] at (0.9, 0.0) (v1) {};
    \node [above = 0.15 cm of v1, anchor=center] (v1+) { $+$ };
    \node [deg6] at (0.4500000000000001, 0.7794228634059948) (v2) {};
    \node [above = 0.15 cm of v2, anchor=center] (v2-) { $-$ };
    \node [deg6] at (-0.44999999999999984, 0.7794228634059949) (v3) {};
    \node [above = 0.15 cm of v3, anchor=center] (v3-) { $-$ };
    \node [deg5] at (-5.551115123125783e-17, 1.5588457268119895) (v4) {};
    \draw [->-] (v0) -- (v1);
    \foreach \u / \v in {v0/v1, v0/v2, v0/v3, v1/v2, v2/v4, v2/v3, v3/v4}
        \draw (\u) -- (\v);
    \foreach \u / \v in {}
        \draw [double] (\u) -- (\v);
\end{tikzpicture}
\end{minipage}
&
\begin{minipage}[t]{0.13\hsize}
\centering
\begin{tikzpicture} []
    \node [deg6] at (0.0, 0.0) (v0) {};
    \node [deg6] at (0.9, 0.0) (v1) {};
    \node [above = 0.15 cm of v1, anchor=center] (v1+) { $+$ };
    \node [deg6] at (0.4500000000000001, 0.7794228634059948) (v2) {};
    \node [deg6] at (-0.44999999999999984, 0.7794228634059949) (v3) {};
    \node [deg6] at (-5.551115123125783e-17, 1.5588457268119895) (v4) {};
    \node [above = 0.15 cm of v4, anchor=center] (v4-) { $-$ };
    \node [deg5] at (0.8999999999999999, 1.5588457268119895) (v5) {};
    \draw [->-] (v0) -- (v1);
    \foreach \u / \v in {v0/v1, v0/v2, v0/v3, v1/v2, v2/v5, v2/v4, v2/v3, v3/v4, v4/v5}
        \draw (\u) -- (\v);
    \foreach \u / \v in {}
        \draw [double] (\u) -- (\v);
\end{tikzpicture}
\end{minipage}
&
\begin{minipage}[t]{0.13\hsize}
\centering
\begin{tikzpicture} []
    \node [deg6] at (0.0, 0.0) (v0) {};
    \node [deg7] at (0.9, 0.0) (v1) {};
    \node [above = 0.15 cm of v1, anchor=center] (v1+) { $+$ };
    \node [deg5] at (0.4500000000000001, 0.7794228634059948) (v2) {};
    \node [deg6] at (-0.44999999999999984, 0.7794228634059949) (v3) {};
    \node [deg6] at (0.29371664009976267, 1.665749841116982) (v4) {};
    \node [above = 0.15 cm of v4, anchor=center] (v4-) { $-$ };
    \node [deg5] at (1.2957233587073176, 1.0872409923990967) (v5) {};
    \draw [->-] (v0) -- (v1);
    \foreach \u / \v in {v0/v1, v0/v2, v0/v3, v1/v2, v1/v5, v2/v5, v2/v4, v2/v3, v3/v4, v4/v5}
        \draw (\u) -- (\v);
    \foreach \u / \v in {}
        \draw [double] (\u) -- (\v);
\end{tikzpicture}
\end{minipage}
&
\begin{minipage}[t]{0.13\hsize}
\centering
\begin{tikzpicture} []
    \node [deg6] at (0.0, 0.0) (v0) {};
    \node [deg7] at (0.9, 0.0) (v1) {};
    \node [above = 0.15 cm of v1, anchor=center] (v1+) { $+$ };
    \node [deg6] at (-0.44999999999999984, 0.7794228634059949) (v2) {};
    \node [deg6] at (0.4500000000000001, 0.7794228634059948) (v3) {};
    \node [deg6] at (-5.551115123125783e-17, 1.5588457268119895) (v4) {};
    \node [deg6] at (0.8999999999999999, 1.5588457268119895) (v5) {};
    \node [above = 0.15 cm of v5, anchor=center] (v5-) { $-$ };
    \node [deg5] at (1.35, 0.7794228634059948) (v6) {};
    \draw [->-] (v0) -- (v1);
    \foreach \u / \v in {v0/v1, v0/v3, v0/v2, v1/v3, v1/v6, v2/v3, v2/v4, v3/v6, v3/v5, v3/v4, v4/v5, v5/v6}
        \draw (\u) -- (\v);
    \foreach \u / \v in {}
        \draw [double] (\u) -- (\v);
\end{tikzpicture}
\end{minipage}
\\
\begin{minipage}[t]{0.13\hsize}
\centering
\begin{tikzpicture} []
    \node [deg7] at (0.0, 0.0) (v0) {};
    \node [deg7] at (0.9, 0.0) (v1) {};
    \node [above = 0.15 cm of v1, anchor=center] (v1+) { $+$ };
    \node [deg6] at (0.5611408216728603, 0.7036483342212269) (v2) {};
    \node [above = 0.15 cm of v2, anchor=center] (v2-) { $-$ };
    \node [deg5] at (-0.20026884056068292, 0.8774351209636413) (v3) {};
    \draw [->-] (v0) -- (v1);
    \foreach \u / \v in {v0/v1, v0/v2, v0/v3, v1/v2, v2/v3}
        \draw (\u) -- (\v);
    \foreach \u / \v in {}
        \draw [double] (\u) -- (\v);
\end{tikzpicture}
\end{minipage}
&
\begin{minipage}[t]{0.13\hsize}
\centering
\begin{tikzpicture} []
    \node [deg7] at (0.0, 0.0) (v0) {};
    \node [deg7] at (0.9, 0.0) (v1) {};
    \node [above = 0.15 cm of v1, anchor=center] (v1+) { $+$ };
    \node [deg6] at (0.5611408216728603, 0.7036483342212269) (v2) {};
    \node [above = 0.15 cm of v2, anchor=center] (v2-) { $-$ };
    \node [deg6] at (-0.20026884056068292, 0.8774351209636413) (v3) {};
    \node [above = 0.15 cm of v3, anchor=center] (v3-) { $-$ };
    \node [deg5] at (-0.8108719811121772, 0.3904953652058024) (v4) {};
    \draw [->-] (v0) -- (v1);
    \foreach \u / \v in {v0/v1, v0/v2, v0/v3, v0/v4, v1/v2, v2/v3, v3/v4}
        \draw (\u) -- (\v);
    \foreach \u / \v in {}
        \draw [double] (\u) -- (\v);
\end{tikzpicture}
\end{minipage}
&
\begin{minipage}[t]{0.13\hsize}
\centering
\begin{tikzpicture} []
    \node [deg7] at (0.0, 0.0) (v0) {};
    \node [deg7] at (0.9, 0.0) (v1) {};
    \node [above = 0.15 cm of v1, anchor=center] (v1+) { $+$ };
    \node [deg5] at (0.5611408216728601, -0.703648334221227) (v2) {};
    \node [deg5] at (-0.20026884056068314, -0.8774351209636413) (v3) {};
    \node [deg5] at (0.5611408216728603, 0.7036483342212269) (v4) {};
    \node [deg5] at (-0.20026884056068292, 0.8774351209636413) (v5) {};
    \node [above = 0.15 cm of v5, anchor=center] (v5+) { $+$ };
    \draw [->-] (v0) -- (v1);
    \foreach \u / \v in {v0/v2, v0/v1, v0/v4, v0/v5, v0/v3, v1/v2, v1/v4, v2/v3, v4/v5}
        \draw (\u) -- (\v);
    \foreach \u / \v in {}
        \draw [double] (\u) -- (\v);
\end{tikzpicture}
\end{minipage}
&
\begin{minipage}[t]{0.13\hsize}
\centering
\begin{tikzpicture} []
    \node [deg7] at (0.0, 0.0) (v0) {};
    \node [deg7] at (0.9, 0.0) (v1) {};
    \node [above = 0.15 cm of v1, anchor=center] (v1+) { $+$ };
    \node [deg5] at (0.5611408216728603, 0.7036483342212269) (v2) {};
    \node [deg6] at (-0.20026884056068292, 0.8774351209636413) (v3) {};
    \node [deg5] at (0.5611408216728605, 1.4846390646328314) (v4) {};
    \node [deg5] at (0.5611408216728601, -0.703648334221227) (v5) {};
    \node [deg6] at (-0.20026884056068314, -0.8774351209636413) (v6) {};
    \node [left = 0.15 cm of v6, anchor=center] (v6+) { $+$ };
    \draw [->-] (v0) -- (v1);
    \foreach \u / \v in {v0/v2, v0/v3, v0/v6, v0/v5, v0/v1, v1/v2, v1/v5, v2/v3, v2/v4, v3/v4, v5/v6}
        \draw (\u) -- (\v);
    \foreach \u / \v in {}
        \draw [double] (\u) -- (\v);
\end{tikzpicture}
\end{minipage}
&
\begin{minipage}[t]{0.13\hsize}
\centering
\begin{tikzpicture} []
    \node [deg7] at (0.0, 0.0) (v0) {};
    \node [deg7] at (0.9, 0.0) (v1) {};
    \node [right = 0.15 cm of v1, anchor=center] (v1+) { $+$ };
    \node [deg5] at (0.5611408216728601, -0.703648334221227) (v2) {};
    \node [deg5] at (-0.20026884056068314, -0.8774351209636413) (v3) {};
    \node [deg5] at (0.56114082167286, -1.4846390646328316) (v4) {};
    \node [deg5] at (0.5611408216728603, 0.7036483342212269) (v5) {};
    \node [deg5] at (-0.20026884056068292, 0.8774351209636413) (v6) {};
    \node [deg5] at (0.5611408216728605, 1.4846390646328314) (v7) {};
    \draw [->-] (v0) -- (v1);
    \foreach \u / \v in {v0/v2, v0/v1, v0/v5, v0/v6, v0/v3, v1/v2, v1/v5, v2/v3, v2/v4, v3/v4, v5/v7, v5/v6, v6/v7}
        \draw (\u) -- (\v);
    \foreach \u / \v in {}
        \draw [double] (\u) -- (\v);
\end{tikzpicture}
\end{minipage}
&
\begin{minipage}[t]{0.13\hsize}
\centering
\begin{tikzpicture} []
    \node [deg7] at (0.0, 0.0) (v0) {};
    \node [deg8] at (0.9, 0.0) (v1) {};
    \node [right = 0.15 cm of v1, anchor=center] (v1+) { $+$ };
    \node [deg5] at (0.5611408216728603, 0.7036483342212269) (v2) {};
    \node [deg6] at (-0.20026884056068292, 0.8774351209636413) (v3) {};
    \node [deg5] at (0.754795608185546, 1.4742849311212725) (v4) {};
    \node [deg5] at (1.0737867867424147, 0.7614096622335433) (v5) {};
    \node [deg5] at (0.5611408216728601, -0.703648334221227) (v6) {};
    \node [deg5] at (-0.20026884056068314, -0.8774351209636413) (v7) {};
    \draw [->-] (v0) -- (v1);
    \foreach \u / \v in {v0/v6, v0/v1, v0/v2, v0/v3, v0/v7, v1/v6, v1/v5, v1/v2, v2/v5, v2/v4, v2/v3, v3/v4, v4/v5, v6/v7}
        \draw (\u) -- (\v);
    \foreach \u / \v in {}
        \draw [double] (\u) -- (\v);
\end{tikzpicture}
\end{minipage}
&
\begin{minipage}[t]{0.13\hsize}
\centering
\begin{tikzpicture} []
    \node [deg7] at (0.0, 0.0) (v0) {};
    \node [deg7] at (0.9, 0.0) (v1) {};
    \node [above = 0.15 cm of v1, anchor=center] (v1+) { $+$ };
    \node [deg6] at (0.5611408216728603, 0.7036483342212269) (v2) {};
    \node [deg5] at (-0.20026884056068292, 0.8774351209636413) (v3) {};
    \node [deg5] at (0.30319593583683485, 1.4408124657820967) (v4) {};
    \node [deg5] at (0.5611408216728601, -0.703648334221227) (v5) {};
    \draw [->-] (v0) -- (v1);
    \foreach \u / \v in {v0/v5, v0/v1, v0/v2, v0/v3, v1/v5, v1/v2, v2/v4, v2/v3, v3/v4}
        \draw (\u) -- (\v);
    \foreach \u / \v in {}
        \draw [double] (\u) -- (\v);
\end{tikzpicture}
\end{minipage}
\\
\begin{minipage}[t]{0.13\hsize}
\centering
\begin{tikzpicture} []
    \node [deg7] at (0.0, 0.0) (v0) {};
    \node [deg7] at (0.9, 0.0) (v1) {};
    \node [above = 0.15 cm of v1, anchor=center] (v1+) { $+$ };
    \node [deg6] at (0.5611408216728603, 0.7036483342212269) (v2) {};
    \node [deg5] at (-0.20026884056068292, 0.8774351209636413) (v3) {};
    \node [deg6] at (-0.8108719811121772, 0.3904953652058024) (v4) {};
    \node [deg5] at (-0.8108719811121773, -0.3904953652058022) (v5) {};
    \node [deg5] at (0.5611408216728601, -0.703648334221227) (v6) {};
    \node [deg7] at (-0.20026884056068314, -0.8774351209636413) (v7) {};
    \node [below = 0.15 cm of v7, anchor=center] (v7+) { $+$ };
    \draw [->-] (v0) -- (v1);
    \foreach \u / \v in {v0/v6, v0/v1, v0/v2, v0/v3, v0/v4, v0/v5, v0/v7, v1/v6, v1/v2, v2/v3, v3/v4, v4/v5, v5/v7, v6/v7}
        \draw (\u) -- (\v);
    \foreach \u / \v in {}
        \draw [double] (\u) -- (\v);
\end{tikzpicture}
\end{minipage}
&
\begin{minipage}[t]{0.13\hsize}
\centering
\begin{tikzpicture} []
    \node [deg7] at (0.0, 0.0) (v0) {};
    \node [deg7] at (0.9, 0.0) (v1) {};
    \node [above = 0.15 cm of v1, anchor=center] (v1+) { $+$ };
    \node [deg6] at (0.5611408216728603, 0.7036483342212269) (v2) {};
    \node [deg6] at (-0.20026884056068292, 0.8774351209636413) (v3) {};
    \node [deg6] at (-0.8108719811121772, 0.3904953652058024) (v4) {};
    \node [deg5] at (-0.8108719811121773, -0.3904953652058022) (v5) {};
    \node [deg5] at (-0.20026884056068314, -0.8774351209636413) (v6) {};
    \node [deg5] at (0.5611408216728601, -0.703648334221227) (v7) {};
    \draw [->-] (v0) -- (v1);
    \foreach \u / \v in {v0/v7, v0/v1, v0/v2, v0/v3, v0/v4, v0/v5, v0/v6, v1/v7, v1/v2, v2/v3, v3/v4, v4/v5, v5/v6, v6/v7}
        \draw (\u) -- (\v);
    \foreach \u / \v in {}
        \draw [double] (\u) -- (\v);
\end{tikzpicture}
\end{minipage}
&
\begin{minipage}[t]{0.13\hsize}
\centering
\begin{tikzpicture} []
    \node [deg7] at (0.0, 0.0) (v0) {};
    \node [deg7] at (0.9, 0.0) (v1) {};
    \node [above = 0.15 cm of v1, anchor=center] (v1+) { $+$ };
    \node [deg6] at (0.5611408216728603, 0.7036483342212269) (v2) {};
    \node [deg5] at (-0.20026884056068292, 0.8774351209636413) (v3) {};
    \node [deg5] at (-0.8108719811121772, 0.3904953652058024) (v4) {};
    \node [deg5] at (-0.20026884056068314, -0.8774351209636413) (v5) {};
    \node [deg5] at (0.5611408216728601, -0.703648334221227) (v6) {};
    \node [deg7] at (-0.8108719811121773, -0.3904953652058022) (v7) {};
    \node [left = 0.15 cm of v7, anchor=center] (v7+) { $+$ };
    \node [deg6] at (0.30319593583683485, 1.4408124657820967) (v8) {};
    \draw [->-] (v0) -- (v1);
    \foreach \u / \v in {v0/v6, v0/v1, v0/v2, v0/v3, v0/v4, v0/v7, v0/v5, v1/v6, v1/v2, v2/v8, v2/v3, v3/v8, v3/v4, v4/v7, v5/v7, v5/v6}
        \draw (\u) -- (\v);
    \foreach \u / \v in {}
        \draw [double] (\u) -- (\v);
\end{tikzpicture}
\end{minipage}
&
\begin{minipage}[t]{0.13\hsize}
\centering
\begin{tikzpicture} []
    \node [deg7] at (0.0, 0.0) (v0) {};
    \node [deg7] at (0.9, 0.0) (v1) {};
    \node [deg5] at (0.5611408216728601, -0.703648334221227) (v2) {};
    \node [deg6] at (-0.20026884056068314, -0.8774351209636413) (v3) {};
    \node [deg5] at (-0.8108719811121773, -0.3904953652058022) (v4) {};
    \node [deg5] at (0.9716628866661935, -1.3106386677394388) (v5) {};
    \node [deg5] at (1.174420480507537, -0.7311907554577427) (v6) {};
    \node [deg6] at (0.5611408216728603, 0.7036483342212269) (v7) {};
    \node [above = 0.15 cm of v7, anchor=center] (v7+) { $+$ };
    \draw [->-] (v0) -- (v1);
    \foreach \u / \v in {v0/v1, v0/v7, v0/v4, v0/v3, v0/v2, v1/v7, v1/v2, v1/v6, v2/v3, v2/v5, v2/v6, v3/v4, v3/v5, v5/v6}
        \draw (\u) -- (\v);
    \foreach \u / \v in {}
        \draw [double] (\u) -- (\v);
\end{tikzpicture}
\end{minipage}
&
\begin{minipage}[t]{0.13\hsize}
\centering
\begin{tikzpicture} []
    \node [deg7] at (0.0, 0.0) (v0) {};
    \node [deg8] at (0.9, 0.0) (v1) {};
    \node [above = 0.15 cm of v1, anchor=center] (v1+) { $+$ };
    \node [deg5] at (0.5611408216728601, -0.703648334221227) (v2) {};
    \node [deg6] at (-0.20026884056068314, -0.8774351209636413) (v3) {};
    \node [deg5] at (-0.8108719811121773, -0.3904953652058022) (v4) {};
    \node [deg5] at (0.56114082167286, -1.4846390646328316) (v5) {};
    \node [deg5] at (1.322550483906403, -0.8774351209636416) (v6) {};
    \node [deg6] at (-0.8108719811121772, 0.3904953652058024) (v7) {};
    \node [deg6] at (-0.20026884056068292, 0.8774351209636413) (v8) {};
    \node [deg6] at (0.5611408216728603, 0.7036483342212269) (v9) {};
    \node [above = 0.15 cm of v9, anchor=center] (v9+) { $+$ };
    \draw [->-] (v0) -- (v1);
    \foreach \u / \v in {v0/v2, v0/v1, v0/v9, v0/v8, v0/v7, v0/v4, v0/v3, v1/v2, v1/v6, v1/v9, v2/v3, v2/v5, v2/v6, v3/v4, v3/v5, v4/v7, v5/v6, v7/v8, v8/v9}
        \draw (\u) -- (\v);
    \foreach \u / \v in {}
        \draw [double] (\u) -- (\v);
\end{tikzpicture}
\end{minipage}
&
\begin{minipage}[t]{0.13\hsize}
\centering
\begin{tikzpicture} []
    \node [deg7] at (0.0, 0.0) (v0) {};
    \node [deg7] at (0.9, 0.0) (v1) {};
    \node [above = 0.15 cm of v1, anchor=center] (v1+) { $+$ };
    \node [deg6] at (0.5611408216728603, 0.7036483342212269) (v2) {};
    \node [deg5] at (-0.20026884056068292, 0.8774351209636413) (v3) {};
    \node [deg5] at (-0.8108719811121772, 0.3904953652058024) (v4) {};
    \node [deg5] at (-0.8108719811121773, -0.3904953652058022) (v5) {};
    \node [deg6] at (-0.20026884056068314, -0.8774351209636413) (v6) {};
    \node [deg6] at (0.5611408216728601, -0.703648334221227) (v7) {};
    \draw [->-] (v0) -- (v1);
    \foreach \u / \v in {v0/v7, v0/v1, v0/v2, v0/v3, v0/v4, v0/v5, v0/v6, v1/v7, v1/v2, v2/v3, v3/v4, v4/v5, v5/v6, v6/v7}
        \draw (\u) -- (\v);
    \foreach \u / \v in {}
        \draw [double] (\u) -- (\v);
\end{tikzpicture}
\end{minipage}
&
\begin{minipage}[t]{0.13\hsize}
\centering
\begin{tikzpicture} []
    \node [deg7] at (0.0, 0.0) (v0) {};
    \node [deg7] at (0.9, 0.0) (v1) {};
    \node [deg6] at (0.5611408216728603, 0.7036483342212269) (v2) {};
    \node [deg6] at (-0.20026884056068292, 0.8774351209636413) (v3) {};
    \node [above = 0.15 cm of v3, anchor=center] (v3-) { $-$ };
    \node [deg5] at (-0.8108719811121772, 0.3904953652058024) (v4) {};
    \node [deg5] at (-0.8108719811121773, -0.3904953652058022) (v5) {};
    \node [deg5] at (1.1744204805075373, 0.7311907554577426) (v6) {};
    \node [deg6] at (0.5611408216728601, -0.703648334221227) (v7) {};
    \node [below = 0.15 cm of v7, anchor=center] (v7+) { $+$ };
    \draw [->-] (v0) -- (v1);
    \foreach \u / \v in {v0/v7, v0/v1, v0/v2, v0/v3, v0/v4, v0/v5, v1/v7, v1/v6, v1/v2, v2/v6, v2/v3, v3/v4, v4/v5}
        \draw (\u) -- (\v);
    \foreach \u / \v in {}
        \draw [double] (\u) -- (\v);
\end{tikzpicture}
\end{minipage}
\\
\begin{minipage}[t]{0.13\hsize}
\centering
\begin{tikzpicture} []
    \node [deg7] at (0.0, 0.0) (v0) {};
    \node [deg7] at (0.9, 0.0) (v1) {};
    \node [deg6] at (0.5611408216728603, 0.7036483342212269) (v2) {};
    \node [deg5] at (1.1744204805075373, 0.7311907554577426) (v3) {};
    \node [deg5] at (-0.20026884056068292, 0.8774351209636413) (v4) {};
    \node [deg6] at (-0.8108719811121772, 0.3904953652058024) (v5) {};
    \node [deg5] at (-0.8108719811121773, -0.3904953652058022) (v6) {};
    \node [deg6] at (0.5611408216728601, -0.703648334221227) (v7) {};
    \node [below = 0.15 cm of v7, anchor=center] (v7+) { $+$ };
    \draw [->-] (v0) -- (v1);
    \foreach \u / \v in {v0/v7, v0/v1, v0/v2, v0/v4, v0/v5, v0/v6, v1/v7, v1/v3, v1/v2, v2/v3, v2/v4, v4/v5, v5/v6}
        \draw (\u) -- (\v);
    \foreach \u / \v in {}
        \draw [double] (\u) -- (\v);
\end{tikzpicture}
\end{minipage}
&
\begin{minipage}[t]{0.13\hsize}
\centering
\begin{tikzpicture} []
    \node [deg7] at (0.0, 0.0) (v0) {};
    \node [deg7] at (0.9, 0.0) (v1) {};
    \node [right = 0.15 cm of v1, anchor=center] (v1+) { $+$ };
    \node [deg7] at (0.5611408216728601, -0.703648334221227) (v2) {};
    \node [deg5] at (0.5611408216728603, 0.7036483342212269) (v3) {};
    \node [deg5] at (-0.20026884056068314, -0.8774351209636413) (v4) {};
    \node [deg5] at (0.16057124038403708, -1.3740888845573092) (v5) {};
    \draw [->-] (v0) -- (v1);
    \foreach \u / \v in {v0/v2, v0/v1, v0/v3, v0/v4, v1/v2, v1/v3, v2/v4, v2/v5, v4/v5}
        \draw (\u) -- (\v);
    \foreach \u / \v in {}
        \draw [double] (\u) -- (\v);
\end{tikzpicture}
\end{minipage}
&
\begin{minipage}[t]{0.13\hsize}
\centering
\begin{tikzpicture} []
    \node [deg7] at (0.0, 0.0) (v0) {};
    \node [deg7] at (0.9, 0.0) (v1) {};
    \node [deg5] at (0.5611408216728603, 0.7036483342212269) (v2) {};
    \node [deg5] at (-0.20026884056068292, 0.8774351209636413) (v3) {};
    \node [deg5] at (-0.8108719811121772, 0.3904953652058024) (v4) {};
    \node [deg5] at (1.1744204805075373, 0.7311907554577426) (v5) {};
    \node [deg6] at (-0.8108719811121773, -0.3904953652058022) (v6) {};
    \node [deg6] at (-0.20026884056068314, -0.8774351209636413) (v7) {};
    \node [deg7] at (0.5611408216728601, -0.703648334221227) (v8) {};
    \node [right = 0.15 cm of v8, anchor=center] (v8+) { $+$ };
    \draw [->-] (v0) -- (v1);
    \foreach \u / \v in {v0/v8, v0/v1, v0/v2, v0/v3, v0/v4, v0/v6, v0/v7, v1/v8, v1/v5, v1/v2, v2/v5, v2/v3, v3/v4, v4/v6, v6/v7, v7/v8}
        \draw (\u) -- (\v);
    \foreach \u / \v in {}
        \draw [double] (\u) -- (\v);
\end{tikzpicture}
\end{minipage}
&
\begin{minipage}[t]{0.13\hsize}
\centering
\begin{tikzpicture} []
    \node [deg7] at (0.0, 0.0) (v0) {};
    \node [deg7] at (0.9, 0.0) (v1) {};
    \node [right = 0.15 cm of v1, anchor=center] (v1+) { $+$ };
    \node [deg5] at (0.5611408216728603, 0.7036483342212269) (v2) {};
    \node [deg5] at (-0.20026884056068292, 0.8774351209636413) (v3) {};
    \node [deg7] at (0.5611408216728601, -0.703648334221227) (v4) {};
    \node [deg5] at (-0.20026884056068314, -0.8774351209636413) (v5) {};
    \node [deg5] at (-0.8108719811121773, -0.3904953652058022) (v6) {};
    \node [deg5] at (0.7689131320083767, -1.4564943928857415) (v7) {};
    \node [deg6] at (0.16057124038403708, -1.3740888845573092) (v8) {};
    \draw [->-] (v0) -- (v1);
    \foreach \u / \v in {v0/v4, v0/v1, v0/v2, v0/v3, v0/v6, v0/v5, v1/v4, v1/v2, v2/v3, v4/v5, v4/v8, v4/v7, v5/v6, v5/v8, v7/v8}
        \draw (\u) -- (\v);
    \foreach \u / \v in {}
        \draw [double] (\u) -- (\v);
\end{tikzpicture}
\end{minipage}
&
\begin{minipage}[t]{0.13\hsize}
\centering
\begin{tikzpicture} []
    \node [deg7] at (0.0, 0.0) (v0) {};
    \node [deg7] at (0.9, 0.0) (v1) {};
    \node [deg6] at (0.5611408216728603, 0.7036483342212269) (v2) {};
    \node [deg6] at (-0.20026884056068292, 0.8774351209636413) (v3) {};
    \node [deg6] at (-0.8108719811121772, 0.3904953652058024) (v4) {};
    \node [deg5] at (-0.8108719811121773, -0.3904953652058022) (v5) {};
    \node [deg5] at (-0.20026884056068314, -0.8774351209636413) (v6) {};
    \node [deg5] at (1.3372208462963133, 0.6162051162491518) (v7) {};
    \node [deg7] at (0.5611408216728601, -0.703648334221227) (v8) {};
    \node [below = 0.15 cm of v8, anchor=center] (v8+) { $+$ };
    \draw [->-] (v0) -- (v1);
    \foreach \u / \v in {v0/v2, v0/v3, v0/v4, v0/v5, v0/v6, v0/v8, v0/v1, v1/v2, v1/v8, v1/v7, v2/v3, v2/v7, v3/v4, v4/v5, v5/v6, v6/v8}
        \draw (\u) -- (\v);
    \foreach \u / \v in {}
        \draw [double] (\u) -- (\v);
\end{tikzpicture}
\end{minipage}
&
\begin{minipage}[t]{0.13\hsize}
\centering
\begin{tikzpicture} []
    \node [deg7] at (0.0, 0.0) (v0) {};
    \node [deg7] at (0.9, 0.0) (v1) {};
    \node [deg7] at (0.5611408216728601, -0.703648334221227) (v2) {};
    \node [deg6] at (0.5611408216728603, 0.7036483342212269) (v3) {};
    \node [deg6] at (1.1744204805075373, 0.7311907554577426) (v4) {};
    \node [deg5] at (1.6181426667187135, 0.3069489065414417) (v5) {};
    \node [deg6] at (1.6181426667187133, -0.306948906541442) (v6) {};
    \node [below = 0.15 cm of v6, anchor=center] (v6+) { $+$ };
    \node [deg6] at (-0.20026884056068314, -0.8774351209636413) (v7) {};
    \node [left = 0.15 cm of v7, anchor=center] (v7-) { $-$ };
    \node [deg5] at (0.16057124038403708, -1.3740888845573092) (v8) {};
    \node [deg7] at (-0.20026884056068292, 0.8774351209636413) (v9) {};
    \node [left = 0.15 cm of v9, anchor=center] (v9+) { $+$ };
    \draw [->-] (v0) -- (v1);
    \foreach \u / \v in {v0/v2, v0/v1, v0/v3, v0/v9, v0/v7, v1/v2, v1/v6, v1/v5, v1/v4, v1/v3, v2/v7, v2/v8, v3/v4, v3/v9, v4/v5, v5/v6, v7/v8}
        \draw (\u) -- (\v);
    \foreach \u / \v in {}
        \draw [double] (\u) -- (\v);
\end{tikzpicture}
\end{minipage}
&
\begin{minipage}[t]{0.13\hsize}
\centering
\begin{tikzpicture} []
    \node [deg8] at (0.0, 0.0) (v0) {};
    \node [deg7] at (0.9, 0.0) (v1) {};
    \node [right = 0.15 cm of v1, anchor=center] (v1+) { $+$ };
    \node [deg5] at (0.6363961030678928, 0.6363961030678927) (v2) {};
    \node [deg5] at (5.5109105961630896e-17, 0.9) (v3) {};
    \node [deg5] at (-0.6363961030678927, 0.6363961030678928) (v4) {};
    \node [deg5] at (0.6363961030678926, -0.636396103067893) (v5) {};
    \node [deg5] at (-1.6532731788489269e-16, -0.9) (v6) {};
    \draw [->-] (v0) -- (v1);
    \foreach \u / \v in {v0/v5, v0/v1, v0/v2, v0/v3, v0/v4, v0/v6, v1/v5, v1/v2, v2/v3, v3/v4, v5/v6}
        \draw (\u) -- (\v);
    \foreach \u / \v in {}
        \draw [double] (\u) -- (\v);
\end{tikzpicture}
\end{minipage}
\\
\begin{minipage}[t]{0.13\hsize}
\centering
\begin{tikzpicture} []
    \node [deg8] at (0.0, 0.0) (v0) {};
    \node [deg7] at (0.9, 0.0) (v1) {};
    \node [right = 0.15 cm of v1, anchor=center] (v1+) { $+$ };
    \node [deg5] at (0.6363961030678928, 0.6363961030678927) (v2) {};
    \node [deg5] at (5.5109105961630896e-17, 0.9) (v3) {};
    \node [deg5] at (-0.6363961030678927, 0.6363961030678928) (v4) {};
    \node [deg5] at (-0.9, 1.1021821192326179e-16) (v5) {};
    \node [deg5] at (0.7263064833220598, 1.3193332436601612) (v6) {};
    \node [deg5] at (0.6363961030678926, -0.636396103067893) (v7) {};
    \draw [->-] (v0) -- (v1);
    \foreach \u / \v in {v0/v7, v0/v1, v0/v2, v0/v3, v0/v4, v0/v5, v1/v7, v1/v2, v2/v6, v2/v3, v3/v6, v3/v4, v4/v5}
        \draw (\u) -- (\v);
    \foreach \u / \v in {}
        \draw [double] (\u) -- (\v);
\end{tikzpicture}
\end{minipage}
&
\begin{minipage}[t]{0.13\hsize}
\centering
\begin{tikzpicture} []
    \node [deg8] at (0.0, 0.0) (v0) {};
    \node [deg7] at (0.9, 0.0) (v1) {};
    \node [right = 0.15 cm of v1, anchor=center] (v1+) { $+$ };
    \node [deg5] at (0.6363961030678928, 0.6363961030678927) (v2) {};
    \node [deg5] at (5.5109105961630896e-17, 0.9) (v3) {};
    \node [deg5] at (0.6363961030678926, -0.636396103067893) (v4) {};
    \node [deg5] at (-0.636396103067893, -0.6363961030678927) (v5) {};
    \node [deg5] at (0.7263064833220592, -1.3193332436601617) (v6) {};
    \node [deg6] at (-1.6532731788489269e-16, -0.9) (v7) {};
    \draw [->-] (v0) -- (v1);
    \foreach \u / \v in {v0/v4, v0/v1, v0/v2, v0/v3, v0/v5, v0/v7, v1/v4, v1/v2, v2/v3, v4/v7, v4/v6, v5/v7, v6/v7}
        \draw (\u) -- (\v);
    \foreach \u / \v in {}
        \draw [double] (\u) -- (\v);
\end{tikzpicture}
\end{minipage}
&
\begin{minipage}[t]{0.13\hsize}
\centering
\begin{tikzpicture} []
    \node [deg8] at (0.0, 0.0) (v0) {};
    \node [deg7] at (0.9, 0.0) (v1) {};
    \node [right = 0.15 cm of v1, anchor=center] (v1+) { $+$ };
    \node [deg5] at (0.6363961030678926, -0.636396103067893) (v2) {};
    \node [deg5] at (-1.6532731788489269e-16, -0.9) (v3) {};
    \node [deg5] at (0.4193332436601614, -1.4464857228137262) (v4) {};
    \node [deg5] at (0.6363961030678928, 0.6363961030678927) (v5) {};
    \node [deg5] at (5.5109105961630896e-17, 0.9) (v6) {};
    \node [deg6] at (-0.636396103067893, -0.6363961030678927) (v7) {};
    \node [above = 0.15 cm of v7, anchor=center] (v7+) { $+$ };
    \node [deg6] at (-0.6363961030678927, 0.6363961030678928) (v8) {};
    \node [above = 0.15 cm of v8, anchor=center] (v8+) { $+$ };
    \node [deg5] at (0.41933324366016156, 1.4464857228137262) (v9) {};
    \node [above = 0.15 cm of v9, anchor=center] (v9+) { $+$ };
    \draw [->-] (v0) -- (v1);
    \foreach \u / \v in {v0/v6, v0/v8, v0/v7, v0/v3, v0/v2, v0/v1, v0/v5, v1/v2, v1/v5, v2/v3, v2/v4, v3/v7, v3/v4, v5/v9, v5/v6, v6/v8, v6/v9}
        \draw (\u) -- (\v);
    \foreach \u / \v in {}
        \draw [double] (\u) -- (\v);
\end{tikzpicture}
\end{minipage}
&
\begin{minipage}[t]{0.13\hsize}
\centering
\begin{tikzpicture} []
    \node [deg8] at (0.0, 0.0) (v0) {};
    \node [deg7] at (0.9, 0.0) (v1) {};
    \node [deg6] at (0.6363961030678928, 0.6363961030678927) (v2) {};
    \node [deg5] at (5.5109105961630896e-17, 0.9) (v3) {};
    \node [deg5] at (-0.6363961030678927, 0.6363961030678928) (v4) {};
    \node [deg5] at (-0.9, 1.1021821192326179e-16) (v5) {};
    \node [deg5] at (-0.636396103067893, -0.6363961030678927) (v6) {};
    \node [deg6] at (-1.6532731788489269e-16, -0.9) (v7) {};
    \node [deg7] at (0.6363961030678926, -0.636396103067893) (v8) {};
    \node [below = 0.15 cm of v8, anchor=center] (v8+) { $+$ };
    \node [deg5] at (1.1636038969321072, 0.6363961030678931) (v9) {};
    \draw [->-] (v0) -- (v1);
    \foreach \u / \v in {v0/v8, v0/v1, v0/v2, v0/v3, v0/v4, v0/v5, v0/v6, v0/v7, v1/v8, v1/v9, v1/v2, v2/v9, v2/v3, v3/v4, v4/v5, v5/v6, v6/v7, v7/v8}
        \draw (\u) -- (\v);
    \foreach \u / \v in {}
        \draw [double] (\u) -- (\v);
\end{tikzpicture}
\end{minipage}
&
\end{tabular}
}
\caption{The rules used in \cite{RSST} (also in \cite{doublecross}).}
\label{fig:rules-RSST}
\end{figure}

%% file: apex_connectivity.tex
\section{Connectivity Lemma}\label{sec:con}\showlabel{sec:con}

In this section, the definition of islands and related terms follows the usage in previous papers on purely cubic graphs, which we define in detail below. (Notice that this is different from the definition in this paper, which includes vertices of degree 1. After deleting vertices of degree one and their incident edges, the degrees of the endpoints become two. Such vertices correspond to vertices of degree 2 for islands in this section.)

\begin{dfn}
    Let $I$ be a graph of all vertices of degree $2$ or $3$.
    We call such graphs \emph{islands}.
    Let $G$ be a bridgeless cubic graph with a cyclic cut $F$ (a cut separating two connected components, each containing a cycle), and assume that one of the two connected components of $G - F$ is isomorphic to an island $I$.
    We say that $I$ \emph{appears} in $G$, and $F$ is the \emph{ring} of $I$.
    The vertices of degree 2 are called \emph{ring vertices}, and the edges in $F$ are called \emph{ring edges} of $I$.
    
    Let $I$ be an island with ring $F$.
    For any $\phi_I: E(I) \to [3]$ which is an edge-coloring of $I$, we can uniquely extend the coloring to $\phi_{I+F}: E(I+F) \to [3]$. 
    We call the restriction of $\phi_{I+F}$ to $F$ the \emph{ring coloring} of $I$. 
\end{dfn}

    We denote the set of all possible ring colorings of $I$ by $\mathcal{C}(I)$. 
    We denote all colorings $\phi$ that satisfy the parity equation $|\phi^{-1}(\{1\})| \equiv |\phi^{-1}(\{2\})| \equiv |\phi^{-1}(\{3\})| \pmod 2$ as $\mathcal{C}^{\text{valid}}$.

\begin{dfn}
    A set $\mathcal C\subseteq C^{\mathrm{valid}}$ is \emph{consistent} if the following holds. For every $\varphi\in\mathcal C$ and every $\kappa\in[3]$, there exists a pairing $\mathcal P$ of
    \[
        \varphi^{-1}([3]\setminus\{\kappa\})
    \]
    such that, for every subcollection $\mathcal P'\subseteq\mathcal P$, the coloring obtained from $\varphi$ by interchanging the two non-$\kappa$ colors on all members of $\bigcup\mathcal P'$ also belongs to $\mathcal C$.
\end{dfn}

Let $\mathcal{C}(I)$ denote the set of ring colorings that extend to a three-edge-coloring of $I$. 
We define $\mathcal{C}^*(I)$ to be the maximal consistent subset of $\mathcal{C}^{\mathrm{valid}}\setminus \mathcal{C}(I)$.
(This is well-defined, since the union of any family of consistent subsets of $\mathcal{C}^{\mathrm{valid}}\setminus \mathcal{C}(I)$ is consistent.)

The island $I$ is \emph{D-reducible} if $\mathcal{C}^*(I)=\emptyset$.

For the purposes of this section, $I$ is \emph{C-reducible} if there exists an island $J$ with the same cyclically ordered ring as $I$ such that $|V(J)|<|V(I)|$ and $C(J)\cap \mathcal{C}^*(I)=\varnothing$.

We first show the following.

\begin{lem}
    Let $I$ be an island that can be embedded in the plane so that all vertices of degree 2 are adjacent to the outer plane.
    (We call such islands \emph{disk islands}.)
    Let $F$ be its ring.
    $I$ is \emph{reducible} (meaning any non-3-edge-colorable 2-edge connected cubic graph containing $I$ can be made to a smaller non-3-edge-colorable 2-edge connected cubic graph by replacing $I$ with a smaller island) if one of the following holds.
    \begin{itemize}
        \item $2 \leq |F| \leq 4$
        \item $|F| = 5$ and $|V(I)| > 5$.
    \end{itemize}
    \label{lem:disk-cut}
\end{lem}

The method we use for this proof is similar to the one provided in \cite{3edgecoloring}.

\begin{proof}
    When $|F|=2,3$, all possible ring colorings that satisfy the parity condition are equivalent.
    So, $\mathcal{C}(I) = \mathcal{C}^{\text{valid}}$ if $\mathcal{C}(I)$ is nonempty.
    If $|F| = 2$, let $G$ be a cubic graph obtained by connecting the two vertices in $I$ that are of degree two with a new edge, and if $|F| = 3$, let $G$ be a cubic graph obtained by identifying the three vertices in $I+F$ that are degree one.
    In either case, $G$ is planar and, by 4CT, colorable.
    Therefore, $\mathcal{C}(I) = \mathcal{C}^{\text{valid}}$, and $I$ is reducible.

    Next, when $|F| = 4$, let us name the four vertices of $I$ having degree 2 in clockwise order as $v_1, v_2, v_3, v_4$.
    We construct six cubic graphs from $I$ as follows. ($u,w$ are new vertices.)
    
    \begin{itemize}
        \item $G^A_{ij} = I + v_iv_j + v_kv_l$ ($\{i,j,k,l\} = \{1,2,3,4\}$, $i < j$)
        \item $G^B_{ij} = I + v_iu + v_ju + v_kw + v_lw + uw$ ($\{i,j,k,l\} = \{1,2,3,4\}$, $i < j$).
    \end{itemize}

    The drawings of the graphs can be seen in \cref{fig:GAB}.

    \begin{figure}
        \centering
        \includegraphics[height=0.2\textheight]{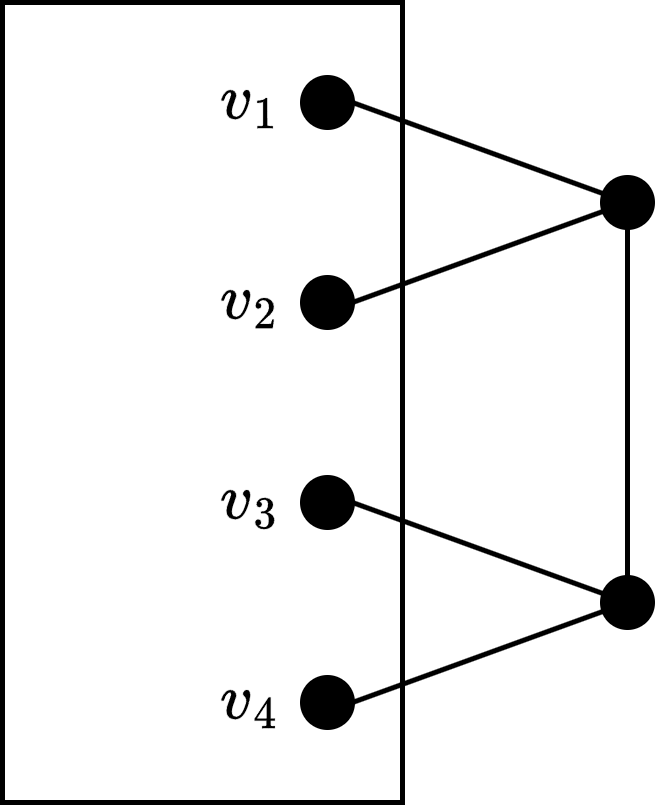}
        \includegraphics[height=0.2\textheight]{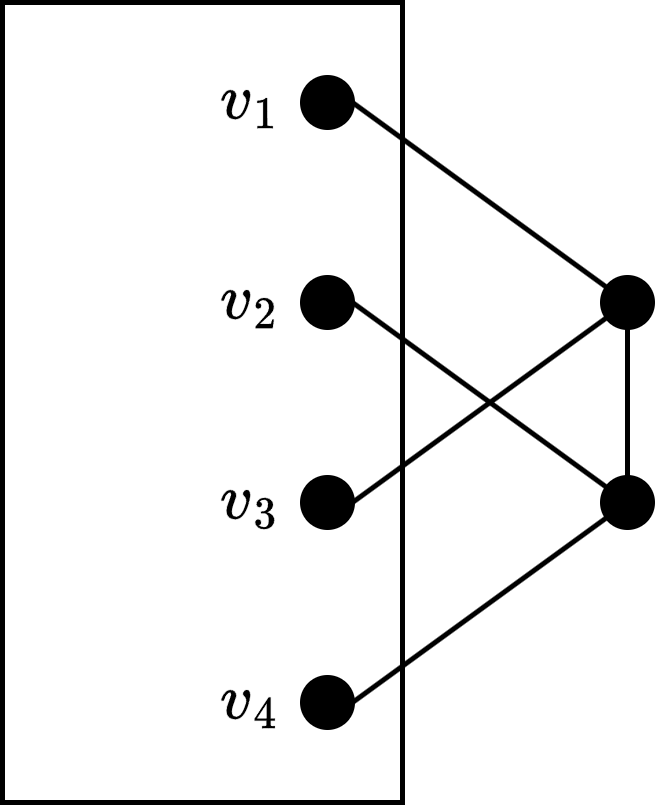}
        \includegraphics[height=0.2\textheight]{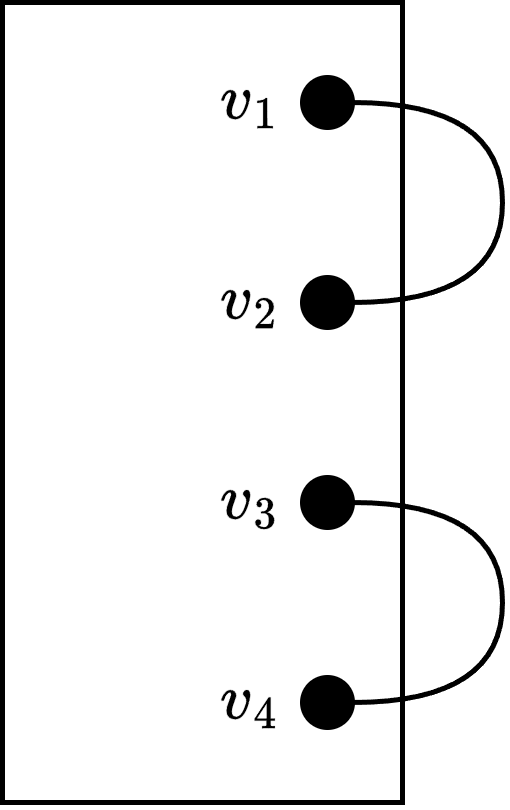}
        \includegraphics[height=0.2\textheight]{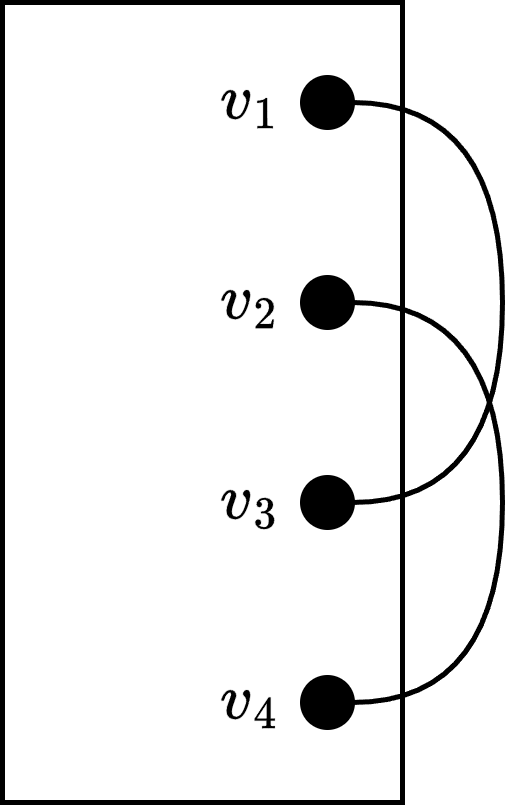}
        \caption{All possible $G^A_{ij}, G^B_{ij}$, excluding rotational symmetry.}
        \label{fig:GAB}
    \end{figure}

    $G^A_{12}, G^A_{14}, G^B_{12}, G^B_{14}$ are planar and $G^A_{13}, G^B_{13}$ are single-cross, so they are all colorable. 
    The multiset expression of the restriction of a coloring of $G^A_{ij}$ to $F$ must be $\{\{i,j,k,l\},\varnothing,\varnothing\}$ or $\{\{i,j\},\{k,l\},\varnothing\}$.
    The multiset expression of the restriction of a coloring of $G^B_{ij}$ to $F$ must be $\{\{i,k\},\{j,l\},\varnothing\}$ or $\{\{i,l\},\{j,k\},\varnothing\}$.
    Note that there are four unique (in the sense that they are pairwise non-equivalent) valid $F$ colorings.
    Since the intersection of $\mathcal{C}(I)$ and the coloring pair above is nonempty, $\mathcal{C}(I)$ contains at least three distinct elements.
    Hence
    \[
        \left|\mathcal{C}^{\mathrm{valid}}\setminus \mathcal{C}(I)\right|\leq 1
    \]
    up to equivalence.
    No nonempty consistent set can consist of only one of these four equivalence classes.
    Indeed, starting from any one of the four classes, choose a color $\kappa$ absent from $\phi$. 
    Switching the two non-$\kappa$ colors on one member of the corresponding pairing produces a different equivalence class. 
    Thus the unique maximal consistent subset of $\mathcal{C}^{\mathrm{valid}}\setminus \mathcal{C}(I)$ is empty.
    Hence $I$ is D-reducible.

    Finally, when $|F| = 5$, we construct 25 graphs as follows.

    \begin{itemize}
        \item $G^C_{ij} = I + v_iv_j + v_ku + v_lu + v_mu$ ($\{i,j,k,l,m\} = \{1,2,3,4,5\}$, $i < j$)
        \item $G^D_{ij, kl} = I + v_iu + v_ju + v_kw + v_lw + uv + wv + v_mv$ ($\{i,j,k,l,m\} = \{1,2,3,4,5\}$, $i < j$, $i < k$, $k < l$)
    \end{itemize}

    The drawings of the graphs are represented in \cref{fig:GCD}.
    
    \begin{figure}
        \centering
        \includegraphics[height=0.2\textheight]{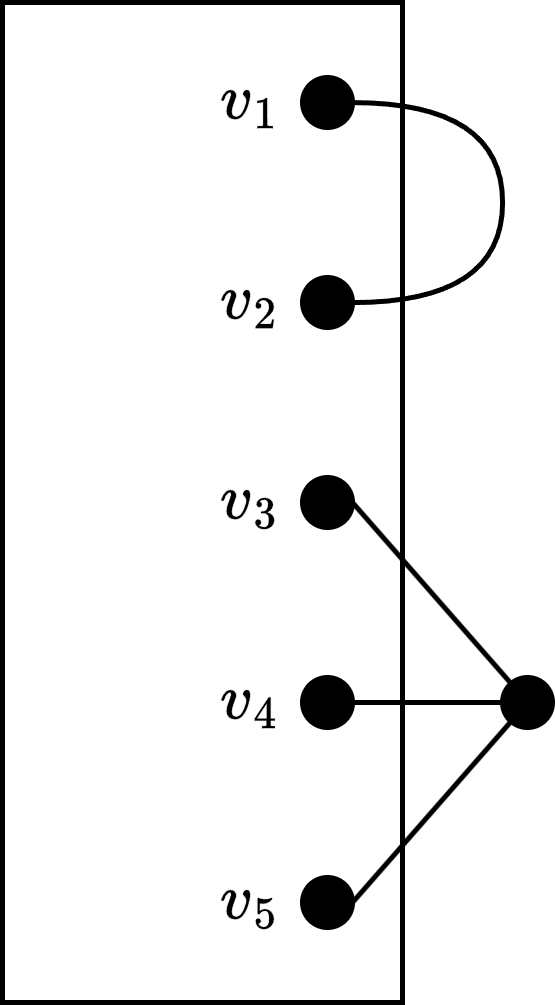}
        \includegraphics[height=0.2\textheight]{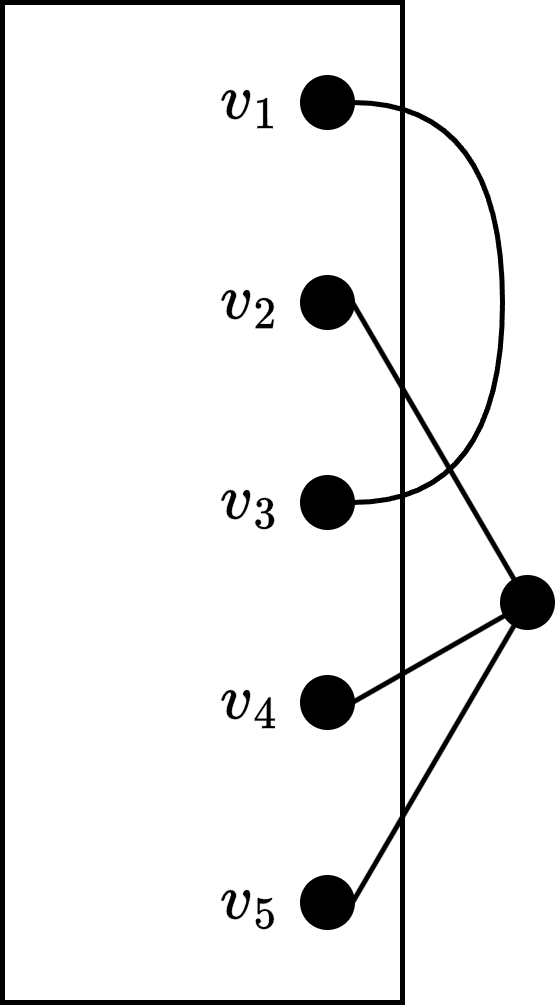}
        \includegraphics[height=0.2\textheight]{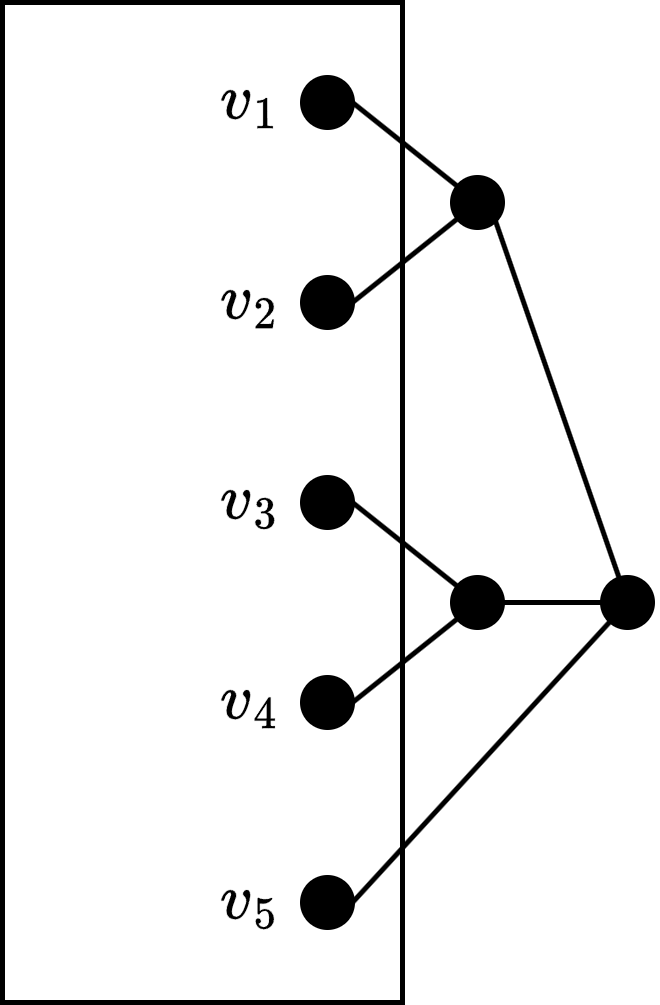}
        \includegraphics[height=0.2\textheight]{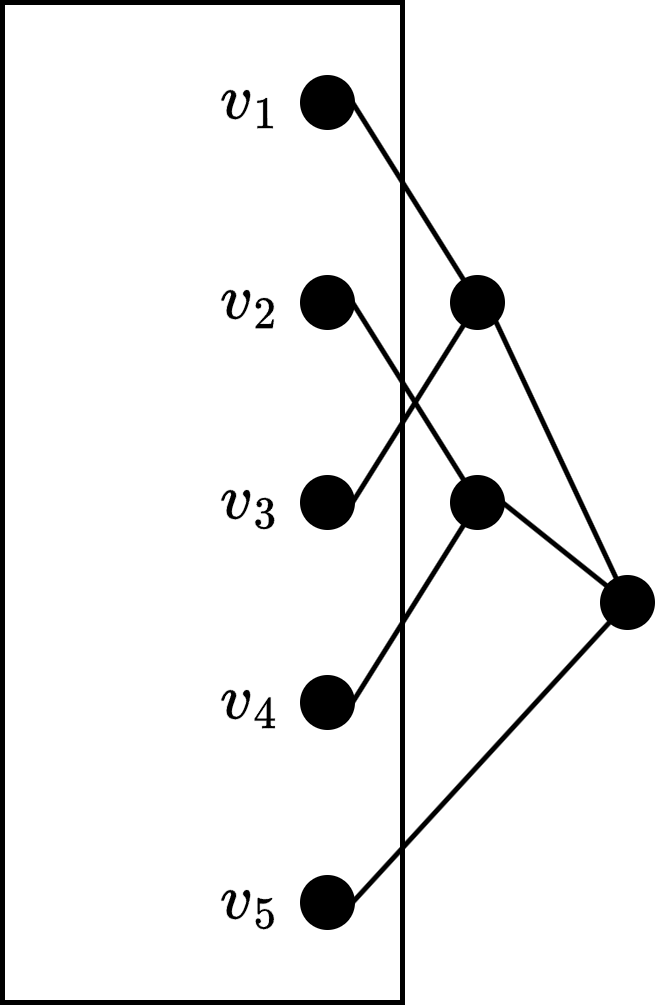}
        \includegraphics[height=0.2\textheight]{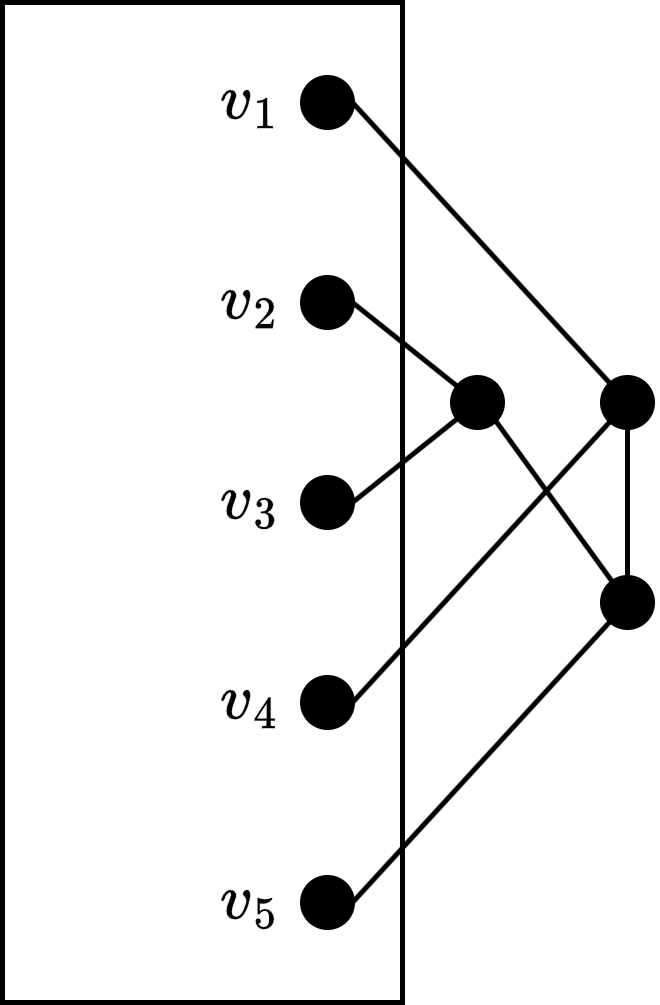}
        \caption{All possible $G^C_{ij}, G^D_{ij, kl}$, excluding rotational symmetry.}
        \label{fig:GCD}
    \end{figure}

    All of $G^C_{ij}$ and $G^D_{ij,kl}$ are planar or single-cross, so every graph in this class is colorable.
    Note that there are 10 equivalence classes of valid $F$ colorings, and their corresponding multisets must be of the form $\{\{i\}, \{j\}, \{k,l,m\}\}$, which we denote by $\phi^5_{ij}$ hereafter.
    
    The restriction of a coloring of $G^C_{ij}$ must be one of $\phi^5_{kl}, \phi^5_{km}, \phi^5_{lm}$.
    The restriction of a coloring of $G^D_{ij,kl}$ must be one of $\phi^5_{ik}, \phi^5_{il}, \phi^5_{jk}, \phi^5_{jl}$.
    We call the conditions above conditions C and D, respectively.

    Let $CG(I)$ be the graph with vertex set $\{1,2,3,4,5\}$ in which $ij$ is an edge if and only if the equivalence class $\varphi^5_{ij}$ belongs to
    $\mathcal{C}^*(I)$. 
    Since $\mathcal{C}^*(I)$ is consistent, $CG(I)$ has no vertex of degree one.
    To see this, suppose that $ij\in E(CG(I))$. Choose a representative of $\varphi^5_{ij}$ in which $f_i$ has color $1$, $f_j$ has color $2$, and the other three ring edges have color $3$ (where $f_i$ is the ring edge incident with $v_i$). 
    Apply consistency with $\kappa=1$. 
    In the resulting pairing, $f_j$ is paired with some $f_k$, where $k\notin\{i,j\}$. Switching on this pair gives the class $\varphi^5_{ik}\in \mathcal{C}^*(I)$. Hence $i$ has another neighbor in $CG(I)$.
    Interchanging the roles of $i$ and $j$ gives the same conclusion for $j$.
    
    By condition C, every set of three vertices contains an edge corresponding to a
    coloring in $\mathcal{C}(I)$. 
    Since $\mathcal{C}^*(I)$ is disjoint from $\mathcal{C}(I)$, the graph $CG(I)$ contains no triangle. 
    By condition D, $CG(I)$ does not contain all four edges between two pairs of vertices. Equivalently, $CG(I)$ contains no $4$-cycle.
    
    If $CG(I)$ is empty, then $I$ is D-reducible. 
    Suppose that $CG(I)$ is nonempty.
    Since $CG(I)$ has no vertex of degree one, it contains a cycle. 
    It has no triangle or $4$-cycle and has only five vertices, so $CG(I)$ is a $5$-cycle.
    Moreover, it has no additional edge, since every chord of a $5$-cycle creates a triangle or a $4$-cycle.
    
    The complement of a $5$-cycle in $K_5$ is another $5$-cycle. Let $J$ be a $5$-cycle island whose cyclic ordering of the ring is chosen so that
    \[
    C(J)=E(K_5)\setminus E(CG(I)).
    \]
    Then
    \[
    C(J)\cap \mathcal{C}^*(I)=\emptyset.
    \]
    Since $|V(J)|=5<|V(I)|$, $I$ is C-reducible.

\end{proof}

We now consider a three-edge-coloring of an apex cubic graph. 
To this end, let $G$ be a minimal counterexample to the 3-edge-coloring of an apex cubic graph.
We prove the following:

\begin{lem}
    If there exists an edge set $F$ with $|F| = 5$ such that $G - F$ has two connected components $I_1, I_2$, one of the following is true:
    \begin{enumerate}
        \item $\min\{|V(I_1)|, |V(I_2)|\} \leq 5$.
        \item One of $I_1, I_2$ is isomorphic to a \emph{domino} (a graph of vertex size 7, depicted in \cref{fig:domino}), and the apex vertex is in $V(I_2), V(I_1)$, respectively.
    \end{enumerate}
\end{lem} 

\begin{proof}
    Suppose that such an edge set $F$ exists, but none of the conclusions stated in the lemma are satisfied.

    We assume, without loss of generality, that the apex vertex lies in $V(I_2)$.
    If $F$ is not incident to the apex vertex,
    $I_1$ is an island embedded in a disk, in which we can apply Lemma \ref{lem:disk-cut} to obtain a contradiction.

    Now, we consider the case where an edge $f_0 \in F$ is incident to the apex vertex. 
    Let $f_1, f_2, f_3, f_4$ be edges in $F \setminus \{f_0\}$, labeled in clockwise order of the embedding.
    Let $G_{ij}^\mathcal{C}(I)$ and $G_{ij,kl}^D(I)$ be the graphs defined in the image as above, but for island $I$.

    Any combination of $G_{ij}^C(I_1)$ is apex, and $G_{ij,kl}^D(I_1)$ is apex for every combination except for $\{\{i,j\},\{k,l\}\} = \{\{1, 3\},\{2, 4\}\}$.
    For $I_2$, $G_{ij}^C(I_2)$ is apex with $i=0$ or $\{i,j\} \in \{\{1,2\},\{2,3\},\{3,4\},\{1,4\}\}$.

    We construct a graph $CG(I)$ with vertex set $F$, in which there is an edge between $i$ and $j$ if and only if $\phi_{ij}^5$ can be extended to a 3-edge-coloring of $I$.

    By the Kempe chain lemma, $CG(I_1)$ and $CG(I_2)$ have no vertices of degree 1. 
    Since $G$ is not 3-edge-colorable, $E(CG(I_1)) \cap E(CG(I_2)) = \varnothing$, and since it is a minimal counterexample, both $CG(I_1)$ and $CG(I_2)$ have at least one edge.

    With these two conditions on $CG(I_1)$ and $CG(I_2)$, we can say that one of $E(CG(I_1))$ or $E(CG(I_2))$ forms a cycle.
    (This is because $|E(CG(I_1))| + |E(CG(I_2))| \le |E(K_5)| = 10$, and any nonempty graph with no vertex of degree one and fewer than five edges has one nontrivial component, and that component is a cycle. The only graph that realizes this is $K_4^-$ (a graph obtained by deleting one edge from $K_4$), but we cannot pack two $K_4^-$s into $K_5$, resulting in a contradiction.)

    Now we do a case analysis:
    
    \begin{itemize}
        \item One of $E(CG(I_1)), E(CG(I_2))$ is a cycle of order 3.
        If $E(CG(I_1))$ is a triangle, we can find a triplet of vertices $k,l,m$ such that $kl, lm, mk \notin CG(I_1)$, and this will violate the colorability of $G_{ij}^C(I_1)$ (where $\{i,j\} = \{0,1,2,3,4\} \setminus \{k,l,m\}$).
        If $E(CG(I_2))$ is a triangle, we can simply use this triangle $k,l,m$ to confirm non-colorability of $G_{ij}^C(I_1)$.
        \item One of $E(CG(I_1)), E(CG(I_2))$ is a cycle of order 4.
        Let $ik, kj, jl, li$ be the four edges of the cycle. For every case except for when $E(CG(I_2))$ is the one containing the four cycle and $\{\{i,j\},\{k,l\}\} = \{\{1, 3\},\{2, 4\}\}$, we know that $G_{ij,kl}^D(I_1)$ or $G_{ij,kl}^D(I_2)$ (depending on which graph is the four cycle) is apex, and it is non-colorable. 
        The only case where this graph is (potentially) non-apex is when $\{\{i,j\},\{k,l\}\} = \{\{1, 3\},\{2, 4\}\}$, and in this case, let us connect the \emph{domino island} and $I_1$, as depicted in \cref{fig:domino}. This is non-colorable because the coloring graph of the domino is the same as $E(CG(I_2))$.
        \item One of $CG(I_1), CG(I_2)$ is a five-cycle.
        In this case, both $CG(I_1)$ and $CG(I_2)$ must be a five-cycle.
        Let $f_i, f_j, f_k, f_l, f_m$ be the order of the cycle in $CG(I_2)$. 
        Then, by connecting each of the degree 1 endpoints of $I_2 + F$ to a five-cycle connecting edges in order $f_i,f_k,f_m,f_j,f_l$, we can say this is non-colorable.
    \end{itemize}
\end{proof}

\begin{figure}[H]
    \centering
    \includegraphics[width=0.2\linewidth, valign=m]{img/lowcuts-domino.pdf}
    \qquad
    \includegraphics[width=0.45\linewidth, valign=m]{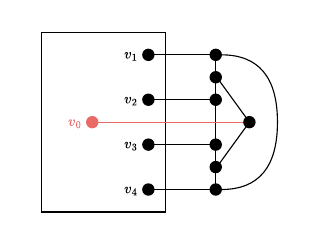}
    \caption{The domino graph and its attachment to an island, producing an apex graph.}
    \label{fig:domino}
\end{figure}

%% file: pseudo/reducibility_check.tex
\input{pseudo/github_commands}
\subsection{Reducibility checking}\label{appreduce}

The notation used in Algorithm \ref{alg:get_valid_parens_matching} and \ref{alg:get_planar_half_kempes_matching} is as follows: $[n]$ denotes $\{1,\ldots,n\}$.
If $M$ is a partial matching on $[n]$, define $U_n(M) := [n] \setminus \bigcup_{e \in M} e$, which is the set of unmatched positions.
If $M$ is a matching on $[n]$ and $B=\{b_1<\cdots<b_m\}$, define $B[M] := \{\{b_r,b_s\} : \{r,s\}\in M\}$ and $M+c := \{\{r+c,s+c\} : \{r,s\}\in M\}$.

These algorithms describe how we generate semi-matchings in Definition \ref{dfn:semi-matching} for an island with a single boundary.
Semi-matchings for a multi-boundary island are generated by concatenating the semi-matchings for its individual boundary components.

\begin{algorithm}[H]
    \caption{\getValidParensWithLink$(q)$}
    \label{alg:get_valid_parens_matching}
    \begin{algorithmic}[1]
        \Require{An integer $q \geq 1$.}
        \Ensure{The Catalan family of noncrossing perfect matchings on $[2q]$.}
        \If{$q=1$}
            \State \Return $\bigl\{\{\{1,2\}\}\bigr\}$
        \EndIf

        \State $\mathcal{P} \gets \emptyset$

        \For{$i \gets 1$ to $q-1$}
            \ForAll{$L \in \Call{GetValidParens}{i}$}
                \ForAll{$R \in \Call{GetValidParens}{q-i}$}
                    \State $\mathcal{P} \gets \mathcal{P} \cup \{L \cup (R+2i)\}$
                \EndFor
            \EndFor
        \EndFor

        \ForAll{$M \in \Call{GetValidParens}{q-1}$}
            \State $\mathcal{P} \gets \mathcal{P} \cup \{(M+1) \cup \{\{1,2q\}\}\}$
        \EndFor

        \State \Return $\mathcal{P}$
    \end{algorithmic}
\end{algorithm}

\begin{algorithm}[H]
    \caption{\getPlanarHalfKempesWithLink$(n)$}
    \label{alg:get_planar_half_kempes_matching}
    \begin{algorithmic}[1]
        \Require{An integer $n \geq 1$.}
        \Ensure{The set of nonredundant planar partial matchings on the cyclically ordered set $[n]$.}
        \If{$n=1$}
            \State \Return $\{\emptyset\}$
        \EndIf
        \State $\mathcal{R} \gets \emptyset$ \Comment{The set to be returned.}
        \State $\mathcal{A} \gets \emptyset$ \Comment{All generated partial matchings, including redundant ones.}

        \For{$q \gets \lfloor n/2 \rfloor$ down to $1$}
            \ForAll{$B \subseteq [n]$ with $|B|=2q$}
                \State $H \gets [n]\setminus B$
                \If{there exists $i \in [n]$ such that $i \in H$ and $i+1 \in H$ cyclically}
                    \State \textbf{continue}
                \EndIf

                \State write $B=\{b_1<\cdots<b_{2q}\}$

                \ForAll{$M_0 \in \Call{GetValidParens}{q}$}
                    \State $M \gets B[M_0]$
                    \State $\mathcal{A} \gets \mathcal{A} \cup \{M\}$
                    \State $\mathsf{redundant} \gets \mathrm{false}$

                    \ForAll{$\{x,y\} \subseteq U_n(M)$}
                        \State $M' \gets M \cup \{\{x,y\}\}$
                        \If{$M' \in \mathcal{A}$}
                            \State $\mathsf{redundant} \gets \mathrm{true}$
                        \EndIf
                    \EndFor

                    \If{$\mathsf{redundant}=\mathrm{false}$}
                        \State $\mathcal{R} \gets \mathcal{R} \cup \{M\}$
                    \EndIf
                \EndFor
            \EndFor
        \EndFor

        \State \Return $\mathcal{R}$
    \end{algorithmic}
\end{algorithm}

\subsubsection{Checking semi-C-reducibility}
\newcommand{\Ependant}{E_\textsf{pendant}}

In this section, we provide the pseudocode to enumerate the deletable edge set (see Definition \ref{dfn:deletable}), which is necessary to check C-reducibility of a given multi-boundary island.
When representing a multi-boundary island $I$ together with a set
$\Ependant$ of auxiliary edges, one added at each vertex of degree two, we assume that the edges are indexed so that the edges of $E_R(I)$ come first,
followed by the edges of $\Ependant$.

\begin{algorithm}[H]
    \caption{\getDeletableEdgeSetWithLink($I$)}
    \label{alg:get-deletable-edge-set}
    \begin{algorithmic}[1]
        \Require{A multi-boundary island $I$ together with a set
$\Ependant$ of auxiliary edges, one added at each vertex of degree two.}
        \Ensure{A collection of deletable edge sets.}
        \State $\texttt{exists} \gets$ the array of length $|E(I) \cup \Ependant|$ with all $-1$. \Comment{unset: $-1$, delete: $0$, retain: $1$}
        \For{$i \gets 0$ \textbf{to}
        $|E_R(I)|+|\Ependant|-1$}
            \State $\texttt{exists}[i]\gets 1$ \Comment{always retain edges in $E_R(I)$ and auxiliary edges.}
        \EndFor
        \State \Return
        \Call{GetDeletableEdgeSetInternal}{$I, \texttt{exists}$}
    \end{algorithmic}
\end{algorithm}

\begin{algorithm}[H]
\caption{\getDeletableEdgeSetInternalWithLink($I, \texttt{exists}$)}
\label{alg:get-deletable-edge-set-internal}
\begin{algorithmic}[1]
    \Require{A multi-boundary island $I$ together with a set
$\Ependant$ of auxiliary edges, one added at each vertex of degree two.}
    \Ensure{A collection of deletable edge sets.}
    
    \State $\texttt{EtoF}[e]\gets\emptyset$ for every edge $e \in E(I) \cup \Ependant$
    \For{each face $f \in F(I) \setminus F_R(I)$}
        \For{each edge $e$ appearing on the boundary walk of $f$}
            \State $\texttt{EtoF}[e] \gets \texttt{EtoF}[e] \cup \{f\}$
        \EndFor
    \EndFor
    
    \State $\texttt{existsList} \gets$ an empty array
    \State $\texttt{max\_delete\_count} \gets 4$
    
    \Procedure{Place}{$\texttt{exists},e_0$} \Comment{\texttt{exists} is modified inside this function.}
        \State $\texttt{Q} \gets$ an empty queue
        \State $\texttt{Q.push}(e_0)$
    
        \While{$\texttt{Q}$ is not empty}
            \State $e \gets \texttt{Q.pop()}$
            \For{each pair $(f,g)$ where $e,f,g \in E(I) \cup \Ependant$ incident with the same vertex }
                \State $\texttt{unset} \gets \{h \in \{e,f,g\} \mid \texttt{exists}[h]=-1\}$
                \State $\texttt{count1} \gets |\{h \in \{e,f,g\} \mid \texttt{exists}[h]=1\}|$
                \If{$|\texttt{unset}|=0$ \textbf{and} $\texttt{count1}=1$}
                    \State \Return \textbf{false} \Comment{Check the first condition of deletable.}
                \ElsIf{$|\texttt{unset}|=1$ \textbf{and} $\texttt{count1}\leq 1$}
                    \State Let $h$ be the unique edge in $\texttt{unset}$
                    \State $\texttt{exists}[h]\gets\texttt{count1}$  \Comment{Ensure the first condition of deletable.}
                    \State \texttt{Q.push($h$)}
                \EndIf
            \EndFor
        \EndWhile
    
        \State \Return \textbf{true}
    \EndProcedure
    
    \Procedure{Recurse}{$\texttt{exists},e$}
        \State $\texttt{delete\_count}\gets$ the number of $0$ in \texttt{exists}
        \If{$\texttt{delete\_count}>\texttt{max\_delete\_count}$}
            \State \Return
        \EndIf
        \While{$e<|E(I) \cup \Ependant|$ \textbf{and} $\texttt{exists}[e]\geq 0$}
            \State $e\gets e+1$
        \EndWhile
        \If{$e=|E(I) \cup \Ependant|$}
            \State $\texttt{boolExists}[h] \gets (\texttt{exists}[h]=1)$ for every $h\in E(I) \cup \Ependant$
            \If{$0<\texttt{delete\_count}<\texttt{max\_delete\_count}$}
                \State $\texttt{existsList} \gets \texttt{existsList} \cup \{\texttt{boolExists}\}$
            \ElsIf{$\texttt{delete\_count}=\texttt{max\_delete\_count}$
            \textbf{and}
            \Call{IsDeletable}{$I, \texttt{boolExists}, \texttt{EtoF}$}}
                 \State $\texttt{existsList} \gets \texttt{existsList} \cup \{\texttt{boolExists}\}$
            \EndIf
            \State \Return
        \EndIf
        \ForAll{$a \in \{0,1\}$} \Comment{delete: $0$, retain: $1$}
            \State $\texttt{exists2} \gets \texttt{exists}$
            \State $\texttt{exists2}[e] \gets a$
            \If{\Call{Place}{$\texttt{exists2}, e$}}
                \State \Call{Recurse}{$\texttt{exists2}, e+1$}
            \EndIf
        \EndFor
    \EndProcedure
    \State \Call{Recurse}{$\texttt{exists}, 0$}
    \State \Return $\texttt{existsList}$
\end{algorithmic}
\end{algorithm}

\begin{algorithm}[H]
    \caption{\isDeletableWithLink($I,\texttt{boolExists}, \texttt{EtoF}$)}
    \label{alg:is-deletable}
    \begin{algorithmic}[1]
        \Require{A multi-boundary island $I$ together with a set
$\Ependant$ of auxiliary edges, one added at each vertex of degree two, \texttt{boolExists} represents an edge set, and \texttt{EtoF} represents the set of faces in $F(I) \setminus F_R(I)$ incident with each edge $e \in E(I) \cup \Ependant$.}
        \Ensure{\True if the edge set represented by \texttt{boolExists} satisfies the second condition of deletable.}
        \For{each face $f \in F(I)\setminus F_R(I)$}
            \State $\texttt{deletion} \gets \emptyset$
            \For{each edge $e$ appearing on the boundary walk of $f$}
                \If{$\texttt{boolExists}[e] = \False$}
                    \State $\texttt{deletion}
                        \gets \texttt{deletion}\cup\{e\}$
                \EndIf
            \EndFor
            \If{$|\texttt{deletion}| \geq 3$}
                \State \Return \textbf{true}
            \EndIf
        \EndFor
        \For{each edge $e \in E(I) \cup \Ependant$}
            \State $\texttt{deletion} \gets \emptyset$
            \For{each face $f \in \texttt{EtoF}[e]$}
                \For{each edge $g$ appearing on the boundary walk of $f$}
                    \If{$\texttt{boolExists}[g] = \False$}
                        \State $\texttt{deletion}
                            \gets \texttt{deletion}\cup\{g\}$
                    \EndIf
                \EndFor
            \EndFor
            \If{$|\texttt{deletion}| = 4$}
                \State \Return \textbf{true}
            \EndIf
        \EndFor
        \State \Return \textbf{false}
    \end{algorithmic}
\end{algorithm}

%% file: pseudo/github_commands.tex
\newcommand{\githubbase}{https://github.com/three-edge-coloring-apex-cubic-graphs/computer-checks/blob/f15d3c7eeafc8814aef478bff59344bdad88bd0f}
\newcommand{\githublink}[3]{\href{\githubbase/#1\##2}{#3}} 

\newcommand{\enumDigonsWithLink}{\githublink{src/cartwheel.cpp}{L125-L143}{enumDigons}}
\newcommand{\distFromCenterWithLink}{\githublink{src/cartwheel.cpp}{L145-L167}{distFromCenter}}
\newcommand{\lowerBoundOfDigonChargeWithLink}{\githublink{src/cartwheel.cpp}{L169-L201}{lowerBoundOfDigonCharge}}
\newcommand{\enumWheelsWithLink}{\githublink{src/cartwheel.cpp}{L203-L227}{enumWheels}}
\newcommand{\enumDigonIncidentWheelsWithLink}{\githublink{src/cartwheel.cpp}{L229-L247}{enumDigonIncidentWheels}}
\newcommand{\generateCartwheelWithLink}{\githublink{src/cartwheel.cpp}{L249-L277}{generateCartwheel}}
\newcommand{\enumPossibleBadWheelsWithLink}{\githublink{src/cartwheel.cpp}{L279-L297}{enumPossibleBadWheels}}
\newcommand{\fixInRulesWithLink}{\githublink{src/cartwheel.cpp}{L300-L328}{fixInRules}}
\newcommand{\updateDegreeByRuleWithLink}{\githublink{src/cartwheel.cpp}{L330-L345}{updateDegreeByRule}}
\newcommand{\concreteDegreeExceptTailWithLink}{\githublink{src/cartwheel.cpp}{L347-L365}{concreteDegreeExceptTail}}
\newcommand{\pruneWithLink}{\githublink{src/cartwheel.cpp}{L367-L383}{prune}}
\newcommand{\pruneByNonAssociatedRuleWithLink}{\githublink{src/cartwheel.cpp}{L385-L400}{pruneByNonAssociatedRule}}
\newcommand{\upperBoundOfChargeWithLink}{\githublink{src/cartwheel.cpp}{L402-L420}{upperBoundOfCharge}}
\newcommand{\fixOutRulesWithLink}{\githublink{src/cartwheel.cpp}{L423-L467}{fixOutRules}}
\newcommand{\shouldRefineWithLink}{\githublink{src/cartwheel.cpp}{L469-L480}{shouldRefine}}
\newcommand{\refinementWithLink}{\githublink{src/cartwheel.cpp}{L482-L490}{refinement}}
\newcommand{\verifyNoBadCartwheelsWithLink}{\githublink{src/cartwheel.cpp}{L492-L506}{verifyNoBadCartwheels}}

\newcommand{\extendFromCutVerticesWithLink}{\githublink{src/configuration.cpp}{L170-L199}{extendFromCutVertices}}
\newcommand{\findCutTuplesWithLink}{\githublink{src/configuration.cpp}{L201-L228}{findCutTuples}}
\newcommand{\removeRingWithLink}{\githublink{src/configuration.cpp}{L230-L287}{removeRing}}

\newcommand{\allHomImagesWithLink}{\githublink{src/configuration_homomorphism.cpp}{L83-L120}{allHomImages}}
\newcommand{\makeOuterExtensionWithLink}{\githublink{src/configuration_homomorphism.cpp}{L164-L184}{makeOuterExtension}}
\newcommand{\findFourDartsWithLink}{\githublink{src/configuration_homomorphism.cpp}{L186-L205}{findFourDarts}}
\newcommand{\ensureOuterExtensionWithLink}{\githublink{src/configuration_homomorphism.cpp}{L208-L218}{ensureOuterExtension}}
\newcommand{\hasSeparatingCycleWithLink}{\githublink{src/configuration_homomorphism.cpp}{L220-L232}{hasSeparatingCycle}}
\newcommand{\enumCyclesWithLink}{\githublink{src/configuration_homomorphism.cpp}{L234-L277}{enumCycles}}
\newcommand{\labelDartsWithLink}{\githublink{src/configuration_homomorphism.cpp}{L279-L309}{labelDarts}}
\newcommand{\numSeparatedVerticesWithLink}{\githublink{src/configuration_homomorphism.cpp}{L311-L356}{numSeparatedVertices}}
\newcommand{\freeCompletionFromOuterExtensionWithLink}{\githublink{src/configuration_homomorphism.cpp}{L359-L383}{freeCompletionFromOuterExtension}}
\newcommand{\islandFromFreeCompletionWithLink}{\githublink{src/configuration_homomorphism.cpp}{L385-L396}{islandFromFreeCompletion}}
\newcommand{\indexBoundaryEdgesWithLink}{\githublink{src/configuration_homomorphism.cpp}{L398-L414}{indexBoundaryEdges}}
\newcommand{\indexPendantEdgesWithLink}{\githublink{src/configuration_homomorphism.cpp}{L416-L437}{indexPendantEdges}}
\newcommand{\indexOtherEdgesWithLink}{\githublink{src/configuration_homomorphism.cpp}{L439-L453}{indexOtherEdges}}
\newcommand{\constructIslandWithLink}{\githublink{src/configuration_homomorphism.cpp}{L455-L473}{constructIsland}}
\newcommand{\confirmDegreeOneWithLink}{\githublink{src/configuration_homomorphism.cpp}{L475-L483}{confirmDegreeOne}}
\newcommand{\confirmFreeCompletionWithLink}{\githublink{src/configuration_homomorphism.cpp}{L485-L505}{confirmFreeCompletion}}

\newcommand{\freeHomomorphismWithLink}{\githublink{src/free_homomorphism.hpp}{L62-L73}{freeHomomorphism}}
\newcommand{\dartIdentificationWithLink}{\githublink{src/free_homomorphism.hpp}{L77-L85}{dartIdentification}}
\newcommand{\freeHomomorphismOnlyDegreesWithLink}{\githublink{src/free_homomorphism.hpp}{L89-L154}{freeHomomorphismOnlyDegrees}}
\newcommand{\resolveDegreeIssuesWithLink}{\githublink{src/free_homomorphism.hpp}{L157-L186}{resolveDegreeIssues}}
\newcommand{\innerSubdegreeErrorWithLink}{\githublink{src/free_homomorphism.hpp}{L189-L198}{innerSubdegreeError}}
\newcommand{\vertexSingleDegreeIssueWithLink}{\githublink{src/free_homomorphism.hpp}{L201-L215}{vertexSingleDegreeIssue}}
\newcommand{\fixSingleDegreeIssueWithLink}{\githublink{src/free_homomorphism.hpp}{L218-L236}{fixSingleDegreeIssue}}
\newcommand{\singleOutLowerDegreeWithLink}{\githublink{src/free_homomorphism.hpp}{L240-L251}{singleOutLowerDegree}}
\newcommand{\enforceSingleDigonIncidenceWithLink}{\githublink{src/free_homomorphism.hpp}{L255-L276}{enforceSingleDigonIncidence}}
\newcommand{\freeHomomorphismAndEnforceSingleDigonIncidenceWithLink}{\githublink{src/free_homomorphism.hpp}{L280-L293}{freeHomomorphismAndEnforceSingleDigonIncidence}}

\newcommand{\boundaryCompletionsConfigurationWithLink}{\githublink{src/pseudo_configuration.cpp}{L61-L72}{boundaryCompletions}}
\newcommand{\addBoundaryDartsWithLink}{\githublink{src/pseudo_configuration.cpp}{L74-L87}{addBoundaryDarts}}
\newcommand{\identifyNeighborsWithLink}{\githublink{src/pseudo_configuration.cpp}{L89-L107}{identifyNeighbors}}
\newcommand{\alwaysApplyWithLink}{\githublink{src/pseudo_configuration.cpp}{L109-L111}{alwaysApply}}
\newcommand{\neverApplyWithLink}{\githublink{src/pseudo_configuration.cpp}{L113-L117}{neverApply}}
\newcommand{\amountOfChargeSendWithLink}{\githublink{src/pseudo_configuration.cpp}{L119-L127}{amountOfChargeSend}}
\newcommand{\amountOfPossibleChargeSendWithLink}{\githublink{src/pseudo_configuration.cpp}{L129-L139}{amountOfPossibleChargeSend}}

\newcommand{\fromVRotationsWithLink}{\githublink{src/pseudo_embedding.cpp}{L40-L89}{fromVRotations}}
\newcommand{\isBoundaryWithLink}{\githublink{src/pseudo_embedding.cpp}{L122-L130}{isBoundary}}
\newcommand{\firstDartWithLink}{\githublink{src/pseudo_embedding.cpp}{L132-L139}{firstDart}}
\newcommand{\lastDartWithLink}{\githublink{src/pseudo_embedding.cpp}{L141-L148}{lastDart}}
\newcommand{\anyDartWithLink}{\githublink{src/pseudo_embedding.cpp}{L150-L157}{anyDart}}
\newcommand{\sucKTimesWithLink}{\githublink{src/pseudo_embedding.cpp}{L159-L168}{sucKTimes}}
\newcommand{\getERotationsWithLink}{\githublink{src/pseudo_embedding.cpp}{L170-L185}{getERotations}}
\newcommand{\getDartsWithLink}{\githublink{src/pseudo_embedding.cpp}{L187-L195}{getDarts}}
\newcommand{\twoDigonsIncidentWithSameVertexWithLink}{\githublink{src/pseudo_embedding.cpp}{L197-L218}{twoDigonsIncidentWithSameVertex}}
\newcommand{\numDigonsWithLink}{\githublink{src/pseudo_embedding.cpp}{L220-L232}{numDigons}}

\newcommand{\removeIsolatedVertexWithLink}{\githublink{src/pseudo_embedding_with_degree.cpp}{L57-L70}{removeIsolatedVertex}}
\newcommand{\addBoundaryDartsDirectlyWithLink}{\githublink{src/pseudo_embedding_with_degree.cpp}{L72-L84}{addBoundaryDartsDirectly}}
\newcommand{\linkIncidenceListEndsWithLink}{\githublink{src/pseudo_embedding_with_degree.cpp}{L86-L104}{linkIncidenceListEnds}}
\newcommand{\getWalksWithLink}{\githublink{src/pseudo_embedding_with_degree.cpp}{L106-L130}{getWalks}}
\newcommand{\isPlanarWithLink}{\githublink{src/pseudo_embedding_with_degree.cpp}{L132-L139}{isPlanar}}
\newcommand{\boundaryCompletionsEmbeddingWithLink}{\githublink{src/pseudo_embedding_with_degree.cpp}{L142-L149}{boundaryCompletions}}
\newcommand{\containConfWithLink}{\githublink{src/pseudo_embedding_with_degree.cpp}{L151-L175}{containConf}}
\newcommand{\dartsByDegreeWithLink}{\githublink{src/pseudo_embedding_with_degree.cpp}{L177-L194}{dartsByDegree}}
\newcommand{\rootedContainConfWithLink}{\githublink{src/pseudo_embedding_with_degree.cpp}{L196-L198}{rootedContainConf}}
\newcommand{\blockedByReducibleConfigurationWithLink}{\githublink{src/pseudo_embedding_with_degree.cpp}{L200-L208}{blockedByReducibleConfiguration}}
\newcommand{\representativeDegreeWithLink}{\githublink{src/pseudo_embedding_with_degree.cpp}{L211-L250}{representativeDegree}}

\newcommand{\verifyNoBadCartwheelsForAllWithLink}{\githublink{verify_no_bad_cartwheels_for_all.sh}{L20-L28}{verifyNoBadCartwheelsForAll}}


\newcommand{\githubbaseReducibleCheck}{https://github.com/three-edge-coloring-apex-cubic-graphs/semi-reducibility-checker/blob/b38ddc531cc64a2bd58fc1904f8b378c3a5fa799}
\newcommand{\githublinkReducibleCheck}[3]{\href{\githubbaseReducibleCheck/#1\##2}{#3}} 

\newcommand{\getValidParensWithLink}{\githublinkReducibleCheck{kempe.hpp}{L20-L48}{GetValidParens}}
\newcommand{\getPlanarHalfKempesWithLink}{\githublinkReducibleCheck{kempe.hpp}{L63-L115}{GetPlanarHalfKempes}}

\newcommand{\isDeletableWithLink}{\githublinkReducibleCheck{multi_boundary_island.hpp}{L185-L203}{IsDeletable}}
\newcommand{\getDeletableEdgeSetInternalWithLink}{\githublinkReducibleCheck{multi_boundary_island.hpp}{L205-L277}{GetDeletableEdgeSetInternal}}
\newcommand{\getDeletableEdgeSetWithLink}{\githublinkReducibleCheck{multi_boundary_island.hpp}{L279-L285}{GetDeletableEdgeSet}}

%% file: pseudo/free_homomorphism.tex
\input{pseudo/github_commands}
\subsection{Free Homomorphism respecting identification requests}\label{sec:pseudo1}
For a pseudo-embedding or a pseudo-triangulation with digons $Z$, we denote the sets of vertices and darts of $Z$ by $V(Z)$ and $D(Z)$, respectively.

\begin{algorithm}[H]
    \caption{\boundaryCompletionsEmbeddingWithLink($(Z, \delta^-, \delta^+), v$)}
    \label{alg:boundary_completions_normal}
    \begin{algorithmic}[1]
        \Require{A pseudo-embedding with degree-range functions $(Z, \delta^-, \delta^+)$ and a boundary vertex $v$ with $\delta^-(v)=\delta^+(v)=d_{Z}(v)$.}
        \Ensure{A set of pseudo-embeddings with degree-range functions after the boundary completion steps.}
        \State $\efirst \gets$ the first dart whose head is $v$ (i.e., $\predecessor(\efirst)=\nil$)
        \State $\elast \gets$ the last dart whose head is $v$ (i.e., $\successor(\elast)=\nil$)
        \State $\predecessor(\efirst) \gets \elast$
        \State $\successor(\elast) \gets \efirst$
        \State \Return $\{((Z, \delta^-, \delta^+), \id{V(Z) \cup D(Z)})\}$
    \end{algorithmic}
\end{algorithm}

\begin{algorithm}[H]
    \caption{\boundaryCompletionsConfigurationWithLink($(Z, \delta^-, \delta^+), v$)}
    \label{alg:boundary_completions_digons}
    \begin{algorithmic}[1]
        \Require{A pseudo-triangulation with digons with degree-range functions and a boundary vertex $v$ with $\delta^-(v)=\delta^+(v)=d_Z(v)$.}
        \Ensure{A set of pseudo-triangulations with digons with degree-range functions after the boundary completion steps.}
        \State $\mathcal{Z}' \gets \emptyset$
        \State $Z_1 \gets$ \Call{addBoundaryDarts}{$(Z, \delta^-, \delta^+), v$} \Comment{Algorithm \ref{alg:add_boundary_darts} is the same as Algorithm A.4.8 in \cite{inoue2026four}.}
        \If {$Z_1 \neq \texttt{null}$}
            \State $\mathcal{Z}' \gets \mathcal{Z}' \cup \{(Z_1, \id{V(Z) \cup D(Z))}\}$
        \EndIf
        \State $Z_2 \gets$ \Call{identifyNeighbors}{$(Z, \delta^-, \delta^+), v$}
        \If {$Z_2 \neq \texttt{null}$}
            \State $\mathcal{Z}' \gets \mathcal{Z}' \cup \{Z_2\}$
        \EndIf
        \State \Return $\mathcal{Z}'$
    \end{algorithmic}
\end{algorithm}

\begin{algorithm}[H]
    \caption{\addBoundaryDartsWithLink($(Z, \delta^-, \delta^+)$, $v$)}
    \label{alg:add_boundary_darts}
    \begin{algorithmic}[1]
        \Require{A pseudo-triangulation with digons with degree-range functions $(Z, \delta^-, \delta^+)$, a boundary vertex $v$ with $\delta^-(v)=\delta^+(v)=d_Z(v)$.}
        \Ensure{A pseudo-triangulation with digons with degree-range functions obtained from $(V, \delta^-, \delta^+)$ by adding darts between two neighbors of $v$ and making $v$ an inner vertex or \texttt{null} if two neighbors coincide and the operation would create a loop.}
        \State $\efirst \gets$ the first dart whose head is $v$ (i.e., $\predecessor(\efirst)=\nil$)
        \State $\elast \gets$ the last dart whose head is $v$ (i.e., $\successor(\elast)=\nil$)
        \If {$\tail(\efirst)=\tail(\elast)$}
            \State \Return \texttt{null}
        \EndIf
        \State $(Z', \delta^{'-}, \delta^{'+}) \gets$ a copy of $(Z, \delta^{-}, \delta^{+})$
        \State In $Z'$, $\predecessor(\efirst) \gets \elast$
        \State In $Z'$, $\successor(\elast) \gets \efirst$
        \State \Call{addBoundaryDartsDirectly}{$(Z', \delta^{'-}, \delta^{'+}), \efirst, \elast$} \Comment{$Z'$ is updated inside this function.}
        \State \Return $((Z', \delta^{'-}, \delta^{'+})$
    \end{algorithmic}
\end{algorithm}

\begin{algorithm}[H]
    \caption{\addBoundaryDartsDirectlyWithLink($(Z, \delta^-, \delta^+)$, $\efirst$, $\elast$)}
    \label{alg:add_boundary_darts_directly}
    \begin{algorithmic}[1]
        \Require{A pseudo triangulation with digons $(Z, \delta^-, \delta^+)$ and the first dart $\efirst$ and the last dart $\elast$ of $v$.}
        \Ensure{Update $(Z, \delta^-, \delta^+)$ by adding darts between $\tail(\efirst)$ and $\tail(\elast)$.}
        \State $f \gets$ a dart whose \head is $\tail(\efirst)$, \reverse is $g$, \successor is \nil, and \predecessor is $\reverse(\efirst)$.
        \State $g \gets$ a dart whose \head is $\tail(\elast)$, \reverse is $f$, \successor is $\reverse(\elast)$, and \predecessor is \nil.
        \State $\successor(\reverse(\efirst)) \gets f$
        \State $\predecessor(\reverse(\elast)) \gets g$
        \State \Return
    \end{algorithmic}
\end{algorithm}

\begin{algorithm}[H]
    \caption{\identifyNeighborsWithLink($(Z, \delta^-, \delta^+)$, $v$)}
    \label{alg:identify_neighbors}
    \begin{algorithmic}[1]
        \Require{A pseudo-triangulation with digons with degree-range functions $(Z, \delta^-, \delta^+)$, a boundary vertex $v$ with $\delta^-(v)=\delta^+(v)=d_Z(v)$.}
        \Ensure{A pseudo-triangulation with digons with degree-range functions obtained from $(V, \delta^-, \delta^+)$ by identifying neighbors of $v$ and making $v$ an inner vertex with a mapping $\phi$ or \texttt{null} if such an identification is impossible by degree conditions or makes a loop.}
        \State $\efirst \gets$ the first dart whose head is $v$ (i.e., $\predecessor(\efirst)=\nil$)
        \State $\elast \gets$ the last dart whose head is $v$ (i.e., $\successor(\elast)=\nil$)
        \If {$[\delta^-(\tail(\efirst)), \delta^+(\tail(\efirst))] \cap [\delta^-(\tail(\elast)), \delta^+(\tail(\elast))] = \emptyset$}
            \State \Return \texttt{null}
        \EndIf
        \State $(Z', \delta^{'-}, \delta^{'+}) \gets$ a copy of $(Z, \delta^{-}, \delta^{+})$
        \State In $Z'$, $\predecessor(\efirst) \gets \elast$
        \State In $Z'$, $\successor(\elast) \gets \efirst$
        \State $\phi' \gets$ \Call{linkIncidenceListEnds}{$(Z', \delta^{'-}, \delta^{'+}), \reverse(\elast), \reverse(\efirst)$} \Comment{$Z'$ is updated inside this function.}
        \If {$Z'$ has a loop}
            \State \Return \texttt{null}
        \EndIf
        \State \Return $((Z', \delta^{'-}, \delta^{'+}), \phi')$
    \end{algorithmic}
\end{algorithm}

\begin{algorithm}[H]
    \caption{\linkIncidenceListEndsWithLink($(Z, \delta^-, \delta^+)$, $\eufirst$, $\ewlast$)}
    \label{alg:link_indidence_list_ends}
    \begin{algorithmic}[1]
        \Require {A pseudo-embedding with degree-range functions $(Z, \delta^-, \delta^+)$,
        $\eufirst$ is the first dart whose head is $u$,
        $\ewlast$ is the last dart whose head is $w$.}
        \Ensure {Update $(Z, \delta^{-}, \delta^{+})$ so that the incidence list of $u,w$ are linked by assigning $\predecessor(\eufirst)=\ewlast$, $\successor(\ewlast)=\eufirst$ ($u$ can be possibly the same as $w$.), and return a mapping from the original $Z$ to the updated $Z$.}
        \State $u \gets \head(\eufirst)$
        \State $w \gets \head(\ewlast)$
        \State $\predecessor(\eufirst) \gets \ewlast$
        \State $\successor(\ewlast) \gets \eufirst$
        \If {$u=w$}
            \State \Return $\id{V(Z) \cup D(Z)}$
        \EndIf
        \State $e\gets \eufirst$
        \While {$e \neq \nil$}
            \State $\head(e) \gets w$
            \State $e \gets \successor(e)$
        \EndWhile
        \State $[\delta^-(w), \delta^+(w)] \gets [\delta^-(u), \delta^+(u)] \cap [\delta^-(w), \delta^+(w)]$
        \State $\phi \gets \id{V(Z) \cup D(Z)}$
        \State $\phi(u) \gets w$
        \State remove $u$ from $Z$
        \State \Return $\phi$
    \end{algorithmic}
\end{algorithm}

\subsubsection{Blocked by configurations}
From the set $\mathcal{K}$ of configurations, we construct the set $\bar{\mathcal{K}}$.
This set contains pseudo-embeddings obtained from configurations in $\mathcal{K}$ by adding an outer edge that is incident to a cut-vertex, as described in Section \ref{subsect:free-homomorphism}, together with their mirror images.

The procedure for constructing such a pseudo-embedding from a configuration with a cut vertex is almost the same as Algorithms A.6.1--A.6.3 in \cite{inoue2026four}.
However, it differs from those algorithms in the way the configuration is extended at the cut-vertex.
Algorithm \ref{alg:extend_from_cut_vertices}, \ref{alg:find_cut_tuples}, \ref{alg:remove_ring} correspond to Algorithms A.6.1--A.6.3 in \cite{inoue2026four}.
By using them and Algorithm A.6.4, A.6.5 in \cite{inoue2026four}, we construct the set $\bar{\mathcal{K}}$.

In these algorithms, we use list data structures whose size is changing dynamically.
In adding a new element $x$ to the back of the list \texttt{L1}, we write $\texttt{L1.push\_back}(x)$.

In Algorithm \ref{alg:remove_ring}, we use Algorithm \ref{alg:from_v_rotations}.
This algorithm constructs dart representation from \texttt{rotations} and \texttt{digons}.
\texttt{rotations} is a list such that $\texttt{rotations}[i]$ represents the rotations of neighbors of $i$-th vertex.
\texttt{digons} is a list of endpoints of digons.
We use this function when constructing dart representations after reading files of configurations, rules, and cartwheels.

\begin{algorithm}[H]
    \caption{\extendFromCutVerticesWithLink($N$, $R$, \texttt{degrees}, \texttt{rotations}, \texttt{digons})}
    \label{alg:extend_from_cut_vertices}
    \begin{algorithmic}[1]
        \Require{The number of vertices $N$ in the free completion of a configuration and $R$ in its ring; a vertex set $V = \{v_0, \ldots, v_{N-1}\}$ containing the ring vertices $V_R = \{v_0, \ldots, v_{R-1}\}$, degree functions $\delta^-, \delta^+$ on $V \setminus V_R$, and a list \texttt{rotations} representing rotations of neighbors for each vertex in $V$. and a list \texttt{digons} representing the digons.}
        \Ensure{The set of configurations obtained by extending at each cut-vertex, with the special dart.}
        \State $P \gets$ \Call{findCutTuples}{$N, R, \texttt{rotations}$}
        \State $\mathcal{K} \gets \emptyset$
        \ForAll{$S \gets \{0, \ldots, 2^{|P|}-1\}$}
            \State $\texttt{adjacent\_cutvertex} \gets$ an array of size $R$, initialized by $-1$
            \ForAll{$i \gets \{0, \ldots, |P|-1\}$}
                \State $(v, a, b) \gets P[i]$
                \If{the $i$-th bit of $S$ is $1$}
                    \State $\texttt{adjacent\_cutvertex}[a] \gets v$
                \Else
                    \State $\texttt{adjacent\_cutvertex}[b] \gets v$
                \EndIf
            \EndFor
            \State $(Z, \delta^-, \delta^+) \gets$ \Call{removeRing}{$N, R, \texttt{degrees}, \texttt{rotations}, \texttt{adjacent\_cutvertex}, \texttt{digons}$}
            \State $f \gets$ \Call{maximumDegreeDart}{$(Z, \delta^-, \delta^+)$}
            \State $\mathcal{K} \gets \mathcal{K} \cup \{((Z, \delta^-, \delta^+), f)\}$
        \EndFor
        \State \Return $\mathcal{K}$
    \end{algorithmic}
\end{algorithm}

\begin{algorithm}[H]
    \caption{\findCutTuplesWithLink($N, R, \texttt{rotations}$)}
    \label{alg:find_cut_tuples}
    \begin{algorithmic}[1]
        \Require{The number of vertices $N$ in the free completion of a configuration and $R$ in its ring; a vertex set $V = \{v_0, \ldots, v_{N-1}\}$ containing the ring vertices $V_R = \{v_0, \ldots, v_{R-1}\}$, and a list \texttt{rotations} representing rotations of neighbors for each vertex in $V$.}
        \Ensure{The array of triples consisting of a cut-vertex and the two ring vertices adjacent to it, or raise an error if the input is an invalid configuration.}
        \State $\texttt{P} \gets$ an empty array
        \ForAll{$v_i \in \{v_R, \ldots, v_{N-1}\}$}
            \State $U_R \gets \emptyset$ \Comment{$U_R$ is the set of ring vertices adjacent to $v_i$.}
            \State $t \gets 0$ \Comment{$t$ is the number of transitions from a ring vertex to an internal vertex around $v_i$.}
            \State $d \gets |\texttt{rotations}[i]|$
            \ForAll{$j \gets \{0, \ldots, d-1\}$}
                \State $k_1 \gets \texttt{rotations}[i][j]$
                \If{$k_1 < R$}
                    \State $U_R \gets U_R \cup \{v_{k_1}\}$
                \EndIf
                \State $k_2 \gets \texttt{rotations}[i][(j+1) \bmod d]$
                \If{$k_1 < R$ and $k_2 \geq R$}
                    \State $t \gets t+1$ \Comment{A transition from the ring to an internal vertex.}
                \EndIf
            \EndFor
            \If{$t \geq 2$ and $|U_R| \neq 2$}
                \State \textbf{raise an error} \Comment{not normal configuration}
            \EndIf
            \If{$t=2$ and $|U_R|=2$} \Comment{$v_i$ is a cut-vertex.}
                \State $v_a, v_b \gets U_R$
                \State $\texttt{P.push\_back}((v_i,v_a,v_b))$
            \EndIf
        \EndFor
        \State \Return \texttt{P}
    \end{algorithmic}
\end{algorithm}

\begin{algorithm}[H]
    \caption{\removeRingWithLink$(N, R, \texttt{degrees}, \texttt{rotations}, \texttt{adjacent\_cutvertex}, \texttt{digons})$}
    \label{alg:remove_ring}
    \begin{algorithmic}[1]
        \Require{The number of vertices $N$ in the free completion of a configuration and $R$ in its ring; a vertex set $V = \{v_0, \ldots, v_{N-1}\}$ containing the ring vertices $V_R = \{v_0, \ldots, v_{R-1}\}$, degree functions $\delta^-, \delta^+$ on $V \setminus V_R$ , a list \texttt{rotations} representing rotations of neighbors for each vertex in $V$, a list \texttt{adjacent\_cutvertex}, and a list \texttt{digons}.}
        \Ensure{The pseudo-embedding with degree obtained by deleting each ring vertex $v_i$ with $\texttt{adjacent\_cutvertex}[i]=-1$ and retaining only the edge between each remaining ring vertex and its selected adjacent cut-vertex.}

        \Comment{Step 1: Assign new vertex IDs}
        \State $\texttt{old2new} \gets$ an array of length $N$, each element initialized by $-1$
        \State $\texttt{new\_id} \gets 0$
        \ForAll{$i \gets \{0, \ldots, N-1\}$}
            \If{$i<R$ and $\texttt{adjacent\_cutvertex}[i]=-1$}
                \Continue
            \EndIf
            \State $\texttt{old2new}[i] \gets \texttt{new\_id}$
            \State $\texttt{new\_id} \gets \texttt{new\_id}+1$
        \EndFor
        \State $N' \gets \texttt{new\_id}$

        \Comment{Step 2: Construct new rotations}
        \State $\texttt{new\_rotations} \gets$ a list of $N'$ lists
        \ForAll{$i \gets \{0, \ldots, N-1\}$}
            \If{$i<R$ and $\texttt{adjacent\_cutvertex}[i]=-1$}
                \Continue
            \EndIf
            \ForAll{$j \gets \texttt{rotations}[i]$}
                \If{$j=-1$ or ($i<R$ and $j\neq\texttt{adjacent\_cutvertex}[i]$) or ($j<R$ and $i\neq\texttt{adjacent\_cutvertex}[j]$)}
                    \State $\texttt{new\_rotations[old2new}[i]\texttt{].push\_back}(-1)$
                \Else
                    \State $\texttt{new\_rotations[old2new}[i]\texttt{].push\_back(old2new}[j]\texttt{)}$
                \EndIf
            \EndFor
        \EndFor

        \Comment{Step 3: Update degrees}
        \State $U' \gets \{u_0, \ldots, u_{N'-1}\}$
        \State $\delta^{\prime-},\delta^{\prime+}\colon U'\to\mathbb{N}\cup\{\infty\}$
        \ForAll{$i \gets \{0, \ldots, R-1\}$}
            \If{$\texttt{adjacent\_cutvertex}[i]=-1$}
                \Continue
            \EndIf
            \State $k \gets \texttt{old2new}[i]$
            \State $d \gets$ the number of elements that are not $-1$ in $\texttt{new\_rotations}[k]$
            \State $\delta^{\prime-}(u_k) \gets d+1$, $\delta^{\prime+}(u_k) \gets \infty$
        \EndFor
        \ForAll{$i \gets \{R, \ldots, N-1\}$}
            \State $k \gets \texttt{old2new}[i]$
            \State $\delta^{\prime-}(u_k) \gets \delta^-(v_i)$, $\delta^{\prime+}(u_k) \gets \delta^+(v_i)$
        \EndFor
        \algstore{removeRing}
    \end{algorithmic}
\end{algorithm}
\begin{algorithm}[H]
    \begin{algorithmic}
        \algrestore{removeRing}
        \State $\texttt{new\_digons} \gets$ an empty list  \Comment{Step 4: Construct new digons}
        \ForAll{$(a,b) \gets \texttt{digons}$}
            \If{$\texttt{old2new}[a]=-1$ or $\texttt{old2new}[b]=-1$}
                \Continue
            \EndIf
            \State $\texttt{new\_digons.push\_back}((\texttt{old2new}[a],\texttt{old2new}[b]))$
        \EndFor

        \State \Return The pseudo-embedding obtained by \Call{fromVRotations}{$N'$, $\texttt{new\_rotations}$, $\texttt{new\_digons}$} with degree functions $\delta^{\prime-},\delta^{\prime+}$.
    \end{algorithmic}
\end{algorithm}

\begin{algorithm}[H]
    \caption{\fromVRotationsWithLink($N, \texttt{rotations}, \texttt{digons}$)}
    \label{alg:from_v_rotations}
    \begin{algorithmic}[1]
        \Require{We denote the number of vertices by $N$ and a vertex set by $V=\{v_0,\ldots,v_{N-1}\}$. A list \texttt{rotations}, where $\texttt{rotations}[i]$ represents the rotation of indices of vertices around $v_i \in V$, where $-1$ marks the boundary. A list \texttt{digons} represents pairs of vertices that bound digons.}
        \Ensure{The pseudo-embedding represented by $V$, \texttt{rotations}, and \texttt{digons}, or raise an error if some vertex is adjacent to the same vertex twice except for digons, or the rotations of two vertices have a discrepancy.}

        \State $\texttt{is\_digon} \gets$ an $N \times N$ array, initialized by $\False$
        \ForAll{$(a,b) \in \texttt{digons}$}
            \State $\texttt{is\_digon}[a][b] \gets \True$
            \State $\texttt{is\_digon}[b][a] \gets \True$
        \EndFor

        \State $\texttt{darts} \gets$ an $N \times N$ array of darts, initialized by \nil
        \State $\texttt{digon\_darts} \gets$ an $N \times N$ array of darts, initialized by \nil
        \ForAll{$a \in \{0,\ldots,N-1\}$}
            \ForAll{$i \in \{0,\ldots,|\texttt{rotations}[a]|-1\}$}
                \State $b \gets \texttt{rotations}[a][i]$
                \If{$b=-1$}
                    \Continue
                \EndIf
                \If{$\texttt{darts}[a][b]\neq\nil$}
                    \State \textbf{raise an error} \Comment{A vertex that is adjacent to the same vertex twice except for digons.}
                \EndIf
                \State $\texttt{darts}[a][b] \gets$ a fresh dart
                \If{$\texttt{is\_digon}[a][b]=\True$}
                    \State $\texttt{digon\_darts}[a][b] \gets$ a fresh dart
                \EndIf
            \EndFor
        \EndFor
        \algstore{fromVRotations}
    \end{algorithmic}
\end{algorithm}
\begin{algorithm}[H]
    \begin{algorithmic}
        \algrestore{fromVRotations}
        \State $D \gets \emptyset$ \Comment{$D$ becomes the dart set.}
        \ForAll{$a \in \{0,\ldots,N-1\}$}
            \State $\texttt{size} \gets |\texttt{rotations}[a]|$
            \ForAll{$i \in \{0,\ldots,\texttt{size}-1\}$}
                \State $b \gets \texttt{rotations}[a][i]$
                \If{$b=-1$}
                    \Continue
                \EndIf
                \If{$\texttt{darts}[b][a]=\nil$}
                    \State \textbf{raise an error} \Comment{The rotations of two vertices have a discrepancy.}
                \EndIf
                \State $\texttt{head} \gets v_a$
                \State $s \gets \texttt{rotations}[a][i+1]$ if $i<\texttt{size}-1$ otherwise $\texttt{rotations}[a][0]$
                \State $\texttt{succ} \gets \texttt{darts}[a][s]$ if $s \neq -1$ otherwise $\nil$
                \State $p \gets \texttt{rotations}[a][i-1]$ if $i>0$ otherwise $\texttt{rotations}[a][\texttt{size}-1]$
                \State $\texttt{pred} \gets \texttt{digon\_darts}[a][p]$ if $p \neq -1$ and $\texttt{is\_digon}[a][p]=\True$, otherwise $\texttt{darts}[a][p]$ if $p \neq -1$, otherwise \nil
                \State $e \gets \texttt{darts}[a][b]$
                \If{$\texttt{is\_digon}[a][b]=\True$}
                    \State $e' \gets \texttt{digon\_darts}[a][b]$
                    \State Assign $\head(e) \gets \texttt{head}, \reverse(e) \gets \texttt{digon\_darts}[b][a], \successor(e) \gets e'$, and $ \predecessor(e) \gets \texttt{pred}$
                    \State Assign $\head(e') \gets \texttt{head}, \reverse(e') \gets \texttt{darts}[b][a], \successor(e') \gets \texttt{succ}$, and $ \predecessor(e') \gets e$
                    \State $D \gets D \cup \{e,e'\}$
                \Else
                    \State Assign $\head(e) \gets \texttt{head}$, $\reverse(e) \gets \texttt{darts}[b][a]$,  $\successor(e) \gets \texttt{succ}$, and $\predecessor(e) \gets \texttt{pred}$
                    \State $D \gets D \cup \{e\}$
                \EndIf
            \EndFor
        \EndFor

        \State \Return The pseudo-embedding that consists of $(V,D)$
    \end{algorithmic}
\end{algorithm}

\subsubsection{Enforce single digon incidence}

\begin{algorithm}[H]
    \caption{\enforceSingleDigonIncidenceWithLink($(Z, \delta^-, \delta^+)$)}
    \label{alg:enforce_single_digon_incidence}
    \begin{algorithmic}[1]
        \Require{The input $(Z, \delta^-, \delta^+)$ is a pseudo-embedding with degree or a pseudo-triangulation with digons with degree.}
        \Ensure{A set of free homomorphic images enforcing single digon incidence.}
        \State $\mathcal Z^* \gets \emptyset$
        \State $\texttt{Q}\gets$ an empty queue
        \State $\texttt{Q.push}(((Z, \delta^-, \delta^+), \id{V(Z) \cup D(Z)}))$
        \While{$\texttt{Q}$ is not empty}
            \State $((\tilde{Z}, \tilde{\delta}^-, \tilde{\delta}^+), \tilde{\phi}) \gets \texttt{Q.pop}()$
            \State $C \gets $ \Call{twoDigonsIncidentWithSameVertex}{$\tilde{Z}$}
            \If{$C \neq \texttt{null}$}
                \State $e_0, f_0 \gets C$
                \State $\tilde{\mathcal Z}^* \gets$ \Call{FreeHomomorphism}{$(\tilde{Z}, \tilde{\delta}^-, \tilde{\delta}^+),\{(e_0, f_0)\}$} \Comment{The algorithm to compute free homomorphic images depend on whether the input is a pseudo-embedding or a pseudo-triangulation with digons.}
                \ForAll{$((\tilde{Z}^*, \tilde{\delta}^{*-}, \tilde{\delta}^{*+}), \tilde{\phi}^*) \in \tilde{\mathcal Z}^*$}
                    \State $\texttt{Q.push}(((\tilde{Z}^*, \tilde{\delta}^{*-}, \tilde{\delta}^{*+}), \tilde{\phi}^* \circ \tilde{\phi}))$
                \EndFor
                \Continue
            \EndIf
            \State $\mathcal Z^* \gets \mathcal Z^* \cup \{(\tilde{Z}, \tilde{\delta}^-, \tilde{\delta}^+), \tilde{\phi})\}$
        \EndWhile
        \State \Return $\mathcal Z^*$
    \end{algorithmic}
\end{algorithm}

\begin{algorithm}[H]
    \caption{\twoDigonsIncidentWithSameVertexWithLink($Z$)}
    \label{alg:two_digons_incident_with_same_vertex}
    \begin{algorithmic}[1]
        \Require{A pseudo-embedding $Z$}
        \Ensure{A pair of darts contained in two digons incident with the same vertex, or \texttt{null}}
        \ForAll{$v \in V(Z)$}
            \State $\texttt{digon\_edge} \gets \nil$
            \ForAll{dart $e_0$ with $\head(e_0)=v$}
                \If {$\successor(e_0)=\nil$}
                    \Continue
                \EndIf
                \State $e_1 \gets \reverse(\successor(e_0))$
                \If {$\successor(e_1)=\nil$}
                    \Continue
                \EndIf
                \State $e_2 \gets \reverse(\successor(e_1))$
                \If {$e_0=e_2$}
                    \If {$\texttt{digon\_edge}=\nil$}
                        \State $\texttt{digon\_edge} \gets e_0$
                    \Else
                        \State \Return $(\texttt{digon\_edge}, e_0)$
                    \EndIf
                \EndIf
            \EndFor
        \EndFor
        \State \Return \texttt{null}
    \end{algorithmic}
\end{algorithm}

\begin{algorithm}[H]
    \caption{\freeHomomorphismAndEnforceSingleDigonIncidenceWithLink($(Z, \delta^-, \delta^+), S$)}
    \label{alg:free_hom_single_digon}
    \begin{algorithmic}[1]
        \Require{The input $(Z, \delta^-, \delta^+)$ is a pseudo-embedding with degree or a pseudo-triangulation with digons with degree and a set $S$ of identification requests.}
        \Ensure{A set of free homomorphic images respecting $S$ and enforcing single digon incidence.}
        \State $\mathcal Z^* \gets$ \Call{FreeHomomorphism}{$(Z, \delta^-, \delta^+), S$}
        \State $\tilde{\mathcal Z} \gets \emptyset$
        \ForAll{$((Z^*, \delta^{*-}, \delta^{*+}), \phi^*)\in \mathcal Z^*$}
            \State $\mathcal Z' \gets$ \Call{enforceSingleDigonIncidence}{$Z^*, \delta^{*-}, \delta^{*+}$}
            \ForAll{$((Z', \delta^{'-}, \delta^{'+}), \phi')\in \mathcal Z'$}
                \State $\tilde{\mathcal Z} \gets \tilde{\mathcal Z} \cup\{((Z', \delta^{'-}, \delta^{'+}), \phi' \circ \phi^*)\}$
            \EndFor
        \EndFor
        \State \Return $\tilde{\mathcal Z}$
    \end{algorithmic}
\end{algorithm}

\begin{lem}
\label{comp:lem:combined_rules}
  Every possible combination $R^*$ of the rules in $\mathcal{R}$ that are not blocked by $\mathcal{K}$ satisfies $r(R^*) \leq 5$.
\end{lem}
     For the proof, we took the algorithm obtained from Algorithm A.8.2 in \cite{inoue2026four} by replacing the routine to compute a set of free homomorphic images of pseudo-triangulation with that for pseudo-triangulation with digons and enforcing single digon incidence (which is Algorithm \ref{alg:free_hom_single_digon}).
    We used our rule set $\mathcal{R}$, and our configuration set $\bar{\mathcal{K}}$ as inputs.
    In the obtained set $\cRstarnoK$, each combined rule $R^*$ turned out to have $r(R^*) \leq 5$.

%% file: pseudo/cartwheel_algorithm.tex
\input{pseudo/github_commands}
\subsection{The final charges of vertices with degrees from 7 to 11}

\subsubsection{Wheels}

$\CARTWHEELDEGREES = [(5, 5), (6, 6), (7, 7), (8, 8), (9, \infty)]$ is the array representing all possible degree-ranges as neighbors of the center vertex in wheels.
\begin{algorithm}
    \caption{\enumDigonIncidentWheelsWithLink($d$)}
    \label{alg:enum_digon_incident_wheels}
    \begin{algorithmic}[1]
        \Require{A center degree $d$.}
        \Ensure{The set of wheels of center degree $d$ with a digon.}
        \State $\mathcal W\gets \emptyset$
        \State $\texttt{degrees} \gets$ an array of length $d-1$.
        \Procedure{enumDegree}{$\texttt{degrees},i$}
            \If{$i=d-1$}
                \State $W\gets \Call{generateCartwheel}{d,\texttt{degrees},\True}$
                \State $\mathcal W\gets \mathcal W\cup\{W\}$
                \State \Return
            \EndIf
            \For{$j \in \{0, \ldots, 4\}$}
                \State Let $(a,b)$ be $\CARTWHEELDEGREES[j]$
                \State $\texttt{degrees}[i]\gets [a,b]$
                \State \Call{enumDegree}{$\texttt{degrees},i+1$}
            \EndFor
            \State \Return
        \EndProcedure
        \State \Call{enumDegree}{$\texttt{degrees},0$}
        \State \Return $\mathcal W$
    \end{algorithmic}
\end{algorithm}

\begin{algorithm}[H]
    \caption{\enumWheelsWithLink($d$)}
    \label{alg:enum_wheels}
    \begin{algorithmic}[1]
        \Require{A center degree $d$.}
        \Ensure{The set of wheels of center degree $d$.}
        \State $\mathcal{W} \gets \emptyset$
        \State $\texttt{degrees}\gets$ an array of length $d$
        \Procedure{enumDegree}{$\texttt{degrees},i,i_\text{lowest}$}
            \If{$i=d$}
                \If{some rotation of \texttt{degrees} is lexicographically strictly smaller than \texttt{degrees}}
                    \State \Return
                \EndIf
                \State $W\gets \Call{generateCartwheel}{d,\texttt{degrees},\False}$
                \State $\mathcal W\gets \mathcal W\cup\{W\}$
                \State \Return
            \EndIf
            \ForAll{$j\in\{i_\text{lowest},\ldots,4\}$}
                \State Let $(a,b)$ be $\CARTWHEELDEGREES[j]$
                \State $\texttt{degrees}[i]\gets [a,b]$
                \State \Call{enumDegree}{$\texttt{degrees},i+1,i_\text{lowest}$}
            \EndFor
            \State \Return
        \EndProcedure
        \ForAll{$j\in\{0,\ldots,4\}$}
            \State Let $(a,b)$ be $\CARTWHEELDEGREES[j]$
            \State $\texttt{degrees}[0]\gets [a,b]$
            \State \Call{enumDegree}{$\texttt{degrees},1,j$}
        \EndFor
        \State \Return $\mathcal W$
    \end{algorithmic}
\end{algorithm}

\begin{algorithm}[H]
    \caption{\generateCartwheelWithLink($d,\texttt{degrees},\texttt{incident\_digon}$)}
    \label{alg:generate_cartwheel}
    \begin{algorithmic}[1]
        \Require{A center degree $d$, an array \texttt{degrees} of degree ranges, and a boolean \texttt{incident\_digon}.}
        \Ensure{A cartwheel of center degree $d$.}
        \If{$\texttt{incident\_digon}$}
            \State $n \gets d-1$
        \Else
            \State $n \gets d$
        \EndIf

        \State $\texttt{rotations} \gets$ an array of length $n+1$.
        \State $\texttt{rotations}[0] \gets [1, \ldots, n]$
        
        \ForAll{$i \in \{1, \ldots, n\}$}
            \State $i' \gets i+1 \text{ if } i < n \text{ otherwise } 1$
            \State $i'' \gets i-1 \text{ if } i > 1 \text{ otherwise } n$
            \State $\texttt{rotations}[i] \gets [i', 0, i'', -1]$
        \EndFor

        \State $\mathcal D\gets \emptyset$
        \If{$\texttt{incident\_digon}$}
            \State $\mathcal D\gets \mathcal D\cup\{\{0,1\}\}$
        \EndIf

        \State $Z \gets \Call{fromVRotations}{n + 1, \texttt{rotations}, \mathcal{D}}$
        \State $\delta^-, \delta^+ \colon V(Z) \to \mathbb{N}$ such that $\delta^-(v_0)=\delta^+(v_0)=d$, $(\delta^-(v_i), \delta^+(v_i))=\texttt{degrees}[i-1]$ if $0 < i \leq n$.
        \State \Return the cartwheel $(Z, \delta^-, \delta^+)$ with center vertex $v_0$.
    \end{algorithmic}
\end{algorithm}

\subsubsection{Charges along an edge and a bound on the final charge}

\begin{algorithm}[H]
    \caption{\neverApplyWithLink($(Z^*, \delta^{*-}, \delta^{*+}), e^*, R$)}
    \label{alg:never_apply}
    \begin{algorithmic}[1]
        \Require{A pseudo triangulation with digons with degree-range functions $(Z^*, \delta^{*-}, \delta^{*+})$ with a dart $e^*$, and a rule $R$.}
        \Ensure{\True if $R$ never applies to the dart $e^*$ in $(Z^*, \delta^{*-}, \delta^{*+})$, and \False otherwise}
        \State $((Z_R, \delta_R^-, \delta_R^+), e_R) \gets R$
        \State $\mathcal{Z} \gets$ \Call{freeHomomorphismAndEnforceSingleDigonIncidence}{$(Z^*, \delta^{*-}, \delta^{*+}) \sqcup (Z_R, \delta_R^-, \delta_R^+), \{(e^*, e_R)\}$}
        \If {$\mathcal Z=\emptyset$}
            \State \Return \True
        \Else
            \State \Return \False
        \EndIf
    \end{algorithmic}
\end{algorithm}

\subsubsection{A lower bound of charges given to digons}

\begin{algorithm}[H]
    \caption{\lowerBoundOfDigonChargeWithLink($(Z_C,\delta_C^-,\delta_C^+)$)}
    \label{alg:lower_bound_of_digon_charge}
    \begin{algorithmic}[1]
        \Require{A cartwheel $(Z_C,\delta_C^-,\delta_C^+)$ with center vertex $v$}
        \Ensure{A lower bound on the charge given to digons}
        \State $\texttt{dist\_pairs}\gets ((0,1),(1,0),(1,1),(1,2),(2,1))$
        \State $\texttt{degree\_pairs}\gets
                \begin{aligned}[t]
                \bigl(&
                    (([7,10],[5,\infty])),\\
                    &(([5,\infty],[7,10])),\\
                    &(([5,5],[5,5]),([5,5],[6,6]),([6,6],[5,5]),([5,\infty],[5,\infty])),\\
                    &(([5,5],[5,5]),([5,5],[6,7]),([6,7],[5,5])),\\
                    &(([5,5],[5,5]),([5,5],[6,7]),([6,7],[5,5]))
                \bigr)
                \end{aligned}$
        \State $\texttt{charges}\gets
            \bigl(
                (4),
                (4),
                (5,3,3,2),
                (4,2,2),
                (4,2,2)
            \bigr)$
        \State $\texttt{charge}\gets 0$
        \State $\mathcal{D} \gets \Call{enumDigons}{Z_C}$
        \ForAll{digons $(u_1,u_2)\in\mathcal D$}
            \For{$i\gets 1$ to $5$}
                \If{$(\textsf{dist}(v,u_1),\textsf{dist}(v,u_2))=\texttt{dist\_pairs}[i]$}
                    \For{$j\gets 1$ to $|\texttt{degree\_pairs}[i]|$}
                        \State Let $(D_1, D_2)$ be $\texttt{degree\_pairs}[i][j]$
                        \If{$[\delta^-_C(u_1), \delta_C^+(u_1)] \subseteq D_1$ and $[\delta^-_C(u_2), \delta_C^+(u_2)] \subseteq D_2$}
                            \State $\texttt{charge}\gets \texttt{charge}+\texttt{charges}[i][j]$
                            \Break
                        \EndIf
                    \EndFor
                    \Break
                \EndIf
            \EndFor
        \EndFor
        \State \Return $\texttt{charge}$
    \end{algorithmic}
\end{algorithm}

\begin{algorithm}[H]
    \caption{\enumDigonsWithLink($Z_C$)}
    \label{alg:enum_digons}
    \begin{algorithmic}[1]
        \Require{A cartwheel $Z_C$}
        \Ensure{The set of digons of $Z_C$, represented by their endpoint sets}
        \State $\mathcal D\gets \emptyset$
        \ForAll{darts $e_0$ of $Z_C$}
            \If{$\successor(e_0)=\nil$}
                \Continue
            \EndIf
            \State $e_1\gets \reverse(\successor(e_0))$
            \If{$\successor(e_1)=\nil$}
                \Continue
            \EndIf
            \State $e_2\gets \reverse(\successor(e_1))$
            \If{$e_2=e_0$}
                \State $u_1\gets \head(e_0)$
                \State $u_2\gets \head(\reverse(e_0))$
                \State $\mathcal D\gets \mathcal D\cup\{\{u_1,u_2\}\}$
            \EndIf
        \EndFor
        \State \Return $\mathcal D$
    \end{algorithmic}
\end{algorithm}

\subsubsection{Fixing rules applied from neighbors to the center}

\begin{algorithm}[H]
    \caption{\fixInRulesWithLink($(Z_{W}, \delta_{W}^-, \delta_{W}^+), \mathcal{R}, \mathcal{R}^{*-\mathcal{K}}, \bar{\mathcal{K}}$)}
    \label{alg:fix_in_rules}
    \begin{algorithmic}[1]
        \Require{A cartwheel $(Z_{W}, \delta_{W}^-, \delta_{W}^+)$ with the center vertex $v$ and with darts $e_1, \ldots, e_d$ whose head is $v$, a set of rules $\mathcal{R}$, a set of combined rules $\mathcal{R}^{*-\mathcal{K}}$, a set of configurations $\bar{\mathcal{K}}$.}
        \Ensure{The set of cartwheels obtained from $(Z_{W}, \delta_{W}^-, \delta_{W}^+)$ by fixing a set of rules $\mathcal{R}_i$ applied to each dart $e_i$.}
        \State $\mathcal{C}_0 \gets \{((Z_{W}, \delta_{W}^-, \delta_{W}^+), \emptyset)\}$
        \ForAll{$i \in \{1, \ldots, d\}$}
            \State $\mathcal{C}_i \gets \emptyset$
            \ForAll{$((Z_C, \delta_C^-, \delta_C^+), \{R^*_j\}_{j=1}^{i-1}) \in \mathcal{C}_{i-1}$}
                \ForAll{$R_i^* \in \mathcal{R}^{*-\mathcal{K}}$}
                    \State $\mathcal{C}' \gets$ \Call{updateDegreeByRule}{$(Z_C, \delta_C^-, \delta_C^+), e_i, R_i^*$}
                    \ForAll{$(Z_{C'}, \delta_{C'}^-, \delta_{C'}^+) \in \mathcal{C}'$}
                        \If {\Call{prune}{$(Z_{C'}, \delta_{C'}^-, \delta_{C'}^+), \{R_j^*\}_{j=1}^i, \mathcal{R}, \mathcal{R}^{*-\mathcal{K}}, \bar{\mathcal{K}}$}}
                            \Continue
                        \EndIf
                        \State $\mathcal{C}_i \gets \mathcal{C}_i \cup \{((Z_{C'}, \delta_{C'}^-, \delta_{C'}^+), \{R_j^*\}_{j=1}^i)\}$
                    \EndFor
                \EndFor
            \EndFor
        \EndFor
        \State \Return $\mathcal{C}_d$
    \end{algorithmic}
\end{algorithm}

\begin{algorithm}[H]
    \caption{\updateDegreeByRuleWithLink($(Z_C,\delta_C^-,\delta_C^+), e_C, R$)}
    \label{alg:update_degree_by_rule}
    \begin{algorithmic}[1]
        \Require{A cartwheel $(Z_C,\delta_C^-,\delta_C^+)$ with center vertex $v$ and darts $e_1,\ldots,e_d$ whose head is $v$, a dart $e_C$, and a rule $R$}
        \Ensure{A set of cartwheels where degree-ranges are only tail-ranges by taking free homomorphic images of $(Z_C, \delta_C^-, \delta_C^+) \sqcup (Z_R, \delta_R^-, \delta_R^+)$ respecting $\{(e_C, e_R)\}$.}
        \State $((Z_R,\delta_R^-,\delta_R^+),e_R) \gets R$
        \State $\mathcal Z^*
            \gets
            \Call{freeHomomorphismAndEnforceSingleDigonIncidence}
            {(Z_C,\delta_C^-,\delta_C^+) \sqcup (Z_R,\delta_R^-,\delta_R^+), \{(e_C,e_R)\}}$
        \State $\mathcal{C}^* \gets \emptyset$
        \ForAll{$((Z^*, \delta^{*-}, \delta^{*+}), \phi^*) \in \mathcal Z^*$}
            \State $C^* \gets$ the cartwheel generated from $(Z^*,\delta^{*-},\delta^{*+})$ with center vertex $\phi^*(v)$ and center darts $\phi^*(e_i)$.
            \State $\mathcal{C}' \gets \Call{concreteDegreeExceptTail}{C^*}$
            \State $\mathcal{C}^* \gets \mathcal{C}^* \cup \mathcal{C}'$
        \EndFor
        \State \Return $\mathcal{C}^*$
    \end{algorithmic}
\end{algorithm}

\begin{algorithm}[H]
    \caption{\concreteDegreeExceptTailWithLink($(Z_C,\delta_C^-,\delta_C^+)$)}
    \label{alg:concrete_degree_except_tail}
    \begin{algorithmic}[1]
        \Require{A cartwheel $(Z_C,\delta_C^-,\delta_C^+)$}
        \Ensure{The set of cartwheels obtained by fixing all non-tail degree ranges}
        \State $\mathcal C\gets \{(Z_C,\delta_C^-,\delta_C^+)\}$
        \ForAll{vertices $u$ of $Z_C$}
            \If{$\delta_C^-(u)=\delta_C^+(u)$ or $\delta_C^+(u)=\infty$}
                \Continue
            \EndIf
            \State $\mathcal C'\gets \emptyset$
            \For{$d\gets \delta_C^-(u)$ to $\delta_C^+(u)$}
                \ForAll{$(Z,\delta^-,\delta^+)\in\mathcal C$}
                    \State Let $(Z',\delta^{\prime -},\delta^{\prime +})$ be a copy of $(Z,\delta^-,\delta^+)$
                    \State $\delta^{\prime -}(u)\gets d$
                    \State $\delta^{\prime +}(u)\gets d$
                    \State $\mathcal C'\gets \mathcal C'\cup\{(Z',\delta^{\prime -},\delta^{\prime +})\}$
                \EndFor
            \EndFor
            \State $\mathcal C\gets \mathcal C'$
        \EndFor
        \State \Return $\mathcal C$
    \end{algorithmic}
\end{algorithm}

\subsubsection{Pruning}

\begin{algorithm}[H]
    \caption{\pruneWithLink($(Z_C,\delta_C^-,\delta_C^+),\{R^*_j\}_{j=1}^{i},\mathcal R,\mathcal R^{*-\mathcal K},\bar{\mathcal K}$)}
    \label{alg:prune}
    \begin{algorithmic}[1]
        \Require{A cartwheel $(Z_C,\delta_C^-,\delta_C^+)$ with center vertex $v$, combined rules with spokes $\{R^*_j\}_{j=1}^{i}$, rules $\mathcal R$, combined rules $\mathcal R^{*-\mathcal K}$, and reducible configurations $\bar{\mathcal K}$}
        \Ensure{\True if the cartwheel is pruned, and \False otherwise}
        \If{\Call{pruneByNonAssociatedRule}{$(Z_C,\delta_C^-,\delta_C^+),\{R^*_j\}_{j=1}^{i},\mathcal R$}} \Comment{Algorithm A.9.12 in \cite{inoue2026four}.}
            \State \Return \True 
        \EndIf
        \If{\Call{upperBoundOfCharge}{$(Z_C,\delta_C^-,\delta_C^+),\{R^*_j\}_{j=1}^{i},\mathcal R,\mathcal R^{*-\mathcal K}$} $\leq$ \Call{lowerBoundOfDigonCharge}{$(Z_C,\delta_C^-,\delta_C^+)$}} \Comment{upperBoundOfCharge routine is Algorithm A.9.13 in \cite{inoue2026four}.}
            \State \Return \True
        \EndIf
        \If{\Call{blockedByReducibleConfiguration}{$(Z_C,\delta_C^-,\delta_C^+),v,\bar{\mathcal K}$}}
            \State \Return \True \Comment{Algorithm A.7.1 in \cite{inoue2026four}.}
        \EndIf
        \State \Return \False
    \end{algorithmic}
\end{algorithm}

\subsubsection{Refinement}

\begin{algorithm}[H]
    \caption{\fixOutRulesWithLink($\mathcal{C}_d, \mathcal{R}, \mathcal{R}^{*-\mathcal{K}}, \bar{\mathcal{K}}, \Rauxiliary$)}
    \label{alg:fix_out_rules}
    \begin{algorithmic}[1]
        \Require{A set of cartwheels $\mathcal{C}_d$, the set of rules $\mathcal{R}$, the set of combined rules $\mathcal{R}^{*-\mathcal{K}}$, the set of configurations $\bar{\mathcal{K}}$, and an auxiliary set of rules with its homomorphic cover $\Rauxiliary$.}
        \Ensure{The set of cartwheels $\mathcal{C}$, obtained by refining the cartwheels in $\mathcal{C}_d$ based on whether a rule in $\Rauxiliary$ applies from the center.}
        \State \texttt{Q} $\gets$ an empty queue
        \ForAll{$((Z_C, \delta_C^-, \delta_C^+), \{R^*_j\}_{j=1}^d) \in \mathcal{C}_d$}
            \State $\texttt{Q.push}(((Z_C, \delta_C^-, \delta_C^+), \{R^*_j\}_{j=1}^d))$
        \EndFor
        \State $\mathcal{C} \gets \emptyset$
        \While{not \texttt{Q.empty()}}
            \State $((Z_C, \delta_C^-, \delta_C^+), \{R^*_j\}_{j=1}^d) \gets \texttt{Q.pop}()$
            \State \texttt{refined\_flag} $\gets$ \False
            \ForAll{$i \in [1,d]$}
                \ForAll{$(R, \{R_j\}_{j=1}^k) \in \Rauxiliary$}
                    \If {\Call{shouldRefine}{$(Z_C, \delta_C^-, \delta_C^+), i, (R, \{R_j\}_{j=1}^k)$} $=$ \False}
                        \Continue
                    \EndIf
                    \State \texttt{refined\_flag} $\gets$ \True
                    \State $\mathcal{Z}_{C'} \gets$ \Call{refinement}{$(Z_C, \delta_C^-, \delta_C^+), i, (R, \{R_j\}_{j=1}^k)$}
                    \ForAll {$(Z_{C'}, \delta_{C'}^-, \delta_{C'}^+) \in \mathcal{Z}_{C'}$}
                         \If {\Call{prune}{$(Z_{C'}, \delta_{C'}^-, \delta_{C'}^+), \{R_j^*\}_{j=1}^d, \mathcal{R}, \mathcal{R}^{*-\mathcal{K}}, \bar{\mathcal{K}}$}}
                             \Continue
                         \EndIf
                         \State \texttt{Q.push}(($(Z_{C'}, \delta_{C'}^-, \delta_{C'}^+), \{R^*_j\}_{j=1}^d$))
                    \EndFor
                    \Break
                \EndFor
                \If {\texttt{refined\_flag} = \True}
                    \Break
                \EndIf
            \EndFor
            \If {\texttt{refined\_flag} = \False}
                \State $\mathcal{C} \gets \mathcal{C} \cup \{((Z_C, \delta_C^-, \delta_C^+), \{R^*_j\}_{j=1}^d)\}$
            \EndIf
        \EndWhile
        \State \Return $\mathcal{C}$
    \end{algorithmic}
\end{algorithm}

\begin{algorithm}[H]
    \caption{\shouldRefineWithLink($(Z_C,\delta_C^-,\delta_C^+),i,(R,\{R_j\}_{j=1}^k)$)}
    \label{alg:should_refine}
    \begin{algorithmic}[1]
        \Require{A cartwheel $(Z_C,\delta_C^-,\delta_C^+)$, an index $i$ ($1 \leq i \leq d$), a rule $R$ with its homomorphic cover $\{R_j\}_{j=1}^k$}
        \Ensure{\True if we should refine by $(R,\{R_j\}_{j=1}^k)$, otherwise \False.}
        \If {\Call{alwaysApply}{$(Z_C,\delta_C^-,\delta_C^+),\reverse(e_i),R$} $=\False$} \Comment{alwaysApply is in Algorithm A.9.1 in \cite{inoue2026four}.}
            \State \Return \False
        \EndIf
        \ForAll{$j \in \{1, 2, \ldots, k\}$}
            \If{\Call{alwaysApply}{$(Z_C, \delta_C^-, \delta_C^+),\reverse(e_i),R_j$} $=\True$}
                \State \Return \False
            \EndIf
        \EndFor
        \State \Return \True
    \end{algorithmic}
\end{algorithm}

\begin{algorithm}[H]
    \caption{\refinementWithLink($(Z_C,\delta_C^-,\delta_C^+),i,(R,\{R_j\}_{j=1}^k)$)}
    \label{alg:refinement}
    \begin{algorithmic}[1]
        \Require{A cartwheel $(Z_C,\delta_C^-,\delta_C^+)$, an index $i$ ($1 \leq i \leq d$), a rule $R$ with its homomorphic cover $\{R_j\}_{j=1}^k$}
        \Ensure{The set of cartwheels updated by each rule $R_j$.}
        \State $\mathcal C\gets \emptyset$
        \ForAll{$j \in \{1,2, \ldots, k\}$}
            \State $\mathcal C_j\gets \Call{updateDegreeByRule}{(Z_C,\delta_C^-,\delta_C^+), \reverse(e_i) ,R_j}$
            \State $\mathcal C\gets \mathcal C\cup \mathcal C_j$
        \EndFor
        \State \Return $\mathcal C$
    \end{algorithmic}
\end{algorithm}

\subsubsection{Overall algorithm}

\begin{algorithm}[H]
    \caption{\enumPossibleBadWheelsWithLink($d,\mathcal R,\mathcal R^{*-\mathcal K},\bar{\mathcal K}$)}
    \label{alg:enum_possible_bad_wheels}
    \begin{algorithmic}[1]
        \Require{A center degree $d$, rules $\mathcal R$, combined rules $\mathcal R^{*-\mathcal K}$, and configurations $\bar{\mathcal K}$.}
        \Ensure{The set of possibly bad wheels of center degree $d$.}
        \State $\mathcal C_0^d \gets \emptyset$
        \State $\mathcal W \gets \Call{enumWheels}{d}$
        \State $\mathcal W_{\mathrm{digon}} \gets \Call{enumDigonIncidentWheels}{d}$
        \ForAll{$(Z_W,\delta_W^-,\delta_W^+) \in \mathcal{W} \cup \mathcal{W}_{\mathrm{digon}}$}
            \If{\Call{prune}{$(Z_W,\delta_W^-,\delta_W^+),\emptyset,\mathcal R,\mathcal R^{*-\mathcal K},\bar{\mathcal K}$}}
                \Continue
            \EndIf
            \State $\mathcal C_0^d\gets \mathcal C_0^d\cup\{(Z_W,\delta_W^-,\delta_W^+)\}$
        \EndFor
        \State \Return $\mathcal C_0^d$
    \end{algorithmic}
\end{algorithm}

\begin{algorithm}
    \caption{\verifyNoBadCartwheelsWithLink($(Z_C,\delta_C^-,\delta_C^+),\mathcal R,\mathcal R^{*-\mathcal K},\bar{\mathcal K},\Rauxiliary$)}
    \label{alg:enum_bad_cartwheels}
    \begin{algorithmic}[1]
        \Require{A cartwheel $(Z_C,\delta_C^-,\delta_C^+)$, rules $\mathcal R$, combined rules $\mathcal R^{*-\mathcal K}$, reducible configurations $\bar{\mathcal K}$, and the auxiliary set of rules and their homomorphic cover $\Rauxiliary$.}
        \State $\mathcal C_d \gets \Call{fixInRules}{(Z_C,\delta_C^-,\delta_C^+),\mathcal R,\mathcal R^{*-\mathcal K},\bar{\mathcal K}}$
        \State $\mathcal C \gets \Call{fixOutRules}{\mathcal C_d,\mathcal R,\mathcal R^{*-\mathcal K},\bar{\mathcal K},\Rauxiliary}$
        \State \textbf{assert}($\mathcal{C} = \emptyset$)
    \end{algorithmic}
\end{algorithm}

\begin{algorithm}[H]
    \caption{verifyNoBadCartwheelsForAll($\mathcal{R}, \mathcal{R}^{*-\mathcal{K}}, \bar{\mathcal{K}}, \Rauxiliary$)}
    \label{alg:enum_all_bad_cartwheel}
    \begin{algorithmic}[1]
        \Require{The set of rules $\mathcal{R}$, the set of combined rules $\mathcal{R}^{*-\mathcal{K}}$, the set of configurations $\bar{\mathcal{K}}$, and the auxiliary set of rules and their homomorphic cover $\Rauxiliary$.}
        \ForAll{$d \in \{7,8,9,10,11\}$ \textbf{in parallel}}
            \State $\mathcal{C}_0 \gets$ \Call{enumPossibleBadWheels}{$d, \mathcal{R}, \mathcal{R}^{*-\mathcal{K}}, \bar{\mathcal{K}}$}
            \ForAll{$(Z_{C_0}, \delta_{C_0}^-, \delta_{C_0}^+) \in \mathcal{C}_0$ \textbf{in parallel}} 
                \State \Call{verifyNoBadCartwheels}{$(Z_{C_0}, \delta_{C_0}^-, \delta_{C_0}^+), \mathcal{R}, \mathcal{R}^{*-\mathcal{K}}, \bar{\mathcal{K}}, \Rauxiliary$}
            \EndFor
        \EndFor
    \end{algorithmic}
\end{algorithm}
We do not have source code that implements Algorithm \ref{alg:enum_all_bad_cartwheel} literally.
Instead, we first run \textsf{enumPossibleBadWheels}, and then, for each cartwheel produced, we run \textsf{verifyNoBadCartwheels} using a shell script.
This shell script implements lines 3--5 of Algorithm \ref{alg:enum_all_bad_cartwheel}.


\begin{lem}
\label{comp:lem:deg7-11}
   In every call to Algorithm \ref{alg:enum_bad_cartwheels} made during its execution, the assertion does not fail.
   In other words, $\mathcal{C}_\textsf{all}$ in Lemma \ref{lem:cartwheel-algorithm-is-correct} is empty.
\end{lem}

%% file: pseudo/conf_hom.tex
\input{pseudo/github_commands}
\subsection{Homomorphic images of configurations}

Note that, in line 2 of Algorithm \ref{alg:all_hom_images}, we use the \textsf{blockedByReducibleConfiguration} routine from \cite{inoue2026four}.
In \cite{inoue2026four}, this routine selects a distinguished vertex, called the  ``center", which is the only vertex allowed to be the image of a vertex $v$ with $\delta_K(v) > 8$ of a configuration.
In Algorithm \ref{alg:prune}, this routine is used by selecting the center vertex of a cartwheel as the center.
This distinction is important in \cite{inoue2026four}, but unnecessary here.
Hence, we do not specify a center.

Consequently, every vertex of a configuration, including vertices $v$ with $\delta_K(v)>8$, may be mapped to any vertex of the target pseudo-embedding.
Thus, we implement the uncentered variant of \textsf{blockedByReducibleConfiguration} in Algorithm A.7.1 from \cite{inoue2026four}, together with Algorithms A.6.6 and A.7.2.

\begin{algorithm}[H]
    \caption{\allHomImagesWithLink($\widehat K$, $\Ksmaller$)}
    \label{alg:all_hom_images}
    \begin{algorithmic}[1]
        \Require{The outer extension of possibly non-planar multi-boundary configuration $\widehat K$, a set of configurations $\Ksmaller$.}
        \Ensure{A set of multi-boundary islands $\mathcal{L}$. If a homomorphism from $\widehat K$ to $G^*$ exists but no homomorphism from configurations in $\Ksmaller$ to $G^*$ exists, then there exists an island $I$ in $\mathcal{L}$ such that $I$ or its trimmed island appears in $G$.}
        \State $\mathcal{L} \gets \emptyset$
        \If {\Call{blockedByReducibleConfiguration}{$\widehat K, \Ksmaller$} = \True} \Comment{This routine is Algorithm A.7.1 in \cite{inoue2026four}, but we do not choose the center.}
            \State \Return $\mathcal{L}$
        \EndIf
        \If {\Call{isPlanar}{$\widehat K$} = \True and \Call{hasSeparatingCycle}{$\widehat K$} = \False}
            \State $S \gets$ \Call{freeCompletionFromOuterExtension}{$\widehat K$}
            \State $I \gets$ \Call{islandFromFreeCompletion}{$S$}
            \If {the number of vertices of degree 2 in $I$ is at most 3} \Comment{This property is only used here.}
                \State $\mathcal{L} \gets \mathcal{L} \cup \{I\}$
            \EndIf
        \EndIf
        \State \texttt{checked} $\gets$ a map from $D(\widehat K)$ to $\{\True, \False\}$, initialized all \False.
        \ForAll{$e \gets D(\widehat K)$}
            \ForAll{$f \gets D(\widehat K)$}
                \If {$e=f$ or $\texttt{checked}(e) = \True$ or $\texttt{checked}(f) = \True$}
                     \Continue
                \EndIf
                \State $\mathcal{Z}^*_{e,f} \gets$ \Call{FreeHomomorphismAndEnforceSingleDigonIncidence}{$\widehat K$, $\{(e, f)\}$} \Comment{Since $\widehat K$ is a pseudo-embedding, free homomorphic images for pseudo-embeddings are computed.}
                \ForAll{$((Z^*_{e,f}, \delta^{*-}_{e,f}, \delta^{*+}_{e,f}), \phi^*_{e,f}) \gets \mathcal{Z}^*_{e,f}$}
                    \State $\mathcal{S} \gets$ \Call{makeOuterExtension}{$(Z^*_{e,f}, \delta^{*-}_{e,f}, \delta^{*+}_{e,f})$}
                    \ForAll{$(\widehat{S}, \phi) \gets \mathcal{S}$}
                         \State $\mathcal{M} \gets$ \Call{allHomImages}{$\widehat S, \Ksmaller$}
                         \State $\mathcal{L} \gets \mathcal{L} \cup \mathcal{M}$
                    \EndFor
                \EndFor
            \EndFor
            \State $\texttt{checked}(e) \gets \True$
            \State $\texttt{checked}(\reverse(e)) \gets \True$
        \EndFor
        \State \Return $\mathcal{L}$
    \end{algorithmic}
\end{algorithm}

\subsubsection{Check planarity}

\algdef{SE}[DOWHILE]{Do}{doWhile}{\algorithmicdo}[1]{\algorithmicwhile\ #1}

\begin{algorithm}[H]
    \caption{\getWalksWithLink($Z$)}
    \label{alg:get_walks}
    \begin{algorithmic}[1]
        \Require{A pseudo-embedding $Z$.}
        \Ensure{The set of closed walks.}
        \State \texttt{visited} is a map from $D(Z)$ to $\{\True, \False\}$, initialized all \False.
        \State $F \gets \emptyset$
        \ForAll{$e \gets D(Z)$}
            \If {$\texttt{visited}(e)=\True$}
                \Continue
            \EndIf
            \State $\texttt{W} \gets$ an empty array
            \State $e_\texttt{current} \gets e$
            \Do
                \State $\texttt{W.push\_back}(e_\texttt{current})$
                \State $\texttt{visited}(e_\texttt{current}) \gets \True$
                \If {$\successor(e_\texttt{current}) = \nil$}
                    \State $f \gets$ the first dart of $\head(e_\texttt{current})$ (i.e., $\head(f)=\head(e_\texttt{current})$ and $\predecessor(f)=\nil$)
                    \State $e_\texttt{current} \gets \reverse(f)$
                \Else
                    \State $e_\texttt{current} \gets \reverse(\successor(e_\texttt{current}))$
                \EndIf
            \doWhile{$e_\texttt{current} \neq e$}
            \State $F \gets F \cup \{\texttt{W}\}$
        \EndFor
        \State \Return $F$
    \end{algorithmic}
\end{algorithm}

\begin{algorithm}[H]
    \caption{\isPlanarWithLink($Z$)}
    \label{alg:is_planar}
    \begin{algorithmic}[1]
        \Require{A pseudo-embedding $Z$}
        \Ensure{\True or \False}
        \State $F \gets$ \Call{getWalks}{$Z$}
        \State \Return \True if $|V(Z)|-|D(Z)|/2+|F|=2$ otherwise \False
    \end{algorithmic}
\end{algorithm}

\subsubsection{Check separating cycles}

\begin{algorithm}[H]
    \caption{\hasSeparatingCycleWithLink($Z, \delta^-, \delta^+$)}
    \label{alg:has_separating_cycle}
    \begin{algorithmic}[1]
        \Require{A pseudo-embedding with degree-range functions $(Z, \delta^-, \delta^+)$}
        \Ensure{\True if $(Z, \delta^-, \delta^+)$ has a separating cycle contradicting the connectivity of $G^*$, otherwise \False.}
        \ForAll{$C \gets$ \Call{enumCycles}{$(Z, \delta^-, \delta^+), 4$}}
            \State $L_D \gets$ \Call{labelDarts}{$(Z, \delta^-, \delta^+), C$}
            \State $n_\textsf{left}, n_\textsf{right} \gets$ \Call{numSeparatedVertices}{$(Z, \delta^-, \delta^+), C, L_D$}
            \If {$|C| \leq 3$, $n_\textsf{left} > 0$, and $n_\textsf{right} > 0$}
                \State \Return \True
            \ElsIf {$|C| = 4$, $n_\textsf{left} > 2$, and $n_\textsf{right} > 2$}
                \State \Return \True
            \EndIf
        \EndFor
        \State \Return \False
    \end{algorithmic}
\end{algorithm}

\begin{algorithm}
    \caption{\enumCyclesWithLink($(Z, \delta^-, \delta^+), l$)}
    \label{alg:enum_cycles}
    \begin{algorithmic}[1]
    \Require {A pseudo-embedding with degree-range functions $(Z, \delta^-, \delta^+)$ and a maximum length $l$}
    \Ensure {All cycles of length at most $l$}
    \State \texttt{visitedV} is a map from $V(Z)$ to $\{\True, \False\}$, initialized all $\False$.
    \State \texttt{visitedD} is a map from $D(Z)$ to $\{\True, \False\}$, initialized all $\False$.
    \Procedure{DFS}{\texttt{P}}
        \If{$|\texttt{P}| > l$}
            \State \Return
        \EndIf
        \State Let $e$ be the last dart of \texttt{P}
        \For{every dart $e'$ with $\head(e')=\head(e)$}
            \If {$e'=e$}
                \Continue
            \EndIf
            \State $f \gets \reverse(e')$
            \If{$f$ is the first dart of \texttt{P}}
                \State \textbf{yield} the cycle represented by \texttt{P}
                \Continue
            \EndIf
            \If{$\texttt{visitedD}(f)$ or $\texttt{visitedV}(\head(f))$}
                \Continue
            \EndIf
        
            \State \texttt{P.push\_back}$(f)$
            \State $\texttt{visitedV}(\head(f)) \gets \True$
            \State \Call{DFS}{\texttt{P}}
            \State $\texttt{visitedV}(\head(f))\gets \False$
            \State \texttt{P.pop\_back}()
        \EndFor
    \EndProcedure
    
    \ForAll{dart $e$ of $Z$}
        \State $\texttt{P} \gets$ the list containing only $e$.
        \State $\texttt{visitedV}(\head(e))\gets \True$
        \State \Call{DFS}{\texttt{P}}
        \State $\texttt{visitedV}(\head(e))\gets \False$
        \State $\texttt{visitedD}(e) \gets \True$
    \EndFor
    \end{algorithmic}
\end{algorithm}

\begin{algorithm}[H]
    \caption{\labelDartsWithLink($(Z, \delta^-, \delta^+), C$)}
    \label{alg:label_darts}
    \begin{algorithmic}[1]
        \Require{A pseudo-embedding with degree-range functions $(Z, \delta^-, \delta^+)$ with a cycle $C=(d_i)_{i=0}^{l-1}$.}
        \Ensure{The label $L_D$ which represents whether the dart is on the left side or the right side of $C$ for the direction of this cycle.}
        \State $L_D$ is a map from $D(Z)$ to \textsf{L}, \textsf{R}, or $\bot$, initialized all $\bot$.
        \ForAll{dart $d_i$ in $C$}
            \State $L_D(d_i) \gets \textsf{L}$
            \State $L_D(\reverse(d_i)) \gets \textsf{R}$
        \EndFor

        \Procedure{Propagate}{$e, \Delta$} \Comment{$\Delta$ is either \textsf{L} or \textsf{R}.}
            \State $L_D(e) \gets \Delta$
            \ForAll{$f \gets \{\successor(e), \predecessor(e), \reverse(e)\}$}
                \If {$f \neq \nil$ and $L_D(f) = \bot$}
                    \State \Call{Propagate}{$f, \Delta$}
                \EndIf
            \EndFor
            \State \Return
        \EndProcedure

        \ForAll{dart $d_i$ in $C$}
            \If {$\successor(d_i) \neq \nil$ and $\successor(d_i) \neq \reverse(d_{(i+1) \bmod l})$}
                \State \Call{Propagate}{$\successor(d_i), \textsf{L}$}
            \EndIf
            \If {$\successor(\reverse(d_{(i+1) \bmod l})) \neq \nil$ and $\successor(\reverse(d_{(i+1) \bmod l})) \neq d_i$}
                 \State \Call{Propagate}{$\successor(\reverse(d_{(i+1) \bmod l})), \textsf{R}$}
            \EndIf
        \EndFor
        \State \Return $L_D$
    \end{algorithmic}
\end{algorithm}

\begin{algorithm}[H]
    \caption{\numSeparatedVerticesWithLink($(Z, \delta^-, \delta^+), C, L_D$)}
    \label{alg:num_seperated_vertices}
    \begin{algorithmic}[1]
        \Require{A pseudo-embedding with degree-range functions $(Z, \delta^-, \delta^+)$, a cycle $C$, the label $L_D$.}
        \Ensure{If a homomorphism from $(Z, \delta^-, \delta^+)$ to $G^*$ exists, compute a lower bound of the number of separated vertices by the image of the cycle $C$.}
        \State $V_C \gets \emptyset$
        \ForAll{dart $d_i$ in $C$}
            \State $V_C \gets V_C \cup \{\head(d_i)\}$
        \EndFor
        \State $V_\textsf{L}, V_\textsf{R} \gets \emptyset, \emptyset$
        \ForAll{$d \in D(Z)$}
            \If {$\head(d) \in V_C$}
                \Continue
            \EndIf
            \If {$L_D(d) = \textsf{L}$}
                \State $V_\textsf{L} \gets V_\textsf{L} \cup \{\head(d)\}$
            \ElsIf {$L_D(d) = \textsf{R}$}
                \State $V_\textsf{R} \gets V_\textsf{R} \cup \{\head(d)\}$
            \EndIf
        \EndFor
        \State initialize all $n_{\textsf{L,inner}}, n_{\textsf{L,boundary}}, n_{\textsf{R,inner}}, n_{\textsf{R,boundary}}$ by 0.
        \ForAll{$v \in V(Z)$}
            \If {$v \in V_\textsf{L}$}
                \If {$v$ is a boundary vertex}
                    \State $n_{\textsf{L,boundary}} \gets n_{\textsf{L,boundary}} + 1$
                \Else
                    \State $n_{\textsf{L,inner}} \gets n_{\textsf{L,inner}} + 1$
                \EndIf
            \EndIf
            \If {$v \in V_\textsf{R}$}
                \If {$v$ is a boundary vertex}
                    \State $n_{\textsf{R,boundary}} \gets n_{\textsf{R,boundary}} + 1$
                \Else
                    \State $n_{\textsf{R,inner}} \gets n_{\textsf{R,inner}} + 1$
                \EndIf
            \EndIf
        \EndFor
        \State $n_\textsf{L} \gets n_{\textsf{L,inner}} + 1$ if $n_{\textsf{L,boundary}} > 0$, otherwise $n_{\textsf{L,inner}}$.
        \State $n_\textsf{R} \gets n_{\textsf{R,inner}} + 1$ if $n_{\textsf{R,boundary}} > 0$, otherwise $n_{\textsf{R,inner}}$.
        \State \Return $(n_\textsf{L}, n_\textsf{R})$
    \end{algorithmic}
\end{algorithm}

\subsubsection{Make an outer extension}

\begin{algorithm}[H]
    \caption{\makeOuterExtensionWithLink{$(Z^*, \delta^{*-}, \delta^{*+})$}}
    \label{alg:make_outer_extension}
    \begin{algorithmic}[1]
    \Require{A pseudo-embedding with degree-range functions $(Z^*, \delta^{*-}, \delta^{*+})$.}
    \Ensure{A set of pseudo-embeddings with degree-range functions  representing an outer extension of some possibly non-planar multi-boundary configuration.}
    \State $\texttt{Q} \gets$ an empty queue
    \State $\texttt{Q.push}(((Z^*, \delta^{*-}, \delta^{*+}), \id{V(Z^*) \cup D(Z^*)}))$
    \State $\mathcal{S} \gets \emptyset$
    \While{\texttt{Q} is not empty}
        \State $((Z, \delta^{-}, \delta^{+}), \phi) \gets \texttt{Q.pop}()$
        \State $A \gets \Call{findFourDarts}{(Z, \delta^{-}, \delta^{+})}$
        
        \If {$A \neq \texttt{null}$}
            \State $\tilde{\mathcal{Z}} \gets \Call{ensureOuterExtension}{(Z, \delta^{-}, \delta^{+}), A}$
            \ForAll{$((\tilde{Z}, \tilde{\delta}^{-}, \tilde{\delta}^{+}), \tilde{\phi}) \in \tilde{\mathcal{Z}}$}
                \State $\texttt{Q.push}(((\tilde{Z}, \tilde{\delta}^{-}, \tilde{\delta}^{+}), \tilde{\phi} \circ \phi))$
            \EndFor
        \Else
            \State $\mathcal{S} \gets \mathcal{S} \cup \{((Z, \delta^{-}, \delta^{+}), \phi)\}$
        \EndIf
    \EndWhile
    \State \Return $\mathcal{S}$
    \end{algorithmic}
\end{algorithm}

\begin{algorithm}[H]
    \caption{\findFourDartsWithLink($Z, \delta^-, \delta^+$)}
    \label{alg:find_four_darts}
    \begin{algorithmic}[1]
        \Require{A pseudo-embedding with degree-range functions $(Z, \delta^-, \delta^+)$.}
        \Ensure{Four darts $(e_0,e_1,e_2,e_3)$ where $e_{i+1}=\reverse(\successor(e_i))$ for $0 \leq i < 3$, $e_0 \neq e_2$ and $e_0 \neq e_3$, or \texttt{null}}
        \ForAll{darts $e_0$ of $D(Z)$}
            \If{$\successor(e_0)=\nil$}
                \Continue
            \EndIf
            \State $\texttt{found\_boundary\_vertex}\gets \False$
            \For{$i \in \{0, 1, 2\}$}
                \If {($i=1$ or $i=2$) and $\successor(e_i)=\nil$}
                    \State $\texttt{found\_boundary\_vertex} \gets \True$
                    \Break
                \EndIf
                \State $e_{i+1} \gets \reverse(\successor(e_i))$
            \EndFor
            \If{$\texttt{found\_boundary\_vertex} = \False$ and $e_0 \neq e_2$ and $e_0 \neq e_3$}
                \State \Return $(e_0, e_1, e_2, e_3)$
            \EndIf
        \EndFor
        \State \Return \texttt{null}
    \end{algorithmic}
\end{algorithm}

\begin{algorithm}[H]
    \caption{\ensureOuterExtensionWithLink($(Z, \delta^-, \delta^+), (e_0, e_1, e_2, e_3)$)}
    \label{alg:enusre_outer_extension}
    \begin{algorithmic}[1]
        \Require{A pseudo-embedding with degree-range functions $(Z, \delta^-, \delta^+)$ and four darts $(e_0, e_1, e_2, e_3)$}
        \Ensure{A set of pseudo-embeddings with degree-range functions together with mappings}
        \State $\mathcal{Z}_3 \gets$ \Call{FreeHomomorphismAndEnforceSingleDigonIncidence}{$((Z, \delta^-, \delta^+), {(e_0, e_3)})$}
        \State $\mathcal{Z}_2 \gets$ \Call{FreeHomomorphismAndEnforceSingleDigonIncidence}{$((Z, \delta^-, \delta^+), {(e_0, e_2)})$}
        \State \Return $\mathcal{Z}_3 \cup \mathcal{Z}_2$
    \end{algorithmic}
\end{algorithm}

\subsubsection{How to get a multi-boundary island from an outer extension}

\begin{algorithm}[H]
    \caption{\freeCompletionFromOuterExtensionWithLink($\widehat{K}$)}
    \label{alg:free_completion_from_outer_extension}
    \begin{algorithmic}[1]
        \Require{An outer extension $\widehat{K}$ represented as a pseudo-embedding with degree.}
        \Ensure{A free completion of $K$.}
        \State $F \gets \Call{getWalks}{\widehat{K}}$
        \State $Z \gets \widehat{K}$
        \State $\phi \gets \id{V(\widehat{K}) \cup D(\widehat{K})}$
        \For{each walk $W \in F$}
            \ForAll{darts $e\in W$}
                \State $e_0\gets \phi(e)$
                \If{$\predecessor(\reverse(e_0))=\nil$}
                    \State $e_1\gets \reverse(\successor(e_0))$
                    \If{$\successor(e_1)=\nil$}
                        \State \Call{addBoundaryDartsDirectly}{$Z,\reverse(e_1),e_0$}
                        \Continue
                    \EndIf
                    \State $e_2\gets \reverse(\successor(e_1))$
                    \If{$\successor(e_2)=\nil$}
                        \State $\phi'\gets \Call{linkIncidenceListEnds}{Z,\reverse(e_0),e_2}$
                        \State $\phi\gets \phi'\circ \phi$
                        \Continue
                    \EndIf
                    \State \textbf{assert}(\False) \Comment{Unreachable}
                \EndIf
            \EndFor
        \EndFor
        \State \Return $Z$
    \end{algorithmic}
\end{algorithm}

\begin{algorithm}[H]
    \caption{\islandFromFreeCompletionWithLink($S$)}
    \label{alg:island-from-free-completion}
    \begin{algorithmic}[1]
        \Require{A free completion $S$}
        \Ensure{A multi-boundary island constructed from $S$, together with a set of pendant edges, one added at each vertex of degree two.}
        
        \State $F_S \gets \Call{getWalks}{S}$
        \State $m \gets 0$
        \State Initialize $\texttt{darts\_to\_edge}[e] = -1$ for every dart $e$ of $S$
        
        \State $\texttt{ring\_sizes} \gets
            \Call{indexBoundaryEdges}
            {S,F_S,\texttt{darts\_to\_edge},m}$
        
        \State $(\texttt{n\_pendant\_edge}, \texttt{digon\_to\_pendant\_edge}) \gets
            \Call{indexPendantEdges}
            {S,F_S,\texttt{darts\_to\_edge},m}$
        
        \State $\Call{indexOtherEdges}
            {S,F_S,\texttt{darts\_to\_edge},m}$
        
        \State \Return $\Call{constructIsland}
            {S,F_S,\texttt{darts\_to\_edge},\texttt{ring\_sizes},
             \texttt{n\_pendant\_edge}, \texttt{digon\_to\_pendant\_edge}}$
    \end{algorithmic}
\end{algorithm}

\begin{algorithm}[H]
    \caption{\indexBoundaryEdgesWithLink($S,F_S,\texttt{darts\_to\_edge},m$)}
    \label{alg:index-boundary-edges}
    \begin{algorithmic}[1]
        \Require{A free completion $S$, its walks $F_S$, a dart-to-edge index map
        $\texttt{darts\_to\_edge}$, and the next available edge index $m$}
        \Ensure{The sizes of the boundary rings; $\texttt{darts\_to\_edge}$ and $m$ are updated}
        
        \State $\texttt{ring\_sizes} \gets$ an empty list
        
        \For{each walk $W \in F_S$}
            \If{$\successor(W[0]) = \nil$ and $|W| > 1$} \Comment{A walk in boundary of size more than 1.}
                \For{each dart $e \in W$}
                    \State $\texttt{darts\_to\_edge}[e] \gets m$
                    \State $\texttt{darts\_to\_edge}[\reverse(e)] \gets m$
                    \State $m \gets m+1$
                \EndFor
                \State $\texttt{ring\_sizes.push\_back}(|W|)$
            \EndIf
        \EndFor
        
        \State \Return $\texttt{ring\_sizes}$
    \end{algorithmic}
\end{algorithm}

\begin{algorithm}[H]
    \caption{\indexPendantEdgesWithLink($S,F_S,\texttt{darts\_to\_edge},m$)}
    \label{alg:index-pendant-edges}
    \begin{algorithmic}[1]
        \Require A free completion $S$, its walks $F_S$, a dart-to-edge index map
        $\texttt{darts\_to\_edge}$, and the next available edge index $m$
        \Ensure The number of pendant edges, each of which is incident with a vertex of degree 2 and a map from digons to their corresponding
        pendant-edge indices; $\texttt{darts\_to\_edge}$ and $m$ are updated
        
        \State $\texttt{n\_pendant\_edge} \gets 0$
        \State $\texttt{digon\_to\_pendant\_edge} \gets$ an empty map
        
        \For{each walk $W \in F_S$}
            \State $e_0 \gets W[0]$
        
            \If{$\successor(e_0) = \nil$ and $|W| = 1$} \Comment{A walk in boundary of size 1.}
                \State $\texttt{darts\_to\_edge}[e_0] \gets m$
                \State $\texttt{darts\_to\_edge}[\reverse(e_0)] \gets m$
                \State $m \gets m+1$
                \State $\texttt{n\_pendant\_edge} \gets
                \texttt{n\_pendant\_edge}+1$
            \EndIf
        
            \If{$\successor(e_0) \neq \nil$ and $|W| = 2$} \Comment{A walk bounding a digon.}
                \State $\texttt{digon\_to\_pendant\_edge}[W] \gets m$ \Comment{The index of a pendant edge.}
                \State $m \gets m+1$
                \State $\texttt{n\_pendant\_edge} \gets
                \texttt{n\_pendant\_edge}+1$
            \EndIf
        \EndFor
        
        \State \Return
        $(\texttt{n\_pendant\_edge},
        \texttt{digon\_to\_pendant\_edge})$
    \end{algorithmic}
\end{algorithm}

\begin{algorithm}[H]
\caption{\indexOtherEdgesWithLink($S,F_S,\texttt{darts\_to\_edge},m$)}
    \label{alg:index-other-edges}
    \begin{algorithmic}[1]
        \Require{A free completion $S$, its walks $F_S$, a dart-to-edge index map
        $\texttt{darts\_to\_edge}$, and the next available edge index $m$}
        \Ensure{$\texttt{darts\_to\_edge}$ and $m$ are updated}
        \For{each walk $W \in F_S$}
            \If{$\successor(W[0]) \neq \nil$} \Comment{A walk bounding a digon or a triangle.}
                \For{each dart $e \in W$}
                    \If{$\texttt{darts\_to\_edge}[e] = -1$ and
                    $\texttt{darts\_to\_edge}[\reverse(e)] = -1$}
                        \State $\texttt{darts\_to\_edge}[e] \gets m$
                        \State $\texttt{darts\_to\_edge}[\reverse(e)] \gets m$
                        \State $m \gets m+1$
                    \EndIf
                \EndFor
            \EndIf
        \EndFor
    \end{algorithmic}
\end{algorithm}

\begin{algorithm}[H]
    \caption{\constructIslandWithLink($S,F_S,\texttt{darts\_to\_edge}, \texttt{ring\_sizes}, \texttt{n\_pendant\_edge}, \texttt{digon\_to\_pendant\_edge}$)}
    \label{alg:construct-island}
    \begin{algorithmic}[1]
        \Require{A free completion $S$, its walks $F_S$, a dart-to-edge index map
        $\texttt{darts\_to\_edge}$, the boundary-ring sizes $\texttt{ring\_sizes}$,
        the number of pendant edges $\texttt{n\_pendant\_edge}$, and a map
        $\texttt{digon\_to\_pendant\_edge}$}
        \Ensure{A multi-boundary island constructed from $S$, together with a set of pendant edges, one added at each vertex of degree two.}
        \State $\texttt{incident\_edges} \gets$ an empty list
        \For{each walk $W \in F_S$}
            \If{$\successor(W[0]) \neq \nil$}
                \If{$|W|=2$} \Comment{A walk bounding a digon}
                    \State $\texttt{incident\_edges.push\_back}
                    (\texttt{darts\_to\_edge}[W[0]], \texttt{darts\_to\_edge}[W[1]], \texttt{digon\_to\_pendant\_edge}[W])$
                \ElsIf{$|W|=3$} \Comment{A walk bounding a triangle}
                    \State $\texttt{incident\_edges.push\_back}
                    (\texttt{darts\_to\_edge}[W[0]], \texttt{darts\_to\_edge}[W[1]], \texttt{darts\_to\_edge}[W[2]])$
                \EndIf
            \EndIf
        \EndFor
        \State $I \gets$ A multi-boundary island where \texttt{incident\_edges[$i$]} is the indices of edges that is incident to $i$-th vertex of this island, and \texttt{ring\_sizes[$j$]} is the number of edges in $E_R(I)$ that is incident to $j$-th face in $F_R(I)$.
        \State \Return $I$
    \end{algorithmic}
\end{algorithm}

\begin{lem}
\label{comp:lem:reducible}
All multi-boundary islands in $\mathcal{I}$, generated from $\mathcal{K}$, are semi-reducible.
\end{lem}
For the proof, we order configurations in $\mathcal{K}$ by $K_1,\ldots,K_{|\mathcal{K}|}$ and their mirrors $K_1^\textsf{mirror},\ldots,K_{|\mathcal{K}|}^\textsf{mirror}$.
Then, we compute $\mathcal{I} = \cup_{1 \leq i \leq |\mathcal{K}|} \allHomImages(K_i, \{K_j, K_j^\textsf{mirror}\}_{j=1}^{i-1})$ by running Algorithm \ref{alg:all_hom_images}.
We verified that all multi-boundary islands in $\mathcal{I}$ are semi-reducible by computer.
Note that we order configurations in $\mathcal{K}$ by their filenames.

%% file: 4CT.bib
@book {MT,
    AUTHOR = {Mohar, Bojan and Thomassen, Carsten},
     TITLE = {Graphs on surfaces},
    SERIES = {Johns Hopkins Studies in the Mathematical Sciences},
 PUBLISHER = {Johns Hopkins University Press, Baltimore, MD},
      YEAR = {2001},
     PAGES = {xii+291},
   MRCLASS = {05C10 (57M15)},
  MRNUMBER = {1844449},
}

@article {4ct2,
    AUTHOR = {Appel, K. and Haken, W. and Koch, J.},
     TITLE = {Every planar map is four colorable. {II}. {R}educibility},
   JOURNAL = {Illinois J. Math.},
  FJOURNAL = {Illinois Journal of Mathematics},
    VOLUME = {21},
      YEAR = {1977},
    NUMBER = {3},
     PAGES = {491--567},
   MRCLASS = {05C15},
  MRNUMBER = {0543793},
       URL = {http://projecteuclid.org/euclid.ijm/1256049012},
}

@article {4ct1,
    AUTHOR = {Appel, K. and Haken, W.},
     TITLE = {Every planar map is four colorable. {I}. {D}ischarging},
   JOURNAL = {Illinois J. Math.},
  FJOURNAL = {Illinois Journal of Mathematics},
    VOLUME = {21},
      YEAR = {1977},
    NUMBER = {3},
     PAGES = {429--490},
   MRCLASS = {05C15},
  MRNUMBER = {0543792},
       URL = {http://projecteuclid.org/euclid.ijm/1256049011},
}

@article {RSST,
    AUTHOR = {Robertson, Neil and Sanders, Daniel and Seymour, Paul and
              Thomas, Robin},
     TITLE = {The four-colour theorem},
   JOURNAL = {J. Combin. Theory Ser. B},
  FJOURNAL = {Journal of Combinatorial Theory. Series B},
    VOLUME = {70},
      YEAR = {1997},
    NUMBER = {1},
     PAGES = {2--44},
   MRCLASS = {05C15},
  MRNUMBER = {1441258},
       URL = {https://doi.org/10.1006/jctb.1997.1750},
}

@article{HopcroftT74,
  author       = {John E. Hopcroft and
                  Robert Endre Tarjan},
  title        = {Efficient Planarity Testing},
  journal      = {J. {ACM}},
  volume       = {21},
  number       = {4},
  pages        = {549--568},
  year         = {1974},
  url          = {https://doi.org/10.1145/321850.321852},
  doi          = {10.1145/321850.321852},
  bibsource    = {dblp computer science bibliography, https://dblp.org}
}

@article{doublecross,
title = {Three-edge-colouring doublecross cubic graphs},
journal = {Journal of Combinatorial Theory, Series B},
volume = {119},
pages = {66-95},
year = {2016},
issn = {0095-8956},
doi = {https://doi.org/10.1016/j.jctb.2015.12.006},
url = {https://www.sciencedirect.com/science/article/pii/S0095895615001410},
author = {Katherine Edwards and Daniel P. Sanders and Paul Seymour and Robin Thomas}
}

@article{proj2024,
  author = {Inoue, Yuta and Kawarabayashi, Ken-ichi and Miyashita, Atsuyuki and Mohar, Bojan and Sonobe, Tomohiro},
  title = {Three-edge-coloring projective planar cubic graphs: A generalization of the Four Color Theorem},
  journal = {FOCS'24. ArXiv:2405.16586},
  year = {2024}
}

@article{3edgecoloring,
    AUTHOR = {Robertson, Neil and Seymour, Paul and
              Thomas, Robin},
     TITLE = {Tutte's three-edge coloring conjecture},
   JOURNAL = {J. Combin. Theory Ser. B},
  FJOURNAL = {Journal of Combinatorial Theory. Series B},
    VOLUME = {70},
      YEAR = {1997},
    NUMBER = {1},
     PAGES = {166-183},
   MRCLASS = {05C15}
}

@inproceedings {RSST-STOC,
    AUTHOR = {Robertson, Neil and Sanders, Daniel P. and Seymour, Paul and
              Thomas, Robin},
     TITLE = {Efficiently four-coloring planar graphs},
 BOOKTITLE = {Proceedings of the {T}wenty-eighth {A}nnual {ACM} {S}ymposium
              on the {T}heory of {C}omputing ({P}hiladelphia, {PA}, 1996)},
     PAGES = {571--575},
 PUBLISHER = {ACM, New York},
      YEAR = {1996},
      ISBN = {0-89791-785-5},
   MRCLASS = {05C85 (05C15)},
  MRNUMBER = {1427555},
       DOI = {10.1145/237814.238005},
       URL = {https://doi.org/10.1145/237814.238005},
}

@article{tutte,
title = {On the algebraic theory of graph colorings},
journal = {J. Combin. Theory},
volume = {1},
pages = {15--50},
year = {1966},
author = {W. Tutte},
}

@book {Diestel_book,
    AUTHOR = {Diestel, Reinhard},
     TITLE = {Graph theory},
    SERIES = {Graduate Texts in Mathematics},
    VOLUME = {173},
   EDITION = {Fifth},
 PUBLISHER = {Springer, Berlin},
      YEAR = {2018},
     PAGES = {xviii+428},
  MRNUMBER = {3822066},
}

@article{petersen-1898, 
  title={Sur le théorème de {T}ait}, 
  volume={5}, 
  journal={L'Intermédiaire des mathématiciens}, 
  author={Petersen, Julius}, 
  year={1898}, 
  pages={225–227}
}

@article{torus2024,
    title={Three-edge-coloring ({T}ait coloring) cubic graphs on the torus: A proof of  {G}r\"{u}nbaum's conjecture},
  author={Inoue, Yuta and Kawarabayashi, Ken-ichi and Miyashita, Atsuyuki and Mohar, Bojan and Sonobe, Tomohiro},
  journal={SODA'26, Arxiv:2505.07002}, 
  year={2025},
}

@article{inoue2026four,
  title={The Four Color Theorem with Linearly Many Reducible Configurations and Near-Linear Time Coloring},
  author={Inoue, Yuta and Kawarabayashi, Ken-ichi and Miyashita, Atsuyuki and Mohar, Bojan and Thomassen, Carsten and Thorup, Mikkel},
  journal={To appear in FOCS'26. ArXiv preprint arXiv:2603.24880},
  year={2026}
}

@article {Klein1879,
    AUTHOR = {Klein, Felix},
     TITLE = {Ueber die {A}ufl\"osung gewisser {G}leichungen vom siebenten
              und achten {G}rade},
   JOURNAL = {Math. Ann.},
  FJOURNAL = {Mathematische Annalen},
    VOLUME = {15},
      YEAR = {1879},
    NUMBER = {2},
     PAGES = {251--282},
      ISSN = {0025-5831,1432-1807},
   MRCLASS = {99-04},
  MRNUMBER = {1510011},
       DOI = {10.1007/BF01444143},
       URL = {https://doi-org.proxy.lib.sfu.ca/10.1007/BF01444143},
}

@incollection{Tutte1971,
    AUTHOR = {Tutte, William T.},
     TITLE = {What is a map?},
 BOOKTITLE = {New directions in the theory of graphs ({P}roc. {T}hird {A}nn
              {A}rbor {C}onf., {U}niv. {M}ichigan, {A}nn {A}rbor, {M}ich.,
              1971)},
     PAGES = {309--325},
 PUBLISHER = {Academic Press, New York-London},
      YEAR = {1973},
   MRCLASS = {05C10},
  MRNUMBER = {376413},
MRREVIEWER = {Laszlo\ Lovasz},
}

@article{excludedcubic,
  title={Excluded minors in cubic graphs},
  author={Robertson, Neil and Seymour, Paul and Thomas, Robin},
  journal={Journal of Combinatorial Theory, Series B},
  volume={138},
  pages={219--285},
  year={2019},
  publisher={Elsevier}
}

@article{cyclically5,
title = {Cyclically five-connected cubic graphs},
journal = {Journal of Combinatorial Theory, Series B},
volume = {125},
pages = {132-167},
year = {2017},
doi = {https://doi.org/10.1016/j.jctb.2017.03.003},
url = {https://www.sciencedirect.com/science/article/pii/S0095895617300187},
author = {Neil Robertson and P.D. Seymour and Robin Thomas}
}
